%% file: main.tex
\documentclass[11pt]{article}
\usepackage[margin = 1in]{geometry}
\usepackage{xparse}

\usepackage{amsmath, amsthm, amssymb}
\usepackage[mathscr]{euscript}

\DeclareSymbolFont{rsfs}{U}{rsfs}{m}{n}
\DeclareSymbolFontAlphabet{\mathrsfs}{rsfs}

\usepackage{nccmath}
\usepackage{booktabs}
\usepackage{empheq}
\usepackage{pifont}
\usepackage{enumitem}
\usepackage{graphicx} % more modern
\usepackage{wrapfig}
\usepackage{subfigure}
\usepackage{empheq}

\usepackage{caption}

\usepackage{thmtools}
\usepackage{thm-restate}

\usepackage{nicefrac}
\usepackage{booktabs}
\usepackage{multirow}
\usepackage{rotating}
\usepackage{bm}

\usepackage{tikz}
\usetikzlibrary{shapes}
\usetikzlibrary{decorations.markings}
\usetikzlibrary{plotmarks}
\usetikzlibrary{patterns}
\usepackage{ctable}

\usepackage{color, colortbl}
\definecolor{mydarkblue}{rgb}{0,0.08,0.45}
\usepackage[colorlinks=true,
    linkcolor=mydarkblue,
    citecolor=mydarkblue,
    filecolor=mydarkblue,
    urlcolor=mydarkblue,
    pdfview=FitH]{hyperref}
    
\usepackage{empheq}

\definecolor{myteal}{RGB}{27,158,119}
\definecolor{myorange}{RGB}{217,95,2}
\definecolor{myred}{RGB}{231,41,138}
\definecolor{mypurple}{RGB}{152,78,163}
\definecolor{myblue}{RGB}{55,126,184}
\definecolor{mygreen}{RGB}{0,100,0}

\newtheorem{definition}{Definition}[section]

\newtheorem{proposition}{Proposition}[section]
\newtheorem{lemma}{Lemma}[section]
\newtheorem{theorem}{Theorem}[section]
\newtheorem{remark}{Remark}[section]
\newtheorem{corollary}{Corollary}[section]
\newtheorem*{theorem*}{Theorem}

\definecolor{myblue}{HTML}{D2E4FC}
\definecolor{Gray}{gray}{0.92}

\newif\ifshowcomments
\showcommentstrue
\showcommentsfalse

\ifshowcomments
\newcommand{\ep}[1]{\textcolor{olive}{[EP: #1]}}

\else
\newcommand{\ep}[1]{}
\newcommand{\ba}[1]{}

\title{Hitting the High-Dimensional Notes: An ODE for SGD \\
\Large Learning dynamics on GLMs and multi-index models
}

\date{}

\newtheorem{assumption}{Assumption}

\graphicspath{{./figures/}}

\usepackage[sort,numbers]{natbib}
\makeatletter
\newcommand{\vast}{\bBigg@{4}}
\newcommand{\Vast}{\bBigg@{5}}
\makeatother

\usepackage[framemethod=tikz]{mdframed}

\global\mdfdefinestyle{exampledefault}{%
outerlinewidth=0pt,innerlinewidth=0pt,
roundcorner=5pt,backgroundcolor=myblue,
leftmargin=0.2cm,rightmargin=0.2cm
}

\include{macros}

\begin{document}

% \author{
% Elizabeth Collins-Woodfin
%  \and Courtney Paquette\thanks{Google DeepMind}
%  \and Elliot Paquette\thanks{Department of Mathematics and Statistics, McGill University, Montreal, QC}
%  \and Inbar Seroussi \thanks{Tel Aviv University, Department of Mathematical Sciences}
% }

\author{%
     Elizabeth Collins-Woodfin \footnotemark[3]
   \and Courtney Paquette\thanks{Corresponding author: email address: \url{courtney.paquette@mcgill.ca} } %
\thanks{Google DeepMind} \footnotemark[3] % 
   \and Elliot Paquette \thanks{Department of Mathematics and Statistics, McGill University, Montreal, QC; C. Paquette is a Canadian Institute for Advanced Research (CIFAR) AI chair, Quebec AI Institute (MILA) and C. Paquette was supported by a Discovery Grant from the
 Natural Science and Engineering Research Council (NSERC) of Canada, NSERC CREATE grant Interdisciplinary Math and Artificial Intelligence Program (INTER-MATH-AI)", and Fonds de recherche du Québec – Nature et technologies (FRQNT) New University Researcher's Start-Up Program; Research by E. Paquette was supported by a Discovery Grant from the
 Natural Science and Engineering Research Council (NSERC) of Canada.} %
    \and Inbar Seroussi \thanks{Department of Applied Mathematics, School of Mathematical Sciences, Tel Aviv University, Tel Aviv, Israel} %
   }
\date{\today}

% \editor{}

\maketitle
\begin{abstract}
We analyze the dynamics of streaming stochastic gradient descent (SGD) in the high-dimensional limit when applied to generalized linear models and multi-index models (e.g. logistic regression, phase retrieval) with general data-covariance.  In particular, we demonstrate a deterministic equivalent of SGD in the form of a system of ordinary differential equations that describes a wide class of statistics, such as the risk and other measures of sub-optimality.
This equivalence holds with overwhelming probability when the model parameter count grows proportionally to the number of data.  %Our analysis holds for general loss functions as well as non-isotropic covariance, and it illustrates the role of noise in the SGD dynamics.  
This framework allows us to obtain learning rate thresholds for stability of SGD as well as convergence guarantees.  
In addition to the deterministic equivalent, 
we introduce an SDE with a simplified diffusion coefficient (homogenized SGD)
%stochastic version of the limiting process (referred to as homogenized SGD), 
which allows us to analyze the dynamics of general statistics of SGD iterates.  
Finally, we illustrate this theory on some standard examples
and show numerical simulations which give an excellent match to the theory.
\end{abstract}

\section{Introduction} \label{sec:introduction}

\newcommand{\ex}{x}
\renewcommand{\tr}{\operatorname{Tr}}

Optimization theory seeks to design efficient algorithms for finding solutions of optimization problems, which are conventionally formulated as minimization problems 
\[
  \min_X \mathcal{R}(X)
\]
for an objective function or \emph{risk} $\mathcal{R}$.
The design of these algorithms and the measurement of their performance is then done within a class of functions $\{\mathcal{R}\}$, which is typically referred to as the structure of the optimization problem.
Typical examples of this structure are convexity, smoothness, or architectural assumptions on the function $\mathcal{R}$ such as the finite-sum structure or the convex-composite structure.

With the growth of machine learning and large-scale statistics,
an important feature of these objective functions is that they live in an intrinsically high-dimensional space;
if $X$ represents the parameters in a statistical model or neural network,
then the dimensionality itself of $X$ represents a tunable parameter, and this dimensionality
can easily grow into the millions or beyond.
Very frequently, optimization theory designed without consideration of this high-dimensionality will fail
to adequately describe the properties of these objective functions when the dimension is made large.

In this article, we consider a general class of risk minimization problems which can be considered as a composite of high-dimensional linear structure with low dimensional, non-linear structure.
We denote by $\mathcal{A} \cong \mathbb{R}^d$ the \textit{ambient} space;
%or in some contexts the \textit{data} space.
the parameter $d$ will be large and all the content in this paper will suppose that $d \geq d_0$ some large value.
We let $\mathcal{O} \cong \mathbb{R}^{\ell}$ be the \emph{observable} space, which we will consider to be fixed-dimensional, independent of $d$ and which will have dimensions that are accessible to the optimization algorithm.
The full \emph{parameter space} over which we will minimize will be $\mathcal{A} \otimes \mathcal{O} \cong \mathbb{R}^{d} \otimes \mathbb{R}^{\ell} \cong \mathbb{R}^{d\ell}$.
Lastly, we let $\mathcal{T} \cong \mathbb{R}^{\ell^{\star}}$
be the \emph{latent} space % or \emph{target space} 
of channels through which the objective function is influenced but which are hidden from the optimization algorithm.
In some cases, we will need to formally work on the full space, that is
we define $\mathcal{O}^+ \defas \mathcal{O} \oplus \mathcal{T}$
and look at $\mathcal{A} \otimes \mathcal{O}^+ \cong \mathbb{R}^d \otimes ( \mathbb{R}^{\ell} \oplus \mathbb{R}^{\ell^{\star}} ).$  
We shall use $|\mathcal{O}|$ and $|\mathcal{T}|$ to denote the dimensions of these spaces, which will be fixed throughout; all constants may depend on these dimensions and we do not quantify this dependence.

%\paragraph{Key contributions and roadmap for this paper:}
\paragraph{Key contributions:}
\begin{itemize}
\item We formulate a  class of optimization problems \eqref{eq:nlgc} -- a composition of a high-dimensional linear function with a general low-dimensional outer function -- 
where dimensionality enters as an explicit parameter. Consequently, for this class, one can take dimensionality to infinity while preserving non-linearity and other structures in the problem. This class includes standard inference problems such as GLMs.

\item Our main result 
is a comparison of SGD dynamics on \eqref{eq:nlgc} to a solution of deterministic ODEs
(Theorem \ref{thm:learning_curves}), which holds when dimension $d$ grows large (as opposed to the canonical small learning rate approximation). Solving these ODEs gives predictions for the risk curves of SGD with vanishing error as $d \to \infty$.

% This applies to for a class of high-dimensional, structured problems \eqref{eq:nlgc} where dimensionality enters as an explicit parameter, which includes standard statistical inference problems.

\item We further introduce a new SDE \eqref{eq:HSGD} which behaves the same way as SGD, when dimension grows large, even for large learning rate at or above the convergence threshold.  This can be compared to SGD or the deterministic equivalent on a large class of statistics (including most standard measures of suboptimality, Theorem \ref{Thm:SGD_HSGD_convergence}).
\item We analyze the deterministic equivalent to give a precise characterization of \emph{descent} \eqref{eq:DODE}, which is to say that we give a formula for the maximal learning rate that decreases suboptimality in a dimension-independent way.  This naturally leads to easy conditions for convergence, as well as rates of convergence under standard assumptions on the risk.  See Propositions \ref{prop:RSI} and \ref{prop:Courtneyrate_main}.%, which gives step-size thresholds for the convergence of SGD in terms of the average eigenvalue of the covariance matrix (an improvement on previous results that instead rely upon the largest eigenvalue).  %Proposition \ref{prop:Courtneyrate_main} gives the threshold for strongly-convex objectives, and Proposition \ref{prop:RSI} provides it in a more general setting.
\item In Section \ref{sec:examples}, we apply our results to some key examples in learning theory including multivariate linear regression, multi-class logistic regression, phase retrieval, and phase chase -- a new model illustrating implicit bias effects of SGD in a high-dimensional nonconvex setting.
%\item Section \ref{sec:notation} gives some preliminaries and notations, including the tensor calculus that we use.
%\item In Section \ref{sec:approximate_solutions_stability} we show that analyzing the dynamics of SGD, homogenized SGD, and their deterministic limit (especially for non-isotropic covariance) boils down to understanding one key quantity, referred to as $S(W,z)$, which involves resolvents of the covariance matrix.
%\item Section \ref{sec:SGD_homogenized_SGD} provides the technical details for the proofs of our main theorem including the conditioning that we use to apply our result for general loss functions and the concentration bounds used to control random terms.
%\item Section \ref{sec:optimization} provides further details and proofs for our optimization results.
%\item Appendices \ref{sec:Volterra_equation} and \ref{sec:analysis_examples} provide technical analysis of the differential equations and the examples respectively.
\end{itemize}

\paragraph{Tensor notation.}
We briefly summarize here the tensor notation used in this article;
see Section \ref{sec:notation} for full details.
We suppose that all of $\mathcal{A},\mathcal{O},$ and $\mathcal{T}$ are equipped with inner products and hence are finite-dimensional Hilbert spaces.  
This allows us to define the inner product of tensor products of these spaces, by the property that for simple tensors,
\[
  \ip{a_1\otimes o_1,a_2\otimes o_2 }_{\mathcal{A}\otimes \mathcal{O}}
  =
  \ip{a_1,a_2}_{\mathcal{A}}
  \ip{o_1,o_2}_{\mathcal{O}},
\]
and then extending this by bilinearity.
For higher tensors we also use the $\ip{A,B}_{\mathcal{A}}$ operator to denote 
\emph{partial contraction}, where the first $\mathcal{A}$ axis from each of $A$ and $B$ are contracted, and the output tensor has the shape of the uncontracted axes of $A$ followed by the uncontracted axes of $B$.  
Thus for example if $A,B$ are $2$-tensors in $\mathcal{A}^{\otimes 2}$, 
\[
  \ip{A,B}_{\mathcal{A}^{\otimes 2}} = \tr(AB^T)
  \quad\text{and}\quad
  \ip{A,B}_{\mathcal{A}} = A^TB \in \mathcal{A}^{\otimes 2}.
\]
When no space is indicated in the contraction, i.e., $\ip{\cdot, \cdot}$, we mean one does a full contraction across all spaces. Finally we let $\|\cdot\|$ be the Hilbert-space norm (which for the case of $2$-tensors/matrices is the Frobenius norm).  We will use $\|\cdot\|_\sigma$ for the injective norm:
\[
  \|A\|_{\sigma} = \sup_{\substack{\|f_j\|=1 \\ 1 \leq j \leq k}} \langle A, \otimes_1^k f_j\rangle,
\]
which for the case of matrices gives the $\ell^2$-operator norm.
%$c$ and $d$ represent the axes between $A$ and $B$ which are contracted.  When $c$ and $d$ are $\mathcal{A}$, the 
%For a $2$-tensor $X \in \mathcal{A} \otimes \mathcal{O},$ 
%and an element $a \in \mathcal{A}$ we let $\langle a, X\rangle_{\mathcal{A}} \in \mathcal{O}$ 
%denote the contraction of these tensors along the $\mathcal{A}$--axis;
%if we represent $X \in \mathcal{A} \otimes \mathcal{O}$ as a matrix $X=( \sum_{ao} X_{ao} e_a \otimes f_o)$ for orthonormal bases 
%$\{e_a \in \mathcal{A}\}$ 
%and 
%$\{f_o \in \mathcal{O}\}$, 
%this could also be represented as $X^T a$.  

\paragraph{High-dimensional structure.}
We shall consider objective functions $\mathcal{R}$ which are \emph{high-dimensional linear composites}
with outer function $f : \mathcal{O} \oplus \mathcal{T} \oplus \mathcal{T} \to \R$
, data distribution $\mathcal{D}$ on $\mathcal{A}\oplus \mathcal{T}$ 
%and $\ell_2$-regularization strength $\delta \geq 0$ in that
\begin{equation}\label{eq:nlgc}
  \mathcal{R}(X) \defas \Exp_{a,\epsilon} \Psi(X; a,\epsilon),
  \quad\text{for}\quad (a,\epsilon) \sim \mathcal{D},
  \quad\text{where}\quad
  \Psi(X; a,\epsilon) \defas f( \langle X, a \rangle_{\mathcal{A}} \oplus \langle X^\star,a \rangle_{\mathcal{A}} ; \epsilon).% + \delta\|X\|^2/2.
\end{equation}
A large class of natural regression problems fit into this framework, such as logistic regression, some simplified neural network training problems, and others; see Section \ref{sec:examples} for concrete examples.
As applied to statistical settings, $\mathcal{R}$ will often represent the expected risk and so we refer to it as the risk.
%We will also allow, in the definition of $\Psi$ for $f$ itself to be random, independent of $a$, to accomodate some models of label noise; in this case, we shall assume all statements and estimates hold almost surely with deterministic constants in the label noise.
Finally, we shall also allow for $\ell^2$--regularized objective functions with regularization strength $\delta > 0$ in defining 
\begin{equation}\label{eq:Rdelta}
 \mathcal{R}_\delta(X) \defas \mathcal{R}(X) + \delta \|X\|^2/2 \quad \text{and} \quad \Psi_{\delta}(X; a, \epsilon) \defas f( \langle X, a \rangle_{\mathcal{A}} \oplus \langle X^\star,a \rangle_{\mathcal{A}} ; \epsilon) + \delta \|X\|^2/2.
\end{equation}

Many idealized machine learning problems fit the high-dimensional linear composite framework \eqref{eq:nlgc}.
The problem class is principally engineered to describe generalized linear models (GLMs) and multi-index models in a student-teacher framework.
We would take for simplicity $\mathcal{T} = \mathcal{O}.$
Then we consider 
a loss function $\ell : \R^m \times \R^m \to \R$,
and a non-linearity or \emph{link function} $g : \mathcal{O} \to \R^m$.
We further allow a source of noise $\epsilon\in\mathcal{T}$
which one could assume for simplicity perturbs the argument of $g$ and hence gives
\begin{equation}\label{eq:student-teacher}
  \Psi(X; a,\epsilon) = \ell( 
  g(\langle X, a \rangle_{\mathcal{A}}),
  g(\langle X^\star, a \rangle_{\mathcal{A}}+\eta\epsilon)),
\end{equation}
with noise level $\eta>0$. We give a more substantial discussion of examples in Section \ref{sec:examples} and provide connections to existing work.

For all the analyses we do of this class, 
we shall impose further restrictions on $(f,\mathcal{D}).$
However, as we shall take gradients of $f$, 
we shall always require, at a minimum:
\begin{assumption}[Pseudo-Lipschitz $f$] 
  \label{ass:pseudo_lipschitz}
  The outer function $f$ is $\alpha$-pseudo-Lipschitz with constant $L(f)$,
  in all its variables.
%  if $x_i,x_i^\star$ are variables
%  \[
%    f(x_1 ; x_1^\star
%  \]
%   in $r \defas \ip{a, W}_{\mathcal{A}} = \ip{a, X \oplus X^{\star}}_{\mathcal{A}}$, 
  That is, for all $r, \hat{r} \in \mathcal{O}^+$ and all $\epsilon \in \mathcal{T}$,
  \begin{equation}
    \begin{gathered}
      |f(r; \epsilon) - f(\hat{r}; {\epsilon})| \le L(f) \|r-\hat{r}\| 
      %+ \|\epsilon - \hat\epsilon\|
      ( 1 + \|r\|^{\alpha} + \|\hat{r}\|^{\alpha} + \|\epsilon\|^{\alpha}).
    \end{gathered}
  \end{equation}
%  Here we have represented in compact notation $r = x \oplus x^\star$.
%  Moreover the regularization function $p \, : \, \mathcal{A} \otimes \mathcal{O} \to \mathbb{R}$ is also $\alpha$-pseudo-Lipschitz, that is, there exists constants $\alpha > 0$ and $L(p) > 0$ such that for any $X, X' \in \mathcal{A} \otimes \mathcal{O}$, 
%  \[
%    \|(\Dif p)(X)\| \le L(p) ( 1 + \|W\|)^{\alpha}.
%  \]
\end{assumption}
\noindent 
For the probabilistic analysis, it is important to express the dependence of $f$ on all its inputs.
For the optimization, in contrast, we would like to view $f$ as a function of $\mathcal{O}$ but where the $\mathcal{T}$-dependence enters as a hidden parameter.
We shall refer to the $\mathcal{O}$--valued input variable as $\ex$, the $\mathcal{O}^+$--valued input as $r$ and the $\mathcal{A} \otimes \mathcal{O}$--valued variables as $X$ (which for example appears as an input to $\Psi$).

%
%\begin{assumption}[Loss function is $\alpha$-pseudo-Lipschitz] \label{assumption:loss_pseudo_Lip}
%For almost every realization of the label noise $\varepsilon$, the loss function $f$ is $\alpha$-pseudo-Lipschitz with constant $L(f)$ in $r \defas \ip{a, W}_{\mathcal{A}} = \ip{a, X \oplus X^{\star}}_{\mathcal{A}}$, that is, for all $r, \hat{r} \in \mathcal{O}^+$,
%\begin{equation}
%\begin{gathered}
%|f(r) - f(\hat{r})| \le L(f) \|r-\hat{r}\| ( 1 + \|r\|^{\alpha} + \|\hat{r}\|^{\alpha} ).
%% \\
%% \text{and} \quad \|\nabla f(r) - \nabla f(\hat{r}) \| \le L(f) \|r-\hat{r}\| ( 1 + \|r\| + \|\hat{r}\| )^{\alpha}.
%\end{gathered}
%\end{equation}
%Moreover the regularization function $p \, : \, \mathcal{A} \otimes \mathcal{O} \to \mathbb{R}$ is also $\alpha$-pseudo-Lipschitz, that is, there exists constants $\alpha > 0$ and $L(p) > 0$ such that for any $X, X' \in \mathcal{A} \otimes \mathcal{O}$, 
%\[
%\|(\Dif p)(X)\| \le L(p) ( 1 + \|W\|)^{\alpha}.
%\]
%\end{assumption}

\paragraph{Streaming Stochastic Gradient Descent (SGD).} 

For the problem class \eqref{eq:nlgc} satisfying Assumption \ref{ass:pseudo_lipschitz}, we consider \emph{streaming} SGD (also known as \emph{online} SGD, \emph{one-pass} SGD, or SGD with \emph{sample splitting}).
So we suppose that we are provided with a sequence of independent samples $\left\{ (a_k,y_k) \right\}_1^\infty$ drawn from the distribution $\mathcal{D}$, where $y_k$ is the target, which is a function of $\epsilon_k$ and $\langle X^\star,a_k \rangle_{\mathcal{A}}$.  Therefore, what determines the distribution of the data is only the input feature and the noise, i.e. the pair $(a,\epsilon$). 
Having specified an initial state $X_0 \in \mathcal{A} \otimes \mathcal{O}$, 
and a sequence of step-sizes $\gamma_k/d$ (which may be adapted to $\{a_j : j \leq k\}$), 
we define a sequence of iterates $\{ X_k \}$ which obeys the recurrence,
\begin{equation} \label{eq:SGD_1}
  X_{k+1} = X_k - \frac{\gamma_k}{d} ( \nabla_X \Psi(X_k; a_{k+1},\epsilon_{k+1}) + \delta X_k), 
  %= X_k - \frac{\gamma_k}{d} \big ( a_{k+1} \otimes \nabla_{\ex} f( \ip{X_k,a_{k+1}}_{\mathcal{A}}) + \delta X_k  \big ),
\end{equation}
where $\nabla_X$ is the usual gradient operator with respect to the $X$ variable.
%and $\nabla_x$ are the usual gradient operators with respect to the variables in their subscript (see Lemma~\ref{lem:derivative_loss} for the computation of $\nabla_X \Psi$).
%We note that in taking the gradient of both $\Psi$ and $f$, the derivative is computed with respect to its primary variable (so $\nabla \Psi \in \mathcal{A}\otimes \mathcal{O}$ and $\nabla f \in \mathcal{O}$) and not its auxiliary parameters.
%We remark that $\nabla f$ is taken with respect to $\tilde{r} = \ip{X, a}_{\mathcal{A}}$ and not with $r = \ip{W,a}$ (as described in Lemma~\ref{lem:derivative_loss}). We will use the notation $\nabla f(r) \defas \nabla f(\tilde{r})$. Moreover, we need a boundedness assumption on the initialization $X_0 \in \mathcal{A} \otimes \mathcal{O}$.

We shall work in a formulation where the norms of the iterates $\{X_k\}$ remain bounded, independent of dimension.
Within the class of high-dimensional linear composites, we note that the contractions $\langle X, a\rangle_{\mathcal{A}}$ should not carry dimension dependence, as otherwise the outer function $f$ (which can very well be non-linear) degenerates to its behavior at infinity.  Hence, we pose the following initialization assumption:
\begin{assumption}[Parameter scaling]
  \label{assumption:scaling} 
  The initialization point, 
  $X_0 \in \mathcal{A} \otimes \mathcal{O}$
  and the hidden parameters $X^\star \in \mathcal{A} \otimes \mathcal{T}$ 
  are bounded independent of $d$, i.e., $\max\{\|X^{\star}\|,\|X_0\|\} \le C$ for some $C > 0$ independent of $d$.
\end{assumption}
\noindent This must be matched by an appropriate assumption on the data distribution $\mathcal{D}$.  We will consider a generic centered Gaussian distribution $\mathcal{D}.$
\begin{assumption}[Data] \label{ass:data_normal}
  We assume that samples $(a,\epsilon) \sim \mathcal{D}$ are normally distributed $N(0,K \oplus \Id_\mathcal{T})$ (and so $a$ and $\epsilon$ are independent),
  with covariance $K \in \mathcal{A}^{\otimes 2}$ which is 
  bounded in operator norm independent of $d$, i.e.\ $\|K\|_{\sigma} \leq \bar{K}$ for $\bar{K}$.
  Hence in particular $\epsilon$ is independent of $a$.
%  We shall also suppose that $K \succ 0$, which will be 
  %$\|K\|_{\sigma}$ is bounded independent of $d$.
\end{assumption}
\noindent Generalizing this is an interesting direction of research.  There is a small class of nice data distributions -- at the very least those which satisfy Lipschitz concentration -- for which the proof strategy in this paper should hold.
It would be interesting to generalize this in the direction of finitely supported distributions, which would allow one to consider multi-pass SGD methods.

%Throughout this section, we use the notation established in Section~\ref{sec:notation}, particularly, 
%\begin{equation*}
%    \begin{aligned}
%        W_k \defas X_k \oplus X^{\star} \in \mathcal{A} \otimes \mathcal{O}^+, \quad r_k \defas \ip{W_k, a_{k+1}}_{\mathcal{A}} \in \mathcal{O}^+, \quad \text{and} \quad  \nabla f(r_k) \defas \nabla f(\ip{X_k,a_{k+1}}_{\mathcal{A}}).
%    \end{aligned}
%\end{equation*}
%Here $\mathcal{O}^+ = \mathcal{O} \oplus \mathcal{T}$, that is, it is the space of both the targets and the observed space. Using this notation, we have that the SGD update \eqref{eq:SGD_1} simplifies as follows 
%\begin{equation} \label{eq:SGD}
%    X_{k+1} = X_k - 
%    \frac{\gamma_k}{d} \big ( a_{k+1} \otimes \nabla f(r_k) + \delta X_k \big ) , \quad k = 0, 1, 2, \hdots
%\end{equation}

The learning rate $\gamma_k/d$ in \eqref{eq:SGD_1} is scaled in a way that the SGD behaves well across 
different dimensions; without the factor of $d$, the algorithm would degenerate to pure noise or to gradient flow as dimension increases.  However, the $\gamma_k$ can still be sufficiently large 
to capture the stability threshold of the algorithm. 
\begin{assumption}\label{assum:learning_rate}
  There is a $\bar{\gamma} < \infty$ and 
  a deterministic scalar function $\gamma : [0,\infty) \to [0,\infty)$ which is bounded by
 $\bar{\gamma} < \infty$ so that $\gamma_k = \gamma(k/d).$
%    $\gamma \, : \, \mathcal{A} \otimes \mathcal{O} \to \mathbb{R}$ 
%    which may depend on the previous history of the process $B_t = \ip{X_t \otimes X_t, K }_{\mathcal{A}^{\otimes 2}}$, such that
%$\gamma_t\defas\gamma( \mathrsfs{B}_t)\le \bar{\gamma}$
%where $\mathrsfs{B}_t \defas \sigma ( \{B_i\}_{i=0}^t)$.  The constant $\bar{\gamma}$ is positive and finite. 
\end{assumption}
\noindent We are principally motivated by the constant step-size case,
but in a sufficiently non-uniform geometry, 
it would make more sense to consider adaptive (and hence random)
step-size algorithms such as 
Adagrad norm \cite{ward2020adagrad}.
%\begin{remark}
%  With a small modification, this assumption can be relaxed to adapted step sizes 
%  which satisfy that
%  there is a deterministic function $\gamma(\cdot)$
%  and an $\epsilon > 0$ so that
%  for all $T > 0$
%  \[
%    \frac{d^{\epsilon}}{d}\sum_{k=1}^{Td}|\gamma_k - \gamma(k/d)| \Prto[d] 0.
%  \]
%  The would allow us to include algorithms such as Adagrad norm \cite{ward2020adagrad}, In particular, the learning rate is normalized by the history of the gradient in the following way. $\gamma_k = \gamma /(b_k+\e)$, with $b^2_{k+1} = b^2_{k} +\tfrac1d \|\nabla \Psi(X_k;a_{k+1})\|^2$, with some constants $b_0, \gamma>0$. The constant $\e>0$   ensures that the step size is bounded.   
%\end{remark}

\begin{figure}[t]
    \centering
\includegraphics[scale =0.3]{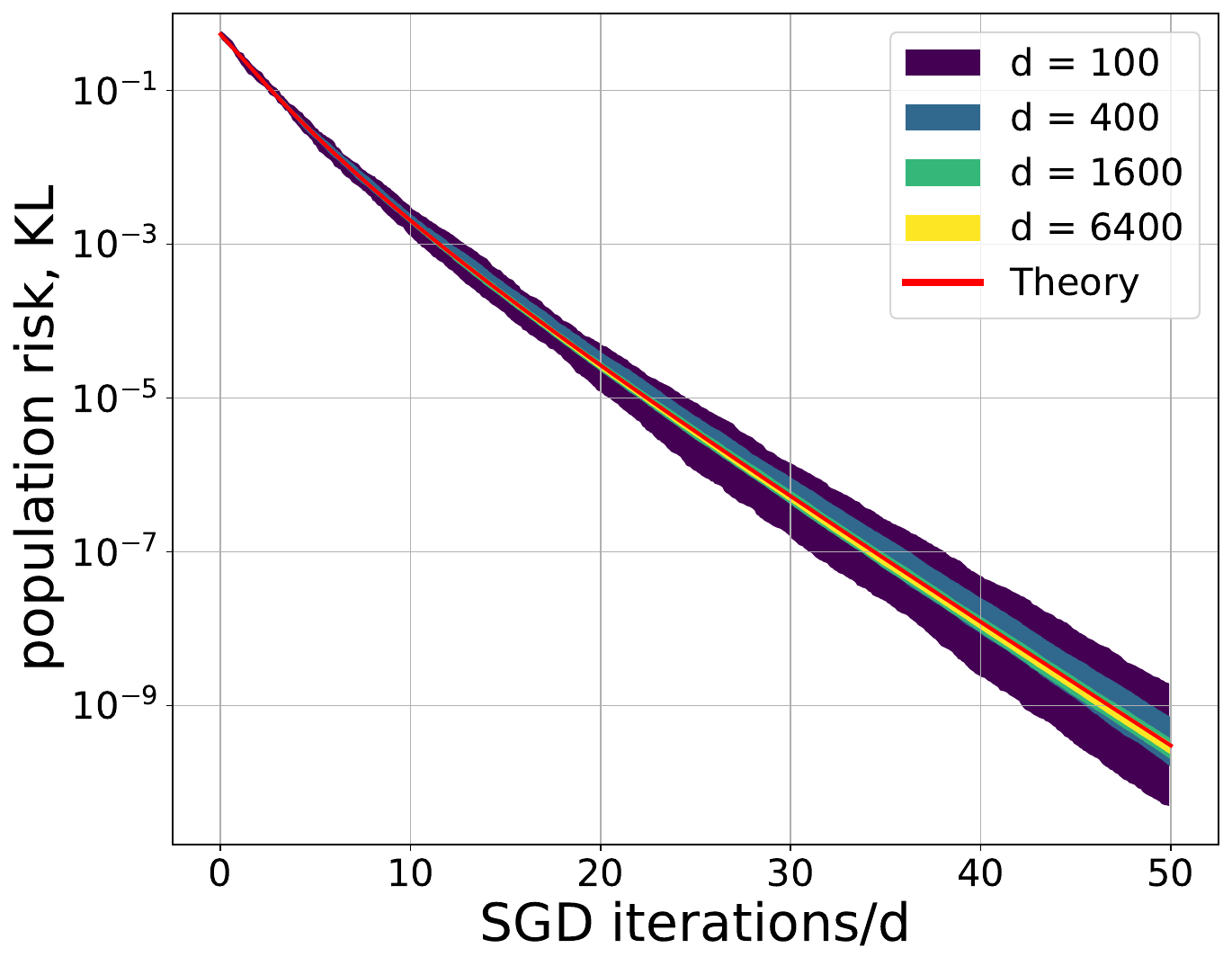}
 \caption{\textbf{Concentration of KL divergence (logistic regression) for SGD} on a (noiseless) binary logistic regression problem (Section~\ref{sec:logistic_regression}) where $X^{\star} \sim 1/\sqrt{d} \cdot N(0,I_d)$ is the ground truth signal and label noise $\epsilon = 0$, SGD was initialized at $X_0 = \tfrac{1.3}{\sqrt{d}} (1, 1, \hdots, 1)$, covariance matrix $K$ has spectrum generated from the Marchenko-Pastur distribution \cite{marvcenko1967distribution} with parameter $4$; an $80\%$ confidence interval (shaded region) over $10$ runs for each $d$, a constant learning rate for SGD was applied, $\gamma = 1.0$. The KL divergence becomes non-random in the large limit and all runs of SGD converge to a deterministic function $\phi$ (red) solving a system of ODEs (Theorem~\ref{Thm:SGD_HSGD_convergence}).
     }
\label{fig:logistic_concentration}

\end{figure}

\paragraph{High-dimensional deterministic equivalent.}
Our first result gives a 
deterministic description of the risk evolution
under streaming SGD (see, e.g., Figure~\ref{fig:logistic_concentration} for logistic regression). 
By assumption, $\mathcal{R}(X)$ involves an expectation over the correlated Gaussians
$\langle X, a \rangle$
and
$\langle X^\star, a \rangle$.
It follows that if we set $W \defas X \oplus X^{\star}$ (which as a matrix may be considered as the block matrix $(X, X^\star)$),
we may represent this expectation $\mathcal{R}(X) \defas h( W^T K W ),$
for some function $h : (\mathcal{O}^+)^{\otimes 2} \to \R$.
We note that it will be convenient to represent $W^T K W$ as the tensor contraction $\langle  W^{\otimes 2}, K\rangle_{\mathcal{A}^{\otimes 2}}$ (see Section \ref{sec:notation} for details). 
Now we need to connect the gradients of the risk to the gradient estimators in SGD \eqref{eq:SGD_1}.
Hence we assume the following:
% Under sufficient smoothness assumptions, this is no issue, but this precludes natural examples.  
\begin{assumption}[Risk representation]\label{assumption:unbiased}
  There is an open set $\mathcal{U} \subseteq (\mathcal{O}^+)^{\otimes 2}$
  such that $\langle (X_0 \oplus X^\star)^{\otimes 2}, K\rangle \in \mathcal{U}$
  and so that provided $\langle W^{\otimes 2} , K\rangle \in \mathcal{U}$
  the map
  $X \mapsto \mathcal{R}(X) \defas h(\langle W^{\otimes 2}, K \rangle)$ is differentiable and satisfies
  \[
    \nabla_X \mathcal{R}(X) = \Exp_{a,\epsilon} \nabla_X \Psi(X ; a,\epsilon).
  \]
Furthermore $h$ is continuously differentiable on $\mathcal{U}$ and its derivative $\nabla h$ is $\alpha$-pseudo-Lipschitz, i.e.
  there is a constant $L(h)>0$, so that for all $B, \hat{B} \in \mathcal{U}$, 
  \begin{equation}
    \begin{gathered}
      \|\nabla h(B) - \nabla h(\hat{B})\| \leq L(h) \|B-\hat{B}\| ( 1 + \|B\|^{\alpha} + \|\hat{B}\|^{\alpha} ).
    \end{gathered}
  \end{equation}
\end{assumption}
\noindent We emphasize that this commutation of expectation and gradient
holds trivially on $\mathcal{U}=(\mathcal{O}^+)^{\otimes 2}$ once $\Psi$ is continuously differentiable (in addition to Assumption \ref{ass:pseudo_lipschitz}).  See Section \ref{sec:examples} for some examples where the $\mathcal{U}$ is needed.
%To give a describption of the state evolution, we 
%Throughout this section, we use the notation established in Section~\ref{sec:notation}, particularly, 
%\begin{equation*}
%    \begin{aligned}
%        W_k \defas X_k \oplus X^{\star} \in \mathcal{A} \otimes \mathcal{O}^+, \quad r_k \defas \ip{W_k, a_{k+1}}_{\mathcal{A}} \in \mathcal{O}^+, \quad \text{and} \quad  \nabla f(r_k) \defas \nabla f(\ip{X_k,a_{k+1}}_{\mathcal{A}}).
%    \end{aligned}
%\end{equation*}
%Here $\mathcal{O}^+ = \mathcal{O} \oplus \mathcal{T}$, that is, it is the space of both the targets and the observed space. Using this notation, we have that the SGD update \eqref{eq:SGD_1} simplifies as follows 
%\begin{equation} \label{eq:SGD}
%    X_{k+1} = X_k - 
%    \frac{\gamma_k}{d} \big ( a_{k+1} \otimes \nabla f(r_k) + \delta X_k \big ) , \quad k = 0, 1, 2, \hdots
%\end{equation}

The final assumption we require is the well-behavior of the Fisher information matrix of the gradients of the outer function
on the same convex set.
\begin{assumption}[$\alpha$-pseudo-Lipschitz of the Fisher matrix] 
  \label{assumption:E_loss_pseudo_Lip}
Define  $I(B) \defas \EE_{a,\epsilon} [ \nabla_\ex f(r;\epsilon)^{\otimes 2}]$, 
where $I:\, (\mathcal{O}^{+})^{\otimes 2} \to \mathcal{O}^{\otimes 2}$ 
where $r = \ip{W,a}_{\mathcal{A}}$, $x = \ip{X,a}_{\mathcal{A}}$,
and $B = \langle W^{\otimes 2}, K \rangle$.
The function $I$ is $\alpha$-pseudo-Lipschitz with constant $L(I)>0$, 
that is, for all $B, \hat{B} \in \mathcal{U}$, 
%with $W,\hat{W}\in \mathcal{O}^+$
% \begin{equation}
% \begin{gathered}
% |\hat{f}(B) - \hat{f}(\hat{B})| \le L(\hat{f}) \|B-\hat{B}\| ( 1 + \|B\|^{\alpha} + \|\hat{B}\|^{\alpha} ).
% % \\
% % \text{and} \quad \|\nabla f(r) - \nabla f(\hat{r}) \| \le L(f) \|r-\hat{r}\| ( 1 + \|r\| + \|\hat{r}\| )^{\alpha}.
% \end{gathered}
% \end{equation}
\begin{equation}
\begin{gathered}
\|I(B) - I(\hat{B})\| \leq L(I) \|B-\hat{B}\| ( 1 + \|B\|^{\alpha} + \|\hat{B}\|^{\alpha} ), 
\end{gathered}
\end{equation}
% and for $H(B)\defas\EE_{(a,\varepsilon)}\big [\nabla_r^2 f(r)\big ]  $, where $H :\, (\mathcal{O}^{+})^{\otimes 2} \to \mathcal{O}^{+}\otimes \mathcal{O}$, 
% \begin{equation}
% \begin{gathered}
% \|H(B) - H(\hat{B})\| \leq L(H) \|B-\hat{B}\| ( 1 + \|B\|^{\alpha} + \|\hat{B}\|^{\alpha} ). 
% \end{gathered}
% \end{equation}
\end{assumption}

The functions $h$ and $I$ allow us to construct closed, deterministic dynamics that describe
the high-dimensional limit of stochastic gradient descent.
To condense the notation, we shall use 
\begin{equation*}
    \begin{aligned}
        W_k \defas X_k \oplus X^{\star} \in \mathcal{A} \otimes \mathcal{O}^+, \quad r_k \defas \ip{W_k, a_{k+1}}_{\mathcal{A}} \in \mathcal{O}^+
	, \quad \text{and} \quad 
	B(W_k) \defas 
	\langle W_k^{\otimes 2}, K \rangle.
    \end{aligned}
\end{equation*}
%Here $\mathcal{O}^+ = \mathcal{O} \oplus \mathcal{T}$, that is, it is the space of both the targets and the observed space.
Using this notation, we have that the SGD update \eqref{eq:SGD_1} simplifies as follows, 
\begin{equation} \label{eq:SGD}
    X_{k+1} = X_k - 
    \frac{\gamma_k}{d} \big ( a_{k+1} \otimes \nabla_\ex f(r_k;\epsilon_{k+1}) + \delta X_k \big ) 
    , \quad k = 0, 1, 2, \hdots
\end{equation}
where $\nabla_x$ gradient operators with respect to the $x = \ip{X,a}$ variable which is part of the vector $r$ (see Lemma~\ref{lem:derivative_loss} for the computation of $\nabla_X \Psi$).

To describe the limiting dynamics,
we define a coupled family of ordinary differential equations.
These coupled differential equations need to be sufficiently rich to describe the covariance matrix 
that enters into $h$ and $I$, and in particular, we give a high-dimensional limit of the covariance matrix 
\begin{equation}
  B(W_k) \defas 
   \begin{bmatrix}
    B_{11}(W_k) & B_{12}(W_k) \\
    B_{12}^T(W_k) & B_{22}(W_k)
  \end{bmatrix} 
  % \begin{bmatrix}
  %   Q_k & R_k \\
  %   R_k^T & T_k
  % \end{bmatrix} 
  \defas
  \begin{bmatrix}
    \langle X_k \otimes X_k, K \rangle_{\mathcal{A}^{\otimes 2}} & \langle X_k \otimes X^\star, K \rangle_{\mathcal{A}^{\otimes 2}}
    \\
    \langle X^\star \otimes X_k, K \rangle_{\mathcal{A}^{\otimes 2}}
    &\langle X^\star \otimes X^\star, K \rangle_{\mathcal{A}^{\otimes 2}}
  \end{bmatrix} 
  ,
  \quad k = 0, 1, 2, \hdots
  \label{eq:Bk}
\end{equation}
where the block structure corresponds to the $\mathcal{O}$ and $\mathcal{T}$ spaces, respectively.

The corresponding limit variables, 
which evolve continuously in time, 
will be defined by an average over a $d$-dimensional family
of limit variables.
We let $((\lambda_i,\omega_i): 1 \leq i \leq d)$ be the eigenvalues and orthonormal eigenvectors of $K$.  Then we introduce the following ODEs on positive semidefinite matrices:
\begin{equation}   \label{eq:Bi}
  \mathrsfs{B}(t) \defas 
    \begin{bmatrix}
    \mathrsfs{B}_{11}(t) & \mathrsfs{B}_{12}(t) \\
    \mathrsfs{B}_{12}^T(t) & \mathrsfs{B}_{22}(t)
  \end{bmatrix},
  % \begin{bmatrix}
  %   \mathrsfs{Q}_t & \mathrsfs{R}_t \\
  %   \mathrsfs{R}_t^T & \mathrsfs{T}_t
  % \end{bmatrix},
  \quad\text{and}\quad
  \mathrsfs{B}_{i}(t) \defas 
  \begin{bmatrix}
      \mathrsfs{B}_{11,i}(t) & \mathrsfs{B}_{12,i}(t) \\
    \mathrsfs{B}_{12,i}^T(t) & \mathrsfs{B}_{22,i}(t)
    % \mathrsfs{Q}_{t,i} & \mathrsfs{R}_{t,i} \\
    % \mathrsfs{R}_{t,i}^T & \mathrsfs{T}_{t,i}
  \end{bmatrix},
  \quad t \geq 0, \quad i \in \{1,2,\dots,d\}.
\end{equation}
These are then related by averaging over $i$. We also introduce at this time a secondary average:
%To relate these, we let $((\lambda_i,\omega_i): 1 \leq i \leq d)$ be the eigenvalues and orthonormal eigenvectors of $K$, and then we require
\begin{equation}
  \label{eq:Bkavg}
  \mathrsfs{B}(t) 
  = \frac{1}{d} \sum_{i=1}^d \lambda_i \mathrsfs{B}_{i}(t)
  \quad\text{and}\quad
  \mathrsfs{N}(t)
  \defas \frac{1}{d} \sum_{i=1}^d \tr(\mathrsfs{B}_{i}(t)).
\end{equation}

Now we suppose that $h$ is defined symmetrically, so that $h( \sum x_i \otimes y_i) = h( \sum y_i \otimes x_i)$ for all $x_i,y_i \in \mathcal{O}^+$ (or as matrices $h(P) = h(P^T)$ for all $P \in (\mathcal{O}^+)^{\otimes 2}$.  Then we define
\[
  H_t \defas
  \nabla h( \mathrsfs{B}(t) ) = 
  \begin{bmatrix}
    H_{1,t} & H_{2,t} \\
    H_{2,t}^T & H_{3,t}
  \end{bmatrix}
  \quad\text{and}\quad
  I_t \defas I( \mathrsfs{B}(t) ).
\]
%where $Dh$ is the differential of $h$, which is therefore a mapping from $(\mathcal{O}^+)^{\otimes 2}$ to symmetric tensors.
Finally, we give a family of coupled ODEs (c.f.\ \cite{Yoshida} where this is introduce for a class of problems with squared loss)
\begin{equation} 
  \begin{aligned}
  &\dif \mathrsfs{B}_{11,i}(t) = -2\lambda_i\gamma_t( \mathrsfs{B}_{11, i}(t) H_{1,t} + H_{1,t}  \mathrsfs{B}_{11,i}(t)+\mathrsfs{B}_{12,i}(t) H_{2,t}) - 2\delta\gamma_t \mathrsfs{B}_{11,i}(t) + \lambda_i\gamma_t^2I_t, \\
  &\dif \mathrsfs{B}_{12,i}(t) = -2\lambda_i\gamma_t( H_{1,t} \mathrsfs{B}_{12,i}(t) + H_{2,t}^T\mathrsfs{B}_{22,i}(t)) - 2\delta\gamma_t \mathrsfs{B}_{12,i}(t), \\
\end{aligned}
  \label{eq:coupledODElimit}
\end{equation}
%where the functions $\mathcal{T}(\lambda)$, $H_{i,t}$, $I_t$ and the initial conditions remain to be defined.  We let $\mu_K$ denote the empirical spectral measure of $K$, so that for any continuous function $f : [0, \bar K] \to \R$ 
%\[
%  \int_\R f(x) \mu_K(dx) \defas \frac{1}{d} \sum_{i=1}^d f(\lambda_i),
%\]
%where the $(\lambda_i,\omega_i)$ are the eigenvalues and orthonormal eigenvectors of $K$.
%Then we set 
%\[
%  \mathcal{Q}_t = \int \mathcal{Q}_t(\lambda)\mu_K(d\lambda)
%  \quad\text{and}\quad 
%  \mathcal{V}_t = \int \mathcal{V}_t(\lambda)\mu_K(d\lambda).
%\]
%In fact, we shall only needs the solutions of the ODEs in \eqref{eq:coupledODElimit} for $\lambda$ which are eigenvalues of $K$.
with the initialization of $\mathrsfs{B}_{11,i},\mathrsfs{B}_{12,i},\mathrsfs{B}_{22,i}$ given by
\[
  \begin{bmatrix} 
    \mathrsfs{B}_{11,i}(0) & \mathrsfs{B}_{12,i}(0) \\
    \mathrsfs{B}_{12,i}^T(0) & \mathrsfs{B}_{22,i}(0)
  \end{bmatrix}
  = d \cdot\langle W_0^{\otimes 2}, \omega_i^{\otimes 2} \rangle
  =
  d \cdot
  \begin{bmatrix} 
    \langle X_0 \otimes X_0, \omega_i^{\otimes 2} \rangle
 & \langle X_0 \otimes X^{\star}, \omega_i^{\otimes 2} \rangle
 \\
    \langle  X^\star \otimes X_0, \omega_i^{\otimes 2} \rangle
 & \langle  X^\star \otimes X^{\star}, \omega_i^{\otimes 2} \rangle
  \end{bmatrix}.
%  = \langle W_0^{\otimes 2}, \omega_i^{\otimes 2} \rangle
\]
We shall also show in Section \ref{susec:SGD_optimality} how to analyze this system with general covariance to gain some optimization insights about SGD on GLMs and multi-index models.

The matrix $\mathrsfs{B}_{22,i}(t) = \mathrsfs{B}_{22,i}(0)$ is constant.
Note that \eqref{eq:coupledODElimit} is a coupled ($d$-dependent but finite) system of differential equations with locally Lipschitz coefficients, which therefore has unique solution up to the first time $\Theta$ that $\mathrsfs{B}_t$ either exits $\mathcal{U}$ or explodes (meaning it has norm that tends to $\infty$ in finite time).  It is also possible to efficiently numerically solve this system with standard ODE methods, which are the basis of the numerical simulations shown throughout the paper.

Under these assumptions, we can describe the limiting matrix of order parameters.
We say an event holds \emph{with overwhelming probability} 
if there is a function $\omega: \N \to \R$ with $\omega(d)/\log d \to \infty $ so that the event holds with probability at least $1-e^{-\omega(d)}.$
\begin{theorem}[Learning curves] \label{thm:learning_curves}
  Suppose 
  Assumptions 
  \ref{ass:pseudo_lipschitz},
  \ref{assumption:scaling},
  \ref{ass:data_normal},
  \ref{assum:learning_rate},
  \ref{assumption:unbiased},
  \ref{assumption:E_loss_pseudo_Lip}
  hold. 
  Let $\vartheta_M$ be the first time 
  that either $\mathrsfs{B}(t)$ or $B(W_{\lfloor td \rfloor })$ exits $\mathcal{U}$ 
  or that $\mathrsfs{N}(t) \geq M.$
  Then there is an $\varepsilon >0$ so that for any $T,M,$ with overwhelming probability
  \[
    \sup_{0 \leq t \leq T \wedge \vartheta_M} \| 
    \mathrsfs{B}(t)
    - 
    B(W_{\lfloor td \rfloor}) 
    \| \leq d^{-\varepsilon}.
  \]
  \label{thm:deterministic}
\end{theorem}
\noindent We shall further extend the class of statistics of the coupled family of ODEs $(\mathrsfs{B}_i(t) : 1 \leq i \leq d)$ which can be compared to SGD statistics in Theorem \ref{Thm:SGD_HSGD_convergence}. We 
 also note that $\tfrac{1}{d}\sum_{i=1}^d\tr(\mathrsfs{B}_{i}(t))$ plays the role of $\|W_{\lfloor td \rfloor }\|^2$ for the family of ODEs, and we shall give some simple sufficient conditions that ensure $\tfrac{1}{d}\sum_{i=1}^d\tr(\mathrsfs{B}_{i}(t))$ remains bounded independent of dimension of all time in Section \ref{susec:SGD_optimality}.

We also note that in the case of identity covariance, the system simplifies dramatically: as all $\lambda_i = 1$, we may directly take the average on both sides of \eqref{eq:coupledODElimit} to conclude:
\begin{corollary}[Learning curves in identity covariance]\label{cor:lc}
Under the same hypotheses as Theorem \ref{thm:deterministic},
if we suppose that $K = \Id_d$, then $\mathrsfs{B}(t)$ solves the autonomous equation
\begin{equation}
  \begin{aligned}
    &\dif \mathrsfs{B}_{11}(t) = -2\gamma_t( \mathrsfs{B}_{11}(t) H_{1,t} + H_{1,t}  \mathrsfs{B}_{11}(t)+\mathrsfs{B}_{12}(t) H_{2,t}) - 2\delta\gamma_t \mathrsfs{B}_{11}(t) + \gamma_t^2I_t, \\
    &\dif \mathrsfs{B}_{12}(t) = -2\gamma_t( H_{1,t} \mathrsfs{B}_{12}(t) + H_{2,t}^T\mathrsfs{B}_{22}) - 2\delta\gamma_t \mathrsfs{B}_{12}(t), \\
  \end{aligned}
  \label{eq:identityODElimit}
\end{equation}
with initial conditions $\mathrsfs{B}(0) = \langle W_0, W_0\rangle_{\mathcal{A}}$.
\end{corollary}
\noindent Many instances of these ODEs have appeared in the literature before (see the discussion in Section \ref{sec:relatedwork}).

\paragraph{High-dimensional diffusion approximation.} 
This system of ODEs \eqref{eq:coupledODElimit} 
%is efficient for numerical simulation.
has complexity that increases substantially with dimension,
since the number of equations grows with the dimensionality of $K$.
It is possible to formulate this in a dimension independent way, either as a measure-valued process or (equivalently) as a evolution on resolvent-like curves (see Section \ref{sec:approximate_solutions_stability}).
Nonetheless, it does not give access to the iterates on parameter space, and one may wish to understand, for example, how the iterates $\{X_k\}$ evolve when tested against another interesting fixed direction $\{\hat X\}$.

So we introduce another tool, 
which is a  stochastic differential equation \textit{homogenized SGD},
and which is amenable to sharp dimension-independent analysis along more traditional 
optimization theory lines.
\begin{equation} \label{eq:HSGD}
\dif \WHSGD_t = 
-\gamma_t \nabla_X \mathcal{R}_\delta(\WHSGD_t)\dif t 
%+ \gamma \ip{\sqrt{K / d  \otimes  I(\langle \HSGD_t^{\otimes 2}, K \rangle)}, \dif B_t}_{\mathcal{A \otimes O}},
+\gamma_t \ip{\sqrt{K/d} \otimes \sqrt{ \EE_{a,\epsilon} [ \nabla_\ex f( \ip{\WHSGD_t \oplus X^{\star}, a}_{\mathcal{A}};\epsilon)^{\otimes 2} ] }, \dif B_t}_{\mathcal{A \otimes O}},
\end{equation}
where the initial conditions are given by $\WHSGD_0 = X_0$ and $(B_t, t \ge 0)$ a $d \times \ell$ dimensional standard Brownian motion. 
Analogously to the $(W_k,r_k)$ notation, we define
\[
\HSGD_t = \WHSGD_t \oplus X^{\star} \quad \text{and} \quad 
\rho_t \defas \ip{\HSGD_t, a }_{\mathcal{A}}.
\]
Homogenized SGD is connected to the coupled ODEs in the same way as SGD:
\begin{proposition}\label{prop:HSGD_median dynamics}
  Suppose 
  Assumptions 
  \ref{ass:pseudo_lipschitz},
  \ref{assumption:scaling},
  \ref{ass:data_normal},
  \ref{assum:learning_rate},
  \ref{assumption:unbiased},
  \ref{assumption:E_loss_pseudo_Lip}
  hold. 
  We let, for any $\eta > 0$,  %$\mathcal{U}_\eta$ be set 
  \begin{equation} \label{eq:U_eta}
  \mathcal{U}_\eta \defas 
  \{ B \in \mathcal{U} : \inf_{V \in \mathcal{U}^c} \| B - V\| \geq \eta\}.
  \end{equation}
  Let $M > 0$ and let $\vartheta_M$ be the first time $\|\HSGD_t^{\otimes 2}\| \geq M$, or that $\HSGD_t$ exits $\mathcal{U}_\eta$.  
  There is an $\varepsilon > 0$ so that for any $T,M$ with overwhelming probability
  \[
    \max_{0 \leq t \leq T \wedge \vartheta_M}
    \|\mathrsfs{B}(t) - \langle \HSGD_t^{\otimes 2}, K \rangle_{\mathcal{A}^{\otimes 2}}\| \leq d^{-\varepsilon}.
  \]
\end{proposition}
\noindent This proposition shows that in high-dimensions, SGD noise becomes effectively continuous (in time) and moreover has a diffusion coefficient that looks like $\tfrac{1}{d} K  \otimes  I(\ip{ \HSGD_t^{\otimes 2}, K}_{\mathcal{A}^{\otimes 2}})$.  The presence of the $1/d$ may at first suggest that the noise is becoming negligible as $d \to \infty$; however, this exactly balances the effect of the growing dimensionality in that it can be viewed as the origin of the non-negligible quadratic-in-$\gamma$ terms, i.e.,\ those with $I(\mathrsfs{B}(t))$, in \eqref{eq:coupledODElimit}.

We also note that we have formulated Proposition~\ref{prop:HSGD_median dynamics} in terms of the first time homogenized SGD has a norm-squared larger than $M$, and hence boundedness of homogenized SGD can be used to show boundededness of the system of ODEs. One can also reverse the roles of these, first showing boundedness for the ODEs to conclude the same for homogenized SGD
%f the deterministic curves, it would suffice to consider the behavior of homogenized SGD. 
%The same holds in reverse as well.  

\paragraph{Other statistics.} 
While $\mathrsfs{B}$ is the most important statistic to describe if one wishes to capture the dynamical evolution of SGD, there are other natural statistics to consider such as contractions without the covariance $K$ 
(e.g., $\|X\|^2$ and $\|X-X^\star\|^2$) and functions such as $\mathcal{R}_\delta$.
The method transparently extends to the following class:
\begin{assumption}[Smoothness of the statistics, $\varphi$] \label{assumption:statistic} The statistic satisfies a composite structure, 
  %$\varphi \, : \, \mathcal{A} \otimes \mathcal{O} \to \mathbb{R}$  is $C^3$-smooth and 
\begin{equation*}
\begin{gathered}
\varphi(X) = 
g( \ip{W \otimes W,q(K)}_{\mathcal{A}^{\otimes 2}}  ) 
\end{gathered}
\end{equation*}
where ${g} \, : \, \mathcal{O}^+ \otimes \mathcal{O}^+ \to \mathbb{R}$ is $\alpha$-pseudo-Lipschitz on $\mathcal{U}$
and $q$ is a polynomial.
%and satisfy the estimates for some $\alpha \geq 0$ and constants $L$
%\[
%  \|D^3 h\|_\sigma \leq L(h)(1+\|W\|^{\alpha})
%  \quad\text{and}\quad
%  \|D^3 p\|_\sigma \leq L(p)(1+\|W\|^{\alpha}).
%\]
\end{assumption}
For statistics satisfying the above, we may then directly compare SGD,  homogenized SGD, and the deterministic family of ODEs. For the ODEs, the relevant combination is
\[
\phi(t) \defas \frac{1}{d} \sum_{i=1}^d g( \mathrsfs{B}_i(t) q(\lambda_i)).
\]
\begin{theorem}\label{Thm:SGD_HSGD_convergence}
  Suppose 
  Assumptions 
  \ref{ass:pseudo_lipschitz},
  \ref{assumption:scaling},
  \ref{ass:data_normal},
  \ref{assum:learning_rate},
  \ref{assumption:unbiased},
  \ref{assumption:E_loss_pseudo_Lip}
  hold. 
  % We let, for any $\eta > 0$,  %$\mathcal{U}_\eta$ be set 
  % \[
  % \mathcal{U}_\eta = 
  % \{ B \in \mathcal{U} : \inf_{V \in \mathcal{U}^c} \| B - V\| \geq \eta\}.
  % \]
  Let $\vartheta_M$ be the first time 
  that $\langle \HSGD_t^{\otimes 2},K\rangle$ exits $\mathcal{U}_\eta$ (see \eqref{eq:U_eta}) 
  or that $\mathrsfs{N}(t) \geq M.$
For any function $\varphi$, which satisfies Assumption \ref{assumption:statistic},
for any $M$, any $T$, and any $\varepsilon\in (0,1/2)$
there is a constant $C$ (not depending on $d$) so that with overwhelming probability
%The processes $\WHSGD_t$ for $t\in [0,T]$ with $T=n/d\in [0,\infty)$, and $X_k$, for $k\in [0,n]$, starting at $X_0 = \WHSGD_0$,
%satisfy for any $n = O(d)$ and $\varepsilon\in (0,1/2)$,
\begin{equation}\label{eq:phi_phi_final}
\sup_{0\leq t\leq T \wedge \vartheta_M} 
\biggl(
|\varphi(\WHSGD_{t})  - \varphi(X_{\lfloor td \rfloor})|
+|\varphi(\WHSGD_{t})  - \phi(t)|
\biggr)
\leq Cd^{-\varepsilon}.
% (f,T, \|K\|_\sigma, \varphi, \bar{\gamma},\delta)
% (T,f,L(h),\alpha, \|K\|_\sigma)
% e^{C_2
\end{equation}

%The constant $C$ is finite and positive and does not depend on $n$, and $d$. 
\end{theorem}

Finally, we give a simple condition under which one can remove the stopping time $\vartheta_M$ (provided one stays within the good set $\mathcal{U}$), which is to say that we can ensure the ODEs do not go to infinity in finite time.  
%This non-explosiveness proposition provides an easily \textit{checkable} condition on the loss $f$ guarantees that Theorem~\ref{Thm:SGD_HSGD_convergence} holds without the use of a stopping time provided one always stays in the good set, $\mathcal{U}_{\eta}$.
\begin{proposition}[Non-explosiveness] \label{prop:nonexplosiveness}
Suppose that 
Assumptions \ref{ass:pseudo_lipschitz},
\ref{assumption:scaling}, \ref{ass:data_normal} and \ref{assum:learning_rate} hold.
Suppose further that the objective function $f$ is $\alpha$-pseudo-Lipschitz with $\alpha=1$.  
Then there is a constant $C$ depending on $\|K\|_{\sigma}$, $\bar{\gamma}$, $\|X_0\|$, $\|X^\star\|$, $L(f)$ so that
\[
\mathrsfs{N}(t) \leq (1+ \mathrsfs{N}(0))e^{C t}
\]
for all time $t$ such that $\mathrsfs{B}(t)$ is in $\mathcal{U}.$
\end{proposition}

This leads us to the following simplified version of Theorem \ref{Thm:SGD_HSGD_convergence}
\begin{corollary}\label{cor:SGD_HSGD_convergence}
  Suppose 
  Assumptions 
  \ref{ass:pseudo_lipschitz},
  \ref{assumption:scaling},
  \ref{ass:data_normal},
  \ref{assum:learning_rate},
  \ref{assumption:unbiased},
  \ref{assumption:E_loss_pseudo_Lip}
  hold. 
  Suppose further that $\mathcal{U} = \mathcal{O}^+ \otimes \mathcal{O}^+$
  and that $f$ is $\alpha$-pseudo-Lipschitz with $\alpha \leq 1$.
For any function $\varphi$, which satisfies Assumption \ref{assumption:statistic},
any $T$, and any $\varepsilon\in (0,1/2)$
there is a constant $C$ (not depending on $d$) so that with overwhelming probability
\begin{equation*}
\sup_{0\leq t\leq T} 
\biggl(
|\varphi(\WHSGD_{t})  - \varphi(X_{\lfloor td \rfloor})|
+|\varphi(\WHSGD_{t})  - \phi(t)|
\biggr)
\leq Cd^{-\varepsilon}.
\end{equation*}
\end{corollary}

\begin{remark}[Longer time horizons]
In cases where Assumptions \ref{assumption:unbiased},  \ref{assumption:E_loss_pseudo_Lip} and \ref{assumption:statistic} hold with $\alpha  = 0$, i.e. Lipschitz functions, one can show that Eq. \eqref{eq:phi_phi_final} holds for any $Td < c d\log d$ with some fixed constant $c>0$, which depends on the operator norm of $K$ and the Lipschitz constants of $\varphi$ and its derivatives. %\IS{I guess we can make it remark also more explicit.}
\end{remark}

\begin{remark}[Other directions]
Suppose one wishes to consider overlaps of the state $X_k$ of SGD with some other deterministic matrix of directions $\hat{X}$ in $\mathcal{A} \otimes \R^p$.  This is already covered by Theorem \ref{Thm:SGD_HSGD_convergence}, as it is possible to extend $X^\star$ by making the replacement $X^\star \to X^\star \oplus \hat{X}.$  The outer function $f$ should then not consider these additional direction, but Theorem \ref{Thm:SGD_HSGD_convergence} gives a deterministic equivalent.  For example, one may choose $\hat{X}$ to be a minimizer of $\mathcal{R}_\delta(X)$  and then $\varphi(X) = \|X-\hat{X}\|^2$.
\end{remark}

\begin{figure}[t] \centering
\includegraphics[scale =0.3]
{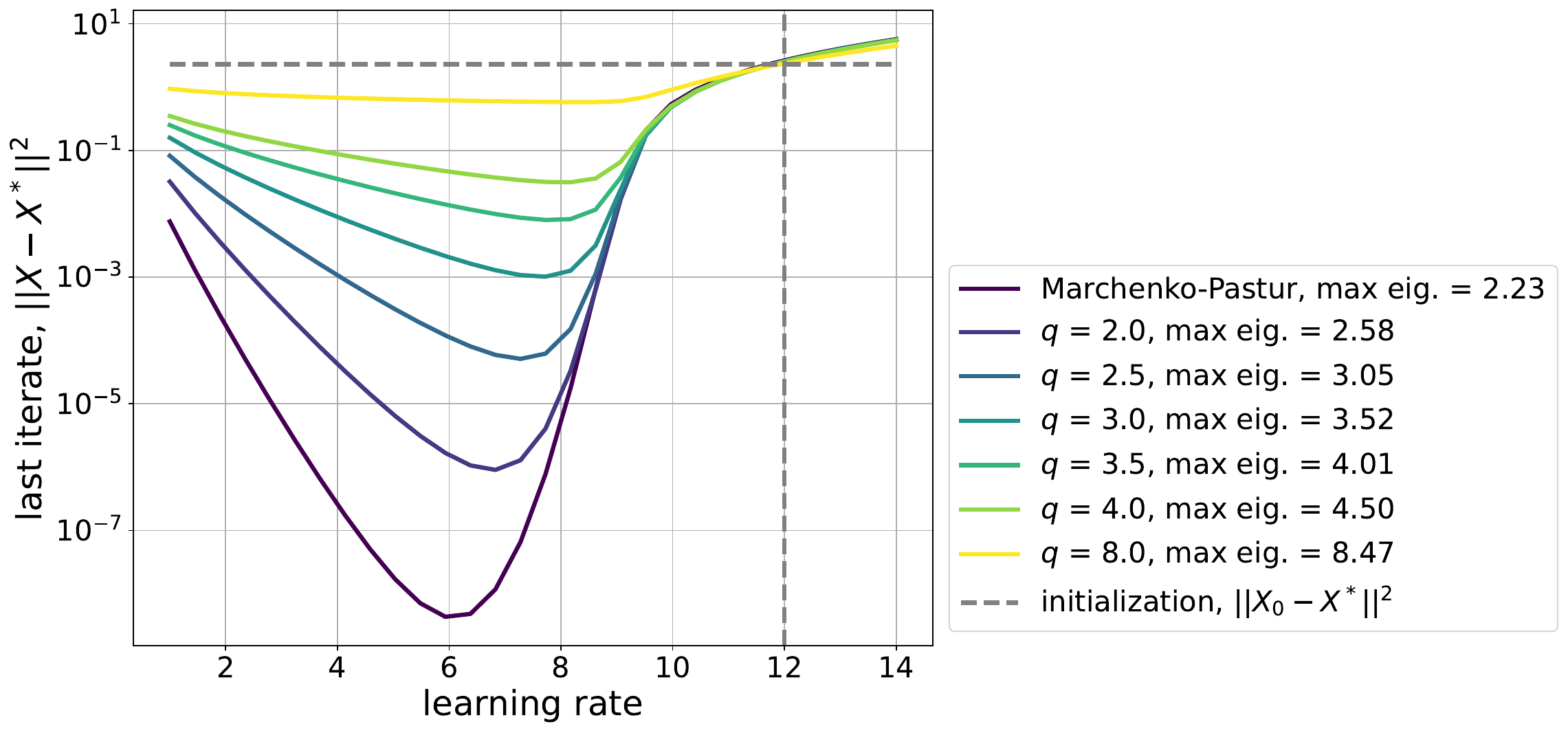}
\caption{\textbf{Descent and critical learning rate} on (binary, noiseless) logistic regression problem. Plotted are the last value of $\mathrsfs{D}^2(t)$ at time $t = 30$, $\mathrsfs{D}^2(t_{30})$, for binary, noiseless (i.e., $\epsilon = 0$) logistic regression problem. From Theorem~\ref{Thm:SGD_HSGD_convergence}, $\mathrsfs{D}^2(t) \approx \|X_{\lfloor td \rfloor}-X^{\star}\|^2$ where $X_{\lfloor td \rfloor}$ are the iterates of SGD. Initialization was random, $X_0 \sim N(0, I_{d})$, and then normalized so that $\|X_0\| = \sqrt{1.1}$ and $X^{\star} \sim \tfrac{1}{\sqrt{d}} N(0, I_{d})$ where $d = 1000$. Covariance matrix was constructed by specifying the spectrum, $\sigma_i \sim \text{Unif}(1.0,2.0)$, $i = 1, \hdots, d = 1000$ and setting the covariance matrix $K = \text{diag}( \sigma_i^{2q} \, : \, i = 1, \hdots, 1000)$.  Also plotted is a covariance matrix with Marchenko-Pastur spectrum (parameter $4$, darkest line). For all covariance matrices, the matrix $K$ was then normalized so that the average eigenvalue of $K$, $\tfrac{1}{d} \tr(K) = 1.0$. As the power $q$ in the spectrum of K, $\sigma_i$, increases, the largest eigenvalue of $K$ also increases while the average eigenvalue is fixed. In spite of $K$ having varying spectral distributions, all the curves reach the same (gray, dashed) initialization line at the same learning rate, $\gamma \approx 12$, suggesting that there is a universal learning rate, depending on the $\tfrac{1}{d}\tr(K)$, that dictates descent. Indeed, this supports our prediction in Corollary~\ref{cor:convexbounded} -- the learning rate threshold for descent \eqref{eq:learning_rate_descent_convex} seems to be controlled by the average eigenvalue and \textit{not} the max eigenvalue of $K$. The optimal learning rates do vary as max eigenvalue changes, as do the rates of convergence. This is also predicted, given that logistic regression satisfies a local strong convexity result, which degrades as the largest eigenvalue changes (see Proposition~\ref{prop:logistic_convergence_rate}) . 
} \label{fig:logistic_descent_critical_learning_rate}
\end{figure}

\subsection{Optimality and descent conditions for SGD \label{susec:SGD_optimality}}
An important part of stochastic optimization is understanding when the distance to optimality decreases; due to the intrinsic stochasticity it is usually too much to ask any measure of suboptimality to decrease at each iteration.  In our setting, the deterministic equivalent gives a method of producing a measure of suboptimality which can be reasonably expected to decrease monotonically and is uniformly close to a traditional metric of suboptimality applied to SGD; this monotone decrease of suboptimality we refer to as \emph{descent}.

Typically in the literature (see \cite{bottou2018optimization} and references therein), sufficient conditions for descent are formulated as upper bounds on the learning rates which depend on the operator norm of the covariance matrix $\|K\|_{\sigma}$, or even the smallest eigenvalue of $K$.\footnote{In fact, typical descent guarantees assume use smoothness or strong convexity constants of the risk $\mathcal{R}$, which when translated to this context involve the smallest and largest eigenvalues of $K$.}  
Instead, our analysis shows for a wide class of GLMs and multi-index models, including convex and strongly convex objectives, that the convergence rate and learning rate thresholds for the descent of SGD can be relaxed to the average eigenvalue of the covariance matrix (i.e., $\tfrac{1}{d} \tr(K)$). This is a significant improvement, as many data sets have $\|K\|_{\sigma} \gg \tfrac{1}{d} \tr(K)$.  Moreover, we can characterize the exact learning rate threshold for descent.

All these conclusions will be drawn by considering 
the evolution of various quadratic functionals.
For simplicity we work in the case $\mathcal{O}=\mathcal{T},$ $\delta=0$ 
and the case that $X^\star$ is itself a minimizer of the risk $\mathcal{R}$. 
Moreover, we assume a result about our outer function $f$, that is, it attains a \textit{global minimizer} at the same point as the global minimizer of the risk $\mathcal{R}$. 
\begin{assumption}[Risk and loss minimizer] \label{assumption:risk_loss_minimizer_1_main} Suppose that
  \[
    X^{\star} \in \argmin_{X} \big \{ \mathcal{R}(X) = \EE_{a, \epsilon}[f(\ip{X,a}_{\mathcal{A}} \oplus \ip{ X^{\star},a}_{\mathcal{A}})] \big \}
  \]
  exists and has norm bounded independent of $d.$
  Then one has, 
  \[ 
    \ip{X^{\star},a}_{\mathcal{A}} \in \argmin_x 
    \{
    f(x \oplus \ip{X^{\star},a}_{\mathcal{A}}) \}, \qquad \text{for almost surely $a \sim N(0, K)$.}
  \]
\end{assumption}
\noindent While at first, this assumption seems quite strong, in fact, in a typical student-teacher setup when label noise is $0$ (i.e., $\epsilon = 0$), where the targets have the same model as the outputs, the assumption is satisfied.   Our goal here is not to be exhaustive, but simply to illustrate that our framework admits a nontrivial and useful analysis and which gives nontrivial conclusions for the optimization theory of these problems.

For the analysis, we use extensively our coupled ODEs, $(\mathrsfs{B}_i(t) \, : \, i = 1, \hdots, d)$.  In particular, we consider the deterministic counterpart for $\|X-X^{\star}\|^2$.
When evolving according to the solution of \eqref{eq:coupledODElimit}, this is exactly:
\begin{equation} \label{eq:deterministic_distance_main}
\begin{aligned}
    \mathrsfs{D}^2(t) 
    &=
    \frac{1}{d}\sum_{i=1}^d 
    \tr\biggl(
    \mathrsfs{B}_{11,i}(t) 
    - 2\mathrsfs{B}_{12,i}(t) 
    +\mathrsfs{B}_{22,i}(t) 
    \biggr).
\end{aligned}
\end{equation}

We will show that for standard outer function assumptions and an upper bound on the learning rate $\gamma_t < \bar{\gamma}$ that the \textit{function $\mathrsfs{D}^2(t)$ is decreasing in $t$}. Since $\|X-X^{\star}\|^2$ is a statistic that satisfies Assumption~\ref{assumption:statistic}, fixing a $T > 0$, we have by Theorem~\ref{Thm:SGD_HSGD_convergence} for some $\varepsilon > 0$,
\[
\sup_{0 \le t \le T} | \|X_{\lfloor td \rfloor}-X^{\star}\|^2 - \mathrsfs{D}^2(t) | \le d^{-\varepsilon} \quad \text{with overwhelming probability} 
\]
In this way, $\mathrsfs{D}^2(t) \approx \|X_{\lfloor td \rfloor}-X^{\star}\|^2$ and since $\mathrsfs{D}^2(t)$ is decreasing, so is the distance to optimality of SGD. Consequently, we say SGD is \textit{descending} if $\mathrsfs{D}^2(t)$ is decreasing. 

As it turns out, the evolution in time of $\mathrsfs{D}^2$ is particularly simple,
as it solves the differential equation
\begin{equation}\label{eq:DODE}
\frac{\dif}{\dif t} \mathrsfs{D}^2(t) 
= -\gamma_t A( \mathrsfs{B}(t)) + \frac{\gamma^2_t}{2d} \tr(K) I(\mathrsfs{B}(t)),
\quad
\left\{
\begin{aligned}
&A( \mathrsfs{B}) = \Exp_{a,\epsilon} [\ip { x-x^\star, \nabla_x f(x\oplus x^\star)}], \\
&I( \mathrsfs{B}) = \Exp_{a,\epsilon} [\|\nabla_x f(x\oplus x^\star)\|^2], \quad \text{where} \\
&(x \oplus x^\star) \sim N(0, \mathrsfs{B}).
\end{aligned}
\right.
\end{equation}
See Lemma \ref{lem:D2} for a proof.  Thus the exact local descent threshold for $\mathrsfs{D}^2$ is given by
\begin{equation}\label{eq:stability}
\gamma_t \leq \gamma^{\text{stable}}_t \defas \frac{A( \mathrsfs{B}(t))}{\tfrac{\tr(K)}{2d}  I(\mathrsfs{B}(t))}.
\end{equation}
This should be compared to the \emph{Polyak step-size} in convex optimization.

%Surprisingly, for this to happen, we will see that the upper bound on the learning rate $\bar{\gamma}$ depends on the average eigenvalue of $K$, $\tfrac{1}{d} \tr(K)$, instead of the largest eigenvalue, $\lambda_{\max}(K)$. As $\tfrac{1}{d} \tr(K) \ll \lambda_{\max}(K)$ for typical datasets, our result shows a larger learning rate can be used in practice and one will still observe the distance to optimality of SGD decreases at each iteration. 

% While we will not give a rate of convergence in this section, we will provide several conditions on the learning rate for which decrease at each iteration is guaranteed. 

\begin{proposition}[Descent of SGD] \label{prop:descent_SGD_main} Suppose the Assumptions of Theorem~\ref{Thm:SGD_HSGD_convergence} hold and suppose that $\mathcal{U} = \mathcal{O}^+ \otimes \mathcal{O}^+$. 
%Suppose $X^{\star} \in \argmin_X \{ \mathcal{R}(X)\}$ satisfies Assumption~\ref{assumption:risk_loss_minimizer_1_main} . 
Moreover, suppose the following inequality holds for some constant $q > 0$,
\begin{equation} \label{eq:stability_2_main}
    q \cdot I(\mathrsfs{B}) \le A(\mathrsfs{B})
    \quad \text{ for all $\mathrsfs{B}$. }
    %\ip{X - X^{\star}, (\nabla \mathcal{R})(X)}, \quad \text{for all $X \in \mathcal{A} \otimes \mathcal{O}$.}
\end{equation}
If the learning rate $\displaystyle \gamma_t < \bar{\gamma}$ for all $t \ge 0$, where 
\begin{equation} 
% \label{eq:stepsize_stability_1}
\bar{\gamma} = \frac{2 q}{\tfrac{1}{d} \tr (K)},
\end{equation}
then, the function $\mathrsfs{D}^2(t)$ defined in \eqref{eq:deterministic_distance_main} is decreasing for all $t \ge 0$. 
Moreover, for some $\varepsilon > 0$ and any $T > 0$, the iterates of SGD $\{X_k\}$ satisfy
\begin{equation}
\label{eq:descent_101_main}
\sup_{0 \le t \le T} | \|X_{\lfloor td \rfloor }-X^{\star}\|^2 - \mathrsfs{D}^2(t) | \le d^{-\varepsilon}, \quad \text{with overwhelming probability.}
\end{equation}
\end{proposition}

The average eigenvalue's significant role in the threshold is supported numerically in Figure~\ref{fig:logistic_descent_critical_learning_rate} on a binary, noiseless logistic regression problem. The threshold for descent, as indicated by the dashed gray line, occurs at the same learning rate for a family of covariances with average eigenvalue $1$ and varying largest eigenvalue. %While the specific value is larger than our predicted, theoretical result (see Corollary~\ref{cor:convexbounded}), it is within a constant factor. 

We shall show that under further structural assumptions, it is possible to check the conditions of Proposition \ref{prop:descent_SGD_main}.
Moreover, we shall put these assumptions on the \emph{outer} function $f$, as opposed to the whole objective function $\mathcal{R}$.  
To start, we shall suppose that $f$ is \textit{$\hat{L}$-smooth}. This type of assumption is typical of many optimization convergence algorithms and it is dimension-independent in our setting.

\begin{definition}[$\hat{L}$-smoothness of outer function $f$] A $C^1$-smooth function $ f \, : \, \mathcal{O} \to \mathbb{R}$ is \textit{$\hat{L}(f)$-smooth} if the following quadratic upper bound holds for any $x, \hat{x} \in \mathcal{O}$
\begin{equation} \label{eq:quadratic_upper_bound_1_main}
     f(\hat{x}) \le   f(x) + \ip{\nabla_x f(x), \hat{x} - x} + \tfrac{\hat{L}(f)}{2} \|\hat{x}-x\|^2. 
\end{equation}
\end{definition}
\noindent Note that if $\nabla_x  f$ is $\hat{L}(f)$-Lipschitz, i.e., $\|\nabla f(x) - \nabla f(\hat{x})\| \le \hat{L}(f) \| x - \hat{x}\|$, then the inequality \eqref{eq:quadratic_upper_bound_1_main} holds with constant $\hat{L}$. Suppose $\displaystyle x^{\star} \in \argmin_x \{f(x)\}$ exists. An immediate consequence of \eqref{eq:quadratic_upper_bound_1_main} is that 
\begin{equation} \label{eq:L_smoothness_bound_main}
\frac{1}{2\hat{L}(f)}\| \nabla f(x) \|^2 \leq  f(x) -f(x^{\star})\leq  \frac{\hat{L}(f)}{2}\|x-x^\star\|^2.
\end{equation}

\begin{corollary}[Descent of convex, $\hat{L}(f)$-smooth outer function] Fix a constant $T > 0$. Suppose the Assumptions of Theorem~\ref{Thm:SGD_HSGD_convergence} hold and suppose that $\sup_{0 \le t \le T} \sup_{V \in \mathcal{U}^c} \| \mathrsfs{B}(t)- V\| > \eta$. In addition, let the outer function $f \, : \, \mathcal{O} \otimes \mathcal{T} \otimes \mathcal{T} \to \mathbb{R}$ be a convex and $\hat{L}(f)$-smooth function with respect to $x \in \mathcal{O}$. Suppose $X^{\star} \in argmin_{X} \{\mathcal{R}(X)\}$ exists bounded, independent of $d$ and Assumption~\ref{assumption:risk_loss_minimizer_1_main} holds. Then the inequality \eqref{eq:stability_2_main} holds with $q = \tfrac{1}{2\hat{L}(f)}$. Moreover, if $\displaystyle  \gamma_t \le \bar{\gamma}$ for all $t$ where
\begin{equation} \label{eq:learning_rate_descent_convex}
\bar{\gamma} = \frac{1}{\hat{L}(f)\tfrac{1}{d} \tr(K)}, 
\end{equation}
then, the function $\mathrsfs{D}^2(t)$ defined in \eqref{eq:deterministic_distance_main} is decreasing for all $t \ge 0$. Moreover, for some $\varepsilon > 0$, the iterates of SGD $\{X_k\}$ satisfy
\[
\sup_{0 \le t \le T} | \|X_{\lfloor td \rfloor }-X^{\star}\|^2 - \mathrsfs{D}^2(t) | \le d^{-\varepsilon}, \quad \text{with overwhelming probability.}
\]

% then, with overwhelming probability,
% \[
% \sup_{0 \le t \le T} \|\WHSGD_t \| \le C
% \]
% for some positive constant $C = C(\|K\|_{\sigma}, \|X_0\|, \|X^{\star}\|, \alpha)$. 
\label{cor:convexbounded}
\end{corollary}

To further guarantee convergence,
we need stronger assumptions, both on the outer function and on the covariance, $K$ (see Section \ref{sec:optimization} for proofs of following propositions).  So we consider functions which satisfy the \textit{restricted secant inequality}. 
% \begin{definition}[Restricted Secant Inequality] A $C^1$-smooth function $f \, : \, \mathcal{O}\to \mathbb{R}$ satisfies the $(\mu,\theta)$--\textit{restricted secant inequality (RSI)} if for any $x \in \mathcal{O}$ and $\hat{x} \in \argmin_x \{ f(x)\}$,
% \[
% \ip{x-\hat{x}, \nabla f(x)} \ge 
% \begin{cases} 
%   \mu \|x-\hat{x}\|^2, & \text{if } \max\{\|\hat{x}\|^2,\|x-\hat{x}\|^2\} \leq \theta,\\
% 0, &\text{otherwise}.
% \end{cases}
% \]
% If $f$ satisfies the above for $\theta = \infty$ then we say $f$ satisfies the $\mu$--RSI.
% \end{definition}
% \begin{remark}
%     If the function $f$ is $\hat{\mu}$-strongly convex, then $f$ satisfies the restricted secant inequality with $\mu = \frac{\hat{\mu}}{2}$.  
% \end{remark}

\begin{definition}[Restricted Secant Inequality] A $C^1$-smooth function $f \, : \, \mathcal{O}\to \mathbb{R}$ satisfies the $(\mu,\theta)$--\textit{restricted secant inequality (RSI)} if, for any $x \in \mathcal{O}$ and $x^{\star} \in \argmin_x \{ f(x)\}$,
\[
\ip{x-x^{\star}, \nabla_x f(x)} \ge 
\begin{cases} 
  \mu \|x-x^{\star}\|^2, & \text{if } \max\{\|x^{\star}\|^2,\|x-x^{\star}\|^2\} \leq \theta,\\
0, &\text{otherwise}.
\end{cases}
\]
If $f$ satisfies the above for $\theta = \infty$, then we say $f$ satisfies the $\mu$--RSI.
\end{definition}

\noindent We note that simple strictly convex examples, such as those built from cross-entropy-loss cannot satisfy traditional uniform restricted secant inequality with $\theta=\infty$.  However, for local convergence, this is unneeded.

\begin{proposition}[Local convergence rate for fixed stepsize, $(\hat{\mu}(f),\hat{\theta}(f))$-RSI, $\hat{L}(f)$-smooth function, with covariance $K \succ 0$] 
\label{prop:RSI}
Fix a constant $T > 0$. Suppose the Assumptions of Theorem~\ref{Thm:SGD_HSGD_convergence} hold and suppose that $\sup_{0 \le t \le T} \sup_{V \in \mathcal{U}^c} \| \mathrsfs{B}(t)- V\| > \eta$. Let the outer function $f \, : \, \mathcal{O} \otimes \mathcal{T} \otimes \mathcal{T} \to \mathbb{R}$ be a $\hat{L}(f)$-smooth function satisfying $(\hat{\mu}(f),\hat{\theta}(f))$--RSI with respect to $x\in \mathcal{O}$. Suppose $X^{\star} \in \argmin_X \{ \mathcal{R}(X)\}$ is bounded, independent of $d$ and Assumption~\ref{assumption:risk_loss_minimizer_1_main} holds. Let the covariance matrix $K$ have a smallest eigenvalue bounded away from $0$, that is $\lambda_{\min}(K) > 0$.
% Fix $T > 0$ and suppose the initialization $X_0$, $X^{\star}$ are bounded independent of $d$. Let the loss function $f \, : \, \mathcal{O} \to \mathbb{R}$ be a $\hat{L}(f)$-smooth, $\alpha$-pseudo-Lipschitz function satisfying the $(\hat{\mu}(f),\hat{\theta}(f))$--RSI. Suppose $\hat{X} \in \argmin_X \{ \mathcal{R}(X)\}$ is bounded, independent of, $d$ and Assumption~\ref{assumption:risk_loss_minimizer} holds. Let the covariance matrix $K$ have the smallest eigenvalue bounded away from $0$, that is $\lambda_{\min}(K) > 0$. 

Suppose the initialization $X_0$ satisfies that, for some $\zeta_0 \in (0,1)$,
\[
10 \exp \biggl( -\frac{\hat{\theta}(f)}{8 \|K\|_\sigma^2
  \max\{\|X_0-X^{\star}\|^2, \|X^{\star}\|^2\}
} \biggr) < \zeta_0,
\]
and suppose that $0 < \zeta < 1-\zeta_0$ and that
\[
\gamma_t = \gamma = \frac{2 \hat{\mu}(f) }{(\hat{L}(f))^2 \tfrac{1}{d} \tr(K)} \zeta.
\]
Then, with $a = \gamma (1-\zeta_0 - \zeta) \hat{\mu}(f) \lambda_{\min}(K)$,
we have, for all $t \ge 0$,
\[
\mathrsfs{D}^2(t) \le 2 e^{-at} \|X_0-X^{\star}\|^2.
\]
% \[
% \|\WHSGD_t - \hat{X}\|^2 \le 2 e^{-at} \|X_0-\hat{X}\|^2, \quad \text{w.o.p}.
% \]
Moreover, for some $\varepsilon> 0$, the iterates of SGD $\{X_k\}$ satisfy
\begin{equation} \label{eq:convergence_101_main}
\sup_{0 \le t \le T} | \|X_{\lfloor td \rfloor }-X^{\star}\|^2 - \mathrsfs{D}^2(t) | \le d^{-\varepsilon}, \quad \text{with overwhelming probability.}
\end{equation}
\end{proposition}

We note as a corollary for $\mu$-strongly-convex (or more generally $(\hat{\mu}(f))$-RSI) objectives, this implies that we have convergence regardless of the initialization.

\begin{proposition}[Global convergence rate for fixed stepsize, $\hat{\mu}(f)$-RSI, $\hat{L}(f)$-smooth function, with covariance $K \succ 0$]  Fix a constant $T > 0$. Suppose the Assumptions of Theorem~\ref{thm:main_concentration_S} hold and suppose that $\sup_{0 \le t \le T} \sup_{V \in \mathcal{U}^c} \| \mathrsfs{B}(t)- V\| > \eta$. Let the outer function $f \, : \, \mathcal{O} \otimes \mathcal{T} \otimes \mathcal{T} \to \mathbb{R}$ be a $\hat{L}(f)$-smooth function satisfying the RSI condition with $\hat{\mu}(f)$ with respect to $x\in \mathcal{O}$. Suppose $X^{\star} \in \argmin_X \{ \mathcal{R}(X)\}$ is bounded, independent of $d$ and Assumption~\ref{assumption:risk_loss_minimizer_1_main} holds. Let the covariance matrix $K$ have a smallest eigenvalue bounded away from $0$, that is $\lambda_{\min}(K) > 0$. If the learning rate satisfies
\[
\gamma_t = \gamma = \frac{2 \hat{\mu}(f) }{(\hat{L}(f))^2 \tfrac{1}{d} \tr(K)} \zeta,
\]
for some $0 < \zeta < 1$, then for all $t \ge 0$
\[
\mathrsfs{D}^2(t) \le e^{-a t} \mathrsfs{D}^2(0), 
\]
where $a = \gamma (1-\zeta) \hat{\mu}(f) \lambda_{\min}(K)$. Moreover, for some $\varepsilon> 0$, the iterates of SGD $\{X_k\}$ satisfy
\begin{equation} \label{eq:convergence_10_main}
\sup_{0 \le t \le T} | \|X_{\lfloor td \rfloor }-X^{\star}\|^2 - \mathrsfs{D}^2(t) | \le d^{-\varepsilon}, \quad \text{with overwhelming probability}.
\end{equation}
\label{prop:Courtneyrate_main}
\end{proposition}

\subsection{Related work}\label{sec:relatedwork}

\subsubsection{Single and multi-index models under SGD}

A single-index model is a high-dimensional model $\mathcal{M}(a; X^\star) =  f(\langle X^\star, a\rangle_{\mathcal{A}})$
in which one may consider both $X^\star$ and the link function $f$ to be unknown.
A classic supervised learning setup is then to estimate both $X^\star$, and also sometimes $\mathcal{M}$ when tested by some data distribution on $a$.
\[
\Psi(X;a,\epsilon) = \ell( \mathcal{M}_1( a; X), \mathcal{M}_2( a; X^\star) + \epsilon),
\]
for some single-index models $\mathcal{M}_1$ and $\mathcal{M}_2$.  This extends to a multi-index model, in our notation, by taking multidimensional $X$ and $X^\star$ and hence having a finite collection of directions in high dimensions which influence the behavior of the algorithm.  

\paragraph{Limit theory: Identity covariance}

An early and influential work in this direction is \cite{saad1995dynamics}, which considered multi-index models of varying size with ReLU activation functions (soft--committee machines) and derived the ODEs in Corollary \ref{cor:lc}.  Many related results appeared around the same time in the physics literature, with different  extensions \cite{biehl1994line, biehl1995learning, saad1995exact}.  These were shown to be exact in \cite{goldt2019dynamics}, building on techniques 
which originate in \cite{wang2019solvable} and \cite{wang2017scaling}. We note that the general strategy of martingale arguments used here is similar to those in \cite{wang2019solvable}.  See also \cite{arnaboldi2023highdimensional} in which these ODEs are compared to other limits.

The ODEs stated can be viewed as describing a class of non-singular setups, in which one does not start too close to some saddle points (as described in the Lipschitz phase retrieval example).  For a large class of single-index models, \cite{ben2022high} considers spherically constrained SGD and characterizes a class, where for a cold initialization longer than $O(d)$, SGD develops a dimension-independent signal. This happens in a wide variety of problems, and this has led to a thread of analyses which study how problem geometries might be changed to improve the performance \cite{Bruno}, \cite{damian2023smoothing}.

Nonetheless, the non-singular setup remains an active area of research \cite{mousavihosseini2023neural} gives generalization guarantees for learning monotone target activation functions, which are a large and important subclass.
In a similar vein, \cite{bietti2022learning} give gradient flow guarantees\footnote{In the system of ODEs, this is achieved by sending $\gamma \to 0$ and rescaling time by a factor $1/\gamma$ in Theorem \ref{thm:deterministic}.}, even applying to some singular setups. 

%Recent analyses have also consider gradient flow in a high-dimensional limit.  The ODEs in Theorem \ref{thm:deterministic} can be taken in a small $\gamma$ scaling limit (in effect setting $\gamma^2$ terms to $0$) to produce the analogous model for gradient flow.  This has led to analyses such as \cite{bietti2022learning}.

\paragraph{Limit theory: Non-identity covariance}

Non-identity covariance might initially appear to have little impact on single and multi-index models, owing to the inner linear structure.  Indeed, for many ``statics'' questions -- such as those connecting empirical and population risks or information theoretic concerns -- there is no gain in considering the covariance.  However, this is no longer true once one considers the optimization: non-identity covariance affects the dynamical behavior of stochastic gradient descent and where the true covariance $K$ is unknown, one may well be compelled to work in a non-identity setting.

The literature is considerably smaller for this case.  A significant step in building a theory for non-identity covariance is given by \cite{Goldt} who give equations of motion supposing Gaussian equivalence principle for some multi-index models; they are in particular motivated by data distributions coming from random-features-model type distributions.  They further derive ODEs like \eqref{eq:coupledODElimit} (but also quite different) in the case of quadratic loss and non-Gaussian data.  In some cases they are able to simplify these ODEs.  This was extended in \cite{goldt2022gaussian} to data input distributions which come from deeper random features models.

The work of \cite{Yoshida} posed the system of ODEs in Theorem \ref{thm:deterministic} in the case of squared loss, although without a precise formulation of the connection of their solution to the learning behavior of SGD.  Hence Theorem \ref{thm:deterministic} can be viewed as a generalization and formal verification of the \cite{Yoshida}.  They further investigate how data covariance leads to long-plateau effects observed in training dynamics.  Finally, we mention 
\cite{CollinsWoodfinPaquette01}, which gives an exact high-dimensional limit as here, but solely for the case of linear regression; \cite{CollinsWoodfinPaquette01} works beyond the case of Gaussian data, however.

\paragraph{High-dimensional optimization literature for online SGD}

The optimization and machine learning literature also contains an independent line of research into properties of SGD, often formulated in terms of guarantees.  Some of these are formulated in such a way to be relevant in a high-dimensional regime like seen here.

Now, the majority of SGD literature considers the finite-sum setup, where multipass SGD is run on a finite-sum problem.  Many results then provide guarantees for the generalization error, and this has led to notions such as algorithmic stability \cite{hardt2016train}.  Others give empirical loss estimates, for example, \cite{schmidt2013fast} and \cite{roux2012stochastic}. %More closely related to single-index models is \cite{soltanolkotabi2017learning} which considers specifically learning a high-dimensional ReLU, although with gradient descent.  %A more general analysis is given in \cite{dudeja2018learning}

Interest in convergence guarantees -- as well as qualitative properties of \emph{streaming} (or online, one-pass, etc.) SGD -- have recently gained attention, especially in the machine learning literature.  
\cite{karimi2016linear} give convergence rates under dimension-independent assumptions on the risk such as Polyak-{\L}ojasiewicz inequalities. \cite{pillaud2018exponential} gives linear convergence for least squares and classification problems.  \cite{dieuleveut2017harder} gives sharp convergence guarantees on least-squares problems.

%\paragraph{Traditional optimization literature for online SGD on GLMs}

%There is an enormous literature on SGD and online SGD, but we mention some of the more closely related theory.  In \cite{toulis14statistical}, online SGD for GLMs with mean-squared-error loss is considered with a decreasing step-size, with traditional bias/variance statistical analysis.  

\subsubsection{Other methods for high-dimensional limits}

\paragraph{Dynamical mean field theory}
A large body theory of high-dimensional limits comes in the form of dynamical mean field theory.  This gives systems of integro-differential equations for covariances, including $\mathrsfs{B}_t$ but also multi-time analogues of this covariance, and other auxiliary covariances.  The strength of this method is that it applies to a wide variety of high-dimensional statistical limits, while arguably the main drawback is the complexity of the resulting characterization.  
\cite{mignacco2020dynamical} gives a DMFT description of SGD for Gaussian mixture classification.
\cite{celentano2021highdimensional} gives a rigorous description of gradient flow dynamics on a similar class of problems, as well as other types of first order algorithms, by a description in terms of dynamical mean field theory.  \cite{gerbelot2022rigorous} performs a related analysis but with proportional batches, and also gives something like a discrete analogue of homogenized SGD.

\paragraph{Gordon methods}

The convex Gaussian minimax theorem \cite{gordon1988milman} has proven to be useful as a way of analyzing learning curve dynamics.  
\cite{chandrasekher2021sharp} gives an extensive analysis of SGD and other algorithms, based on the convex Gaussian minimax theorem, and in particular gives another method to derive some of the descriptions here in the case of identity covariance.  The methods in \cite{celentano2021highdimensional} are also based on this.

\subsubsection{Statics \& information theory and message-passing}

Our goal in this paper is to develop theory for the optimization theory of online SGD in high-dimensions, which may not be the most sample-efficient algorithm for finding the solution to a GLM.  For a large class of GLMs, there is a class of generalized message passing algorithms known to be optimal \cite{barbier2019optimal}.  There are additional specific studies for canonical GLMs such as logistic regression \cite{candes2020phase} and phase retrieval \cite{maillard2020phase}, the latter of which also shows that message passing achieves the information theoretic threshold for the solvability of the problem.

\paragraph{Outline of the paper. } The remainder of the article is structured as follows: in Section~\ref{sec:examples}, we provide some examples and specifically analyze SGD trajectories, applied to these examples, using the system of ODEs introduced in \eqref{eq:coupledODElimit}. For computations of specific example-dependent quantities needed to state the ODEs, see Appendix~\ref{sec:analysis_examples}. We give some preliminary tensor notation and derive derivatives of special functions used to prove Theorem~\ref{Thm:SGD_HSGD_convergence} in Section~\ref{sec:notation}. Our main results, Theorem~\ref{thm:learning_curves} and Theorem~\ref{Thm:SGD_HSGD_convergence} and their corollaries, are shown in Section~\ref{sec:approximate_solutions_stability}  for \textit{approximate solutions} to the system of ODEs \eqref{eq:coupledODElimit} (see for Definition~\ref{def:integro_differential_equation} for precise details). In Section~\ref{sec:SGD_homogenized_SGD}, we show that SGD and the SDE, homogenized SGD \eqref{eq:HSGD}, are approximate solutions to the ODEs in \eqref{eq:coupledODElimit}. Lastly, in Section~\ref{sec:optimization}, the deterministic system of ODEs is analyzed to give (and prove) critical thresholds on learning rates related to descent (proofs of Proposition~\ref{prop:descent_SGD_main}, Corollary~\ref{cor:convexbounded},  Proposition~\ref{prop:RSI}, and Proposition~\ref{prop:Courtneyrate_main}) and simple conditions on the outer function that ensure the ODEs do not go to infinity in finite time (proof of Proposition~\ref{prop:nonexplosiveness}). In Appendix~\ref{sec:Volterra_equation}, alternative interpretations of the ODEs \eqref{eq:coupledODElimit} are presented (e.g., as a solution to a Volterra equation, etc). 

\tableofcontents

% \begin{corollary}[Global convergence rate for fixed stepsize, $(\hat{\mu}(f))$-RSI, $\hat{L}(f)$-smooth function, with covariance $K \succ 0$]
%   \label{cor:RSI}
%   Suppose the same assumptions as in Proposition \ref{prop:RSI} hold but additionally $\hat \theta(f) = \infty$. 
%   Suppose Assumption \ref{assumption:fishernoise} holds with $\eta = 0.$
%   Suppose that $0 < \varepsilon < 1$ and that
% \[
% \gamma_t = \gamma = \frac{2 \hat{\mu}(f) }{(\hat{L}(f))^2 \tfrac{1}{d} \tr(K)} \varepsilon,
% \]
% Then with $a = \gamma (1- \varepsilon) \hat{\mu}(f) \lambda_{\min}(K)$,
% we have for any $0 \le t \le T$
% \[
% \|\WHSGD_t - \hat{X}\|^2 \le 2 e^{-at} \|X_0-\hat{X}\|^2, \quad \text{w.o.p}.
% \]
% \end{corollary}

\section{Examples}\label{sec:examples}

Throughout this section, we refer to the $K$-norm as $\|W\|_K^2 = \tr( \ip{W^{\otimes 2}, K}_{\mathcal{A}^{\otimes 2}})$. This is in comparison to the standard Euclidean norm, $\|W\|^2 = \tr(\ip{W, W}_{\mathcal{A}})$. In many examples, the $K$-norm plays a significant role. 

\subsection{Multivariate Linear regression.} 
The simplest example which satisfies \eqref{eq:student-teacher} is linear regression. 
Here we suppose that $g$ is rather the identity map, and $\ell$ is the squared loss $\ell(u,v)=\tfrac12 \|u-v\|^2$. 
Hence, we arrive at, with $\eta$ a constant
\begin{equation*}%\label{eq:student-teacher}
  \Psi(X;a,\epsilon) = 
  \tfrac12
\|\langle X-X^\star, a \rangle_{\mathcal{A}} + \eta \epsilon\|^2.
%  g(\langle X, a \rangle_{\mathcal{A}}),
%  g(\langle X^\star, a \rangle_{\mathcal{A}}+\eta\epsilon)).
\end{equation*}
Thus averaging over the data distribution and noise, we have
\begin{equation}
    \min_{X \in \mathbb{R}^d} 
    \bigg \{ 
      \mathcal{R}_\delta(X) 
      =
      \tfrac12 \eta^2
      +
      \tfrac12\EE_{a} [\tr(\langle (X-X^\star)^{\otimes 2},a^{\otimes 2} \rangle_{\mathcal{A}^{\otimes 2}})] + \tfrac{\delta}{2} \|X\|^2 \bigg \}.
      %, \quad \text{where $y = a \cdot X^{\star} + \varepsilon$}. 
\end{equation}
We note that this can be further simplified to be 
\[
  \mathcal{R}_\delta(X) 
  =
  \tfrac12 \eta^2
  +
  \tfrac12 \tr(\langle(X-X^\star)^{\otimes 2},K\rangle_{\mathcal{A}^{\otimes 2}}) + \tfrac{\delta}{2} \|X\|^2.
\]
In this case, the pair $h$ and $I$ can be evaluated simply:
\[
  h = \tr(\langle (X-X^\star)^{\otimes 2},K\rangle_{\mathcal{A}^{\otimes 2}})
  \quad
  \text{and}
  \quad
  I = \langle (X-X^\star)^{\otimes 2},K\rangle_{\mathcal{A}^{\otimes 2}},
\]
noting that both of these are linear functions of the block matrix $B(W) = \langle (X\oplus X^{\star})^{\otimes 2},K\rangle.$

The deterministic dynamics \eqref{eq:coupledODElimit} can be rearranged to give a particularly simple equation in this case.  For simplicity, we take $\delta = 0.$
Then we can express the loss $h$ as
\[
  h(\mathrsfs{B}(t)) = \langle (\Id_\mathcal{O} \oplus -\Id_\mathcal{T})^{\otimes 2}, \mathrsfs{B}(t) \rangle
  = \tr \mathrsfs{B}_{11}(t) - 2 \tr\mathrsfs{B}_{12}(t) + \tr\mathrsfs{B}_{22}(t).
\]
This leads us to (see Section~\ref{sec:lsq_analysis} for details)
\[
%\begin{aligned}
  h(\mathrsfs{B}(t)) %& 
= 
\tfrac{1}{2} \tr( \ip{(X_0-X^{\star})^{\otimes 2}, K e^{-2K \gamma t} }_{\mathcal{A}^{\otimes 2}} )  + \tfrac{1}{2} \eta^2 
%\\
%&\qquad 
+ \tfrac{\gamma^2}{d} \int_0^t \tr(K^2 e^{-2 \gamma K(t-s)} ) h(\mathrsfs{B}(s)) 
\, \dif s.
\]
This is a convolution Volterra equation, and it has appeared earlier in \cite{CollinsWoodfinPaquette01,PPAP02,PPAP01, PAquettes03}, in the case of univariate linear regression.
The descent threshold of this equation is simply $\gamma < \tfrac{2d}{\tr K}$.
Note this agrees with the stability threshold in Corollary \ref{cor:convexbounded} up to a factor of $2$.
Under the assumption that $K \succ 0,$ we also have that it converges linearly to $0$,
and this rate of convergence can be determined from solving a certain \emph{Malthusian exponent} problem. Taking $\gamma = d/({\tr K})$, the asymptotic rate is guaranteed to be at least $e^{-\lambda_\text{min}(K)\tfrac{d}{4\tr K}}.$ This objective function is $(1,\infty)$--RSI, and hence Proposition \ref{prop:RSI} gives an equivalent result up to absolute constant factors.  This is sharp up to an absolute constant in the exponent.

\subsection{Multi-class logistic regression.} 
\label{sec:logistic_regression}

\begin{figure}[t]
    \centering
        \includegraphics[scale =0.3]{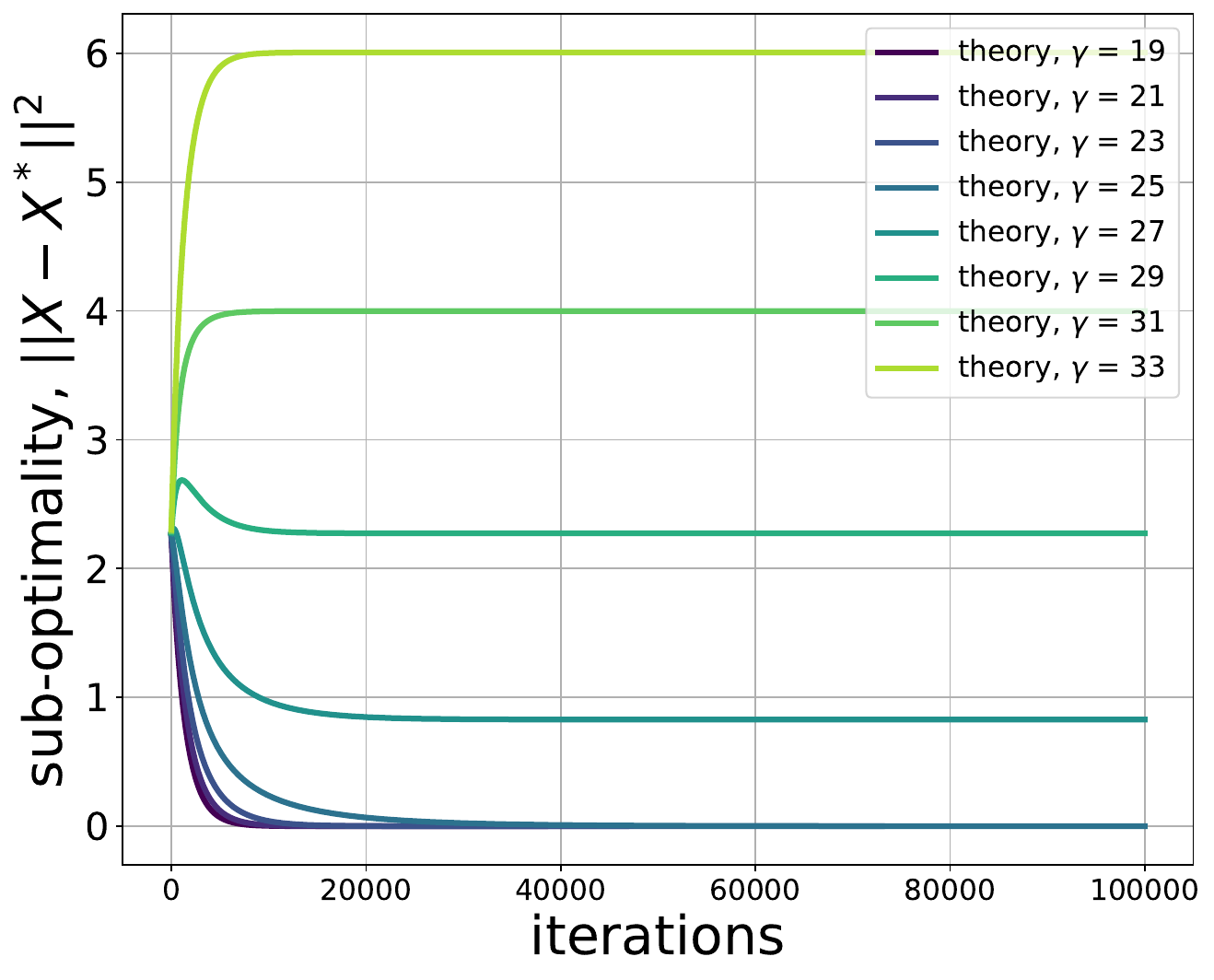} \qquad \includegraphics[scale =0.3]{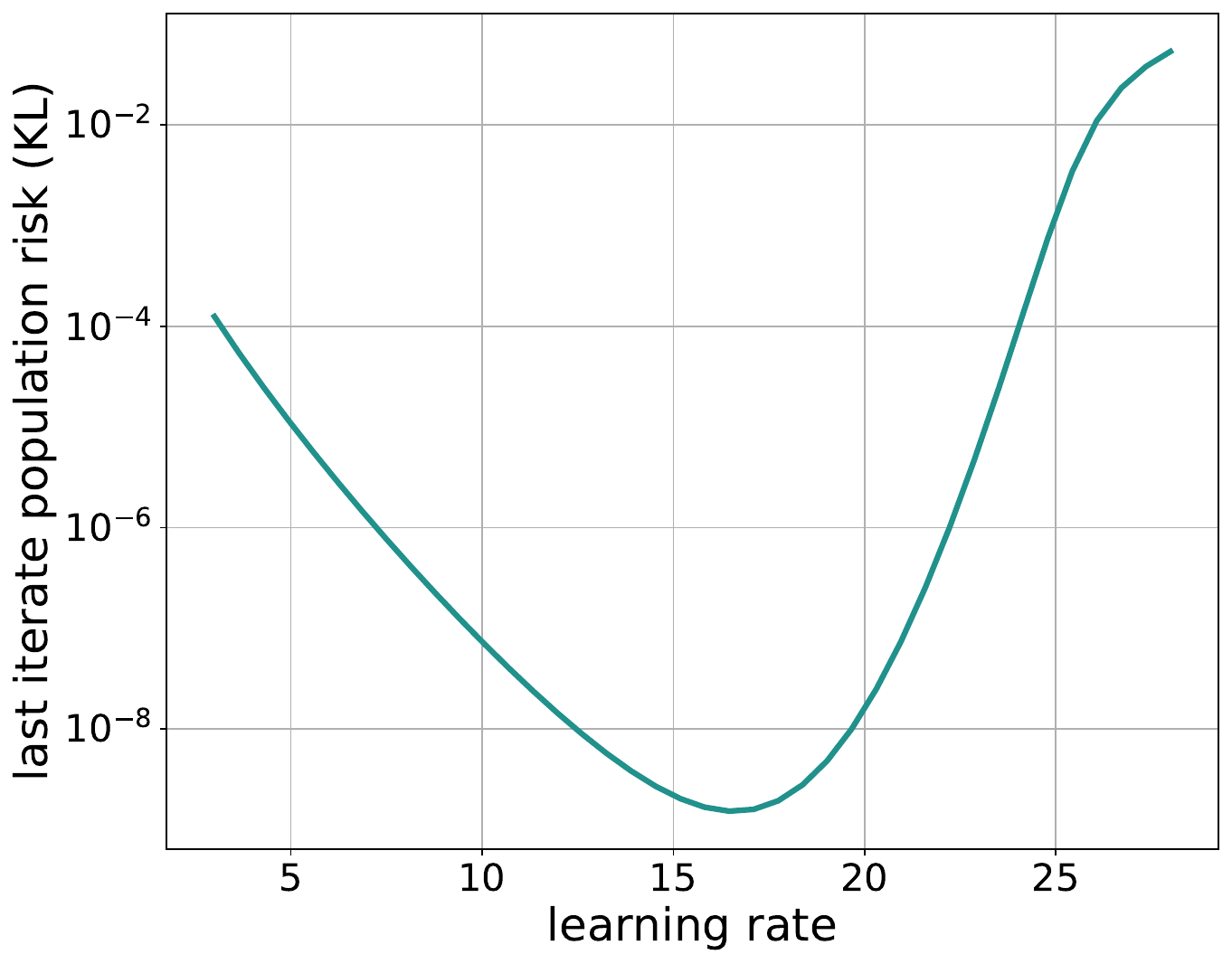}
    \caption{\textbf{Learning rate and stability of logistic regression descent.} Plot of the theory for various learning rates for the noiseless, binary logistic regression problem initialized at $  1.1 \cdot X_0 / \|X_0\|$ with $X_0 \sim N(0, \Id_d)$, $d = 1000$. The ground truth signal is also normally distributed, $X^{\star} \sim \tfrac{1}{\sqrt{d}}N(0, \Id_d)$. The covariance matrix is generated from Marchenko-Pastur (MP) with parameter $4$. \textbf{(Left):} Curves for $\mathrsfs{D}^2(t)$ are plotted for various learning rates $\gamma$. As predicted by Corollary \ref{cor:convexbounded}, there exists a learning rate at which $\mathrsfs{D}^2(t)$ is a decreasing function. Theory guarantees this to occur at $1/ (\hat{L}(f)  \tr(K)/d) \approx 12$ (Here $ \tr(K)/d \approx 1/3$, max. eigenvalue of $K$ is $0.75$, and smoothness constant is $\hat{L}(f) = 1/4$). \textbf{(Right):} last iterate of deterministic curve for the KL divergence, at $t = 25$, is plotted. The optimal learning rate occurs approximately $1/2$ the learning rate threshold where descent $\mathrsfs{D}^2(t)$ occurs.
    } \label{fig:logistic_regression_learning_rate}
\end{figure}

An important and motivating example is logistic regression. 
In this case, the dimension $\ell$ of $\mathcal{O}$ corresponds to the number of classes;
we let $\{o_j\}$ denote an orthonormal basis of $\mathcal{O}$. The data arrives in a pair $(a, y)$, a point $a$ in the feature space and a probability vector $y$, whose coordinates $\langle y, o_j\rangle$ correspond to the probability that $a$ comes from class $j$. 
We then look to fit an exponential model $p(a;X)$ parametrically described by weights $X \in \mathcal{A} \otimes \mathcal{O}$, by the formula 
\begin{equation}
    p(a;X) = \frac{\exp(\ip{X,a}_{\mathcal{A}})}{\mathcal{Z}(X,a)} \in \mathcal{O},
\end{equation}
where $\exp$ is applied entrywise, and $\bm{1} =\sum o_j$, and so 
\begin{equation}
  % \tr( \exp( \ip{X, a \otimes \bm{1}}_{\mathcal{A}} ) ) 
  \mathcal{Z}(X,a)= \sum_{j=1}^\ell\exp(\ip{X, a\otimes o_j})
\end{equation}
is the sum of the exponentials, which ensures that $p(a;X)$ is indeed a probability vector.

The conventional loss to consider in this case is the KL-divergence,
and so we are brought, in a student-teacher setup, to
\[
  \hat{\Psi}(X;a,\epsilon) = \sum_{j=1}^\ell p_j(a;X^\star) \log \frac{p_j(a;X^\star)}{p_j(a;X)},
\]
where $p_j(a;X) = \ip{p({a;X}),o_j}.$
This differs from the cross-entropy only by a constant, namely
\[
  \Psi(X;a,\epsilon) = -\sum_{j=1}^\ell p_j(a;X^\star) \log{p_j(a;X)},
\]
which therefore has the same gradients.  
Setting $x_j = \langle X, a \oplus o_j \rangle$ 
and setting  $x_j^\star = \langle X^\star, a \oplus o_j \rangle,$
we have
\[
  \Psi(X;a,\epsilon) 
  =
  -\sum_{j=1}^\ell \biggl\{\tfrac{\exp(x_j^\star)}{\sum_i\exp(x_i^\star)}x_j\biggr\}
  +\log\biggl(\sum_{j=1}^\ell \exp(x_j)\biggr)
  \defas f( x \oplus x^\star).
  %= f( \langle X\oplus X^\star, a\rangle),
  %\quad 
  %\text{where}
  %\quad
  %f
\]
Cross-entropy is convex and attains a global minimizer at $x^\star$, but also at $x^\star + \alpha \mathbf{1}$ for any $\alpha.$
In the ambient space, we can let $\hat X = X^\star$ shifted to have the same center of mass as the initialization $X_0$ of SGD, i.e. for some $v \in \mathcal{A},$
\[
\hat X = X^{\star} + v \otimes \mathbf{1}
\quad \text{where}\quad
\ip{\hat X, \mathbf{1}}_{\mathcal{O}} 
=
\ip{X_0, \mathbf{1}}_{\mathcal{O}}.
\]
Then $p(a; X^\star) = p(a; \hat X)$.
Since $\nabla_\ex f$ gradient is orthogonal to $\mathbf{1}$,
this property is preserved by the optimization, i.e.\ both SGD and homogenized SGD have $\ip{\hat X, \mathbf{1}}_{\mathcal{O}} 
=
\ip{\WHSGD_t, \mathbf{1}}_{\mathcal{O}}$ for all time.
It follows that Assumption \ref{assumption:risk_loss_minimizer_1_main}
is satisfied with this minimizer.  The Lipschitz constant is known to be given by $1$ (see \cite[Chapter 5]{Beck}), and so we have a stability threshold given by
\[
\bar \gamma = \frac{1}{\tfrac1d \tr(K)}.
\]
by Corollary \ref{cor:convexbounded}. Figure~\ref{fig:logistic_regression_learning_rate} numerically supports this result (up to constants). 

We further claim that the outer function $f$ has a local RSI constant; we note that it suffices to do this for $x$ so that $x-\hat x$ is orthogonal to $\mathbf{1}$. Setting $\mathcal{Z} = \ip{ \exp(x), \mathbf{1}}$ and similarly for $\hat{\mathcal{Z}},$
\[
\ip{ x - \hat x, \nabla_\ex f(\ex)}
=
\ip{ x - \hat x, \frac{e^x}{\mathcal{Z}} - \frac{e^{\hat x}}{\hat{\mathcal{Z}}}}
=
\ip{ x - \hat x + \alpha \mathbf{1}, \frac{e^x}{\mathcal{Z}} - \frac{e^{\hat x}}{\hat{\mathcal{Z}}}},
\]
for any $\alpha \in \R$.  Setting $p = \frac{e^x}{\mathcal{Z}}$ and similarly for $\hat p,$ we thus have
\[
\ip{ x - \hat x, \nabla_\ex f(\ex)}
=
\ip{ \log \frac{p}{\hat p}, p - \hat p}.
\]
Now $\log(p_j/\hat{p}_j) \leq \frac{p_j- \hat{p}_j}{\hat{p}_j}$.  So for coordinates $j$ where $p_j > \hat{p}_j$, we may apply this bound to lower bound the contribution to the inner product by $\log(p_j/\hat{p}_j)^2 \hat{p}_j$.  We may do the same to coordinates where $p_j < \hat{p}_j$ after reversing the roles of the two, and so we conclude that with $u = 
\min \{ \hat{p}_j, p_j\}$,
\[
\ip{ x - \hat x, \nabla_\ex f(\ex)}
\geq
u
\|\log \frac{p}{\hat p}\|^2
=
u\| x - \hat x +\log(\hat{\mathcal{Z}}/\mathcal{Z})\mathbf{1}\|^2
\geq
u\| x - \hat x\|^2,
\]
where the final line follows since $x-\hat x$ is orthogonal to $\mathbf{1}.$  
Now if $\|x-\hat x\|^2 \leq \theta$ and $\|\hat x\|^2 \leq \theta,$ then it follows that $\|x\|_\infty$ and $\|\hat x\|_\infty$ are less than $\sqrt{2 \theta}$.  For these bounds, it follows that logistic regression is $(\mu,\theta)$--RSI with
\[
\mu = \tfrac{1}{\ell e^{\sqrt{4\theta}}}.
\]
Hence we have shown using Proposition \ref{prop:RSI}:
\begin{proposition}[Local convergence of logistic regression] \label{prop:logistic_convergence_rate}
Suppose $\hat X$ is the minimizer of $\mathcal{R}(X)$ with the same center of mass as $X_0,$ and set $\theta = 64 \|K\|_\sigma^2\max\{\|\hat X\|^2, \|X_0\|^2\}$.
Then for 
\[
\gamma_t = \gamma =
\frac{e^{-\sqrt{4\theta}}}{\tfrac \ell d \tr K},
\]
and for $a = c \frac{e^{-4\sqrt{\theta}}}{\tfrac {\ell^2}{ d }\tr K} \lambda_{\min}(K)$,
we have for all $t \ge 0$
\[
\mathrsfs{D}^2(t) \le 2 e^{-at} \|X_0-X^{\star}\|^2.
\]
\end{proposition}
\noindent Unlike for descent threshold, here the operator norm of $K$ plays a role.  
The root of this problem is that for heavily distorted spectral distributions (in particular with many large eigenvalues but with bounded average-trace), the $K$-norm $\langle (\Id_\mathcal{O} \oplus -\Id_\mathcal{T})^{\otimes 2}, \mathrsfs{B}_t \rangle$ might grow quite large.  This in turn pushes the state of SGD to regions where the probabilities $\{p_j\}$ are very close to the extremes $\{0,1\}$, which in turn compresses the gradients (exponentially in the parameters $\|x\|$).

\begin{remark}
Another way to handle the overparameterization is to pin one column at $0$: we could subtract the final column of $X$ from all other columns to produce the same output, i.e. $p(a;X)=p(a;X - \langle X, o_\ell \rangle \otimes \bm{1})$. Hence, one can also work on an $(\ell-1)$--dimensional space $\mathcal{O}$, which is embedded in the $\ell$-dimensional space above, by adding a $0$-column.
In the specific case of two-class logistic regression, this brings us to the problem of \emph{binary logistic regression}, in which $\ell=1$ and the loss is given by
\[
  \begin{aligned}
  \Psi(X;a,\epsilon) 
  &= 
  -\frac
  {\exp(\ip{X^\star,a}_{\mathcal{A}})}
  {\exp(\ip{X^\star,a}_{\mathcal{A}} + 1)}
  \log
  \biggl(
  \frac
  {\exp(\ip{X,a}_{\mathcal{A}})}
  {\exp(\ip{X,a}_{\mathcal{A}} + 1)}
  \biggr)
  -
  \frac
  {1}
  {\exp(\ip{X^\star,a}_{\mathcal{A}} + 1)}
  \log
  \biggl(
  \frac
  {1}
  {\exp(\ip{X,a}_{\mathcal{A}} + 1)}
  \biggr) \\
  &= 
  -\frac
  {\exp(\ip{X^\star,a}_{\mathcal{A}})}
  {\exp(\ip{X^\star,a}_{\mathcal{A}}) + 1}
  %\log
  %\bigl(
  %{\exp(
    \ip{X,a}_{\mathcal{A}}
  %)}
  %\bigr)
  +
  \log
  \bigl(
  {\exp(\ip{X,a}_{\mathcal{A}}) + 1}
  \bigr). 
  \end{aligned}
\]
\noindent Some simplification of $h$ and the $I$ are given in Section \ref{sec:analysis_examples},
but ultimately these must be left as unevaluated Gaussian integrals.
\end{remark}

Logistic regression is a well--studied problem.  Information theoretic recovery bounds are known to exist \cite{candes2020phase} in the proportional scaling done here; in particular one needs sufficiently many samples $n > \alpha d$ for some $\alpha$ depending on $X^\star$ to have an MLE on taking $d\to\infty$.  It is not clear if any such transition in the high-dimensional SGD dynamics, which do not appear to display a phase transition, possibly suggesting some implicit regularization. See also extensions to regularized logistic regression \cite{salehi2019impact} (see also \cite{montanari2019generalization}).

\subsection{Lipschitz phase retrieval.}

%\afterpage{\clearpage} 

\begin{figure}[t]
    \centering
        \includegraphics[scale =0.25]{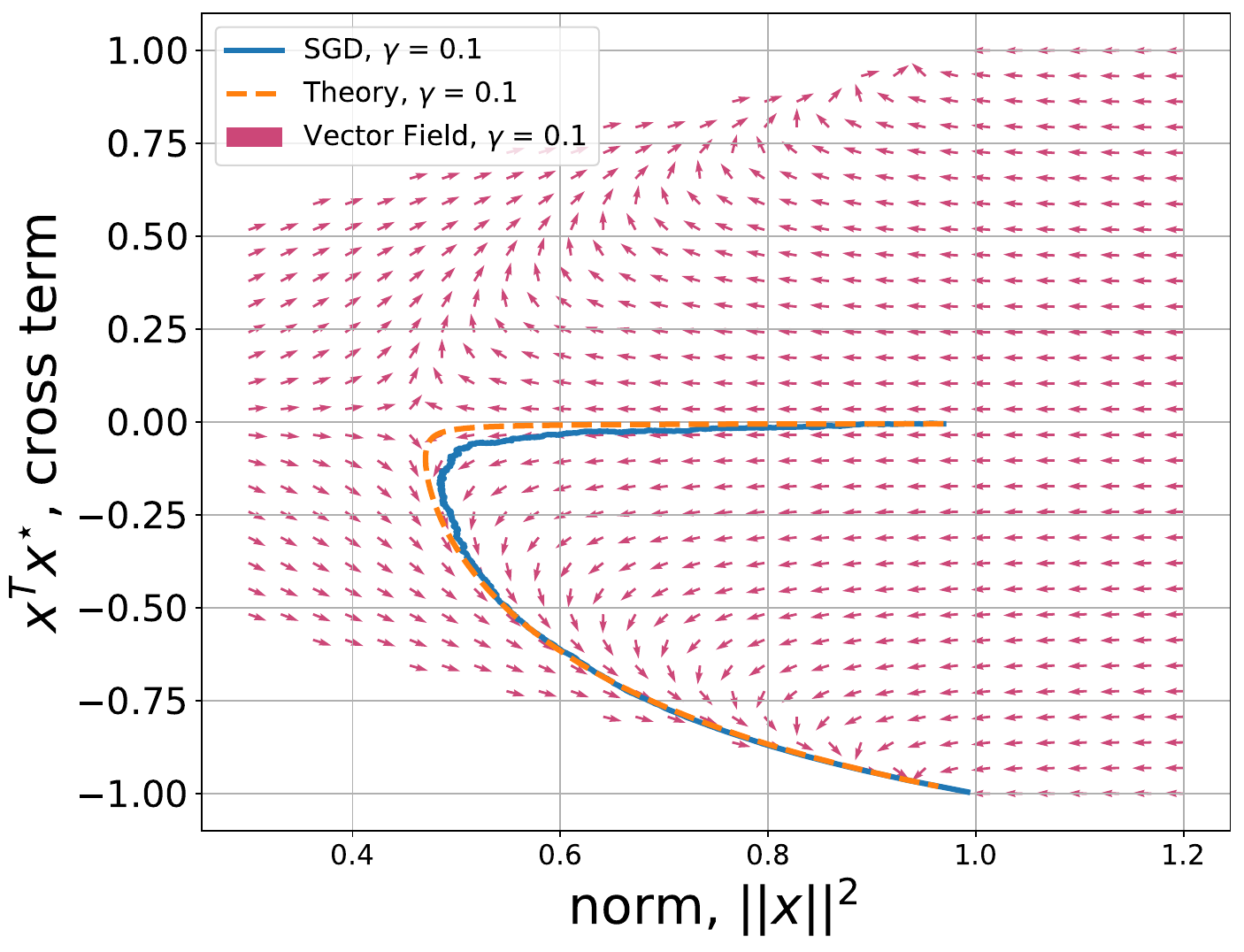}
        \qquad 
    \includegraphics[scale =0.25]{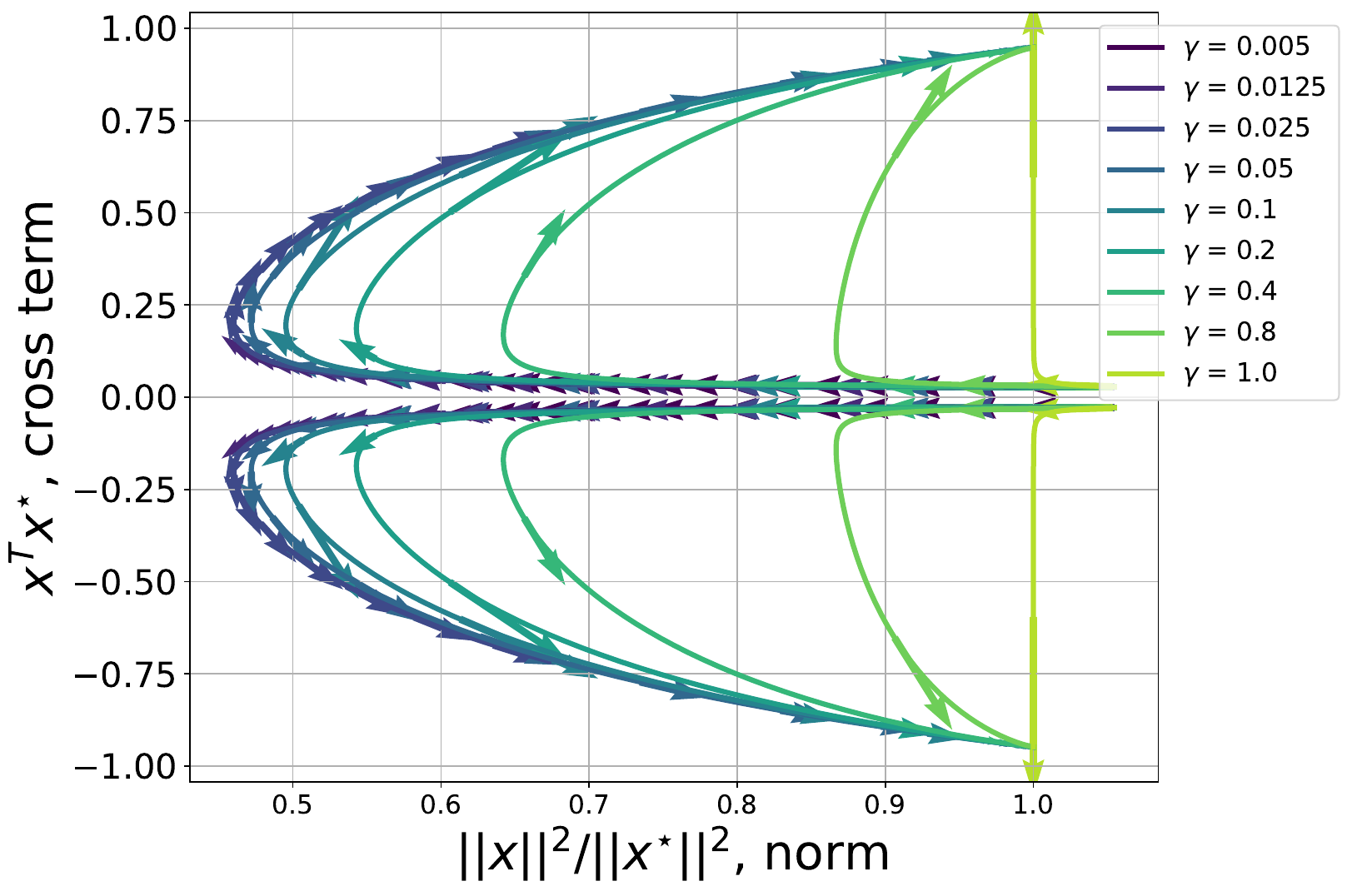}  \includegraphics[scale =0.25]{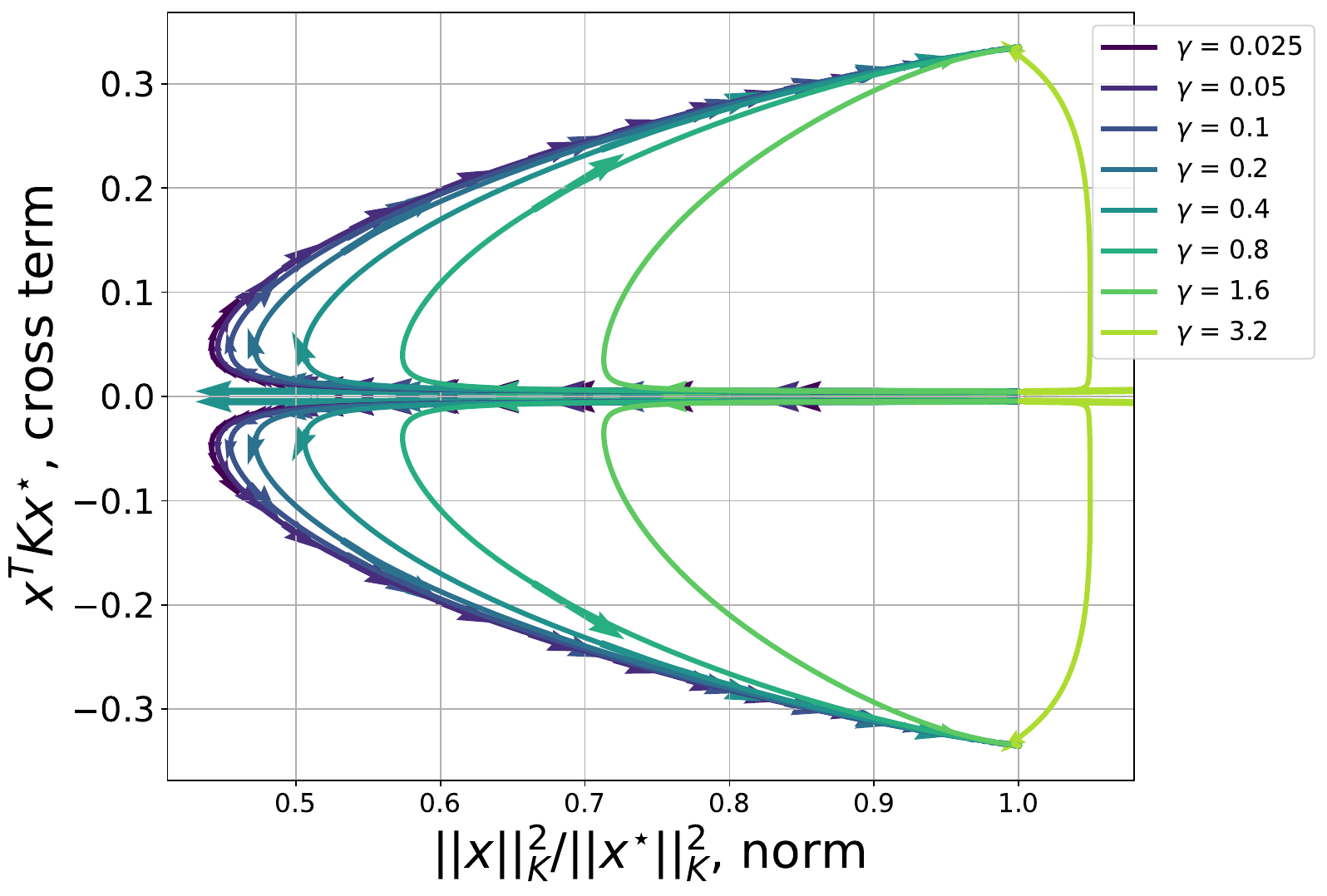} 
   \caption{\textbf{Evolution of the norm versus cross terms.} Plot of the theory for various learning rates for the noiseless phase retrieval problem initialized at $ \pm X_0 / \|X_0\|$ with $X_0 \sim N(0, \Id_d)$, $d = 2000$. The ground truth signal is also normally distributed, $X^{\star} \sim N(0, \Id_d)$. The \textbf{(top row)} are with identity covariance and the \textbf{(bottom row)} has a covariance matrix generated from Marchenko-Pastur (MP) with parameter $4$. The initialization is such that cross-term is initially $0$. All the trajectories converge to either $\pm \|X^{\star}\|$. The trajectories follow a path of first decreasing the norm $\|X\|^2_K = \tr(\ip{X^{\otimes 2}, K}_{\mathcal{A}^{\otimes 2}})$ until some fixed value ($\pi^2 / 4$, identity) and then SGD starts to match the cross term, i.e., $\tr(\ip{X \otimes X^{\star}, K}_{\mathcal{A}^{\otimes 2}}) \to \pm 1$. There exists critical learning rates, $\gamma = 1$ (identity covariance) and $\gamma \approx 3.2$ (MP covariance), such that no movement is observed and the SGD algorithm immediately starts making the cross term $\pm 1$. As learning rate $\to 0$, the trajectories start to behave as gradient flow. 
    } \label{fig:phase_retrieval_vector_field}
\end{figure}

The phase retrieval problem is to recover an underlying signal from linear observations of the modulus of the signal. This is a classic example in optimization theory, in that it is generally tractable to analyze but is nonconvex.  There are multiple formulations, but we consider the following ``Lipschitz'' version
(see also \cite{DDP} for the similar ``robust'' version), with no noise:
\begin{equation} \label{eq:phase_retrieval_nonsmooth_loss}
\mathcal{R}(X) \defas  \tfrac{1}{2} \EE_a[ \big ( |\ip{X,a}_{\mathcal{A}}| - |\ip{X^{\star},a}_{\mathcal{A}}| \big )^2 ]. 
\end{equation}
\noindent Here we take $\mathcal{O} = \mathcal{T} = \R$.

We can explicitly represent the risk in terms of the scalar overlap variables of $B$
\[
B(W) = \ip{W \otimes W, K}_{\mathcal{A}^{\otimes 2}} = 
\begin{pmatrix} 
B_{11}(W) & B_{12}(W)\\
B_{21}(W) & B_{22}(W)
\end{pmatrix}.
\]
We often drop the $W$ in $B$ when it is clear from context. The risk is then given by (using the symmetry of the inputs).
\begin{align*}
    \mathcal{R}(X) = h \left ( \begin{pmatrix} B_{11} & B_{12}\\
    B_{21} & B_{22}
    \end{pmatrix} \right ) 
    &
    = 
    \tfrac{1}{2} B_{11} + \tfrac{1}{2} B_{22} 
    -
    \tfrac{2}{\pi} \left ( {B_{12}} \arcsin \left (\tfrac{B_{12}}{\sqrt{B_{11}} \sqrt{B_{22}} } \right ) + \sqrt{ B_{11}B_{22} - B_{12}^2} \right).\\
\end{align*}
Note in particular that we lose differentiability at the extreme $B_{12}^2 = B_{11}B_{22}$ as well as at $B_{11}=0$ at which the $\arcsin$ degenerates to a step function.  So in particular to apply the theory in this paper to this example, we need to work on a set away from $\mathcal{U}$ given by
\[
\mathcal{U} \defas \{ B: B_{11} > 0, B_{12} < \sqrt{B_{11} B_{22}}\}.
\]
(Here we assume that $B_{22}$ is nonzero).

Computing the derivatives,
\footnote{ On differentiating $h$ with respect to $B_{12}$, one gets twice this formula for $H_2$.  The factor of $2$ is explained by needing to represent $h$ as a symmetric function of its inputs $B_{12}$ and $B_{21}$, and then treating these as independent variables and which effectively divides the derivative in $2$.}
\begin{align*}
H_1 = 
\frac{1}{2} - \frac{1}{\pi} \sqrt{\frac{B_{22}}{B_{11}} - \frac{B_{12}^2}{ B_{11}^2 } }
\quad \text{and} \quad
H_2 = 
-\frac{1}{\pi} \arcsin \left ( \frac{B_{12}}{ \sqrt{B_{11}} \sqrt{B_{22}} } \right ).
\end{align*}
It can also be checked that 
\[
\EE_a [\nabla_x f(\ip{X,a}_{\mathcal{A}})^{\otimes 2} ] = 2\mathcal{R}(X),
\]
and hence $I=2h.$

%The dynamics example displays a natural saddle manifold, where $B_{12}$ is $0$.
%Using \eqref{eq:coupledODElimit},
%\begin{equation*}
%  \begin{aligned}
%  &\dif \mathcal{Q}_{t,i} = -2\lambda_i^2\gamma( 2\mathcal{Q}_{t,i} H_{1,t} + +\mathcal{V}_{t,i} H_{2,t})
%  + \lambda_i^2\gamma^2I_t, \\
%  &\dif \mathcal{V}_{t,i} = -2\lambda_i^2\gamma( H_{1,t} \mathcal{V}_{t,i} + H_{2,t}\mathcal{T}) , \\
%\end{aligned}
%\end{equation*}
%where we have 
%$\sum_i \lambda_i \mathcal{Q}_{t,i} = B_{11}(t)$ and
%$\sum_i \lambda_i \mathcal{V}_{t,i} = B_{12}(t).$

The dynamics example displays a natural saddle manifold, where $B_{12}$ is $0$.
Simplifying to the case of $K=\Id$, and constant learning rate for clarity,
Using \eqref{eq:coupledODElimit},
\begin{equation*}
  \begin{aligned}
  &\dif \mathrsfs{B}_{11}(t) = -2\gamma( 2\mathrsfs{B}_{11}(t) H_{1,t} +\mathrsfs{B}_{12}(t) H_{2,t})
  + \gamma^2 I_t, \\
  &\dif \mathrsfs{B}_{12}(t) = -2\gamma( H_{1,t} \mathrsfs{B}_{12}(t) + H_{2,t}\mathrsfs{B}_{22}) , \\
\end{aligned}
\end{equation*}
where we have 
$\mathrsfs{B}_{11}(t) = B_{11}(W_{\lfloor td \rfloor})$ and
$\mathrsfs{B}_{12} = B_{12}(W_{\lfloor td \rfloor}).$
In particular if we initialize $B_{12}(W_0) = 0,$ then $\mathrsfs{B}_{12} = H_2 = 0$ identically. Thus, in particular the limit dynamics are trapped close to this axis and, in fact, converge to a saddle point defined by (with $\beta = \tfrac{4}{\gamma}$)
\[
\beta B_{11} \biggl(\frac{1}{2} - \frac{1}{\pi} \sqrt{\frac{B_{22}}{B_{11}}}\biggr) =  
\biggl(\tfrac{1}{2} B_{11} + \tfrac{1}{2} B_{22} -\tfrac{2}{\pi}  \sqrt{ B_{11}B_{22} }\biggr)
\,\implies\,
\pi\sqrt{\frac{B_{22}}{B_{11}}}
=2-\beta\pm \sqrt{\beta^2 -(\pi^2-4)(1-\beta)}.
%\quad \beta=\tfrac{4}{\gamma}.
\]
Initializing off of this manifold allows the process to escape linearly provided $\gamma$ is small enough that $H_{2,t} \mathcal{T}$ can exceed $H_{1,t}\mathcal{V}_t$.  
Approximating $H_{2,t}$ in small $B_{12}$ shows that this threshold is determined by
\[
%\frac{1}{\pi}
\sqrt{\frac{B_{22}}{B_{11}}} > 
%\biggl(
\frac{\pi}{4}.
%\biggr),
\]
This, in particular, is always satisfied at the saddle point for small $\gamma$ (at which $\sqrt{\frac{B_{22}}{B_{11}}} \approx \tfrac{\pi}{2}$) (see Figure~\ref{fig:phase_retrieval_vector_field} (top row) when $\gamma$ is small).  Hence, for large initial $B_{11}$ and small $B_{12} \approx \tfrac{1}{\sqrt{d}}$ (which can be guaranteed by random initialization), SGD first pushes $B_{11}$ towards the saddle.   Then it begins to develop a nontrivial overlap $\mathrsfs{B}_{12}(t)$, which then grows exponentially. These dynamics can be explicitly seen in Figure~\ref{fig:phase_retrieval_vector_field} (see this discrepancy of our theory and SGD in Figure~\ref{fig:phase_retrieval_concentration}). 
As it is initialized with $B_{12}$ small in $d,$ HSGD requires $O(\log d)$ time to reach equilibrium (and SGD requires $O(d\log d)$ steps).
See \cite{tanvershynin} in which this is proven rigorously directly for SGD (in part by explicitly considering a diffusion approximation like homogenized SGD).  See also \cite{ArousInformationExponent} in which a general class of related singular models is given, in which $O(d \log d)$ -- or even $O(d^\alpha)$ steps for $\alpha > 1$ -- is required.  

One solution to this problem is to do a ``warm start'' using a spectral method.  This has been shown rigorously to lead to linear sample complexity when combined with gradient methods \cite{candes2015phase}. 
 See also \cite{mondelli2021approximate} for similar considerations in the approximate message passing setting.

There are known information theoretic bounds for the phase retrieval problem. Especially for smooth isotropic phase retrieval, one needs at least $d$ samples to recover any signal in the problem \cite{maillard2020phase}.  
By increasing the amount of overparameterization in the ``student'' network, which is to say one rather considers a sum $\sum_{1}^m |\langle X_j, a\rangle_{\mathcal{A}}|$ for a family of $m$ parameters $(X_j : 1 \leq j \leq m)$ in $\mathcal{A}$, one can improve the rate. See especially \cite{sarao2020optimization}, \cite{damian2023smoothing} and \cite{Bruno} for various investigations of how to improve the landscape in these cases.

%See also \cite{Bruno} for an analysis of a multichannel version of smooth phase retrieval with similar conclusions.  With heavily overparameterized student networks
%sarao2020optimization

\begin{remark}
We note that Theorem \ref{Thm:SGD_HSGD_convergence} does not apply to super-linear time scales in $d$.  In some cases, it is possible to extend the range to $\epsilon d \log d$ for a small absolute constant $\epsilon.$ Nonetheless,  Theorem \ref{Thm:SGD_HSGD_convergence} does show that with small $B_{12}$ initialization, the process does tend towards the saddle (and reaches any small neighborhood in linear time) and it also shows that with a warm start, the process converges linearly (see Figure~\ref{fig:phase_retrieval_concentration} for numerical support).
%To see why not, SGD could jump between the portion of the phase space with $B_{12} > 0$ and $B_{12} < 0$, provided $B_{12}$  
\end{remark}

% \clearpage

\begin{figure}[t]
    \centering
        \includegraphics[scale =0.3]{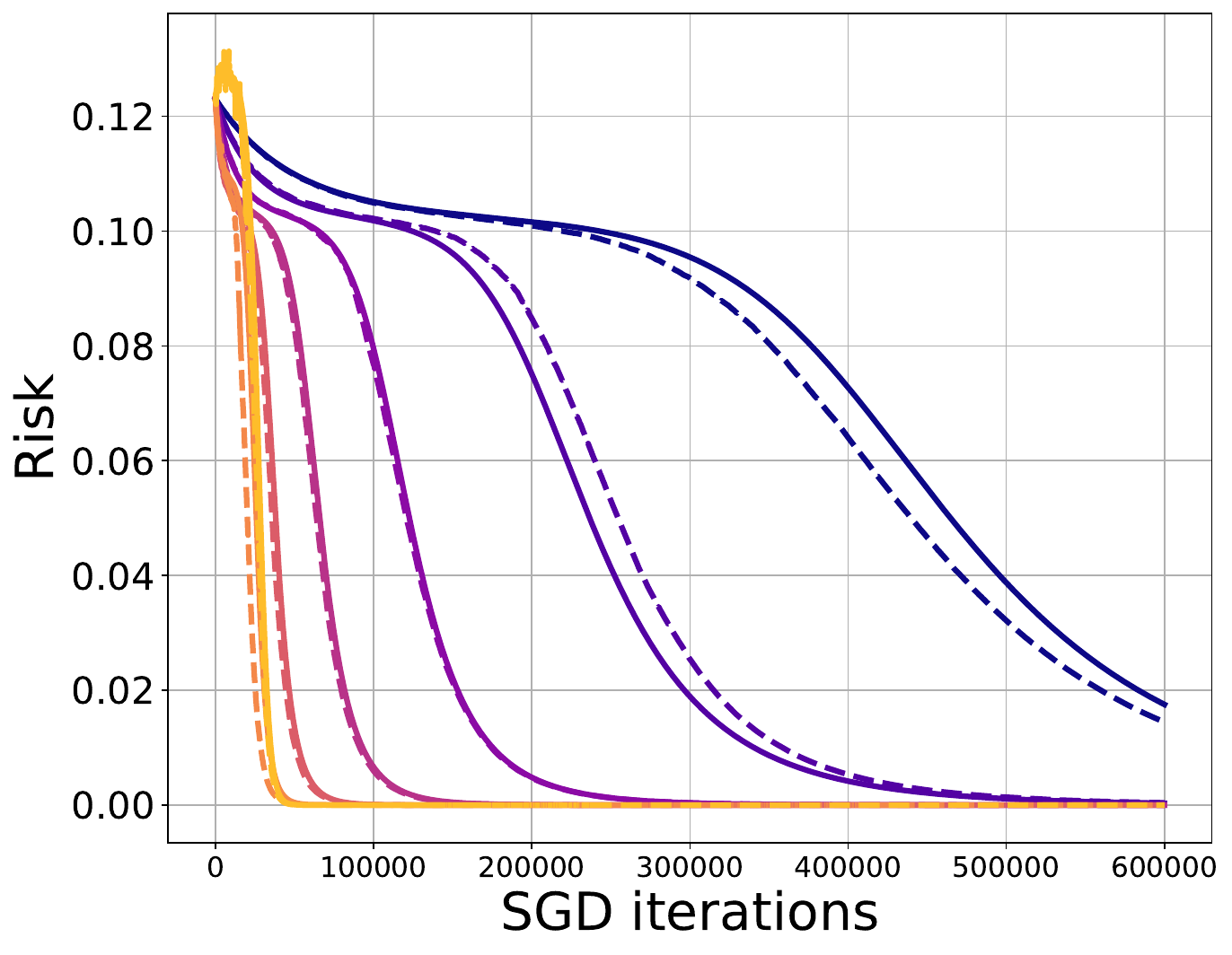}
    \includegraphics[scale =0.3]{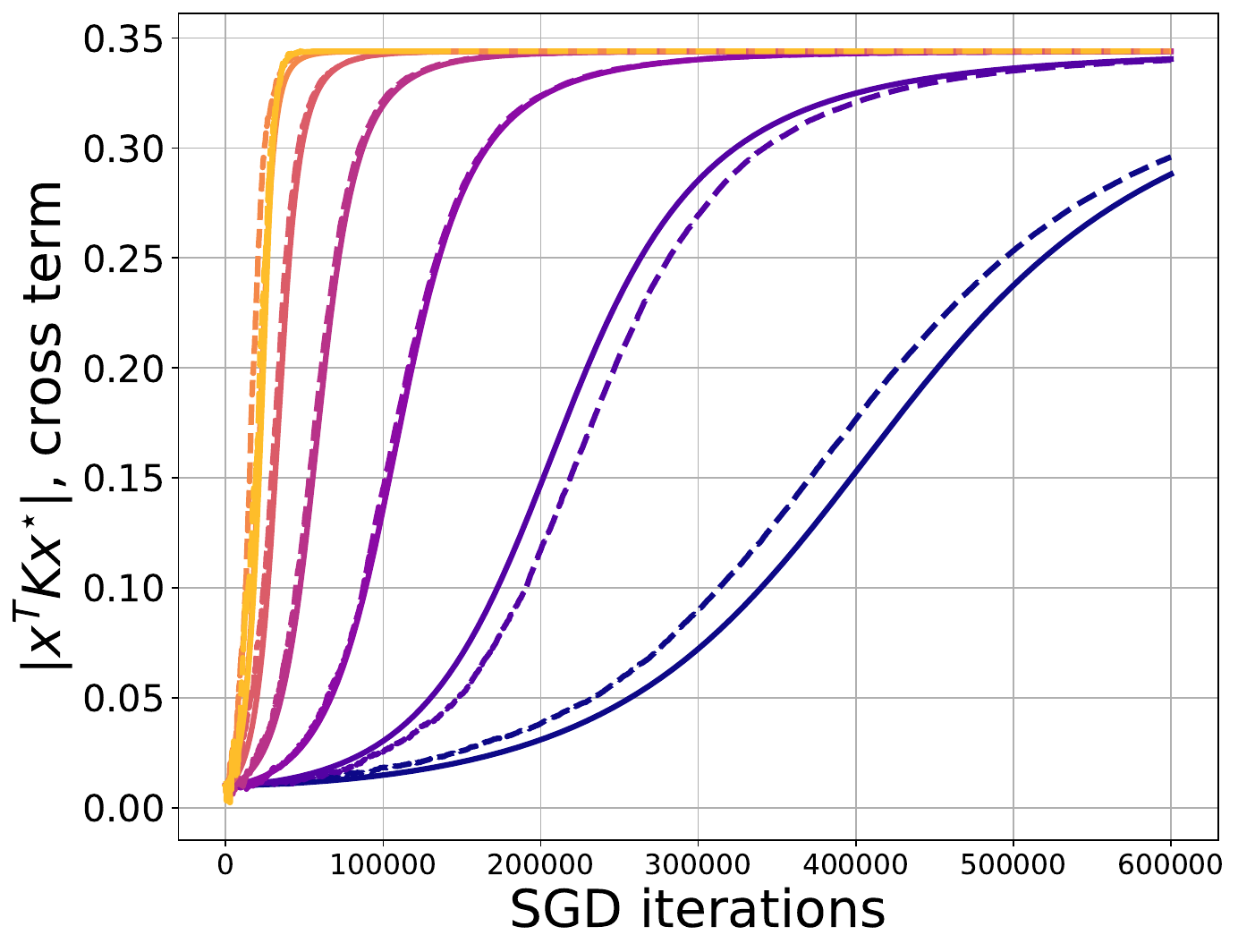} \\
    \includegraphics[scale =0.3]{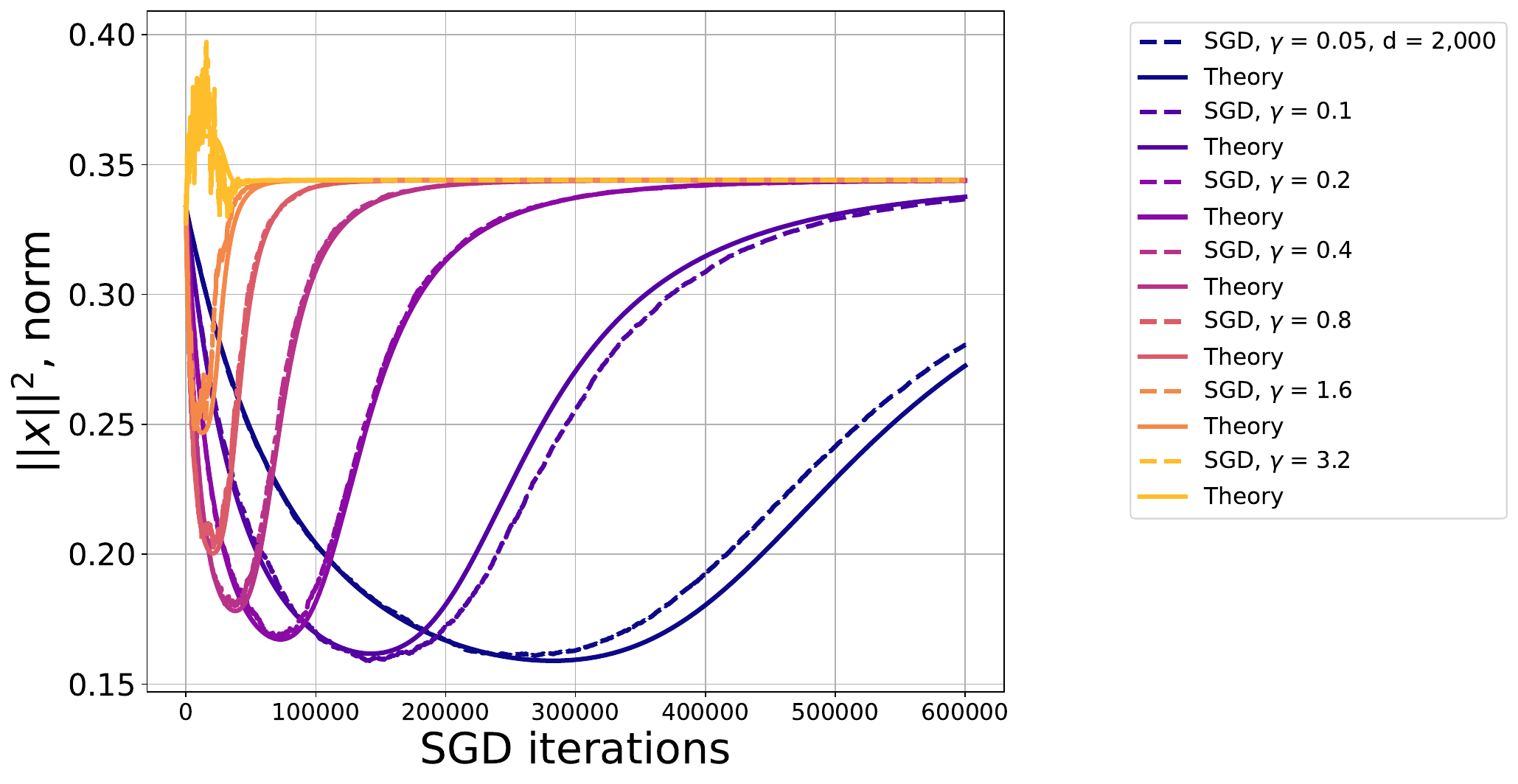} 
    \caption{\textbf{SGD versus Theory.} Plot of SGD in comparison with theory for various learning rates for the noiseless phase retrieval problem initialized at $X_{0,i} = \frac{1}{\sqrt{d}}$ for $i = 1, \hdots, d = 2000$. The ground truth signal is also normally distributed, $X^{\star} \sim \tfrac{1}{\sqrt{d}} N(0, \Id_d)$ and the covariance of the data $a$ is generated from Marchenko-Pastur (MP) with parameter $4$. Our prediction (theory), despite not expecting to match, has a good fit with SGD runs. 
    }\label{fig:phase_retrieval_concentration}
\end{figure}

%\end{aligned}
%This can be rearranged to give an integral equation for the deterministic equivalent of $\mathcal{R}$.  

%
%To describe the limiting dynamics, we fix a curve $\Gamma \subset \mathbb{C}$ which is distance $1$ from $[0, \bar K]$ where $\bar K$ is the constant from Assumption \ref{ass:data_normal}.
%We then introduce the process
%\begin{equation} \label{eq:SGD_obs}
%  S_k(z)
%  \coloneqq
%  \langle R(z;K), W_k^{\otimes 2} \rangle_{\mathcal{A}^{\otimes 2}}
%  \quad  R(z;K) \defas (K - z\Id_d)^{-1},\text{where } z \in \Gamma,
%  \quad k = 0, 1, 2, \hdots
%\end{equation}
%We note that that by contour integration, it is possible to recover two important observables from this:
%\begin{equation}\label{eq:obs_recov}
%  \langle W_k, W_k \rangle_{\mathcal{A}}
%  =\langle \Id_d, W_k^{\otimes 2} \rangle_{\mathcal{A}^{\otimes 2}}
%  =\frac{-1}{2\pi i}
%  \oint_{\Gamma}
%  S_k(z)\dif z
%  \quad\text{and}\quad
%  \langle K, W_k^{\otimes 2} \rangle_{\mathcal{A}^{\otimes 2}}
%  =\frac{-1}{2\pi i}
%  \oint_{\Gamma}
%  zS_k(z)\dif z.
%\end{equation}
%In block matrix notation, these are given by
%\[
%  \langle W_k, W_k \rangle_{\mathcal{A}}
%  =
%  \begin{bmatrix}
%    X_k^TX_k & X_k^T X^\star \\
%    X^{\star T}X_k & X^{\star T} X^\star \\
%  \end{bmatrix}
%  \quad\text{and}\quad
%  \langle K, W_k^{\otimes 2} \rangle_{\mathcal{A}^2}
%  =
%  \begin{bmatrix}
%    X_k^T K X_k & X_k^T K X^\star \\
%    X^{\star T}K X_k & X^{\star T} K X^\star \\
%  \end{bmatrix},
%\]
%which are well-known order parameters for describing the high-dimensional limits of streaming SGD in the isotropic $K=\Id_d$ case.

\subsection{Phase chase.} 

% \begin{figure}[t]
%     \centering
%     \includegraphics[scale =0.2]{figures/chase_phase_dynamics_risk.pdf} \qquad \includegraphics[scale =0.2]{figures/chase_phase_dynamics_norm.pdf} \\
%     \includegraphics[scale =0.2]{figures/chase_phase_dynamics_cross.pdf}
%     \caption{\textbf{SGD vs Theory on Chase Phase Problem}: The covariance $K$ is Marchenko-Pastur. }
%     \vspace{-0.5cm} \label{fig:chase_phase_problem}
% \end{figure}

In this problem, we consider an alteration of the phase-retrieval problem in which one trains both the $X$ and $X^\star$.  This can be considered as an idealization of a high-dimensional non-convex objective function with a high-degree of degeneracy in the set of minimizers (see \cite{agarwala2022second} for related quartic problems).
We can formulate this as the optimization problem:
\begin{equation} \label{eq:chasing_problem}
    \min_{X_1, X_2 \in \mathcal{A}} 
    \bigg \{ \mathcal{R}(X) = \EE_{a}\big ( \ip{X_1,a}_{\mathcal{A}}^2 - \ip{X_2,a}_{\mathcal{A}}^2 \big )^2\bigg \}.
\end{equation}
We have switched to the smooth formulation of phase retrieval for simplicity.

There are many solutions to this problem, all of which satisfy $X_1 = X_2$ or $X_1 = -X_2$, provided $K$ is non-degenerate (in the case of degenerate $K$, you get equality outside the kernel of $K$). Therefore, the dynamics of this problem are such that $X_1$ is \textit{chasing} $X_2$. 

\subsubsection{Dynamics of the $\mathrsfs{B}$ matrix for phase chase, non-symmetric} To understand these dynamics better and, in particular, the role of SGD noise, we invoke our homogenized SGD theorem. For this, we need the expressions for $h, \nabla h, \nabla f,$ and $\EE_a[ \nabla f(r)^{\otimes 2}]$. First, we note the target $X^{\star} = 0$ and thus, $B_{12} = \ip{X \otimes X^{\star}, K}_{\mathcal{A} \otimes 2}$ and $B_{22} = \ip{X^{\star} \otimes X^{\star}, K}_{\mathcal{A}^{\otimes 2}}$ are both identically $0$. This leaves the $B_{11} = \ip{X \otimes X, K}_{\mathcal{A}^{\otimes 2}}$ which is itself a $2 \times 2$ matrix and can be viewed as a norm and cross term with $X_1$ and $X_2$.  

With this in mind, we introduce notation to represent the norm and cross term between $X_1$ and $X_2$, as represented by a symmetric matrix, 
\begin{equation}
\begin{aligned}
B_{11} \defas Q = \begin{pmatrix} Q_{11} & Q_{12}\\
Q_{12} & Q_{22}
\end{pmatrix} = \ip{(X_1 \oplus X_2)^{\otimes 2}, K}_{\mathcal{A}^{\otimes 2}} = 
\begin{pmatrix} 
\ip{X_1 \otimes X_1, K}
&
\ip{X_1 \otimes X_2, K} \\
\ip{X_2 \otimes X_1, K}
&
\ip{X_2 \otimes X_2, K} \\
\end{pmatrix}.
\end{aligned}
\end{equation}
%where we use the $K$-norm, $\|\cdot\|_K = \ip{\cdot \otimes \cdot, K}_{\mathcal{A}^{\otimes 2}}$. 

Under this notation, we represent the function $h$ and $\nabla h$:
\begin{align*}
    h(Q, B_{12}, B_{22}) 
    &
    = 
    3( Q_{11}^2 + Q_{22}^2 ) - 2 ( Q_{11} Q_{22} ) - 4 Q_{12}^2
    \\
    (\nabla h)(Q, B_{12}, B_{22}) 
    &
    =
    \begin{pmatrix}
        6 Q_{11} - 2 Q_{22} & -4 Q_{12}\\
        - 4 Q_{21} & 6 Q_{22} - 2Q_{11}
    \end{pmatrix}.
\end{align*}

The expression for the function $f$ is simply
\[
f(r_1, r_2) = ( r_1^2 - r_2^2)^2 \quad \text{and} \quad \nabla f(r) = 4 (r_1^2 - r_2^2) \begin{bmatrix} r_1 \\
-r_2 
\end{bmatrix},
\]
where $r_1 = \ip{x_1,a}_{\mathcal{A}}$ and $r_2 = \ip{x_2,a}_{\mathcal{A}}$. An application of Wick's formula yields that 
\begin{equation}
\begin{gathered}
\EE_a [ \nabla f( \ip{a, X}_{\mathcal{A}})^{\otimes 2} ] 
= 
16 \begin{bmatrix}
G_{11} & G_{12}\\
G_{12} & G_{22}
\end{bmatrix} \\
\text{where} 
\qquad 
G_{11} = 15 Q_{11}^3 - 6 Q_{11}^2 Q_{22} - 24 Q_{11} Q_{12}^2 + 3 Q_{11} Q_{22}^2 + 12 Q_{12}^2 Q_{22}
\\
G_{12} = - ( 15 Q_{12} Q_{22}^2 + 15 Q_{12} Q_{11}^2 - 18 Q_{11} Q_{12} Q_{22} - 12 Q_{12}^3 )
\\
G_{22} = 15 Q_{22}^3 - 6 Q_{22}^2 Q_{11} - 24 Q_{22} Q_{12}^2 + 3 Q_{22} Q_{11}^2 + 12 Q_{12}^2 Q_{11}.
\end{gathered}
\end{equation}
%It is through these quantities that we can derive an expression for $\mathcal{S}$ when applied to homogenized SGD. The expected matrix $\mathcal{S}$ becomes
%\[
%\EE [ \dot{\mathcal{S}}] = -2 z \gamma \left ( \EE \big [ \mathcal{S} \Dif h + \Dif h \mathcal{S} \big ]\right ) + \gamma^2 \left ( \tfrac{1}{d} \tr(K R(z;K)) \right ) G + \text{constants}. 
%\]
Under the differential equations, note there is an important \textit{symmetry} between $Q_{11} = \|X_1\|_K^2$ and $Q_{22} = \|X_2\|_{K}^2$. Provided that at initialization $X_1$ and $X_2$ have the same norm value, the evolution of $Q_{11}$ will be the same as $Q_{22}$. In essence, we can simplify look at the dynamics of only two quantities $Q_{11}$ and $Q_{12}$ \textit{and} replace $Q_{22}$ with $Q_{11}$ in the expressions.

%\textcolor{red}{Moreover, the stepsizes act like a $\ell^2$-regularization with larger learning rates leading to smaller $K$-norm values of $x_1$ and $x_2$}

\subsubsection{Dynamics when $K = I$} 

We will see from homogenized SGD that the evolution of $Q$ has interesting properties. In particular, for SGD, there are nontrivial effects on the solutions to which it converges. This does not occur for gradient flow, and hence gradient descent-- all learning rates go to the same optimum. 

When the covariance is identity, the expressions for the dynamics of $Q$ simplify to the system of ODEs
\begin{equation} \label{eq:Q_phase_chase}
\begin{aligned}
\dot{Q_{11}} 
&
=
-16 \gamma ( Q_{11}^2 - Q_{12}^2 ) + 192 \gamma^2 (Q_{11}^2 - Q_{12}^2) Q_{11}
\\
\dot{Q_{12}} 
&
=
- 192 \gamma^2 (Q_{11}^2 -Q_{12}^2) Q_{12}.
\end{aligned}
\end{equation}
In comparison to gradient flow with speed $\gamma$, we have that
\begin{equation}
\begin{aligned}
\dot{Q_{11}} 
&
= 
- 16 \gamma (Q_{11}^2 - Q_{12}^2)
\\
\dot{Q_{12}}
&
=
0. 
\end{aligned}
\end{equation}
In both cases, we have $Q^2_{11}-Q^2_{12} \to 0$ although with SGD the rate is slowed.  In gradient flow, $Q_{12}$ remains fixed while under SGD $Q_{12}$ decays.  Hence SGD finds a lower norm solution than gradient flow, and hence can be compared in a sense to a form of implicit regularization, in that an $\ell^2$ regularizer does the same. See Figure~\ref{fig:phase_chase_problem} illustrating numerically these observations even in the non-identity covariance setting.
%Further, 
%In particular, we see that the rate at which $Q_{11}(t)-Q_{12}(t) \to 0$ is slowed down
%\[
%{(\dot{Q}_{11}-\dot{Q}_{12})} = -16 \gamma (Q_{11}^2 - Q_{12}^2)Q_{11} + 192 \gamma^2 (Q_{11}^2 - Q_{12}^2) (Q_{11}^2 + Q_{12}^2)
%\]

%We expect for both SGD and gradient flow that $Q_{11} = Q_{12}$ at the optimum, bu they go about it differently. As we see, for gradient flow (and hence gradient descent scaled by stepsize), the cross term $Q_{12}$ remains constant. The norm, $Q_{11}$, and the risk $\mathcal{R}$, do change, reflecting that \textit{for all stepsizes} gradient descent finds the optimum for which $Q_{11}(t) = Q_{12}(0)$. 

%On the other hand, SGD noise, as illustrated through the $\gamma^2$ terms, does four things:
%\begin{enumerate}
%    \item SGD noise slows down the rate at which $Q_{11}(t)-Q_{12}(t) \to 0$
%    \item The movement in the cross-term, $Q_{12}$, is solely due to the noise in SGD
%    \item Since both the cross term and norm move, SGD find an optimum where the first time $Q_{11}(t) = Q_{12}(t)$. Moreover, because of this, larger learning rates lead to slower movement in $Q_{11} \to Q_{12}$ and faster movement in $Q_{12}$. The result is optimum with lower $K$-norm values, that is, $\|x_1^{\star}\|_K$ and $\|x_2^{\star}\|_K$ have smaller values as learning rate $\gamma$ increases. In this sense, SGD is doing some form of implicit $\ell^2$-regularization. 
%\end{enumerate}

\begin{figure}[t]
\centering
    \includegraphics[scale = 0.25]{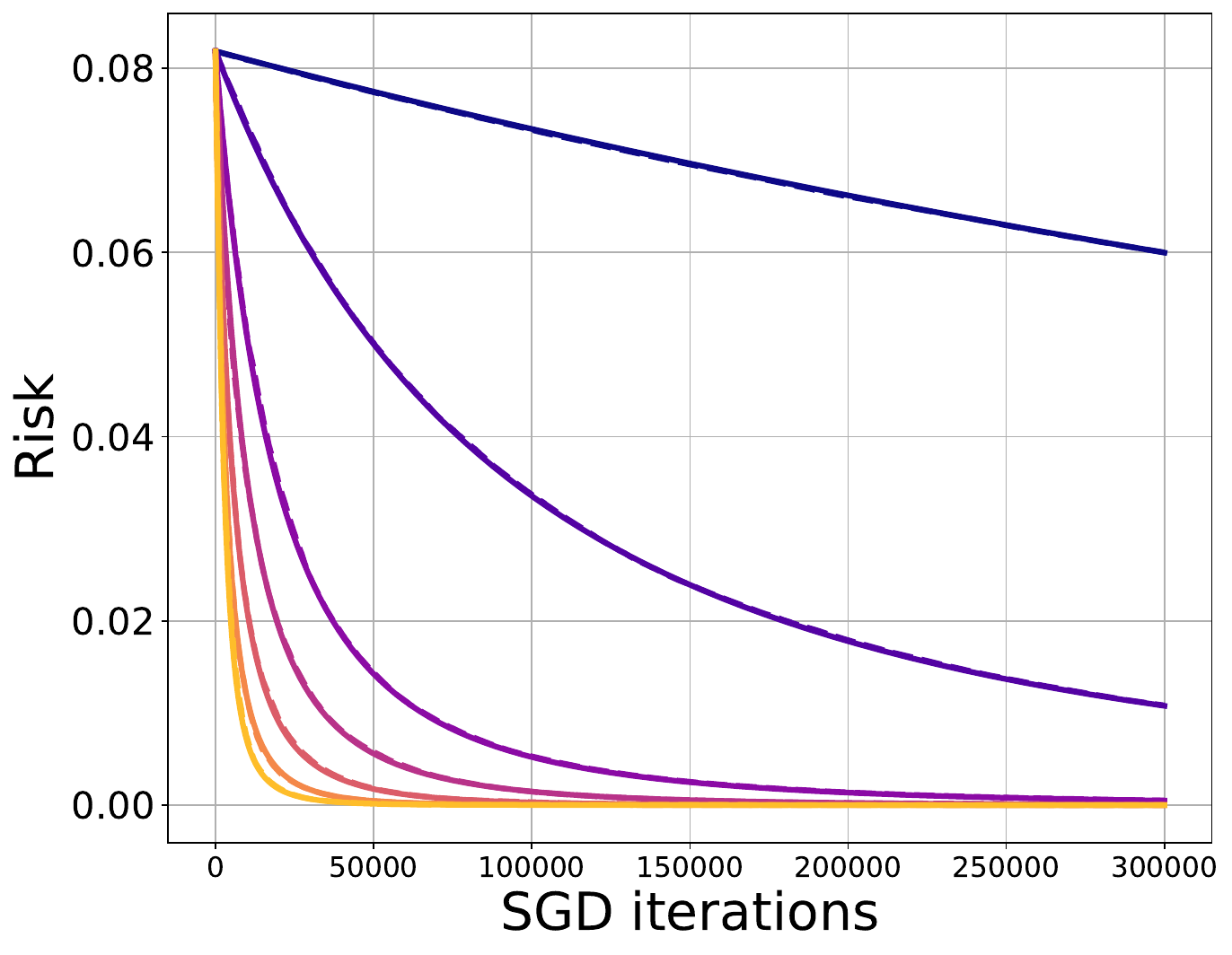} \qquad \includegraphics[scale =0.25]{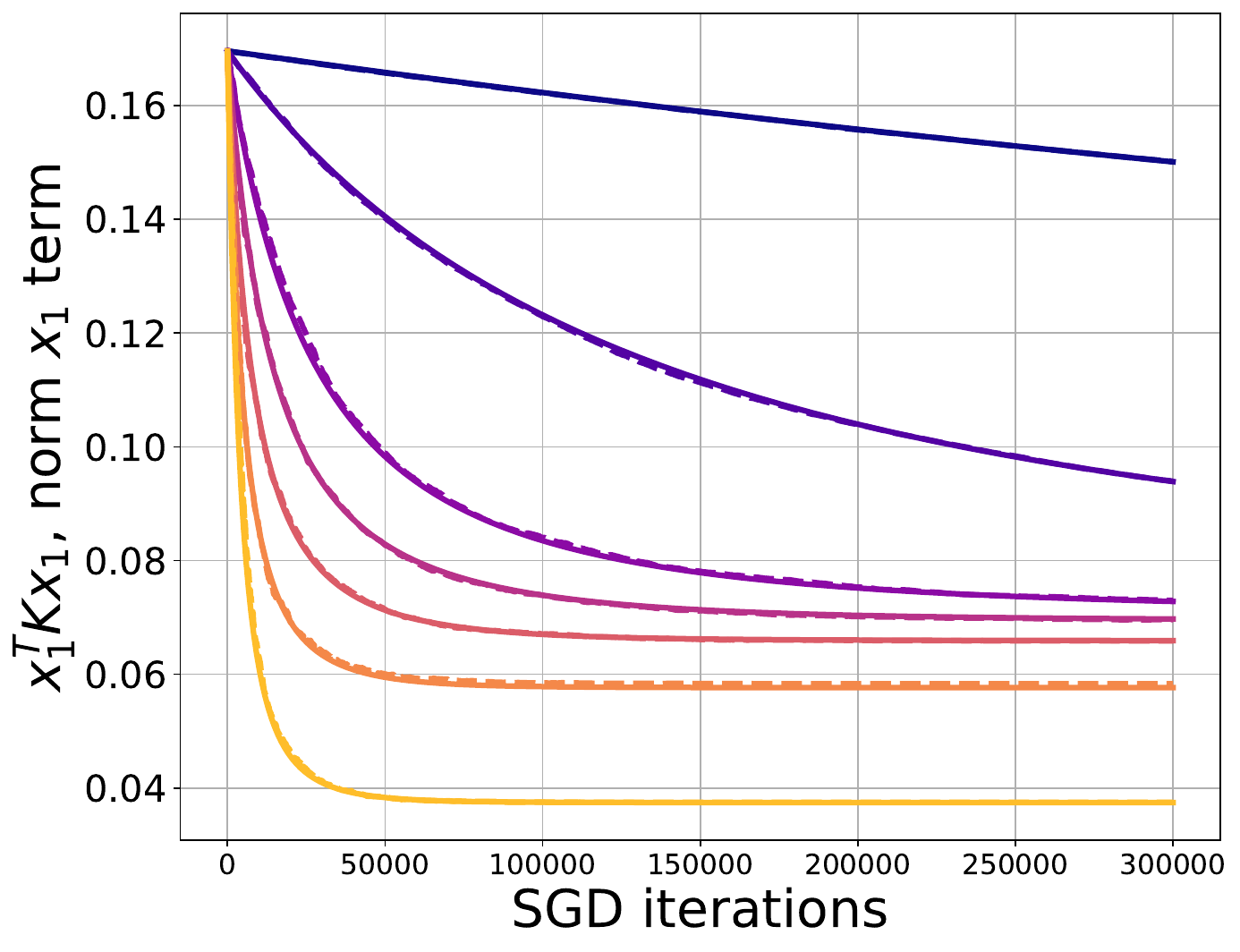} \\
    \includegraphics[scale =0.25]{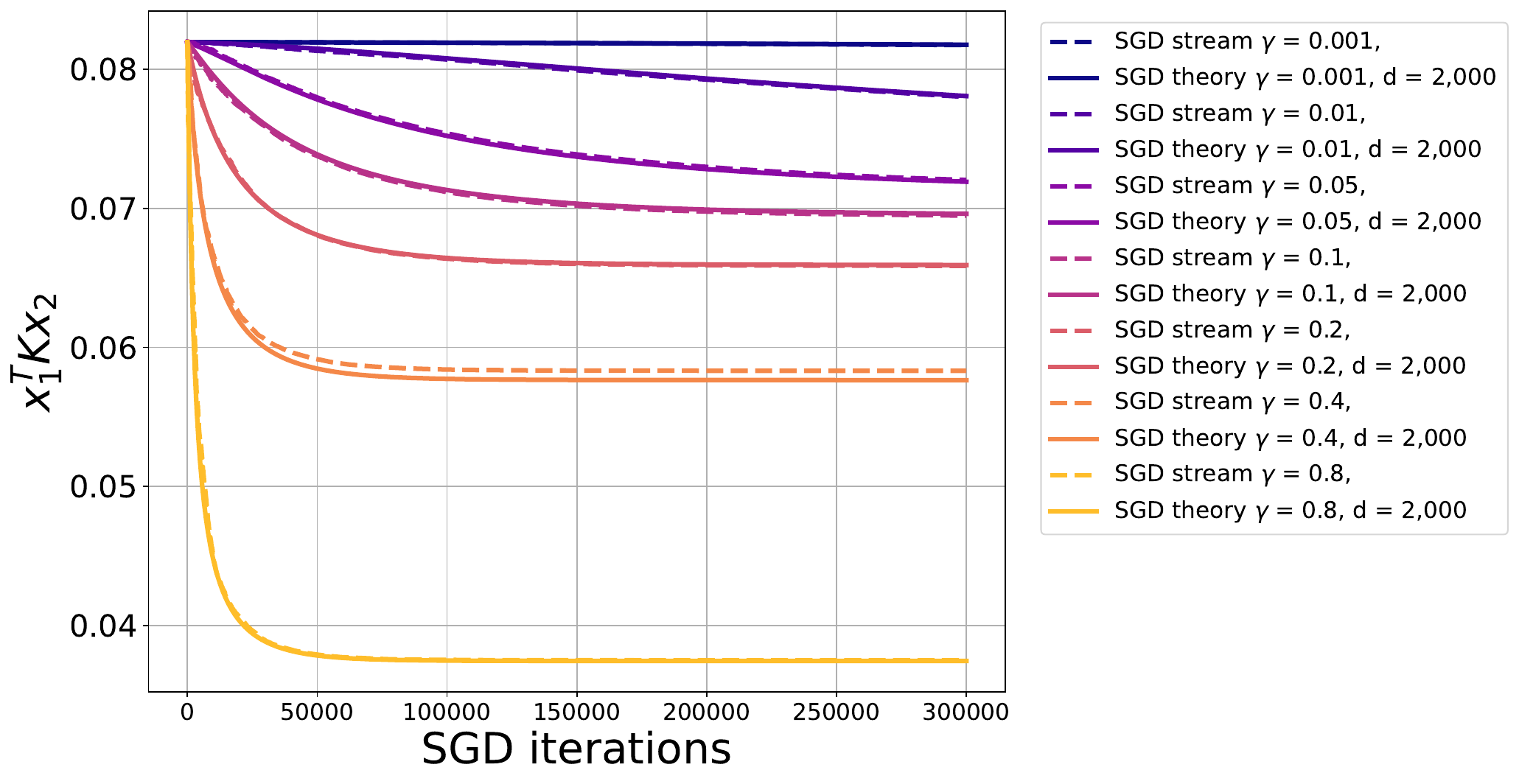}
 \caption{\textbf{SGD vs Theory on (noiseless) Chase Phase Problem.} Plot of SGD in comparison with theory for several statistics spanning learning rates for the noiseless phase chase problem \eqref{eq:chasing_problem} initialized at $X_{0,i} = 0.5 \cdot \frac{1}{\sqrt{d}} N(0, \Id_{d}) + 0.5 \cdot \frac{1}{\sqrt{d}} ( 1, \hdots, 1)^T$ for $i = 1, 2$ and $d = 2000$, a student-teacher model is employed with $X^{\star} = 0$ and covariance matrix $K$ having spectrum generated from a Marchenko-Pastur distribution with parameter $4$. First, the theoretical trajectories (solid) of SGD match \textit{single runs} of SGD (dashed) on all the statistics, see Theorem~\ref{Thm:SGD_HSGD_convergence}. The optimal solution occurs when $\|X_1\|_K = \|X_2\|_K$. We see that various learning rates pick out different solutions; the $K$-norm near convergence changes as the learning rate varies \textbf{(top right)}. Moreover, as the learning rate goes to $0$ (i.e. gradient flow, correctly scaled), we see that the cross term, $X_1^T K X_2$, does not change much from initialization \textbf{(bottom)}. SGD noise causes movement in the cross term, see \eqref{eq:Q_phase_chase}. Moreover, over the larger the learning rate, the slower $\|X_1\|^2 \to \|X_2\|^2$ while simultaneously speeding up the decreasing cross term. The result is we qualitatively see an $\ell^2$-regularized implicit bias, that is, larger learning rates lead to smaller coordinate values, $\|X_1\|_K$ and $\|X_2\|_K$.}
 \label{fig:phase_chase_problem}
\end{figure}

\clearpage

%\subsection{Outline of the proof?}

\section{Preliminaries} \label{sec:notation}

In this section, we give a more thorough discussion of the tensor notation used in this article, expanding on the discussion in the introduction.  We then show how the notation can be used to simplify derivative computations.  
We also include a discussion of the concentration of measure theory required for this work.

%Lastly, we allow our model of the targets $y$ to depend on a small (also \textit{fixed}) number of these high-dimensional vectors. So, we consider a third parameter $\ell^{\star}$, the dimension of the \textit{target} space $\mathcal{T} \cong \mathbb{R}^{\ell^{\star}}$. In some cases, we will need to formally work on the state and target space, that is $\mathcal{A} \otimes \mathcal{O}^+ \cong \mathcal{A} \otimes (\mathcal{O} \oplus \mathcal{T} ) \cong \mathbb{R}^d \otimes ( \mathbb{R}^{\ell} \oplus \mathbb{R}^{\ell^{\star}} )$ where $\mathcal{O}^+ \cong \mathcal{O} \oplus \mathcal{T}$.  

%For convenience we will also use the subgaussian norm $\|\cdot\|_{\psi_2}$ (see e.g., \citep{vershynin2018high} for more details) which is equivalent up to universal constants to the optimal variance proxy in a Gaussian tail bound for a random variable $X$ i.e.,
%\begin{equation}\label{eq:subgaussian_norm}
%\| X \|_{\psi_2} 
%\asymp
%\inf \{ V  > 0 : \forall~t > 0~\Pr( |X| > t) \leq 2 e^{-t^2/V^2}\}.
%\end{equation}
%We say an event $B$ holds with overwhelming probability (w.o.p.) if, for every fixed $D > 0$, $\Pr(B) \ge 1 -C_D d^{-D}$ for some $C_D$ independent of $d$.  

\subsection{Tensor products of Hilbert space}
We have posed three finite-dimensional real vector spaces $\mathcal{A},\mathcal{O}$ and $\mathcal{T}$, which we equip with inner products and so are finite dimensional Hilbert spaces.
Recall that as a vector space $\mathcal{A} \otimes \mathcal{O}$ is all (finite) linear combinations of \textit{simple} tensors, i.e., those of the form $a \otimes b$ where $a \in \mathcal{A}$ and $b \in \mathcal{O}$.  This becomes an algebra, allowing scalars to commute, i.e., for $c \in \mathbb{R}$
\[ c (a \otimes b) = (ca) \otimes b = a \otimes (c b),\]
and by allowing $\otimes$ to distribute over addition, 
\begin{equation}
    \begin{gathered}
        (a + b) \otimes c = (a \otimes c) + (b \otimes c) \quad 
        \text{and} \quad a \otimes (b + c) = ( a \otimes b) + (a \otimes c).
    \end{gathered}
\end{equation}

In what proceeds, we will need to consider general tensor contractions, which generalize matrix multiplication and dot products. We will use the inner product $\ip{\cdot,\cdot}$ operator in various ways to describe this contraction. Each $\mathcal{A}$ and $\mathcal{O}$ carries with it an inner product, and so $\mathcal{A} \otimes \mathcal{O}$ 
has a natural inner product which for simple tensors is defined by
%which for simple tensors is defined by $\ip{\cdot, \cdot}_{\mathcal{A}}$ and $\ip{\cdot, \cdot}_{\mathcal{O}}$ respectively. This induces a natural inner produce on $\mathcal{A} \otimes \mathcal{O}$, which for simple tensors is defined by
\begin{equation}
    \ip{a \otimes b, c \otimes d}_{\mathcal{A} \otimes \mathcal{O}} = \ip{a,c}_{\mathcal{A}} \ip{b, d}_{\mathcal{O}}.
\end{equation}
This is extended to the full space $\mathcal{A} \otimes \mathcal{O}$ by bilinearity. 

This, for example, can be connected to the Frobenius inner product. If we represent an element $A \in \mathbb{R}^d \otimes \mathbb{R}^{\ell}$ in the orthonormal basis $\{e_i \otimes e_j\}$ as 
\begin{equation}
    A = \sum_{i,j} A_{ij} e_i \otimes e_j,
\end{equation}
then we have the identification 
\[\ip{A, B}_{\mathcal{A} \otimes \mathcal{O}} = \sum_{i,j} A_{ij} B_{ij} = \tr(A B^T). \]

\subsection{Higher tensor powers} 
For taking higher derivatives, we will be led naturally to expressions which involve higher order tensor powers. In particular, the dot products written above extend naturally to
\begin{equation}
    (\mathcal{A} \otimes \mathcal{O})^{\otimes 2} \defas (\mathcal{A} \otimes \mathcal{O}) \otimes (\mathcal{A} \otimes \mathcal{O}) \cong \mathcal{A}^{\otimes 2} \otimes \mathcal{O}^{\otimes 2},
\end{equation}
where the last isomorphism corresponds to reshaping the tensor to have its ambient directions listed first, and its observable directions second. In some cases, we also need to consider the target space $\mathcal{T}$ this will be listed third. We will try to always work with this convention. 

We will always sort the simple tensors into $\mathcal{A}$ first and then $\mathcal{O}$, if applicable, but within each space we must preserve the ordering. For instance, supposing $o_i \in \mathcal{O}$ with $i = 1,2,3$ and $\alpha_i \in \mathcal{A}$ with $i = 1, 2$, then
\begin{align*}
    o_1 \otimes a_1 \otimes o_2 \otimes a_2 \otimes o_3 \cong a_1 \otimes a_2 \otimes o_1 \otimes o_2 \otimes o_3,
\end{align*}
but the following is not allowed
\begin{align*}
o_1 \otimes a_1 \otimes o_2 \otimes a_2 \otimes o_3 \not \cong a_1 \otimes a_2 \otimes o_2 \otimes o_1 \otimes o_3.
\end{align*}
The above fails to preserve the ordering in the observable $\mathcal{O}$ space. This, particularly, will be important when we do derivatives.

Tensor computations naturally give rise to an inner product on higher tensor products, which we define first for simple tensors, $t_i \defas (a_i \otimes o_i)$ for $i = 1, 2, 3, 4$,
\begin{equation}
\begin{aligned}
    \ip{t_1 \otimes t_2, t_3 \otimes t_4}_{(\mathcal{A} \otimes \mathcal{O})^{\otimes 2}} 
    &
    =
    \ip{t_1, t_3}_{\mathcal{A} \otimes \mathcal{O}} \ip{t_2, t_4}_{\mathcal{A} \otimes \mathcal{O}}
    \\
    &
    =
    \ip{a_1, a_3}_{\mathcal{A}} \ip{a_2, a_4}_{\mathcal{A}} \ip{o_1, o_3}_{\mathcal{O}} \ip{o_2, o_4}_{\mathcal{O}}.
\end{aligned}
\end{equation}
This is once more extended by multi-linearity, and we further extend it to higher tensor powers. 

\subsection{Partial contractions}
When we contract in the ambient direction (which is to say, we form dot products in the ambient direction), we anticipate concentration of measure and central limit theorem effects. So for working with random tensors, it is especially helpful if we consider partial contractions, in which we contract tensors only in their $\mathcal{A}$ directions. Once more, for simple tensors, $t_i = (a_i \otimes o_i)$ for $i = 1,2$,
\begin{equation} \label{eq:higher_tensor}
    \ip{t_1, t_2}_{\mathcal{A}} \defas \ip{a_1, a_2}_{\mathcal{A}} (o_1 \otimes o_2) \in \mathcal{O}^{\otimes 2}. 
\end{equation}
This is also extended to all $\mathcal{A} \otimes \mathcal{O}$ to be bilinear. This extends to higher tensor powers analogously, and also to the more general situation of products of $V_0 \otimes V_1$ with $V_0 \otimes V_2$ as a bilinear mapping:
\begin{equation}
    \ip{\cdot, \cdot}_{V_0} \, : \, (V_0 \otimes V_1) \otimes (V_0 \otimes V_2) \to V_1 \otimes V_2 
\end{equation}
by the formula for simple tensors in \eqref{eq:higher_tensor}. In particular, one of $V_1$ or $V_2$ may be a 1-dimensional space or a tensor product of other spaces.
To summarize, the contraction operation $\ip{a,b}_{V_0}$ contracts all $V_0$ axes of $a$ with $b$ and outputs a tensor having the shape of the un-contracted axes of $a$ followed by those of $b$.

When we have multiple axes indicated by a tensor power of $\mathcal{O}$, contractions are taken left to right. For instance, for $o_i \in \mathcal{O}$ for $i = 1, 2, 3, 4$, we use 
\begin{align*}
\ip{o_1 \otimes o_2, o_3 \otimes o_4}_{ \mathcal{O}} \cong \ip{o_1, o_3}_{\mathcal{O}}\,\cdot\,o_2 \otimes o_4.%\ncong \ip{o_1, o_4}_{\mathcal{O}} \ip{o_2, o_3}_{\mathcal{O}}.
\end{align*}
% and $V_0$ may also be $\mathcal{O}$. 

We shall reserve the notation $\ip{\cdot, \cdot}$ for the contraction which contracts the most axes possible of the tensor, in whichever space they reside, and we shall add the subscript whenever a partial contraction is needed. We note that having done the partial contraction, it may be helpful to complete the contraction to a full contraction. This is performed by the \textit{trace} operation, which on the Hilbert space $V \otimes V$, is defined for simple tensors by
\begin{equation}
    \text{Tr}(v \otimes w) = \ip{v, w}_V, 
\end{equation}
and which extends to all $V \otimes V$ by linearity. In the context of \eqref{eq:higher_tensor}, we can then write 
\[ \text{Tr} ( \ip{t_1, t_2}_{\mathcal{A}} ) = \ip{a_1, a_2}_{\mathcal{A}} \ip{o_1, o_2}_{\mathcal{O}} = \ip{t_1, t_2}, \]
which by linearity therefore identifies $\tr( \ip{\cdot, \cdot}_{\mathcal{A}})$ as the full contraction.

\subsection{Norms on tensors} 
Recall that for a matrix $A \in \R^{d \times d}$,
which we can identify with a $2$-tensor,
the operator norm can be defined explicitly as 
\[
\sup_{\substack{\|y\|_2 =1,\\ \|z\|_2 =1}} \ip{A, y \otimes z} = \sup_{\substack{\|y\|_2 =1,\\ \|z\|_2 =1}}  y^T A z  = \|A\|_{\text{op}}.
\]
To generalize this idea to higher tensors, one can generalize this as a supremum over simple unit tensors. We will notate this by $\|\cdot\|_{\sigma}$; this norm is also commonly known as the \textit{injective tensor norm}. Explicitly, if $\varphi = x_1 \otimes x_2 \otimes \hdots \otimes x_k \in V_1 \otimes V_2 \otimes \hdots \otimes V_k$, for simple tensors, then we define its $\sigma$-norm by
\[
\|\varphi\|_{\sigma} \defas \sup_{\substack{\|y_i\|_{V_i} = 1 \\ i = 1,2, \hdots, k}} \ip{\varphi, y_1 \otimes y_2 \otimes \hdots \otimes y_k},
\]
where $y_1 \otimes y_2 \otimes \hdots \otimes y_k \in V_1 \otimes V_2 \otimes \hdots \otimes V_k$ is a simple tensor. 

The second norm we will use is the Hilbert-Schmidt norm, or simply the Hilbert-space norm, on a tensor $A$, which is given by
\[
\|A\| = \ip{A,A} = \sup_{\|B\|=1} \langle A, B\rangle.
\]
%In this way, we can define $\sigma$-norms for the derivative operators $(\Dif \varphi)(x), (\Dif^2 \varphi)(x)$, and $\Dif^3 \varphi(x)$ and generalize the idea of largest singular value. 
Finally 
we define 
the dual norm to the injective norm, which we still call the \emph{nuclear norm} by analogy with the matrix case,
and which is given by
    \[
      \|A\|_{*} \defas
      \sup_{\|B\|_\sigma =1} 
      \ip{A, B}.
    \]
Using the variational representations we observe
\begin{equation}\label{eq:tensornorms}
\|A\|_{\sigma} \leq \|A\| \leq \|A\|_*.
\end{equation}
%We also note the following inequality for quadratic forms
%\[
%\tr \langle B^\otimes B, A \rangle
%\]
%Moreover, we have that for matrices $A,B$ such that $B^TAB$ makes sense,
%\[
%\tr(B^T A B) \le \|A\|_{\sigma} \|B\|^2. 
%\]

\subsection{Calculus for tensors} 
We recall briefly how we represent differential calculus with the tensor notation introduced above. 
 For a  (smooth) function $f \, : \, V_0 \to V_1$ on (finite dimensional) Hilbert spaces $V_0, V_1$, its (Fr\'{e}chet) derivative $\Dif f$ can be identified as a mapping from $V_0 \to \mathcal{L}(V_0, V_1)$, the space of linear operators from $V_0 \to V_1$ so that for all $x, h \in V_0$
\[
\lim_{t \downarrow 0} \frac{f(x+th) - f(x)}{t} = (\Dif f)(x)[h].
\]

The space $\mathcal{L}(V_0, V_1)$ can be represented as elements of the tensor product $V_1 \otimes V_0$, by picking an orthonormal basis $\{e_j\}$ for $V_0$ and then identifying, 
\[
(\Dif f)(x) \leftrightarrow \sum_j (\Dif f)(x)[e_j] \otimes e_j,
\]
which is (in effect) its Jacobian matrix representation. This procedure can now be iterated, as $\Dif f$ is a mapping between $V_0$ and a new vector space $\mathcal{L}(V_0, V_1) \cong V_1 \otimes V_0$, and hence
\[
\Dif^2 f \, : \, V_0 \to \mathcal{L}(V_0, \mathcal{L}(V_0, V_1)) \cong V_1 \otimes V_0 \otimes V_0.
\]
In the case that the output of $f$ is $1$-dimensional (so that $V_1 \cong \mathbb{R}$) we may furthermore identify the second derivative $(\Dif^2 f)(x)$ with an element of $V_0 \otimes V_0$. A parallel approach identifies the third derivative as
\[
\Dif^3 f \, : \, V_0 \to \mathcal{L}( V_0,\mathcal{L}(V_0, \mathcal{L}(V_0,V_1))) \cong V_1 \otimes V_0^{\otimes 3}.
\]
In this way, we have that
\[
\Dif^k f \, : \, V_0 \to V_1 \otimes V_0^{\otimes k}. 
\]
Similarly, when $V_1 \cong \mathbb{R}$, we can identify $V_1 \otimes V_0^{\otimes k} \cong V_0^{\otimes k}$. 

\subsubsection{Chain rule with tensors} 

The class of statistics (and losses) we consider are compositions of smooth maps. In this section, we show how one can use the tensor notation to simplify the chain rule for higher order derivatives. Supposing one has two smooth maps $f,g$ with $f \, : \, V_0 \to V_1$ and $g \, : \, V_1 \to V_2$, the chain rule states that $g \circ f$ is a smooth map from $V_0 \to V_2$ and its derivative is a map from $V_0$ to $\mathcal{L}(V_0, V_2)$. Moreover, its derivative is given by 
\[
\Dif (g \circ f)(x)[h] = (\Dif g)(f(x))[(\Dif f)(x)[h]].
\]
If we represent these as tensors, then $(\Dif g)(f(x))$ is in $V_2 \otimes V_1$ and $(\Dif f)(x)$ is in $V_1 \otimes V_0$, and hence we can as well represent the chain rule by
\begin{equation} \label{eq:chain_rule}
    \Dif (g \circ f)(x) = \ip{(\Dif g)(f(x)), (\Dif f)(x)}_{V_1} \in V_2 \otimes V_0,
\end{equation}
showing along which axis the contraction is taken. We note that the ordering is important here. The input space is always taken to be on the right. 

Applying this in the case of a directional derivative, suppose we take a smooth function $\varphi \, : \, V \to \mathbb{R}$. Then for any fixed $x, \Delta \in V$, the map $\psi \, : \, t \mapsto \varphi(x + t \Delta)$ is a smooth function of $\mathbb{R}$, and we may compute its Taylor approximation. In particular, we are interested in approximating $\varphi(x + \Delta)$ or equivalently $\psi(1)$. If we approximate $\varphi(x + \Delta)$ by the third order Taylor expansion at $x$ with remainder, we have
\[
\varphi(x + \Delta) = \psi(1) = \psi(0) + \psi'(0) + \tfrac{1}{2} \psi''(0) +  \frac{1}{2} \int_0^1 (1-t)^2 \psi^{(3)}(t) \, \dif t.
\]
Applying the chain rule, if we set $x(t) = x + t \Delta$, then $(\Dif x)(t)$ is constant and equal to $\Delta$. Therefore, we deduce that
\[
\begin{aligned}
    &
    \psi'(0) = \ip{(\Dif \varphi)(x), \Delta}, 
    \quad 
    \psi''(0) = \ip{(\Dif^2 \varphi)(x), \Delta^{\otimes 2}},
    \quad
    \text{and}
    \quad
    \psi^{(3)}(t) = \ip{(\Dif^3 \varphi)(x(t)), \Delta^{\otimes 3}}.
\end{aligned}
\]
To derive this, in particular, the 2nd and 3rd derivatives, we used linearity to conclude
\[
\begin{aligned}
    \psi''(t) = \Dif ( \ip{ (\Dif \varphi) (x(t)), \Delta}_{V} ) 
    &
    = 
    \ip{\Dif ( (\Dif \varphi) (x(t)) ), \Delta}_{V}
    \\
    &
    =
    \ip{ \ip{ (D^2\varphi) (x(t)), \Delta}_V, \Delta }_V
    \\
    &
    =
    \ip{(D^2 \varphi) (x(t)), \Delta^{\otimes 2}}_{V \otimes V}.
\end{aligned}
\]
We note that in the second line, there is in principle an ambiguity $\ip{(D^2 \varphi)(x(t)), \Delta}_V$, in that $(D^2 \varphi)(x(t))$ is an element of $V \otimes V$. However, as the second derivative is symmetric (as $\varphi$ is smooth and so mixed partials can be interchanged), contraction along either axis works. We summarize with the following generic directional derivative expansion for scalar $C^3$-smooth functions $\varphi \, : \, V \to \mathbb{R}$
\begin{equation}
\varphi(x + \Delta) = \varphi(x) + \ip{ (\Dif \varphi) (x), \Delta} + \frac{1}{2} \ip{(\Dif^2 \varphi) (x), \Delta^{\otimes 2}} + \frac{1}{2} \int_0^1 (1-t)^2 \ip{(\Dif^3 \varphi)( x + t \Delta), \Delta^{\otimes 3}} \, \dif t. 
\end{equation}

\subsection{Derivative of special statistics} 
In this section, we compute the derivatives of the functions $f$, $\Psi_{\delta}$, and the risk function $\mathcal{R}_{\delta}$. 
\paragraph{Derivative of $\Psi_{\delta}$ and bounds on $\nabla_x f$.} The function $f \, : \, \mathcal{O} \oplus \mathcal{T} \oplus \mathcal{T} \to \mathbb{R}$ as in \eqref{eq:nlgc} is $\alpha$-pseudo-Lipschitz and so the derivatives of $f$, $\nabla_x f$ and $\Psi_{\delta} \, : \mathcal{A} \times \mathcal{O} \to \mathbb{R}$, defined in \eqref{eq:Rdelta}, $\nabla_X \Psi_{\delta}$, exist a.e. 

To reduce notation, we write 
\begin{equation}
\begin{aligned}
 \Psi_\delta(X) 
 &
 \defas \Psi_{\delta}(X; a, \epsilon), \quad 
 \Psi(X) 
 \defas 
 \Psi(X; a, \epsilon),
 \\
 \text{and} \quad f(\ip{W, a}_{\mathcal{A}}) 
 &
 \defas f(\ip{X,a}_{\mathcal{A}}) \defas f(\ip{X,a}_{\mathcal{A}} \oplus \ip{X^{\star}, a}_{\mathcal{A}}; \epsilon), \quad \text{where $W = X \otimes X^{\star}$.} 
\end{aligned}
\end{equation}
This is to emphasize various dependencies on $a, X, X^{\star}$, and the noise $\epsilon$ in the proofs that follow. For further simplicity, 
\[
r \defas \ip{W,a}_{\mathcal{A}} \quad \text{and} \quad f(r) \defas f(\ip{W,a}_{\mathcal{A}}) = f(\ip{X,a}_{\mathcal{A}} \oplus \ip{X^{\star}, a}_{\mathcal{A}}; \epsilon).
\]
Analogously, we do the same for gradients:
\begin{equation}
\begin{aligned}
 \nabla_X \Psi_\delta(X) 
 &
 \defas \nabla_X \Psi_{\delta}(X; a, \epsilon), \quad
 \nabla_X \Psi(X) 
 \defas 
 \nabla_X \Psi(X; a, \epsilon)
 \\
 \text{and} \quad 
 \nabla_x f(r) \defas \nabla_x f(\ip{W, a}_{\mathcal{A}}) 
 &
 \defas \nabla_x f(\ip{X,a}_{\mathcal{A}}) \defas \nabla_x f(\ip{X,a}_{\mathcal{A}} \oplus \ip{X^{\star}, a}_{\mathcal{A}}; \epsilon),
\end{aligned}
\end{equation}
% We will also think of $f$ as a function of $W$, that is, 
% \[
% \Psi(X) = f(\ip{W, a}_{\mathcal{A}}) + p(X),
% \]
% where the function $f \, : \, \mathcal{O} \to \mathbb{R}$ by $\tilde{r} \mapsto f(\tilde{r}; \ip{X^{\star},a}_{\mathcal{A}}, \ip{W \otimes W, I_d}_{\mathcal{A}^{\otimes 2}}) = f(\tilde{r})$. Here the notation $\nabla_{\tilde{r}} f = \nabla f$ will be reserved for derivatives with respect to only the $\ip{X,a}_{\mathcal{A}}$, but often evaluated at $r = \ip{W,a}_{\mathcal{A}}$ with the understanding that function is acting as in \eqref{eq:loss_function_1}. Precise and (additional) assumptions on the functions $\Psi, f, \varphi$ and $h$ are given in Section~\ref{sec:introduction}. 

% Next, given the composite structure of both the statistic and the objective function, we compute explicitly several of their derivatives. 

Given the composite structure of $\Psi_{\delta}$, 
\begin{equation} \label{eq:loss_function_1}
\Psi_{\delta}(X; a, \epsilon) = f( \ip{X, a}_{\mathcal{A}} \oplus \ip{X^{\star}, a}_{\mathcal{A}}; \epsilon) + \tfrac{\delta}{2} \|X\|^2,
\end{equation}
we compute its derivative. For this, we need to introduce the \textit{identity mapping} 
\[
\text{Id}_{\mathcal{A} \otimes \mathcal{O}} \, : \, \mathcal{A} \otimes \mathcal{O} \to \mathcal{A} \otimes \mathcal{O}
\quad \text{such that $\text{Id}(X) = X$.}
\]
Moreover, with this, we have that $\Dif (X \mapsto X) \, : \, \mathcal{A} \otimes \mathcal{O} \to \mathcal{L}(\mathcal{A} \otimes \mathcal{O}, \mathcal{A} \otimes \mathcal{O}) \cong \mathcal{A}^{\otimes 3} \otimes \mathcal{O}^{\otimes 3}$. The derivative of the mapping $X \to X$, $\Dif X$, is the identity mapping, 
\[
\Dif X \cong \text{Id}_{\mathcal{A} \otimes \mathcal{O}}.
\]

Let us now consider the derivative of $X \in \mathcal{A} \otimes \mathcal{O} \mapsto \ip{X,a}_{\mathcal{A}}$, $\Dif ( \ip{X, a}_{\mathcal{A}} ) \in \mathcal{L}(\mathcal{A} \otimes \mathcal{O}, \mathcal{O})$. Then we see that 
\begin{align*}
    \Dif (\ip{X, a}_{\mathcal{A}} ) 
    = 
    \ip{ \Dif X, a}_{\mathcal{A}}
    \cong
    \ip{\text{Id}_{\mathcal{A} \otimes \mathcal{O}}, a}_{\mathcal{A}} \in \mathcal{L}(\mathcal{A} \otimes \mathcal{O}, \mathcal{O}).
\end{align*}
We now choose an orthogonal basis $\{e_{\alpha} \otimes f_o\}$ for $\mathcal{A} \otimes \mathcal{O}$, and
\begin{equation} \label{eq:derivative_X_a}
\begin{aligned}
    \Dif (\ip{X, a}_{\mathcal{A}} ) 
    \cong
    \ip{\text{Id}_{\mathcal{A} \otimes \mathcal{O}}, a}_{\mathcal{A}}
    &
    \cong 
    \sum_{\alpha, o} \ip{e_{\alpha} \otimes f_o, a}_{\mathcal{A}} \otimes e_{\alpha} \otimes f_o \\
    &
    = \sum_{\alpha, o} \ip{e_{\alpha}, a}_{\mathcal{A}} f_o \otimes e_{\alpha} \otimes f_o
    \\
    &
    =
    \sum_{\alpha, o} \ip{e_{\alpha}, a} e_{\alpha} \otimes f_o \otimes f_o
    \\
    &
    =
    \sum_{o} a \otimes f_o \otimes f_o
    \\
    &
    \cong a \otimes \text{Id}_{\mathcal{O}}.
\end{aligned}
\end{equation}
We make explicit the connection between the operator definition of $\Dif (X \mapsto \ip{a,X}_{\mathcal{A}})$ and the tensor definition just seen \eqref{eq:derivative_X_a}. Consider a perturbation $H \in \mathcal{A} \otimes \mathcal{O}$ and evaluate $(\Dif \ip{\cdot,a}_{\mathcal{A}})(W)[H]$, 
\[
\Dif(\ip{X, a}_{\mathcal{A}})[H] = \lim_{t \downarrow 0} \frac{\ip{a, X + t H}_{\mathcal{A}} - \ip{a, X}_{\mathcal{A}}}{t} = \ip{a, H}_{\mathcal{A}} = \ip{a \otimes \text{Id}_{\mathcal{O}}, H}. 
\]
Thus, once more sorting the coordinates, the derivative of the loss $\Psi_{\delta}(X) = f( \ip{X,a}_{\mathcal{A}}) + \delta \|X\|^2 /2$ using chain rule \eqref{eq:chain_rule} and the basis $\{f_o\}$ for $\mathcal{O}$
\begin{equation} \label{eq:loss_derivative}
\begin{aligned}
\nabla_X \Psi_{\delta} (X)
&
\cong
\ip{(\nabla_x f)(\ip{X, a}_{\mathcal{A}}), a \otimes \text{Id}_{\mathcal{O}} }_{\mathcal{O}} + \delta X \cong \sum_o \ip{(\nabla_x f)(\ip{X,a}_{\mathcal{A}}), a \otimes f_o}_{\mathcal{O}} \otimes f_o + \delta X
\\
& 
\cong a \otimes (\nabla_x f)(\ip{X, a}_{\mathcal{A}}) + \delta X \in \mathcal{A} \otimes \mathcal{O}.
\end{aligned}
\end{equation}
We have shown our first important result:
\begin{lemma}[Derivative of $\Psi_{\delta}$] \label{lem:derivative_loss}
    Setting the loss $\Psi_\delta(X) = f(\ip{W, a}_{\mathcal{A}}) + p(X)$ and letting $k \in \mathbb{N}$, we define
    \begin{gather*}
        \nabla_X \Psi_{\delta}(X) 
        = 
        a \otimes \nabla_x f( \ip{W,a}_{\mathcal{A}} ) + \delta X
    \end{gather*}
where we represent the differentials in the sorted coordinates $\mathcal{A} \otimes \mathcal{O}$ and preserve the ordering (left to right) of the $\mathcal{A}$ and $\mathcal{O}$ tensor contractions. 
\end{lemma}

We are now ready to compute the derivative of the risk $\mathcal{R}(X)$. 

\begin{lemma}[Derivatives of the statistic, $\varphi$] \label{lem:derivative_risk}
Suppose the risk is $\mathcal{R}(X) = h( \ip{ W \otimes W, K }_{\mathcal{A}^{\otimes 2}} )$. Then, one has
\begin{align*}
    \nabla \mathcal{R}(X) 
    &=
        \ip{\nabla h, (\text{\rm Id}_{\mathcal{O}} \oplus 0_{\mathcal{T}}) \otimes \ip{K,W}_{\mathcal{A}} }_{\mathcal{(O^+)}^{\otimes 2}}
        \\
        & 
        \quad 
        + \ip{\nabla h, \ip{K, W}_{\mathcal{A}} \otimes (\text{ \rm Id}_{\mathcal{O}} \oplus 0_{\mathcal{T}}) }_{(\mathcal{O}^{+})^{\otimes 2}}.
    \end{align*}
where $\nabla h$ is evaluated at $ \ip{W \otimes W, K}_{\mathcal{A}^{\otimes 2} } $. We represent the differentials in the sorted coordinates $\mathcal{A}$ and then $\mathcal{O}$.
\end{lemma}

\begin{proof}
The result is immediate from \eqref{eq:statistic_inner_function} and chain rule. 

\end{proof}

\paragraph{Derivative of the risk $\mathcal{R}$.} Now we turn to evaluate the (composite) risk
\begin{equation} \label{eq:statistic_information_2}
\mathcal{R}(X) = h( \ip{W \otimes W , K}_{\mathcal{A}^{\otimes 2}}) \quad \text{where $W = X \oplus X^{\star}$,}
\end{equation}
and its corresponding chain rule. We introduce the zero tensor in the vector space $\mathcal{T}$, denoted by $0_{\mathcal{T}}$. We emphasize the space in which the zero tensor lives to avoid confusion. First, the mapping $X \mapsto W = X \oplus X^{\star}$ has a nice, simple derivative
\[
\Dif(W) = \Dif(X \oplus X^{\star}) \cong \text{Id}_{\mathcal{A} \otimes \mathcal{O}} \oplus 0_{\mathcal{A} \otimes \mathcal{T}}.
\]
Now to compute the chain rule of \eqref{eq:statistic_information_2}. For this, we need to compute the derivative of the inside function $\Dif (X \mapsto \ip{W \otimes W, K}_{\mathcal{A}^{\otimes 2}} \in \mathcal{L}(\mathcal{A} \otimes \mathcal{O}, (\mathcal{O}^+)^{\otimes 2})$. The product rule gives
\begin{align*}
    \Dif ( \ip{W \otimes W, K}_{\mathcal{A}^{\otimes 2}}) 
    &
    = 
    \ip{\Dif W \otimes W, K}_{\mathcal{A}^{\otimes 2}} + \ip{W \otimes \Dif W, K}_{\mathcal{A}^{\otimes 2}}
    \\
    &
    \cong
    \ip{(\text{Id}_{\mathcal{A} \otimes \mathcal{O}} \oplus 0_{\mathcal{A} \otimes \mathcal{T}}) \otimes W, K}_{\mathcal{A}^{\otimes 2}} + \ip{W \otimes (\text{Id}_{\mathcal{A} \otimes \mathcal{O}} \oplus 0_{\mathcal{A} \otimes \mathcal{T}}), K}_{\mathcal{A}^{\otimes 2}}.   
\end{align*}
Choosing an orthonormal basis $\{e_{\alpha} \otimes f_o\}$ for $\mathcal{A} \otimes \mathcal{O}$, 
\begin{align*}
    \ip{\Dif W \otimes W, K}_{\mathcal{A}^{\otimes 2}} 
    \cong
    \ip{(\text{Id}_{\mathcal{A} \otimes \mathcal{O}} \oplus 0_{\mathcal{A} \otimes \mathcal{T}}) \otimes W, K}_{\mathcal{A}^{\otimes 2}} 
    &
    \cong
    \sum_{o, \alpha} \ip{e_{\alpha} \otimes (f_o \oplus 0_{\mathcal{T}}) \otimes W, K}_{\mathcal{A}^{\otimes 2}} \otimes e_{\alpha} \otimes f_o
    \\
    & 
    =
    \sum_{o, \alpha} (f_o \oplus 0_{\mathcal{T}}) \otimes \ip{e_{\alpha} \otimes W, K}_{\mathcal{A}^{\otimes 2}} \otimes e_{\alpha} \otimes f_o
    \\
    \text{($K = \EE[ a \otimes a]$)} \quad  
    &
    =
    \sum_o (f_o \oplus 0_{\mathcal{T}}) \otimes \ip{W,K}_{\mathcal{A}} \otimes f_o
    \\
    &
    \cong ( \text{Id}_{\mathcal{O}} \oplus 0_{\mathcal{T}} ) \otimes \ip{W, K}_{\mathcal{A}}. 
\end{align*}
A similar computation, making sure to preserve the ordering of the contractions in $\mathcal{O}$, yields 
\[
\ip{W \otimes \Dif W, K}_{\mathcal{A}^{\otimes 2}} \cong \ip{ W \otimes (\text{Id}_{ \mathcal{A} \otimes \mathcal{O}} \oplus 0_{\mathcal{A} \otimes \mathcal{T}} ), K}_{\mathcal{A}^{\otimes 2}} \cong \ip{W, K}_{\mathcal{A}} \otimes (\text{Id}_{\mathcal{O}} \oplus 0_{\mathcal{T}}). 
\]
It immediately follows that 
\begin{equation} \label{eq:statistic_inner_function}
\begin{aligned}
    \Dif ( \ip{W \otimes W, K}_{\mathcal{A}^{\otimes 2}}) 
    &
    =
    \ip{\Dif W \otimes W, K}_{\mathcal{A}^{\otimes 2}} + \ip{W \otimes \Dif W, K}_{\mathcal{A}^{\otimes 2}}
    \\
    &
    \cong
    (\text{Id}_{\mathcal{O}} \oplus 0_{\mathcal{T}}) \otimes \ip{W, K}_{\mathcal{A}} + \ip{W, K}_{\mathcal{A}} \otimes ( \text{Id}_{\mathcal{O}} \oplus 0_{\mathcal{T}}). 
\end{aligned}
\end{equation}

\subsection{Concentration and pseudo-Lipschitz}

 For convenience, we will also use the subgaussian norm $\|\cdot\|_{\psi_2}$ (see e.g., \citep{vershynin2018high} for more details) which is equivalent up to universal constants to the optimal variance proxy in a Gaussian tail bound for a random variable $X$ i.e.,
\begin{equation}\label{eq:subgaussian_norm}
\| X \|_{\psi_2} 
\asymp
\inf \{ V  > 0 : \forall~t > 0~\Pr( |X| > t) \leq 2 e^{-t^2/V^2}\}.
\end{equation}
%We say an event $B$ holds with overwhelming probability (w.o.p.) if, for every fixed $D > 0$, $\Pr(B) \ge 1 -C_D d^{-D}$ for some $C_D$ independent of $d$.  

Gaussian variables are naturally subgaussian.
Moreover, they satisfy a vastly stronger property, \emph{Lipschitz concentration}, which gives concentration inequalities for nonlinear functions of Gaussian vectors.  
If $V_0$ is a Hilbert space, say that a function $f : V_0 \to \R$ is Lipschitz with constant $L$ if for all $x,y \in V_0,$
\[
|f(x) - f(y)| \leq L \|x-y\|.
\]
Then for $Z$ which is an isotropic, centered Gaussian vector on $V_0$ and Lipschitz $f$,
\[
\|f(Z) - \Exp f(Z)\|_{\psi_2} \leq C L(f).
\]
The constant $C$ is an absolute universal constant.  In particular, this concentration is dimension-free. 

\paragraph{Pseudo-Lipschitz.} In our setting, we shall also work with functions which are not-quite Lipschitz, in that they are locally-Lipscthiz (Lipschitz on compact sets) and moreover have polynomial growth of their Lipschitz on norm-balls.  Specifically:
\begin{definition}[Pseudo-Lipschitz functions] For $\alpha \ge 0$ and a function $f \, : \, V_0 \to V_1$ is called pseudo-Lipschitz of order $\alpha$ if there exists a constant $L \defas L(\alpha, f)$ such that 
\begin{equation}
\sup_{x,y \in V_0} \left ( \frac{\|f(x)-f(y)\|_{V_1}}{\|x-y\|_{V_0}}  \right ) \le L ( 1 + \|x\|_{V_0}^{\alpha} + \|y\|_{V_0}^{\alpha}).
\end{equation}
The constant $L$ is the $\alpha$-pseudo-Lipschitz constant for the function $f$ (for shorthand, we will often call $L$ the Lipschitz constant of $f$). 
\end{definition}

We will often work with outer functions and statistics whose gradients are $\alpha$-pseudo-Lipschitz. In order to invoke a bound on the $\alpha$-pseudo-Lipschitz gradient, $\nabla f$, which involves the norms of $\|y\|$ and $\|x\|$, we introduce the \textit{projection operator onto the ball of radius $\beta$, $\text{\rm Proj}_{\beta} \, : \, V_0 \to V_0$}, by
\begin{equation}
    \begin{aligned}
        \text{Proj}_{\beta}(x) 
        &
        \defas
        \argmin_{y \in \beta \mathbb{B} } \, \big \{ \|x-y\|^2_{V_0}  \big \}, \quad \text{where $\mathbb{B}$ is the unit ball in $V_0$}
        \\
        &
        = \begin{cases}
        x & \text{ if $\|x\|_{V_0} \le \beta$ } \\
        \beta \left ( \frac{x}{\|x\|_{V_0}} \right ) & \text{otherwise.}
        \end{cases}
    \end{aligned}
\end{equation}
It immediately follows by taking compositions of projections with $\alpha$-pseudo-Lipschitz functions that we have Lipschitz functions.

\begin{lemma}
    Suppose $f \, : \, V_0 \to V_1$ is $\alpha$-pseudo-Lipschitz with constant $L$. Then the composition $f \circ \text{\rm Proj}_\beta$ is Lipschitz with constant $L(1 + 2 \beta^{\alpha})$. 
\end{lemma}

\begin{proof}
    First, the projection onto any convex set is $1$-Lipschitz. From this, a simple computation shows that
    \begin{equation}
    \begin{aligned}
        \|(f \circ \text{Proj}_{\beta})(x) &- (f \circ \text{Proj}_{\beta})(y)\|_{V_1} 
        \\
        &
        \le L \| \text{Proj}_{\beta}(x)-\text{Proj}_{\beta}(y) \|_{V_0} \big ( 1 + \|\text{Proj}_{\beta}(x)\|^{\alpha}_{V_0} + \|\text{Proj}_{\beta}(y)\|_{V_0}^{\alpha} \big )
        \\
        &
        \le L \|x-y\|_{V_0} \big (1 + 2 \beta^{\alpha} \big ).
    \end{aligned}
    \end{equation}
\end{proof}

The $\alpha$-pseudo-Lipschitz property of $f$, Assumption~\ref{ass:pseudo_lipschitz}, in addition, gives us a rate of growth on moments of $\nabla_x f$ in terms of $W = X \oplus X^{\star}$. 

\begin{lemma}[Growth of $\nabla_x f$] \label{lem:growth_grad_f} Suppose the function $f \, : \, \mathcal{O} \oplus \mathcal{T} \oplus \mathcal{T} \to \mathbb{R}$ is $\alpha$-pseudo-Lipschitz with Lipschitz constant $L(f)$ (see Assumption~\ref{ass:pseudo_lipschitz}) and the noise $\epsilon \sim N(0, I_{\mathcal{T}})$ independent of $a$ (see Assumption~\ref{ass:data_normal}). Then for $p > 0$ and any $r \in \mathcal{O}^+$,
\begin{equation} \label{eq:growth_grad_f}
\| \nabla_x f(r)\|^p \le C(\alpha, p) (L(f))^p (1 + \|r\| + \|\epsilon\|)^{\max \{1, \alpha p\}},
\end{equation}
Moreover, if $r = \ip{W,a}_{\mathcal{A}}$, there is a growth rate on $\nabla_x f(r)$ and sub-Gaussian norm on $r$ in terms of $W$, 
    \begin{equation} \label{eq:expectation_f_growth}
    \begin{gathered}
        \EE_{a,\epsilon}[\|\nabla_x f(r) \|^p] \le C(\alpha, p, |\mathcal{T}|)  (L(f))^p \big (1 + \|K\|_{\sigma}^{1/2} \|W\|\big )^{\max \{1, \alpha p\} }
        \\
            \text{and} \qquad \| (1 + \|r\| + \|\epsilon\| )\|_{\psi_2} \le C ( 1 + \|K\|_{\sigma}^{1/2} \|W\| ).
    \end{gathered}
    \end{equation}
\end{lemma}

\begin{proof} Consider an arbitrary vector $v = v_{\ell} \oplus 0_{\mathcal{T}}$ where $v_{\ell} \in \mathcal{O}$ and $\|v\|_{\mathcal{O}^+} = \|v_{\ell}\|_{\mathcal{O}} = 1$. By the definition of a directional derivative, we can write the norm of the gradient of $f$ as 
\begin{equation}
    \begin{aligned}
        \| \nabla_x f(r) \| 
        &
        =
        \max_{\|v_{\ell}\| =1} \ip{\nabla_x f(r), v_\ell}
        =
        \max_{ \|v_{\ell}\| = 1} \lim_{s \downarrow 0} \frac{f(r + sv) - f(r)}{s}.
    \end{aligned}
\end{equation}
For any $\delta > 0$, there exists an $s < 1$ such that 
\[
\max_{ \|v_{\ell}\| = 1} \lim_{s \downarrow 0} \frac{f(r + sv) - f(r )}{s} \le \max_{ \|v_{\ell}\| = 1}\frac{\|f(r + sv) - f(r )\|}{s\|v\|} + \delta.
\]
By $\alpha$-pseudo-Lipschitz, we deduce that 
\begin{equation}
    \begin{aligned}
        \|\nabla f_x(r) \| 
        &
        \le 
        \max_{\|v_{\ell}\| = 1} \frac{\|f(r + sv)- f(r) \|}{s \|v\|} + \delta
        \\
        &
        \le
        \max_{ \|v_{\ell}\| = 1} L(f) (1 + \|r + sv\|^{\alpha} + \|r\|^{\alpha} + 2\|\epsilon\|^{\alpha}) + \delta
        \\
        &
        \le 
        \max_{\|v_{\ell}\|= 1} L(f) \big (1 + (\|r\| + \|v\| )^{\alpha} + \|r\|^{\alpha} + 2\|\epsilon\|^{\alpha} \big ) + \delta.
    \end{aligned}
\end{equation}
We set $L \defas L(f)$. Sending $\delta \to 0$ and using that $\|v\| = 1$, we get that 
\begin{equation} \label{eq:grad_f_bound}
    \|\nabla_x f(r)\|^p \le C(\alpha, L, p) \big ( 1 + \|r\| + \|\epsilon\| \big )^{\alpha p} \le C(\alpha, L, p) \big ( 1 + \|r\| + \|\epsilon\| \big )^{\max \{1, \alpha p \} }, 
\end{equation}
where $C(\alpha, L, p)$ is a constant depending on $\alpha$, $p$, and the Lipschitz constant $L$. This gives the first expression in \eqref{eq:growth_grad_f}.  

Given the above expression \eqref{eq:grad_f_bound}, we need to compute $\EE[ (1+\|r\| + \|\epsilon\|)^{\alpha'} ] = \EE[ (1 + \|\ip{W, a}_{\mathcal{A}}\| + \|\epsilon\|)^{\alpha'}]$ with the expectation taken over $(a,\epsilon)$ and for some $\alpha' \ge 1$. In the process, we will also get a bound $\| 1 + \|r\| + \|\epsilon\|\|_{\psi_2}$.

The idea is to use Gaussian concentration of Lipschitz functions to get the bound, for any $\alpha' \ge 1$,
\begin{equation} \label{eq:grad_f_Lipschitz_4}
    \EE_{a, \epsilon} [ (1+ \|r\| + \|\epsilon\|)^{\alpha'} ] \le C(\alpha') \big ( 1 + \|K\|_{\sigma}^{1/2} \|W \| \big )^{\alpha'},
\end{equation}
where $C(\alpha')$ is a constant. 

For this, write $a = \sqrt{K} v$ where $v \sim N(0, I_{\mathcal{A}})$. It immediately follows that $\| \ip{W, a}_{\mathcal{A}} \| = \|\ip{ \ip{\sqrt{K}, W }_{\mathcal{A}}, v}_{\mathcal{A}}\|$. We will apply Gaussian concentration of Lipschitz function to the mapping $(v, \epsilon) \mapsto 1 +  \|\ip{ \ip{\sqrt{K}, W }, v}_{\mathcal{A}}\| + \|\epsilon\|$. The mapping is clearly Lipschitz in $(v, \epsilon)$ and the Lipschitz constant is $\|\ip{\sqrt{K}, W}_{\mathcal{A}}\| + 1$. 

Defining $X \defas \|\ip{ \ip{\sqrt{K}, W}_{\mathcal{A}}, v}_{\mathcal{A}}\| + \|\epsilon\|$ and $\hat{X} \defas 1 + X$, Gaussian concentration of Lipschitz functions \cite[Thorem 5.2.2]{vershynin2018high} gives that there exists an absolute constant $C$ such that 
\begin{align*}
    \| \hat{X} - \EE[\hat{X}] \|_{\psi_2} \le C ( 1+  \|\ip{\sqrt{K}, W}_{\mathcal{A}} \| ),
\end{align*}
where the concentration is taken with respect to the sub-Gaussian norm \eqref{eq:subgaussian_norm}. This, in particular, means that 
\begin{equation} \label{eq:grad_f_Lipschitz_1}
\begin{aligned}
\|\hat{X}\|_{\psi_2} 
&\le 
C (1 + \|\ip{\sqrt{K}, W }_{\mathcal{A}}\| )  + \| \EE[\hat{X}] \|_{\psi_2} \le C ( 1 + \|K\|_{\sigma}^{1/2} \|W\| + \|\EE[\hat{X}]\|_{\psi_2} )
\\
&
\le C ( 2 + \|K\|_{\sigma}^{1/2} \|W \|  + \|\EE[X]\|_{\psi_2} ),
\end{aligned}
\end{equation}
where $C$ is an absolute constant. With this expression in mind, we only need to compute a bound on $\|E[X]\|_{\psi_2}$. For this, we first observe that $(\EE[Z])^2 \le \EE[Z^2]$
and 
\begin{equation} 
    \begin{gathered}
        \EE[ \|\ip{W, a}_{\mathcal{A}}\|^2 ] = \tr \big ( \ip{K, W }_{\mathcal{A}} \big ) \le \|W \|^2 \|K\|_{\sigma} 
        \\
        \Rightarrow \qquad \EE[ \|\ip{W, a}_{\mathcal{A}}\|] \le \sqrt{\EE[ \|\ip{W, a}_{\mathcal{A}}\|^2 ]} \le \|K\|_{\sigma}^{1/2} \|W \|.
    \end{gathered}
\end{equation}
Moreover, as $\epsilon \sim N(0, I_{\mathcal{T}})$, we have $\EE[\|\epsilon\|] = \sqrt{|\mathcal{T}|}$ which is independent of $d$. Thus, $\EE[\|X\|] \le \|K\|_{\sigma}^{1/2} \|W\| + \sqrt{|\mathcal{T}|}$. 

By the definition of the sub-gaussian norm \eqref{eq:subgaussian_norm}, we have that there exists an absolute constant $C$ such that 
\begin{equation} \label{eq:grad_f_Lipschitz_2}
\|\EE[X]\|_{\psi_2} \le C (\|K\|_{\sigma}^{1/2} \|W \| + \sqrt{|\mathcal{T}|} ). 
\end{equation}

Now to get a bound on $\EE[(1 +\|r\| + \|\epsilon\|)^{\alpha'}] = \EE[ \|\hat{X}\|^{\alpha'}]$ from a bound on the sub-gaussian norm, we use the property that sub-gaussian norm bounds all norms, \cite[Property (ii), Proposition 2.5.2]{vershynin2018high},
\begin{equation} \label{eq:grad_f_Lipschitz_3}
(\EE [ (1 + \|r\| + \|\epsilon\|)^{\alpha'} ] )^{1/\alpha'} = (\EE [ \|\hat{X}\|^{\alpha'} ] )^{1/\alpha'} \le C \sqrt{\alpha'} \cdot \EE[\|\hat{X}\|_{\psi_2}], 
\end{equation}
where $C$ is an absolute constant. Putting this together, \eqref{eq:grad_f_Lipschitz_1}, \eqref{eq:grad_f_Lipschitz_2}, and \eqref{eq:grad_f_Lipschitz_3}, for any $\alpha' \ge 1$
\begin{equation} \label{eq:grad_f_Lipschitz_5}
\begin{aligned}
    \EE[ (1 + \|r\| + \|\epsilon\|)^{\alpha'} ] = \EE[ \|\hat{X}\|^{\alpha'} ] \le C(\alpha') (\EE[\|\hat{X}\|_{\psi_2} ])^{\alpha'} \le C(\alpha', |\mathcal{T}|) \big ( 1+ \|K\|_{\sigma}^{1/2} \|W \| \big )^{\alpha'},
\end{aligned}
\end{equation}
which shows \eqref{eq:grad_f_Lipschitz_2}. The first result \eqref{eq:expectation_f_growth} immediately follows from \eqref{eq:grad_f_Lipschitz_5} and \eqref{eq:grad_f_bound}. 

By combining \eqref{eq:grad_f_Lipschitz_1} and \eqref{eq:grad_f_Lipschitz_2}, the result \eqref{eq:growth_grad_f} on $\|1 + \|r\|\|_{\psi_2}$ also follows.
\end{proof}

\section{The Dynamical Nexus} \label{sec:approximate_solutions_stability}

A goal of this paper is to show that statistics $\varphi \, : \, \mathcal{A} \otimes \mathcal{O} \to \mathbb{R}$ satisfying Assumption~\ref{assumption:statistic} applied to SGD converge to a deterministic function \textit{and} statistics of homogenized SGD, $\WHSGD_t$, and SGD, $X_{\lfloor td \rfloor}$, are close. This argument hinges on understanding the deterministic dynamics of one important statistic, defined as 
\begin{equation} \label{eq:S_statistic}
S(W,z) = \ip{W \otimes W, R(z;K)}_{\mathcal{A}^{\otimes 2}},
\end{equation}
applied to $\HSGD_t$ (homogenized SGD updates) and $W_{\lfloor td \rfloor}$ (SGD updates). Here $W = X \oplus X^{\star}$ and $R(z;K)= (K-zI_d)^{-1}$ for $z \in \mathbb{C}$ is the resolvent of the matrix $K$. The argument we present is twofold. First, we compare the iterates of homogenized SGD, $\HSGD_t$, and SGD, $W_{\lfloor td \rfloor}$  under $S(\cdot, z)$ and show the two are close. Then we show that $S(W, z)$, with either homogenized SGD or SGD, is, itself, close to a deterministic function $(t,z) \mapsto \mathcal{S}(t,z)$ which satisfies an integro-differential equation (see \eqref{eq:ODE_resolvent_2}). Knowledge about the $S$ statistic is quite powerful as from it we recover the deterministic dynamics of \textit{any} statistic $\varphi$. We will make this idea explicit in Section~\ref{sec:compare_SGD_HSGD}. Beyond this, the dynamics of the mapping $S(W,z)$ itself often provide useful insights into analyzing the optimization trajectories of particular optimization problems (see Section~\ref{sec:analysis_examples}). Indeed, properties of the solutions to which the algorithms converge can be derived by looking at the mapping $S(W,z)$.

\subsection{Approximate solutions and stability} 
To introduce the integro-differential equation, recall by Assumption \ref{assumption:unbiased} and  \ref{assumption:E_loss_pseudo_Lip} that 
\[
\mathcal{R}(X) = h \circ B(W) \quad \text{and} \quad \EE_{a, \epsilon}[\nabla_x f(\ip{W, a}_{\mathcal{A}})^{\otimes 2}] = I \circ B(W) \quad \text{with} \quad \, B(W) = \ip{W^{\otimes 2}, K}_{\mathcal{A}^{\otimes 2}},
\] 
and $\alpha$-pseudo-Lipschitz functions $h \, : \, (\mathcal{O}^+)^{\otimes 2} \to \mathbb{R}$ differentiable and $I \, : \, (\mathcal{O}^+)^{\otimes 2} \to \mathbb{R}$. It will be useful, throughout the remaining paper, to decompose the derivative of $h$, i.e., $\nabla h$, in terms of its $\mathcal{O}$ and $\mathcal{T}$ components. The easiest and succinct way to do this is to consider a matrix structure
\begin{equation} \label{eq:matrix_form_1}
    (a \oplus b) \otimes (c \oplus d) \cong \left [ \begin{array}{c|c}
        a \otimes c & a \otimes d\\
        \hline
        b \otimes c & b \otimes d
    \end{array} \right ].
\end{equation}
In this regard, we express $\nabla h$ in terms of this matrix,
\begin{align*}
\nabla h 
&
\cong
\left[ \begin{array}{c|c} 
    \nabla h_{11} & \nabla h_{12}
    \\
    \hline
    \nabla h_{21} & \nabla h_{22}
    \end{array} \right ] \in \left[ \begin{array}{c|c} 
    \mathcal{O} \otimes \mathcal{O} & \mathcal{O} \otimes \mathcal{T}
    \\
    \hline
    \mathcal{T} \otimes \mathcal{O} & \mathcal{T} \otimes \mathcal{T}
    \end{array} \right ].
\end{align*}
With these recollections, the integro-differential equation is defined below. 

\begin{mdframed}[style=exampledefault]
\textbf{Integro-Differential Equation for $\mathcal{S}(t, z)$.} For any contour $\Gamma \subset \mathbb{C}$ enclosing the eigenvalues of $K$, we have an expression for the derivative of $\mathcal{S}$:
\begin{equation}\label{eq:ODE_resolvent_2}
    \dif \mathcal{S}(t,\cdot) 
    = \mathscr{F}(z, \mathcal{S}(t, \cdot)) \, \dif t
\end{equation}
\begin{align}
    \text{where} \, \, 
    \mathscr{F}(z, \mathcal{S}(t, \cdot))
    &
    \defas
    - 2\gamma_t \bigg ( \bigg ( \frac{-1}{2\pi i} \oint_{\Gamma} \mathcal{S}(t,z) \, \dif z \bigg ) H( \mathrsfs{B}(t)) + H^T(\mathrsfs{B}(t)) \bigg ( \frac{-1}{2\pi i} 
 \oint_{\Gamma} \mathcal{S}(t,z) \, \dif z \bigg ) \bigg ) \,  \nonumber
 \\
    & \qquad
    + \frac{\gamma_t^2}{d}\left [ \begin{array}{c|c} 
    \tr(K R(z;K)) I(\mathrsfs{B}(t)) & 0\\ \hline 0 & 0
    \end{array} \right ]  \label{eq:F}\\
    & \qquad 
    - \gamma_t (\mathcal{S}(t,z) (2z H(\mathrsfs{B}(t)) + \delta D) + ( 2 z H^T( \mathrsfs{B}(t) ) + \delta D) \mathcal{S}(t,z)). \nonumber
\end{align}
\begin{gather} \nonumber
\text{Here} \, \, \mathrsfs{B}(t) = \frac{-1}{2\pi i} \oint_{\Gamma} z \mathcal{S}(t,z) \, \dif z, 
 \quad 
 H(\mathrsfs{B}) = \left [ \begin{array}{c|c} 
\nabla h_{11}(\mathrsfs{B}) & 0\\
 \hline
 \nabla h_{21}(\mathrsfs{B}) & 0
 \end{array} \right ],  \quad \text{and} \quad D = \left [ \begin{array}{c|c} 
 I_{\mathcal{O}} & 0\\
 \hline
 0 & 0
 \end{array} \right ], \\
 \label{eq:ODE_IC}
 \text{and initialization} \quad \mathcal{S}(0,z) = \ip{W_0 \otimes W_0, R(z; K)}_{\mathcal{A}^{\otimes 2}}. 
\end{gather}
\end{mdframed}
In this section, we will be interested in approximate solutions to the integro-differential equation \eqref{eq:ODE_resolvent_2} (see below for specifics). The idea is that both $S(\HSGD_t, z)$ and $S(W_{\lfloor td \rfloor}, z)$, which are functions of both homogenized SGD and SGD respectively, are approximate solutions.  We also note that there is in fact an actual solution to the integro-differential equation, which is a re-representation of \eqref{eq:coupledODElimit}.
\begin{lemma}[Equivalence to coupled ODEs]\label{lem:ODE_S}
The unique solution of \eqref{eq:ODE_resolvent_2} with initial condition \eqref{eq:ODE_IC} is given by
\[
\mathcal{S}(t,z) = \frac{1}{d}\sum_{i=1}^n \frac{1}{\lambda_i - z} \mathrsfs{B}_{t,i}
\quad \text{for all } z \in \Gamma.
\]
\end{lemma}
\begin{proof}
    We first observe that this satisfies \eqref{eq:ODE_resolvent_2}, which can be checked directly from \eqref{eq:coupledODElimit} using the identity
    \[
     \frac{1}{d}\sum_{i=1}^d \frac{\lambda_i}{\lambda_i - z} \mathrsfs{B}_{t,i}
     =
     \frac{1}{d}\sum_{i=1}^d \mathrsfs{B}_{t,i}
     +
     z\frac{1}{d}\sum_{i=1}^d \frac{1}{\lambda_i - z} \mathrsfs{B}_{t,i}
     = \frac{-1}{2\pi i} \oint \mathcal{S}(t,y) \dif y + z\mathcal{S}(t,z).
    \]
    Conversely, given a solution to \eqref{eq:ODE_resolvent_2}, we observe that the process $\mathcal{S}(t,z)$ is a meromorphic function in $z$, with simple poles at the spectrum of $K$ and tending to $0$ as $z \to \infty$.  Hence by analyticity, \eqref{eq:F} holds at all $z$ not in the spectrum of $K$.
    It follows that we have a partial fraction decomposition
    \[
    \mathcal{S}(t,z)
    = \sum_{i=1}^d \frac{1}{\lambda_i - z}\mathrsfs{X}_{t,i}.
    \]
    In the case that $K$ has $d$ distinct eigenvalues, by contour integrating \eqref{eq:F} around a simple contour enclosing a single eigenvalue $\lambda_i$, we conclude that \eqref{eq:coupledODElimit} holds for the family $(d \mathrsfs{X}_{t,i} : 1 \leq i \leq d)$.  By uniqueness of the coupled family of ODEs, we are done.
    In the case of non-simple spectrum, we have that for all $\lambda \in \operatorname{Spec}(K)$
    \[
    \sum_{i : \lambda_i = \lambda} d \mathrsfs{X}_{t,i}
    =
    \sum_{i : \lambda_i = \lambda} \mathrsfs{B}_{t,i},
    \]
    since they both again satisfy \eqref{eq:coupledODElimit} (with $\lambda_i \to \lambda$) and have the same initial conditions -- as those ODEs have unique solutions, we conclude that there is a unique solution of \eqref{eq:ODE_resolvent_2}.
\end{proof}

For working with approximate solutions to \eqref{eq:ODE_resolvent_2}, we introduce some notation. 
We shall always work on a fixed contour $\Gamma$ surrounding the spectrum of $K$, given by 
$\Gamma \defas \{ z \, : \, |z| = \max\{1, 2\|K\|_{\sigma}\} \}$. We note that this contour is always distance at least $\tfrac12$ from the spectrum of $K$.
We define a norm, $\| \cdot \|_{\Gamma}$ on a continuous function $A \, : \, \mathbb{C} \to (\mathcal{O}^+)^{\otimes 2}$ by
\[
\|A\|_{\Gamma} = \max_{z \in \Gamma} \|A(z)\|.
\]
We note that up to constants that depend on $\|K\|_{\sigma}$, this norm applied to $\mathcal{S}(t,\cdot)$, $S(\HSGD_t, \cdot)$ and $S(W_{\lfloor td \rfloor}, \cdot)$ has an equivalent representation in terms of the norm-squared of the parameters:
\begin{lemma}\label{lem:normequivalence}
Let $\mathrsfs{N}(t) \defas  \frac{-1}{2\pi i} \oint_\Gamma \tr \mathcal{S}(t, z) \dif z$ which is positive.  Then for a constant $C$ depending on the $\|K\|_{\sigma}$ and $|\mathcal{O}^+|,$
\[
C \leq 
\frac{\| S(\HSGD_t, \cdot ) \|_\Gamma}{ \|\HSGD_t\|^2},
\frac{\| S(W_{td}, \cdot ) \|_\Gamma}{ \|W_{td}\|^2},
\frac{ \|\mathcal{S}(t, \cdot )\|_{\Gamma}}{\mathrsfs{N}(t)} \leq 2.
\]
\end{lemma}
\begin{proof}
    For homogenized SGD,
    \[
        \|\HSGD_t\|^2 = \frac{-1}{2\pi i} \oint_\Gamma \tr S(\HSGD_t, z) \dif z \leq C \sqrt{ |\mathcal{O}^+|} \|K\|_\sigma \| S(\HSGD_t, \cdot ) \|_\Gamma.
    \]
    On the other hand, 
    \[
    \| S(\HSGD_t, \cdot ) \|_\Gamma
    = 
    \max_{z \in \Gamma} \|\ip{ \HSGD_t^{\otimes 2}, R(z ; K )}\|
    \leq \|\HSGD_t\|^2 \max_{z \in \Gamma} \|R(z ; K )\|_\sigma
    \leq 2\|\HSGD_t\|^2.
    \]
    The same bounds hold for SGD with obvious changes. 

    For the integro-differential equation, 
    we start by observing that
    \[
    \mathrsfs{N}(t) = 
    \frac{-1}{2\pi i} \oint_\Gamma \tr \mathcal{S}(t, z) \dif z = \frac{1}{d}\sum_{i=1}^d \tr(\mathrsfs{B}_{i}(t)),
    \]
    which is positive.  Then with $|\Gamma|$ given by the length of $\Gamma$,
    \[
    \frac{-1}{2\pi i} \oint_\Gamma \tr \mathcal{S}(t, z) \dif z \leq \frac{1}{2\pi} |\Gamma| \sqrt{ |\mathcal{O}^+|}  \| S(\HSGD_t, \cdot ) \|_\Gamma.
    \]
    Using Lemma \ref{lem:ODE_S}, we have
    \[
    \|\mathcal{S}(t, \cdot )\|_{\Gamma}
    \leq
    \frac{1}{d}\sum_{i=1}^d \max_{z \in \Gamma} \biggl| \frac{1}{\lambda_i-z} \biggr| \|\mathrsfs{B}_{i}(t)\|
    \leq 
    \frac{2}{d}\sum_{i=1}^d \|\mathrsfs{B}_{i}(t)\|.
    \]
    As each $\mathrsfs{B}_{i}(t)$ is positive semidefinite, we have $\|\mathrsfs{B}_{i}(t)\| \leq \|\mathrsfs{B}_{i}(t)\|_* = \tr \mathrsfs{B}_{i}(t)$, and so the same bound holds.
    
\end{proof}

We will be working with \textit{approximate solutions to the integro-differential equation} defined as: 
\begin{definition}[$(\varepsilon, M, T)$-approximate solution to the integro-differential equation]  \label{def:integro_differential_equation}
{\rm 
 For constants $M, T, \varepsilon > 0$, we call continuous functions $\mathscr{S} \, : \, \{t \ge 0\} \otimes \mathbb{C} \to (\mathcal{O}^+)^{\otimes 2}$ an \textit{$(\varepsilon, M, T)$-approximate solution} of \eqref{eq:ODE_resolvent_2} if with 
 \[
 \hat{\tau}_{M}(\mathscr{S}) \defas \inf \bigg \{ t \ge 0 \, : \, \|\mathscr{S}(t,\cdot)\|_{\Gamma} > M \quad \text{or}  \quad \frac{-1}{2 \pi i} \oint_{\Gamma} z \mathscr{S}(t,z) \, \dif z \not \in \mathcal{U} \bigg \},
 \]
 then
\[
\sup_{0 \le t \le (\hat{\tau}_M \wedge T)} \big \| \mathscr{S}(t, \cdot) - \mathscr{S}(0, \cdot) - \int_0^t \mathscr{F}(\cdot, \mathscr{S}(s, \cdot) ) \, \dif s \big \|_{\Gamma} \le \varepsilon
\]
and $\mathscr{S}(0,\cdot) = \ip{W_0 \otimes W_0, R(\cdot, K)}_{\mathcal{A}^{\otimes 2}}$, where $W_0 = X_0 \otimes X^{\star}$ is the initialization of SGD. 

We suppress the $\mathscr{S}$ in the notation for $\hat{\tau}_M$, that is $\hat{\tau}_M = \hat{\tau}_M(\mathscr{S})$, when it is clear the function $\mathscr{S}$ from context. 
}
\end{definition}

\begin{remark} In Section~\ref{sec:SGD_homogenized_SGD}, we prove that SGD and homogenized SGD, $S(W_{\lfloor td \rfloor}, z)$ and $S(\HSGD_t, z)$, respectively, are $(\varepsilon, M,T)$-approximate solutions. Note that we must extend the discrete time of SGD to a continuous time (see Section~\ref{sec:SGD_approx_solution} for details). It is clear by the definition of the solution to the deterministic integro-differential equation, $\mathcal{S}$, in \eqref{eq:ODE_resolvent_2} is an $(\varepsilon, M, T)$-approximate solution with $\varepsilon = 0$. 
\end{remark}

Our first result of this section is a \textit{stability} statement, that is, if we have two $(\varepsilon, M,T)$-approximate solutions, $\mathscr{S}_1$ and $\mathscr{S}_2$, then $\mathscr{S}_1$ and $\mathscr{S}_2$ are uniformly close. 

\begin{proposition}[Stability] \label{prop:stability} For all $(\varepsilon, M, T)$-approximate solutions $\mathscr{S}_1$ and $\mathscr{S}_2$, there exists a positive constant $C = C(M,T, \|K\|_{\sigma}, \bar{\gamma})$ such that 
\[
\sup_{0 \le t \le T} \, \|\mathscr{S}_1(t \wedge \tau_M, \cdot)- \mathscr{S}_2(t \wedge \tau_M, \cdot) \|_{\Gamma} \le C \cdot \varepsilon,
\]  
where $\tau_{M} = \min \{ \hat{\tau}_M(\mathscr{S}_1), \hat{\tau}_M(\mathscr{S}_2) \}$.
\end{proposition}

\begin{proof} First note that $\tau_M \le \hat{\tau}_M(\mathscr{S}_1)$ and $\tau_M \le \hat{\tau}_M(\mathscr{S}_2)$. Therefore, we can work on the smaller time $\tau_M$. Write $\mathscr{S}_1$ and $\mathscr{S}_2$ as
\begin{align} \label{eq:S_1_S_2}
\mathscr{S}_1(t, \cdot) 
&
= 
\mathscr{S}_1(0, \cdot) + \int_0^t \mathscr{F}(\cdot, \mathscr{S}_1(s, \cdot)) \, \dif s + \varepsilon(\mathscr{S}_1) \, \,
\text{and} \, \, \mathscr{S}_2(t, \cdot)
= 
\mathscr{S}_2(0, \cdot) + \int_0^t \mathscr{F}(\cdot, \mathscr{S}_2(s, \cdot)) \, \dif s + \varepsilon(\mathscr{S}_2),
\end{align}
where $\varepsilon(\mathscr{S}_i)$ are error terms from the $(\varepsilon, M, T)$-approximate solution inequality and we have for $i = 1,2$ 
\[
\displaystyle \sup_{0 \le t \le (T \wedge \tau_M)} \|\varepsilon(\mathscr{S}_i)\|_{\Gamma} \le \varepsilon.
\] 
Let us suppose that there exists a positive constant $C = C(M, \|K\|_{\sigma}, \bar{\gamma})$ such that for all $s$
\begin{equation} \label{eq:pseudo_Lip_F}
\| \mathscr{F}(\cdot, \mathscr{S}_1(s \wedge \tau_M, \cdot) ) - \mathscr{F}(\cdot, \mathscr{S}_2(s \wedge \tau_M, \cdot) ) \|_{\Gamma} \le C \|\mathscr{S}_1(s \wedge \tau_M, \cdot) - \mathscr{S}_2(s \wedge \tau_M, \cdot) \|_{\Gamma}. 
\end{equation}
We defer the proof of the Lipschitz condition \eqref{eq:pseudo_Lip_F} for $\mathscr{F}$ until later. Equation~\eqref{eq:pseudo_Lip_F} and \eqref{eq:S_1_S_2} imply
\begin{align*}
    \sup_{0 \le t \le T \wedge \tau_M} \| \mathscr{S}_1(t, \cdot) - \mathscr{S}_2(t, \cdot) \|_{\Gamma} 
    &
    \le 
    2 \varepsilon + \sup_{0 \le t \le T \wedge \tau_M} \int_0^t \| \mathscr{F}( \cdot, \mathscr{S}_1(s, \cdot) ) - \mathscr{F}(\cdot, \mathscr{S}_2(s, \cdot) ) \|_{\Gamma} \, \dif s
    \\
    &
    \le
    2 \varepsilon + \sup_{0 \le t \le T} \int_0^t \| \mathscr{F}( \cdot, \mathscr{S}_1(s \wedge \tau_M, \cdot) ) - \mathscr{F}(\cdot, \mathscr{S}_2(s \wedge \tau_M, \cdot) ) \|_{\Gamma} \, \dif s
    \\
    &
    \le 2 \varepsilon + C(M, \|K\|_{\sigma}, \bar{\gamma} ) \int_0^T \| \mathscr{S}_1(s \wedge \tau_M, \cdot) - \mathscr{S}_2(s \wedge \tau_M, \cdot) \|_{\Gamma} \, \dif s.
\end{align*}
Define $\displaystyle Q_T \defas \sup_{0 \le t \le T} \| \mathscr{S}_1( t \wedge \tau_M, \cdot) - \mathscr{S}_2(t \wedge \tau_M, \cdot) \|_{\Gamma}$. Then one has that 
\begin{align*}
    Q_T = \sup_{0 \le t \le T \wedge \tau_M} \| \mathscr{S}_1(t, \cdot) - \mathscr{S}_2(t, \cdot) \|_{\Gamma} \le 2 \varepsilon + C \int_0^T Q_s \, \dif s.
\end{align*}
By an application of Gronwall's inequality, the result is shown. 

It remains now to show that $\mathscr{F}$ is Lipschitz, that is, the expression \eqref{eq:pseudo_Lip_F} holds. We will do this in steps. First, define $\mathscr{B}_i(\cdot) \defas \frac{-1}{2 \pi i} \oint_{\Gamma} z \mathscr{S}_i(\cdot, z) \, \dif z$ and $\mathscr{I}_i(\cdot) \defas \frac{-1}{2 \pi i} \oint_{\Gamma} \mathscr{S}_i(\cdot, z) \, \dif z$ for $i = \{1,2\}$. We will use the shorthand $\mathscr{B}_i^{\tau_M}(s) \defas \mathscr{B}_i(s \wedge \tau_M)$,  $\mathscr{I}_i^{\tau_M}(s) = \mathscr{I}_i( s \wedge \tau_M)$, and $\mathscr{S}_i^{\tau_M}(s, \cdot ) = \mathscr{S}_i(s \wedge \tau_M, \cdot)$. Now by the $\alpha$-pseudo-Lipschitz of $\nabla h$ (Assumption~\ref{assumption:statistic} ), 
\begin{align*}
    \|H(\mathscr{B}_1^{\tau_M}(s) )- H(\mathscr{B}_2^{\tau_M}(s))\| 
    &
    \le 
    ( 1 + \|\mathscr{B}_1^{\tau_M}(s)\|^{\alpha} + \|\mathscr{B}_2^{\tau_M}(s)\|^{\alpha} ) \|\mathscr{B}_1^{\tau_M}(s) - \mathscr{B}_2^{\tau_M}(s)\|
    \\
    &
    \le
    C(M, L(h), \alpha) \|\mathscr{B}_1^{\tau_M}(s) - \mathscr{B}_2^{\tau_M}(s)\|
\end{align*}
since 
\begin{equation} \label{eq:B_bound}
    \|\mathscr{B}_i^{\tau_M}(s)\| = \big \| \frac{-1}{2\pi i} \oint_{\Gamma} z \mathscr{S}_i^{\tau_M}(s, z) \, \dif z \big \| \le C( |\Gamma|) \|\mathscr{S}_i^{\tau_M}(s, \cdot) \|_{\Gamma} \le C( \|K\|_{\sigma} ) \cdot M.
\end{equation}
Here we used the stopping time $\tau_M$ explicitly. Now we see that 
\begin{equation} \label{eq:B_Lip}
    \|\mathscr{B}^{\tau_M}_1(s) - \mathscr{B}^{\tau_M}_2(s)\| \le C \oint_\Gamma |z| \|\mathscr{S}_1^{\tau_M}(s, \cdot) - \mathscr{S}_2^{\tau_M}(s, \cdot) \|_{\Gamma} \, \dif |z| \le C( \|K\|_{\sigma}) \|\mathscr{S}_1^{\tau_M}(s, \cdot) - \mathscr{S}_1^{\tau_M}(s, \cdot) \|_{\Gamma}.
\end{equation}
Consequently, there exists a positive constant (independent of $s$) such that 
\begin{equation} \label{eq:H_Lip}
    \|H(\mathscr{B}_1^{\tau_M}(s)) - H(\mathscr{B}_2^{\tau_M}(s))\| \le C(M, \|K\|_{\sigma}, L(h), \alpha) \cdot \|\mathscr{S}_1^{\tau_M}(s, \cdot) - \mathscr{S}_2^{\tau_M}(s, \cdot) \|_{\Gamma}.
\end{equation}
Analogous to \eqref{eq:B_bound} and \eqref{eq:B_Lip}, 
\begin{equation} \label{eq:I_Lip}
\begin{aligned}
    \|\mathscr{I}_1^{\tau_M}(s) - \mathscr{I}_2^{\tau_M}(s) \| 
    &
    \le 
    C(M, \|K\|_{\sigma}) \cdot \| \mathscr{S}_1^{\tau_M}(s, \cdot) - \mathscr{S}_2^{\tau_M}(s, \cdot)\|_{\Gamma} 
    \\
    \|\mathscr{I}_i^{\tau_M}(s)\|
    &
    \le 
    C(|\Gamma|) \|\mathscr{S}_i^{\tau_M}(s, \cdot)\|_{\Gamma} \le C(\|K\|_{\sigma}) \cdot M. 
\end{aligned}
\end{equation}
Moreover by Assumption~\ref{assumption:unbiased} and the bound on $\mathscr{B}_i^{\tau_M}(s)$ in \eqref{eq:B_bound}
\begin{equation} \label{eq:H_bound}
    \|H(\mathscr{B}_i^{\tau_M}(s)) \| \le L(h) (1 + \|\mathscr{B}_i^{\tau_M}(s)\|)^{\alpha} \le C(\|K\|_{\sigma}, L(h), \alpha, M).
\end{equation}
It follows from Equations~\eqref{eq:B_bound}, \eqref{eq:B_Lip}, \eqref{eq:H_Lip}, \eqref{eq:I_Lip}, and \eqref{eq:H_bound} the existence of a positive constant $C = C(M, \|K\|_{\sigma}, L(h), \alpha, \bar{\gamma})$ such that
\begin{equation} \label{eq:F_1}
\|2 \gamma(s) \mathscr{I}_1^{\tau_M}(s) H( \mathscr{B}_1^{\tau_M}(s)) - 2 \gamma(s) \mathscr{I}_2^{\tau_M}(s) H( \mathscr{B}_2^{\tau_M}(s)) \| \le  C \cdot \| \mathscr{S}_1^{\tau_M}(s, \cdot) - \mathscr{S}_2^{\tau_M}(s, \cdot) \|_{\Gamma}.
\end{equation}
An analogous argument shows
\begin{equation} \label{eq:F_2}
\|2 \gamma(s) H^T( \mathscr{B}_1^{\tau_M}(s)) \mathscr{I}_1^{\tau_M}(s) - 2 \gamma(s) H^T( \mathscr{B}_2^{\tau_M}(s)) \mathscr{I}_2^{\tau_M}(s)  \| \le  C \cdot \| \mathscr{S}_1^{\tau_M}(s, \cdot) - \mathscr{S}_2^{\tau_M}(s, \cdot) \|_{\Gamma}.
\end{equation}

Next we consider the term $\mathscr{S}(s,z) ( 2 z H(\mathscr{B}(s)) + \delta D)$ and noting that an analogous proof holds for $( 2 z H^T(\mathscr{B}(s)) + \delta D)\mathscr{S}(s,z)$. We immediately have that
\begin{equation} \label{eq:S_term}
    \begin{aligned}
    \|\delta D (\mathscr{S}_1^{\tau_M}(s, \cdot) - \mathscr{S}_2^{\tau_M}(s, \cdot) \big )\|_{\Gamma} 
    &
    \le 
    \delta |\mathcal{O}| \|\mathscr{S}_1^{\tau_M}(s, \cdot) - \mathscr{S}_2^{\tau_M}(s, \cdot)\|_{\Gamma}
    \\
    \text{and} \quad 
    \|2z ( \mathscr{S}_1^{\tau_M}(s, z) - \mathscr{S}_2^{\tau_M}(s, z) ) \|_{\Gamma}
    &
    \le
    C(\|K\|_{\sigma}) \|\mathscr{S}_1^{\tau_M}(s, \cdot) - \mathscr{S}_2^{\tau_M}(s, \cdot)\|_{\Gamma}
    \end{aligned}
\end{equation}
and $\|2z \mathscr{S}_i^{\tau_M}(s, \dot)\|_{\Gamma}, \|\delta D \mathscr{S}_i^{\tau_M}(s, \cdot)\|_{\Gamma} \le C(\|K\|_{\sigma}, \delta, |\mathcal{O}|) \cdot M$ where $|\mathcal{O}| = \ell$ is independent of $d$. Consequently, by \eqref{eq:H_Lip} and \eqref{eq:H_bound} for $H(\mathscr{B}(s))$, we have that 
\begin{equation}
\begin{aligned} \label{eq:F_term_2}
\|\gamma(s) ( \mathscr{S}_1^{\tau_M}(s, z) & (2z H(\mathscr{B}_1^{\tau_M}(s) ) + \delta D) - \mathscr{S}_2^{\tau_M}(s, z) (2z H(\mathscr{B}_2^{\tau_M}(s) ) + \delta D) ) \|_{\Gamma} 
\\
&
\le
C \cdot \|\mathscr{S}_1^{\tau_M}(s, \cdot) - \mathscr{S}_2^{\tau_M}(s, \cdot) \|_{\Gamma}
\end{aligned}
\end{equation}
where $C = C(M, \|K\|_{\sigma}, L(h), \alpha, \bar{\gamma}, \delta, |\mathcal{O}|)$ is a positive constant. 

What remains is the third and final term in $\mathscr{F}$, $\tfrac{\gamma(s)^2}{2} \tr(KR(z;K)) I(\mathscr{B}(s))$. Lastly,
\begin{equation} 
\begin{aligned}
\| \tfrac{\gamma(s)^2}{d} \tr (K R(z;K))
& 
\big ( I(\mathscr{B}_1^{\tau_M}(s)) - I(\mathscr{B}_2^{\tau_M}(s)) \big ) \|_{\Gamma}
\\
&
\le \tfrac{\bar{\gamma}^2}{d} | \tr(K) + z \tr(R(z; K))|_{\Gamma} \|I(\mathscr{B}_1^{\tau_M}(s)) - I(\mathscr{B}_2^{\tau_M}(s))\|_{\Gamma} 
\\
&
\le
\bar{\gamma}^2 ( \|K\|_{\sigma} + 1) \|I(\mathscr{B}_1^{\tau_M}(s)) - I(\mathscr{B}_2^{\tau_M}(s))\|_{\Gamma}. 
\end{aligned}
\end{equation}
By $\alpha$-pseudo-Lipschitz of Fisher matrix (Assumption~\ref{assumption:E_loss_pseudo_Lip}) and the inequalities \eqref{eq:B_bound} and \eqref{eq:B_Lip} 
\[
\|I(\mathscr{B}_1^{\tau_M}(s)) - I(\mathscr{B}_2^{\tau_M}(s))\|_{\Gamma} \le C(M, \alpha, L(I)) \|\mathscr{S}_1^{\tau_M}(s, \cdot) - \mathscr{S}_2^{\tau_M}(s, \cdot) \|_{\Gamma}. 
\]
Therefore, we deduce that 
\begin{equation}\label{eq:F_3}
\begin{aligned}
\| \tfrac{\gamma(s)^2}{d} \tr (K R(z;K))
& 
\big ( I(\mathscr{B}_1^{\tau_M}(s)) - I(\mathscr{B}_2^{\tau_M}(s)) \big ) \|_{\Gamma}
\le C \cdot \|\mathscr{S}_1^{\tau_M}(s, \cdot) - \mathscr{S}_2^{\tau_M}(s, \cdot) \|_{\Gamma}, 
\end{aligned}
\end{equation}
where $C = C(M, \|K\|_{\sigma}, L(I), \alpha, \bar{\gamma})$ is a positive constant.

The Lipschitz condition for $\mathcal{F}$ \eqref{eq:pseudo_Lip_F} holds after applying expressions \eqref{eq:F_1}, \eqref{eq:F_2}, \eqref{eq:F_term_2}, and \eqref{eq:F_3}. 
\end{proof}

Having established stability (Proposition~\ref{prop:stability}), we now show the same result holds for any statistic $\varphi(X) = (g \circ Q)(W)$
satisfying Assumption~\ref{assumption:statistic}. Here 
\[Q(W) \defas \ip{W^{\otimes 2}, q(K)}_{\mathcal{A}^{\otimes 2}}, 
\]
where $q(K)$ is a polynomial in $K$. For this, we introduce the notation: for $\mathscr{S}_i$ an $(\epsilon, M, T)$-approximate solution, we define
\begin{equation} \label{eq:q_function}
\mathscr{Q}_i(t) \defas \frac{-1}{2\pi i} \oint_{\Gamma} q(z) \mathscr{S}_i(t, z) \, \dif z.
\end{equation}
The following proposition shows that given two approximate solution, $\mathscr{S}_1$ and $\mathscr{S}_2$, $g \circ \mathscr{Q}_1(t)$ is close to $g \circ \mathscr{Q}_2(t)$. The idea is that the pseudo-Lipschitzness of $g$ allows us to show that 
\[
\sup_{0 \le t \le T} \| g( \mathscr{Q}_1(t \wedge \tau_M) ) - g( \mathscr{Q}_2(t \wedge \tau_M) ) \| \le \sup_{0 \le t \le T} \|\mathscr{S}_1(t \wedge \tau_M, \cdot)- \mathscr{S}_2(t \wedge \tau_M, \cdot) \|_{\Gamma}
\]
and then Proposition~\ref{prop:stability} finishes the result. 

\begin{proposition}\label{prop:stability_varphi} Suppose $\varphi \, : \, \mathcal{A} \otimes \mathcal{O} \to \mathbb{R}$ is a statistic satisfying Assumption~\ref{assumption:statistic} such that $\varphi(X) = g \circ Q(W)$. Suppose $\mathscr{S}_1$ and $\mathscr{S}_2$ are $(\varepsilon, M, T)$-approximate solutions. Then there exists a positive constant $C = C(M, T, \|K\|_{\sigma}, \|q\|_{\Gamma}, \bar{\gamma})$ such that
\[
\sup_{0 \le t \le T} \| g \big ( \tfrac{-1}{2 \pi i} \oint_{\Gamma} q(z) \mathscr{S}_1^{\tau_M}(t, z) \, \dif z \big ) - g \big ( \tfrac{-1}{2 \pi i} \oint_{\Gamma} q(z) \mathscr{S}_2^{\tau_M}(t, z) \, \dif z \big ) \|  \le C \cdot \varepsilon,
\]
where $\tau_{M} = \inf \{ t \ge 0 \, : \, \|\mathscr{S}_1(t, \cdot)\|_{\Gamma} \ge M \, \, \text{or} \, \, \|\mathscr{S}_2(t, \cdot)\|_{\Gamma} \ge M \}$. Here $\mathscr{S}_i^{\tau_M}(t, \cdot) =\mathscr{S}_i(t \wedge \tau_M), \cdot)$.
\end{proposition}

\begin{proof} Since $\tau_M \le \hat{\tau}_M(\mathscr{S}_1)$ and $\tau_M \le \hat{\tau}_M(\mathscr{S}_2)$, we can always work on the smaller time $\tau_M$. We define $\mathscr{Q}_i(t) = \tfrac{-1}{2\pi i} \oint_{\Gamma} q(z) \mathscr{S}_i(t,z) \, \dif z$ and the stopped process $\mathscr{Q}_i^{\tau_M}(t) = \mathscr{Q}_i(t \wedge \tau_M)$ for $i = 1,2$. First, we observe that 
\begin{equation} \label{eq:Q_bound_statistic}
\|\mathscr{Q}_i^{\tau_M}(t)\| \le C \oint_{\Gamma} |q(z)| \|\mathscr{S}_i^{\tau_M}(t,z)\| \, \dif z \le C(\|K\|_{\sigma}, \|q\|_\Gamma) \| \mathscr{S}_i^{\tau_M}(t, \cdot) \|_{\Gamma} \le  C(\|K\|_{\sigma}, \|q\|_\Gamma) \cdot M. 
\end{equation}
Moreover, the function $\mathscr{Q}$ is Lipschitz, that is, 
\begin{equation} \label{eq:Q_Lip_statistic}
\begin{aligned}
\|\mathscr{Q}_1^{\tau_M}(t) - \mathscr{Q}_2^{\tau_M}(t)\|  
&
\le 
C(\|q\|_{\Gamma} ) \oint_{\Gamma} \|\mathscr{S}_1^{\tau_M}(t, z) - \mathscr{S}_2^{\tau_M}(t, z) \| \, \dif |z| \\
&
\le C(\|K\|_{\sigma}, \|q\|_{\Gamma}) \|\mathscr{S}_1^{\tau_M}(t, \cdot) - \mathscr{S}_2^{\tau_M}(t, \cdot) \|_{\Gamma}.
\end{aligned}
\end{equation}
Since $g$ is $\alpha$-pseudo-Lipschitz (Assumption~\ref{assumption:statistic}) and the boundedness and Lipschitzness of  $\mathscr{Q}$ (see \eqref{eq:Q_bound_statistic} and \eqref{eq:Q_Lip_statistic}),
\begin{equation}
\begin{aligned}
    \| g( \mathscr{Q}_1^{\tau_M}(t) ) - g( \mathscr{Q}_2^{\tau_M}(t) ) \| 
    &
    \le 
    L(g) \|\mathscr{Q}_1^{\tau_M}(t) - \mathscr{Q}_2^{\tau_M}(t) \| \big ( 1 + \|\mathscr{Q}_1^{\tau_M}(t)\|^{\alpha} + \|\mathscr{Q}_2^{\tau_M}(t)\|^{\alpha} \big ) \\
    &
    \le
    C \cdot \|\mathscr{S}_1^{\tau_M}(t, \cdot) - \mathscr{S}_2^{\tau_M}(t, \cdot) \|_{\Gamma},
\end{aligned}
\end{equation}
where $C = C(\|K\|_{\sigma}, M, \|q\|_{\Gamma}, L(g), \alpha)$ is a positive constant. Taking the supremum over all $0 \le t \le T$ and applying Proposition~\ref{prop:stability} finishes the result. 
\end{proof}

\subsection{Main argument of the proof -- concentration of SGD and homogenized SGD under $S$} \label{sec:compare_SGD_HSGD}

In this section, we derive one of our main results -- concentration of both homogenized SGD and SGD under the statistic $S$ to the deterministic function $\mathcal{S}(t,z)$ that satisfies the integro-differential equation \eqref{eq:ODE_resolvent_2}. We will first prove a more general result than Theorem~\ref{thm:learning_curves} involving the resolvent, see Theorem~\ref{thm:main_concentration_S}.  The important statistic which will play a pivotal role is 
\begin{equation}
S(W,z) = \ip{W \otimes W, R(z;K)}_{\mathcal{A}^{\otimes 2}},
\end{equation}
as well as the function
\[
B(W) = \ip{W^{\otimes 2}, K}_{\mathcal{A}^{\otimes 2}}. 
\]

We will extend the iterates of SGD, $\{X_k\}$ defined on discrete time $k$, to continuous time. This is so that we can compare SGD and homogenized SGD, $\{\WHSGD_t\}$. We relate the $k$-th iterate of SGD to the continuous time parameter $t$ in homogenized SGD through the relationship $k = \lfloor td \rfloor$. Thus, when $t = 1$, SGD has done exactly $d$ updates. Under this mapping, we write the iterates of SGD with the continuous time parameter as $X_{td} = X_{\lfloor td \rfloor}$ (see Section~\ref{sec:SGD_homogenized_SGD} for additional details). 

We are now ready to state and prove one of our main results. 

\begin{theorem}[Concentration of SGD, Homogenized SGD, and deterministic function $\mathcal{S}(t,z)$] \label{thm:main_concentration_S_min}
Suppose the risk function $\mathcal{R}_{\delta}(X)$ \eqref{eq:Rdelta} satisfies Assumptions~\ref{ass:pseudo_lipschitz}, \ref{assumption:unbiased}, and \ref{assumption:E_loss_pseudo_Lip}. Suppose the learning rate schedule satisfies Assumption \ref{assum:learning_rate}, and the initialization $X_0$ and hidden parameters $X^{\star}$ satisfy Assumption~\ref{assumption:scaling}. Moreover the data $a \sim N(0,K)$ and label noise $\epsilon$ satisfy Assumption~\ref{ass:data_normal}. Let $\{W_{\lfloor td \rfloor}\}$ be generated from the iterates of SGD \eqref{eq:SGD} and $\HSGD_t$ generated from the solution of homogenized SGD \eqref{eq:HSGD} through $W = X \otimes X^{\star}$ and initialized with $X_0 = \WHSGD_0$.  
  Then there is an $\varepsilon >0$ so that for any $T,M > 0$ and $d$ sufficiently large, with overwhelming probability
  \begin{equation}
  \begin{gathered}
    \sup_{0 \leq t \leq T \wedge \tau_{M}(S(W,\cdot), \mathcal{S})} \! \! \! \! \!\! \! \! \! \| 
    S(W_{\lfloor td \rfloor}, \cdot)
    - 
    \mathcal{S}(t,\cdot)
    \|_{\Gamma} \leq d^{-\varepsilon}, \quad \sup_{0 \leq t \leq T \wedge \tau_{M}(S(\HSGD,\cdot), \mathcal{S})} \| 
    S(\HSGD_t, \cdot)
    - 
    \mathcal{S}(t,\cdot)
    \|_{\Gamma} \leq d^{-\varepsilon},
    \\
       \text{and} \qquad \sup_{0 \leq t \leq T \wedge \tau_M(S(W,\cdot), S(\HSGD,\cdot) )} \! \! \! \! \!\! \! \! \!  \| 
    S(W_{\lfloor td \rfloor}, \cdot)
    - 
    S(\HSGD_t,\cdot)
    \|_{\Gamma} \leq d^{-\varepsilon},
    \end{gathered}
  \end{equation}
where the deterministic function $\mathcal{S}(t,z)$ solves the integro-differential equation \eqref{eq:ODE_resolvent_2} and
\[
\tau_{M}(\mathscr{S}_1, \mathscr{S}_2) = \min\{\hat{\tau}_M(\mathscr{S}_1), \hat{\tau}_M(\mathscr{S}_2) \}.
\]
\end{theorem}  

\begin{proof} We will consider $\mathscr{S}_1(t,z) = S(\HSGD_t,\cdot)$ and $\mathscr{S}_2(t,z) = S(W_{ td },z)$ and suppress the notation by setting $\tau_M(\mathscr{S}_1, \mathscr{S}_2) = \tau_M$. We also note that the cases when $\mathscr{S}_1(t,z) = S(W_{ td }, z)$ and $\mathscr{S}_2(t,z)= \mathcal{S}(t,z)$ and $\mathscr{S}_1(t,z) = S(\HSGD_t, z)$ and $\mathscr{S}_2(t,z)= \mathcal{S}(t,z)$ follow an analogous proof, so for brevity, we do not present them. 

By Proposition~\ref{prop:homogenized_SGD_approx_solution}, for some $\tilde{\varepsilon} > 0$, we have that $S(\HSGD_t, z)$ is an $(d^{-\tilde{\varepsilon}}, M, T)$-approximate solution with overwhelming probability. Moreover, by Proposition~\ref{prop:SGD_approx_solution}, the function $\mathcal{S}(W_{td },z)$ is an $(d^{-\tilde{\varepsilon}}, M, T)$-approximate solution. (For the deterministic function $\mathcal{S}$, it is an $(0, M+1, T)$-approximate solution by definition.) We now apply the stability result, Proposition~\ref{prop:stability}, to conclude that there exists a $\varepsilon > 0 $ such that
\begin{equation} \label{eq:A_event}
\sup_{0 \le t \le T \wedge \tau_{M}} \|S(\HSGD_t, z) - S(W_{ td },z)\|_{\Gamma} \le d^{-\varepsilon}, \quad w.o.p. 
\end{equation}
The result immediately follows. 
\end{proof}

In the next theorem, we note that one can remove the condition that \textit{both} processes must remain good and reduce this to show that we need only \textit{one} of the processes to remain good. In this way, we can show, for instance, that homogenized SGD is well-behaving and then conclude that SGD must also be well-behaving. 

For any $(\epsilon, M,T)$- approximate solution $\mathscr{S}(t,\cdot)$, we define
\[
\hat{\tau}_{M, \eta}(\mathscr{S}) = \inf \{ t \ge 0 \, : \|\mathscr{S}(t,\cdot)\|_{\Gamma} > M \, \, \text{or} \, \, \sup_{V \in \mathcal{U}^c} \|\mathscr{B}(t,\mathscr{S})-V\| \le \eta\} \quad \text{where} \, \, \mathscr{B}(t, \mathscr{S}) = \frac{-1}{2 \pi i} \oint_{\Gamma} z\mathscr{S}(t, z) \, \dif z,
\]
and where $\mathcal{U}^c$ is the set complement of $\mathcal{U}$.
Our main theorem requires that \textit{only one} of the statistics stays bounded, and not, in particular, both. To define this, we introduce a stopping time 
\begin{equation} \label{eq:stopping_time_thm}
\begin{aligned}
\Theta_{M, \eta}^{\mathscr{S}_1, \mathscr{S}_2} 
& 
= \max \{
\inf\{t \ge 0 \, : \, \|\mathscr{S}_i(t,\cdot)\|_{\Gamma} > M\} \, : \, i = 1,2 \} 
\\
&
\quad \wedge \max \{ \inf \{ t \ge 0 \, : \, \sup_{V \in \mathcal{U}^c} \|\mathscr{B}(t,\mathscr{S}_i)-V\| \le \eta \} \, : \, i = 1,2 \}.
% \wedge \inf \{t \ge 0 \, : \, \mathscr{B}(t, \mathscr{S}_1) \not \in \mathcal{U} \, \text{or} \, \mathscr{B}(t, \mathscr{S}_2) \not \in \mathcal{U} \}. 
\end{aligned}
\end{equation} 
 We note that $\hat{\tau}_{M,0} = \hat{\tau}_M$ with $\hat{\tau}_M$ defined in the $(\epsilon, M, T)$-approximate solution definition. 

\begin{theorem}[Concentration of SGD, Homogenized SGD, and deterministic function $\mathcal{S}(t,z)$] \label{thm:main_concentration_S}
Suppose the risk function $\mathcal{R}_{\delta}(X)$ \eqref{eq:Rdelta} satisfies Assumptions~\ref{ass:pseudo_lipschitz}, \ref{assumption:unbiased}, and \ref{assumption:E_loss_pseudo_Lip}. Suppose the learning rate schedule satisfies Assumption \ref{assum:learning_rate}, and the initialization $X_0$ and hidden parameters $X^{\star}$ satisfy Assumption~\ref{assumption:scaling}. Moreover the data $a \sim N(0,K)$ and label noise $\epsilon$ satisfy Assumption~\ref{ass:data_normal}. Let $\Theta_M$ be defined as in \eqref{eq:stopping_time_thm} and let $\{W_{\lfloor td \rfloor}\}$ be generated from the iterates of SGD \eqref{eq:SGD} and $\HSGD_t$ generated from the solution of homogenized SGD \eqref{eq:HSGD} through $W = X \otimes X^{\star}$ and initialized with $X_0 = \WHSGD_0$.  
% Let $\Theta_M$ be the first time 
%   that either $\mathrsfs{B}_t$ or $B_{[td]}$ exits $\mathcal{U}$ 
%   or that $\|\mathrsfs{B}_t\| \geq M.$
  Then there is an $\varepsilon >0$ so that for any $T,M, \eta > 0$ and $d$ sufficiently large, with overwhelming probability
  \begin{equation}
  \begin{gathered}
    \sup_{0 \leq t \leq T \wedge \Theta_{M,\eta}^{S(W,\cdot), \mathcal{S}} } \| 
    S(W_{ td }, \cdot)
    - 
    \mathcal{S}(t,\cdot)
    \|_{\Gamma} \leq d^{-\varepsilon}, \quad \sup_{0 \leq t \leq T \wedge \Theta_{M,\eta}^{S(\HSGD,\cdot), \mathcal{S}} } \| 
    S(\HSGD_t, \cdot)
    - 
    \mathcal{S}(t,\cdot)
    \|_{\Gamma} \leq d^{-\varepsilon},
    \\
       \text{and} \qquad \sup_{0 \leq t \leq T \wedge \Theta_{M, \eta}^{S(W,\cdot), S(\HSGD,\cdot)}} \| 
    S(W_{td}, \cdot)
    - 
    S(\HSGD_t,\cdot)
    \|_{\Gamma} \leq d^{-\varepsilon},
    \end{gathered}
  \end{equation}
where the deterministic function $\mathcal{S}(t,z)$ solves the integro-differential equation \eqref{eq:ODE_resolvent_2}.
\end{theorem}     

\begin{proof} Fix an $\eta > 0$. For two mappings $\mathscr{S}_1$ and $\mathscr{S}_2$, we define the stopping time
\begin{equation} \label{eq:tau_M_twice}
\tau_{M+1, 0}^{\mathscr{S}_1, \mathscr{S}_2} = \min \{ \hat{\tau}_{M+1, 0}(\mathscr{S}_1), \hat{\tau}_{M+1, 0}(\mathscr{S}_2) \}. 
\end{equation}
As in the previous theorem, we will consider $\mathscr{S}_1(t,z) = S(\HSGD_t,\cdot)$ and $\mathscr{S}_2(t,z) = S(W_{ td },z)$ and suppress the notation by setting $\tau_{M,\eta}^{\mathscr{S}_1, \mathscr{S}_2} = \tau_{M,\eta}$. We also note that the cases when $\mathscr{S}_1(t,z) = S(W_{ td }, z)$ and $\mathscr{S}_2(t,z)= \mathcal{S}(t,z)$ and $\mathscr{S}_1(t,z) = S(\HSGD_t, z)$ and $\mathscr{S}_2(t,z)= \mathcal{S}(t,z)$ follow an analogous proof so for brevity we do not present them. 

% By Proposition~\ref{prop:homogenized_SGD_approx_solution}, we have that $S(\HSGD_t, z)$ is an $(d^{-\tilde{\varepsilon}}, M+1, T)$-approximate solution with overwhelming probability. Moreover, by Proposition~\ref{}, the function $\mathcal{S}(W_{\lfloor td \rfloor},z)$ is an $(d^{-\tilde{\varepsilon}}, M+1, T)$-approximate solution. (For the deterministic function $\mathcal{S}$, it is a $(0, M+1, T)$-approximate solution by definition.) We now apply the stability result, Proposition~\ref{prop:stability}, to conclude that there exists a $\varepsilon > 0 $ such that
By Theorem~\ref{thm:main_concentration_S_min}, we have that 
\begin{equation} \label{eq:A_event_2}
\sup_{0 \le t \le T \wedge \tau_{M+1, 0}} \|S(\HSGD_t, z) - S(W_{td},z)\|_{\Gamma} \le d^{-\varepsilon}, \quad w.o.p. 
\end{equation}
% Since when we shrink $\mathcal{U}$ by $\eta$ it results in smaller stopping time, we have that $\tau_{M+1, 0} \ge \tau_{M+1, \eta}$. Therefore, we immediately have
% \begin{equation} \label{eq:A_event}
% \max_{0 \le t \le T \wedge \tau_{M+1, \eta}} \|S(\HSGD_t, z) - S(W_{\lfloor td \rfloor},z)\|_{\Gamma} \le d^{-\varepsilon}, \quad w.o.p. 
% \end{equation}

The remaining component is to replace the stopping time $\tau_{M+1,0}$ which requires \textit{both} statistics to have $\Gamma$-norm less than $M+1$ with $\Theta_{M,\eta}$ which only requires \textit{one} of the statistics to remain in the good set. Denote the event that \eqref{eq:A_event_2} occurs by $A_{\varepsilon}$ and its complement by $A_{\varepsilon}^c$. Then for sufficiently large $d$, 
\begin{equation} \label{eq:probability_last}
\Pr(\Theta_{M,\eta} > \tau_{M+1,0}) \le \Pr(A_{\varepsilon}^c). 
\end{equation}
To see this, suppose $\Theta_{M,\eta} > \tau_{M+1,0}$. 
Let $t = \tau_{M+1,0}$. Then four things could have happened either $\|S(\HSGD_t,\cdot)\|_{\Gamma} \ge M+1$ or $\sup_{V \in \mathcal{U}^c} \|\mathscr{B}(t, S(\HSGD_t,\cdot))-V\| \le 0$ or $\|S(W_{td},\cdot)\|_{\Gamma} \ge M+1$ or $\sup_{V \in \mathcal{U}^c} \|\mathscr{B}(t, S(W_{td}, \cdot))-V\| \le 0$. 
On the other hand, since $\tau_{M+1,0} = t < \Theta_{M, \eta}$, then either $\|S(\HSGD_t,\cdot)\|_{\Gamma} \le M$ or  $\|S(W_{td},\cdot)\|_{\Gamma} \le M$ and the following happens $\sup_{V \in \mathcal{U}^c} \|\mathscr{B}(t, S(\HSGD_t, \cdot))-V\| > \eta$ or $\sup_{V \in \mathcal{U}^c} \|\mathscr{B}(t, S(W_{td}, \cdot))-V\| > \eta$. 

Now we consider cases. Suppose $\|S(\HSGD_t, \cdot)\|_{\Gamma} \ge M+1$. Then $\|S(\HSGD_t, \cdot)\|$ can not be less than or equal to $M$ so it must have been that $\|S(W_{td}, \cdot)\|_{\Gamma} \le M$. Since $t = \tau_{M+1,0}$, working on the event that \eqref{eq:A_event_2} occurs, we have that 
\[
\|S(\HSGD_t,\cdot) \|_{\Gamma} \le \|S(\HSGD_t, \cdot)-S(W_ {td},\cdot)\|_{\Gamma} + \|S(W_{td},\cdot)\|_{\Gamma} \le d^{-\varepsilon} + M. 
\]
For sufficiently large $d$, then $\|S(\HSGD_t,\cdot)\|_{\Gamma} < M+1$ which is a contradiction. 

Suppose $\|S(W_{td},\cdot)\| \ge M+1$. Then by reversing the roles of $W_{td}$ and $\HSGD_t$ in the previous case, we see that this cannot occur. 

Next suppose that $\sup_{V \in \mathcal{U}^c} \|\mathscr{B}(t,S(\HSGD_t, \cdot))- V\| \le 0$. Then $\sup_{V \in \mathcal{U}^c} \|\mathscr{B}(t,S(\HSGD_t, \cdot))- V\|$ can not be greater than $\eta$. Thus it had to be the case that $\|\mathscr{B}(t,S(W_{td}, \cdot))- V\| > \eta$. Now working on the event that $\eqref{eq:A_event_2}$ occurs, we have that 
\begin{align*}
    \|\mathscr{B}(t,S(W_{td}, \cdot))- V\| 
    &
    \le 
    \|\mathscr{B}(t,S(W_{ td }, \cdot)) - \mathscr{B}(t,S(\HSGD_t, \cdot))\|
    \\
    &
    \le C \cdot \sup_{z \in \Gamma} |z| \cdot \|S(W_{ td}, \cdot)- S(\HSGD_t, \cdot) \|_{\Gamma}
    \\
    &
    \le \tilde{C} \cdot d^{\varepsilon},
\end{align*}
where $C, \tilde{C}$ are positive constants. Hence for sufficiently large $d$,  $\sup_{V \in \mathcal{U}^c} \|\mathscr{B}(t,S(W_{td}, \cdot))- V\| < \eta$. Hence a contradiction. 

Lastly suppose $\sup_{V \in \mathcal{U}^c} \|\mathscr{B}(t,S(W_{td}, \cdot))- V\| \le 0$. By reversing the roles of $W_{td}$ and $\HSGD_t$, we reach the same conclusion as the previous case. 

Hence the inequality \eqref{eq:probability_last} holds and thus, $\tau_{M+1,0} \ge \Theta_{M,\eta}$ with overwhelming probability. The result immediately follows. 
\end{proof}

We immediately get a corollary which shows that SGD and homogenized SGD concentrates around the deterministic function $\mathcal{S}(t,z)$ which is a solution to the integro-differential equation \eqref{eq:ODE_resolvent_2} provided that either homogenized SGD or the solution to the integro-differential equation stay bounded, i.e., the quantity $\mathrsfs{N}(t) = \frac{-1}{2 \pi i} \oint_{\Gamma} \tr( \mathcal{S}(t,z) ) \, \dif z$ is bounded. 

\begin{corollary}[Bounded $\mathrsfs{N}$ and concentration] \label{cor:bounded_iterates} Suppose the Assumptions of Theorem~\ref{thm:main_concentration_S} hold. Suppose, in addition, for a fixed $T > 0$ and $\eta > 0$ that
\begin{equation} \label{eq:concentration_condition}
 \sup_{0 \le t \le T} \mathrsfs{N}(t) \le M  \quad \text{and} \quad \sup_{0 \le t \le T} \sup_{V \in \mathcal{U}^c} \|\mathrsfs{B}(t)-V\| > \eta \quad \text{hold w.o.p.}
\end{equation}
 by a positive constant $M$ which is independent of $d$.
 Then there is an $\varepsilon > 0$ so that for $d$ sufficiently large, with overwhelming probability, 
  \begin{equation}
  \begin{gathered} \label{eq:something_1_1}
    % \sup_{0 \leq t \leq T} \| 
    % S(W_{\lfloor td \rfloor}, \cdot)
    % - 
    % \mathcal{S}(t,\cdot)
    % \|_{\Gamma} \leq d^{-\varepsilon}, \quad
    \sup_{0 \leq t \leq T } \| 
    S(\HSGD_t, \cdot)
    - 
    \mathcal{S}(t,\cdot)
    \|_{\Gamma} \leq d^{-\varepsilon} \quad 
       \text{and} \qquad \sup_{0 \leq t \leq T} \| 
    S(W_{td}, \cdot)
    - 
    \mathcal{S}(t,\cdot)
    \|_{\Gamma} \leq d^{-\varepsilon}.
    \end{gathered}
  \end{equation}
  Moreover, by a simple triangle inequality, one has
  \begin{equation} \label{eq:something_2_1}
  \sup_{0 \le t \le T} \| S(W_{td}, \cdot) - S(\HSGD_t,\cdot)\|_{\Gamma} \le 2d^{-\varepsilon}. 
  \end{equation}
\end{corollary}

\begin{proof}  Define the following stopping time similar to $\Theta_{M,\eta}$ in \eqref{eq:stopping_time_thm} by 
\begin{align*}
\tilde{\Theta}_{M, \eta}^{\mathscr{S}_1, \mathscr{S}_2} 
& 
\defas \max \big \{ \inf \{t \ge 0 \, : \, \|\frac{-1}{2\pi i} \oint_{\Gamma} \mathscr{S}_i(t,\cdot) \, \dif z\|> M \, \} \, : \, i = 1,2\big \}
\\
&
\quad \wedge \max \{ \inf \{ t \ge 0 \, : \, \sup_{V \in \mathcal{U}^c} \|\mathscr{B}(t,\mathscr{S}_i)-V\| \le \eta \} \, : \, i = 1,2\}.
\end{align*}
Here we think of $\mathscr{S}_1$ as either SGD or homogenized SGD and $\mathscr{S}_2 = \mathcal{S}$. 
The idea is that $\Theta_{M,\eta}$ (see \eqref{eq:stopping_time_thm}) and $\tilde{\Theta}_{M,\eta}$ are related by our assumptions. By Lemma~\ref{lem:normequivalence}, there exists positive constants $c, C > 0$ such that $c \cdot \mathrsfs{N}(t) \ \le \|\mathcal{S}(t,\cdot)\|_{\Gamma} \le C \cdot \mathrsfs{N}(t)$. Consequently, this translates into 
\begin{align*}
 \{ t \ge 0 \, : \, \|\mathcal{S}(t, \cdot)\|_{\Gamma} > C \cdot M \} \subset \{t \ge 0 \, : \, \mathrsfs{N}(t) > M\} 
\end{align*}
and so the infimum of the right-hand-side is smaller than the infimum of the left-hand-side. Moreover, we have by assumption that 
\[
T \le \inf \{t \ge 0 \, : \, \mathrsfs{N}(t) > M \} \quad w.o.p. 
\]
Similarly 
% after noting that $\{ t \ge 0 \, : \, \ip{\HSGD_t^{\otimes 2}, K}_{\mathcal{A}^{\otimes 2}} \in \mathcal{U}\}  = \{t \ge 0, : \, \sup_{V \in \mathcal{U}^c} \|\mathscr{B}(t, S(\HSGD_t),\cdot)-V\| > 0\}$, 
we have that 
\[
T \le \inf \{ t \ge 0 \, : \, \sup_{V \in \mathcal{U}^c} \|\mathscr{B}(t, \mathcal{S}(t,\cdot))-V\| \le \eta\} \quad w.o.p. 
\]
Thus, we have that 
\[
T \le \tilde{\Theta}_{M, \eta}^{\mathcal{S}(t, \cdot), \mathscr{S}_2} \le \Theta_{C \cdot M,\eta}^{\mathcal{S}(t, \cdot), \mathscr{S}_2} \quad w.o.p, 
\]
where $\mathscr{S}_2$ is either $S(W_{td},\cdot)$ or $S(\HSGD_t, \cdot)$. By Theorem~\ref{thm:main_concentration_S}, we immediately get the result \eqref{eq:something_1_1}. A simple triangle inequality gives the result in \eqref{eq:something_2_1}. 
\end{proof}

\begin{remark} One can replace $(\mathrsfs{N}(t), \mathrsfs{B}(t))$ in \eqref{eq:concentration_condition}with $(\|W_{td}\|^2, B(W_{td}))$ or $(\|\HSGD_t\|^2, B(\HSGD_t))$ and the conclusion of Corollary~\ref{cor:bounded_iterates} would still hold. 
% Also note that one could exchange the roles of homogenized SGD with SGD in Corollary~\ref{cor:bounded_iterates} and get a result involving SGD. 

% Moreover one could have replace $T$ in the statement of Corollary~\ref{cor:bounded_iterates} with $T \wedge \tilde{\Theta}_{\tilde{M},\eta}$ while simultaneously removing the assumption in Eq. \eqref{eq:concentration_condition}. 
\end{remark}

In Section~\ref{susec:SGD_optimality}, we gave conditions on the risk function and on the learning rate for which the condition in \eqref{eq:concentration_condition} hold. Lastly, we make one final connection to Theorem~\ref{thm:learning_curves} and Proposition~\ref{prop:HSGD_median dynamics}, proving the result below.  

\begin{proof}[Proof of Theorem~\ref{thm:learning_curves} and Proposition~\ref{prop:HSGD_median dynamics}] The result immediately follows from Theorem~\ref{thm:main_concentration_S} and Corollary~\ref{cor:bounded_iterates} (and the remark following it) after noting that 
\begin{align*}
B(W_{td})
&
= 
\frac{-1}{2 \pi i} \oint_{\Gamma} zS(W_{td}, z) \, \dif z, \, \,
\ip{\HSGD_t^{\otimes 2}, K}_{\mathcal{A}^{\otimes 2}}
= 
\frac{-1}{2 \pi i} \oint_{\Gamma} zS(\HSGD_{td}, z) \, \dif z, 
\, \, 
\text{and} \, \, 
\mathrsfs{B}(t) 
=
\frac{-1}{2 \pi i} \oint_{\Gamma} z \mathcal{S}(t, \cdot) \, \dif z
\end{align*}
and Lipschitzness of the integral, that is, 
\[
\bigg \| \oint_{\Gamma} z \mathscr{S}_1(t,\cdot) \dif z- \oint_{\Gamma} z \mathscr{S}_2(t, \cdot) \, \dif z \bigg \| \le C \cdot \|\mathscr{S}_1(t, \cdot)- \mathscr{S}_2(t,\cdot) \|_{\Gamma}, \quad \text{for some positive $C > 0$.} 
\]
\end{proof}
 
\subsection{Concentration result for any statistic}

In this section, we show an extension of  Theorem~\ref{thm:main_concentration_S} to any statistic $\varphi \, : \, \mathcal{A} \otimes \mathcal{O} \to \mathbb{R}$ satisfying Assumption~\ref{assumption:statistic}. Indeed, this result, Theorem~\ref{thm: concentration_statistic}, a reformulation of Theorem~\ref{Thm:SGD_HSGD_convergence}, applies to the risk curve, $\mathcal{R}_{\delta}(X)$ as well as to a host of other generalization metrics. The result is that SGD under any statistic concentrates around a deterministic function. 

In this section, the statistics $\varphi \, : \mathcal{A} \otimes \mathcal{O} \to \mathbb{R}$ of interest satisfy a composite structure
\begin{equation*}
\begin{gathered}
\varphi(X) = 
g( \ip{W \otimes W,q(K)}_{\mathcal{A}^{\otimes 2}}  ) 
\end{gathered}
\end{equation*}
where ${g} \, : \, \mathcal{O}^+ \otimes \mathcal{O}^+ \to \mathbb{R}$ is $\alpha$-pseudo-Lipschitz on $\mathcal{U}$
and $q$ is a polynomial (see Assumption~\ref{assumption:statistic}). The deterministic equivalence of this statistic for $\varphi(\WHSGD_t)$ and $\varphi(X_{ td })$ is precisely
\begin{equation} \label{eq:statistic_deterministic_equivalence}
\phi(t) \defas g \left ( \frac{-1}{2 \pi i} \oint_{\Gamma} q(z) \mathcal{S}(t,z) \, \dif z \right ), \quad \text{where $\mathcal{S}(t,z)$ solves \eqref{eq:ODE_resolvent_2}.}
\end{equation}
Thus we state our concentration theorem for $\varphi(\WHSGD_t)$ and $\varphi(X_{ td })$. 
% we recall the function $\mathscr{Q}$ \eqref{} defined earlier on $(\varepsilon, M,T)$-approximate solutions $\mathscr{S}$ to be
% \begin{equation}
% \mathscr{Q}(t) \defas \frac{-1}{2\pi i} \oint_{\Gamma} q(z) \mathscr{S}(t, z) \, \dif z, \quad \text{where $q \, : \, \mathbb{C} \to \mathbb{C}$ is a polynomial.}
% \end{equation}
% Thus the deterministic equivalence is precisely
% \begin{equation}
%     \phi(t) \defas g \
% \end{equation}

% From this $\mathscr{Q}$, we can write down the deterministic function for which both $\varphi(\WHSGD_t)$ and $\varphi(X_{\lfloor td \rfloor})$. For the solution to the integro-differential equation \eqref{eq:ODE_resolvent_2}, $\mathcal{S}(t,z)$, we define the deterministic function 
% \[
% \phi(t) \defas g
% \]

% Through this function, we get our concentration result. 
 
\begin{theorem}[Concentration of any statistic] \label{thm: concentration_statistic} Suppose the Assumptions of Theorem~\ref{thm:main_concentration_S} hold. Suppose, in addition, the statistic satisfies a composite structure, 
  %$\varphi \, : \, \mathcal{A} \otimes \mathcal{O} \to \mathbb{R}$  is $C^3$-smooth and 
\begin{equation*}
\begin{gathered}
\varphi(X) = 
g( \ip{W \otimes W,q(K)}_{\mathcal{A}^{\otimes 2}}  ) 
\end{gathered}
\end{equation*}
where ${g} \, : \, \mathcal{O}^+ \otimes \mathcal{O}^+ \to \mathbb{R}$ is $\alpha$-pseudo-Lipschitz on $\mathcal{U}$
and $q$ is a polynomial (see Assumption~\ref{assumption:statistic}). Then there is an $\varepsilon > 0$ so that for any $T,M > 0$ and $d$ sufficiently large, with overwhelming probability
  \begin{equation}
  \begin{gathered}
    \sup_{0 \leq t \leq T \wedge \Theta_M^{S(W,\cdot), \mathcal{S}} } \| 
    \varphi(W_{ td })
    - 
    \phi(t)
    \|_{\Gamma} \leq d^{-\varepsilon}, \quad \sup_{0 \leq t \leq T \wedge \Theta_M^{S(\HSGD,\cdot), \mathcal{S}} } \| 
    \varphi(\HSGD_t)
    - 
    \phi(t)
    \|_{\Gamma} \leq d^{-\varepsilon},
    \\
       \text{and} \qquad \sup_{0 \leq t \leq T \wedge \Theta_M^{S(W,\cdot), S(\HSGD,\cdot)}} \| 
    \varphi(W_{td})
    - 
    \varphi(\HSGD_t)
    \|_{\Gamma} \leq d^{-\varepsilon},
    \end{gathered}
  \end{equation}
where $\phi$ is defined in \eqref{eq:statistic_deterministic_equivalence} and where the stopping time $\Theta_M^{\mathscr{S}_1, \mathscr{S}_2}$ is defined in \eqref{eq:stopping_time_thm}. 
\end{theorem}

\begin{proof} As in the proof of Theorem~\ref{thm:main_concentration_S}, we define the stopping time $\tau_{M+1, \eta}^{\mathscr{S}_1, \mathscr{S}_2}$ as in \eqref{eq:tau_M_twice} and suppress the notation by setting  $\tau_{M+1}^{\mathscr{S}_1, \mathscr{S}_2} = \tau_{M}$. We will consider the case when $\mathscr{S}_1(t,\cdot) = S(\HSGD_t,\cdot)$ and $\mathscr{S}_2(t,\cdot) = S(W_{ td }, \cdot)$. The other cases will follow by analogous proof. 

By Proposition~\ref{prop:homogenized_SGD_approx_solution}, we have that $S(\HSGD_t, z)$ is an $(d^{-\tilde{\varepsilon}}, M+1, T)$-approximate solution with overwhelming probability. Moreover, by Proposition~\ref{prop:SGD_approx_solution}, the function $\mathcal{S}(W_{ td },z)$ is an $(d^{-\tilde{\varepsilon}}, M+1, T)$-approximate solution. (For the deterministic function $\mathcal{S}$, it is a $(0, M+1, T)$-approximate solution by definition.) We observe that 
\[
\frac{-1}{2 \pi i} \oint_{\Gamma} q(z) S(\HSGD_t, z) \, \dif z = q(\HSGD_t) \quad \text{and} \quad  \frac{-1}{2 \pi i} \oint_{\Gamma} q(z) S(W_{td }, z) \, \dif z = q(W_{td}).
\]
Now we apply Proposition~\ref{prop:stability_varphi} to conclude that there exists a $\varepsilon > 0 $ such that
\begin{equation} \label{eq:A_event_1}
\sup_{0 \le t \le T \wedge \tau_{M+1,0}} |\varphi(\HSGD_t) - \varphi(W_{ td}) |_{\Gamma} \le d^{-\varepsilon}, \quad w.o.p. 
\end{equation}
Using the same argument as in Theorem~\ref{thm:main_concentration_S}, we can remove the stopping time $\tau_{{M+1},0}$ and replace it with $\Theta_{M, 0}$ for sufficiently large $d$. 
\end{proof}

Lastly we formulate an immediate corollary which follows immediately from the proofs of Theorem~\ref{thm: concentration_statistic} and Corollary~\ref{cor:bounded_iterates}. 

\begin{corollary} \label{cor:statistic_bounded}
    Suppose the Assumptions of Theorem~\ref{thm: concentration_statistic} and Corollary~\ref{cor:bounded_iterates} hold. Then there is an $\varepsilon > 0$ so that for $d$ sufficiently large, with overwhelming probability, 
  \begin{equation}
  \begin{gathered} \label{eq:something_10}
    % \sup_{0 \leq t \leq T} \| 
    % S(W_{\lfloor td \rfloor}, \cdot)
    % - 
    % \mathcal{S}(t,\cdot)
    % \|_{\Gamma} \leq d^{-\varepsilon}, \quad
    \sup_{0 \leq t \leq T } | 
    \varphi(\WHSGD_t)
    - 
    \phi(t)
    | \leq d^{-\varepsilon} \quad 
       \text{and} \qquad \sup_{0 \leq t \leq T} | 
    \varphi(X_{td})
    - 
    \phi(t)
    | \leq d^{-\varepsilon}.
    \end{gathered}
  \end{equation}
  Moreover, by a simple triangle inequality, one has
  \begin{equation} \label{eq:something_20}
  \sup_{0 \le t \le T} | \varphi(X_{td}) - \varphi(\WHSGD_t) | \le 2d^{-\varepsilon}. 
  \end{equation}
\end{corollary}

\begin{remark} As in the remark after Corollary~\ref{cor:bounded_iterates}, one can replace $(\mathrsfs{N}(t), \mathrsfs{B}(t))$ in \eqref{eq:concentration_condition}with $(\|W_{td}\|^2, B(W_{td}))$ or $(\|\HSGD_t\|^2, B(\HSGD_t))$.
% in \eqref{eq:concentration_condition} the term $\sup_{0\le t \le T} \|\WHSGD_t\|^2 \le M$ with $\sup_{0 \le t \le T} \|\HSGD_t\|^2 \le M$ and the conclusion would still hold. Also note that one could exchange the roles of homogenized SGD with SGD in Corollary~\ref{cor:statistic_bounded} and get a result involving SGD. 

% Moreover one could have replace $T$ in the statement of Corollary~\ref{cor:statistic_bounded} with $T \wedge \tilde{\Theta}_{\tilde{M},\eta}$ while simultaneously removing the assumption in Eq. \eqref{eq:concentration_condition}. 
\end{remark}

The proof of Theorem~\ref{Thm:SGD_HSGD_convergence} immediately follows from Corollary~\ref{cor:statistic_bounded} and the remark that follows it. 
 
\section{SGD and homogenized SGD are approximate solutions} \label{sec:SGD_homogenized_SGD}
In order to compare SGD and homogenized SGD, we use a version of the martingale method in diffusion approximation (see \cite{ethier86markov}). In effect, we show that any statistic $\varphi(X_k)$ applied to SGD \eqref{eq:SGD} is nearly identical to the same statistic under homogenized SGD. The main argument hinges on the dynamics of one important statistic,
defined as,
\begin{equation}
S(W,z) = \ip{W \otimes W, R(z;K)}_{\mathcal{A}^{\otimes 2}}, 
\end{equation}
which plays an overly significant role in our analysis and the function
\[
B(W) = \ip{W \otimes W, K}_{\mathcal{A}^{\otimes 2}}. 
\]
Here $W = X \oplus X^{\star}$ and $R(z;K) = (K-z I_d)^{-1}$ for $z \in \mathbb{C}$ is the resolvent of $K$. We first show that both homogenized SGD and SGD on $S(\cdot, z)$ are $(\e, M, T)$-approximate solutions as defined in Definition~\ref{def:integro_differential_equation}. Then by Proposition \ref{prop:stability}, it is immediately implied that both homogenized SGD and SGD on $S(\cdot, z)$ are uniformly close. Finally, Proposition \ref{prop:stability_varphi}, establishes that the same hold for any statistics $\varphi(X)$ satisfying Assumption \ref{assumption:statistic}. 
% By showing that homogenized SGD and SGD are close on $S(\cdot, z)$, we can show the result for any statistic. 
In order to show that both homogenized SGD and SGD on $S(\cdot, z)$ are $(\e, M, T)$-approximate solutions, we perform a Doob's decomposition for both homogenized SGD and SGD and then show that both martingale terms are small.
% , and hence the predictable parts of $S(W_k, z)$ and $S(\HSGD_t, z)$ are close. 

For the comparison between homogenized SGD and SGD to hold, we introduce a rescaling of time. We relate the $k$-th iteration of SGD to the continuous time parameter $t$ in homogenized SGD through the relationship $k = \lfloor td \rfloor $. Thus, when $t = 1$, SGD has done exactly $d$ updates. Since the parameter $t$ is continuous and the iteration counter $k$ (integer) discrete, to simplify the discussion below, we \textit{extend} $k$ to continuous values through the floor operation, $X_k \defas X_{\lfloor k \rfloor }$. Using the continuous parameter $t$, the iterates are related by
\[ X_{td} = X_{\lfloor td \rfloor} \, \, \text{(SGD)} \, \quad \text{and} \quad \WHSGD_t \, \, \text{(HSGD)}.\]
When $td$ is an integer, we will show that homogenized SGD and SGD agree on statistics. For non-integer values, the two will agree up to a term that vanishes like $1/d$. Throughout the paper, we will generally work with the continuous time parameter. 

Our first argument is a net argument showing that we do not need to work with every $z$, but only polynomially many in $d$. For this, recall the contour $\Gamma = \{ z \, : \, |z| = \{2 \|K\|_{\sigma}, 1\}\}$. For a fixed $\delta > 0$, we say that $\Gamma_{\delta}$ is a \textit{$d^{-\delta}$-mesh of $\Gamma$} if $\Gamma_{\delta} \subset \Gamma$ and for every $z \in \Gamma$ there exists a $\bar{z} \in \Gamma_{\delta}$ such that $|z-\bar{z}| < d^{-\delta}$. We can achieve this with $\Gamma_{\delta}$ having cardinality, $|\Gamma_{\delta}| = C(|\Gamma|) d^{\delta}$. 

\begin{lemma}[Net argument] \label{lem:net_argument} Fix $T, M > 0$ and let $\delta > 0$. Suppose $\Gamma_{\delta}$ is a $d^{-\delta}$ mesh of $\Gamma$ with $|\Gamma_{\delta}| = C \cdot d^{\delta}$ and positive $C > 0$. Let the function $S(t,z) = S(W_{td}, z)$ or $S(\HSGD_t, z)$ satisfy
\begin{equation}
\label{eq:net_argument_gamma_delta}
\sup_{0 \le t \le (\hat{\tau}_M \wedge T)} \| S(t, \cdot) - S(0, \cdot) - \int_0^t \mathscr{F}(\cdot, S(s, \cdot) ) \, \dif s \|_{\Gamma_\delta} \le \varepsilon
\end{equation}
with $\hat{\tau}_M = \inf \{ t \ge 0 \, : \, \|S(t, \cdot)\|_{\Gamma} > M\}$. Then $S$ is a $(\varepsilon + C(M,T, \|K\|_{\sigma})d^{-\delta}, M, T)$-approximate solution to the integro-differential equation, that is, 
\[
\sup_{0 \le t \le (\hat{\tau}_M \wedge T)} \|S(t, \cdot) - S(0, \cdot) - \int_0^t \mathscr{F}(\cdot, S(s, \cdot) ) \, \dif s \|_{\Gamma} \le \varepsilon + C \cdot d^{-\delta},
\]
where $C = C(M,T, \|K\|_{\sigma}, \bar{\gamma}, L(I), L(h), |\mathcal{O}|)$ is a positive constant. 
\end{lemma}

\begin{proof} We consider only $S(t,z) = S(\HSGD_t, z)$ as the same argument will also hold for SGD. We also will always work with the stopped process, that is, $S(t \wedge \hat{\tau}_M, z)$, where $\hat{\tau}_M = \inf \{ t \ge 0 \, : \, \|S(t,z)\|_{\Gamma} \ge M\}$. To simplify the notation, we suppress the $\hat{\tau}_M$ and use $S(t,z)$. First, we note for any contour $\tilde{\Gamma}$ containing the spectrum of $K$, 
\begin{equation}
\begin{aligned} \label{eq:B_W_identity}
B(t) = \frac{-1}{2\pi i} \oint_{\tilde{\Gamma}} zS(t,z) \, \dif z = \ip{\HSGD_t^{\otimes 2}, K}_{\mathcal{A}^{\otimes 2}} \quad \text{and} \quad \frac{-1}{2\pi i} \oint_{\tilde{\Gamma}} S(t,z) \, \dif z = \ip{\HSGD_t^{\otimes 2}, I_{\mathcal{A}}}_{\mathcal{A}^{\otimes 2}}. 
\end{aligned}
\end{equation}
In this regard, these two quantities do not dependent on the specific contour. 

Next we state some resolvent identities. One such resolvent identity gives
\begin{equation} \label{eq:resolvent_identity}
\|R(z; K) - R(\bar{z}; K)\|_{\sigma} \le |z-\bar{z}| \|R(z;K) R(\bar{z}; K)\|_{\sigma}, \quad \text{for any $z, \bar{z} \in \Gamma$.}
\end{equation}
Furthermore, by Neumann series, $(K-zI_{\mathcal{A}})^{-1} = -1/z (I_{\mathcal{A}} - 1/z K)^{-1} = -\tfrac{1}{z} \sum_{j=0}^\infty ( \tfrac{1}{z} K)^j$. So, using $|z| = \max \{ 1, 2 \|K\|_{\sigma}\}$, we immediately get 
\begin{equation} \label{eq:resolvent_bound_1}
\displaystyle \sup_{z \in \Gamma} \|R(\cdot;K)\|_{\sigma} \le 2.
\end{equation}
These bounds will be useful later in the proof. 

Next, with these bounds, we can get estimates on quantities involving $S(t, \cdot)$ where $t$ is fixed and $z$ varies. Fix $z \in \Gamma$ and let $\bar{z} \in \Gamma_{\delta}$ be such that $|z-\bar{z}| < d^{-\delta}$. Then, using the resolvent identity \eqref{eq:resolvent_identity} (and the stopping time $\hat{\tau}_M$)
\begin{equation}
\begin{aligned} \label{eq:S_3}
\|S(t,z) - S(t,\bar{z})\| 
&
\le 
|z-\bar{z}| \|\HSGD_t\|^2 \|R(z; K)\|_{\sigma} \|R(\bar{z};K)\|_{\sigma}\\
& 
\le 
C \cdot d^{-\delta} \cdot \left \| \frac{-1}{2 \pi i} \oint_{\Gamma} S(t,z) \, \dif z \right \| \\
&
\le
C  \cdot d^{-\delta}  \left (\oint_{\Gamma} \|S(t,z)\| \dif |z| \right )
\\
&
\le C(\|K\|_{\sigma}) \cdot d^{-\delta} \cdot M,
\end{aligned}
\end{equation}
where we used the identity in \eqref{eq:B_W_identity} and the boundedness of the contour $|\Gamma|$ in the last inequality. Similarly, using the same identity for $\|\HSGD_t\|$ \eqref{eq:B_W_identity} as well as \eqref{eq:resolvent_bound_1}, for any $z \in \Gamma$, 
\[
\|S(t,z)\| \le \|\HSGD_t\|^2 \|R(z; K)\|_{\sigma} \le C(\|K\|_{\sigma}) \cdot M.  
\]
Thus, since $z, \bar{z} \in \Gamma$ and the contour $\Gamma$ is bounded,
\begin{equation}
\label{eq:S_2}
\|zS(t,z) - \bar{z} S(t,\bar{z})\| \le C(\|K\|_{\sigma}) \cdot M \cdot d^{-\delta}. 
\end{equation}
Furthermore, we will need a bound on the $\tr(K R(z; K))$. Again for $z \in \Gamma$ with $|z-\bar{z}| \le d^{-\delta}$ and $\bar{z} \in \Gamma_{\delta}$, we have that 
\begin{equation} \label{eq:trace_net}
\frac{1}{d} | \tr(KR(z;K)) - \tr(KR(\bar{z}; K)) | \le \|K\|_{\sigma} \|R(z;K) - R(\bar{z};K)\|_{\sigma} \le \|K\|_{\sigma} \cdot d^{-\delta}
\end{equation}
where we applied \eqref{eq:resolvent_identity} and \eqref{eq:resolvent_bound_1}.

Now we are ready to prove the main result of the proposition. For a fixed $t \le \hat{\tau}_{M}$ and $z \in \Gamma$ with $\bar{z} \in \Gamma_{\delta}$ such that $|z-\bar{z}| \le d^{-\delta}$,
\begin{equation}
\begin{aligned} \label{eq: S_1}
\big \| S(t,z) - 
&
S(0,z) - \int_0^t \mathscr{F}(z, S(s, \dot)) \, \dif s  \big \| \\
&
\le
\| S(t,z)- S(t, \bar{z}) \| + \| S(0, z) - S(0, \bar{z}) \| + \int_0^t \| \mathscr{F}(z, S(s,\cdot)) - \mathcal{F}(\bar{z}, S(s, \cdot)) \| \, \dif s 
\\
& 
\quad + \big \| S(t,\bar{z}) - 
S(0,\bar{z}) - \int_0^t \mathscr{F}(\bar{z}, S(s, \dot)) \, \dif s  \big \| 
\\
&
\le C(\|K\|_{\sigma}) \cdot M^2 \cdot d^{-\delta} + \int_0^t \tfrac{\bar{\gamma}^2}{d} \| I(B(s)) \| \big | \tr(KR(z;K)) - \tr(K R(\bar{z}; K)) \big | \, \dif s
\\
& \quad + 4 \bar{\gamma} \int_0^t \big ( \|H(B(s))\| \|zS(s,z) - \bar{z}  S(s,\bar{z})\| + \delta \|D\| \| S(s,z)-S(s, \bar{z})\| \big ) \, \dif s + \varepsilon. 
\end{aligned}
\end{equation}
Here we used \eqref{eq:S_3} to bound the first two terms in the first inequality and $\varepsilon$ for the last term by the assumption \eqref{eq:net_argument_gamma_delta} in the statement. For the difference in $\mathscr{F}(z, S(s,\cdot))$, we see that many of the terms in \eqref{eq:ODE_resolvent_2} are independent of $z$, that is, they only depend on $t$ (or in this case $s$) (see e.g., $\big ( \tfrac{-1}{2 \pi i} \oint_{\Gamma} S(s,z) \, \dif s \big ) H(B(s))$). Since we have fixed $s$ to be the same and we are only varying $z$, these terms drop out. The only surviving terms, which depend on $z$ from the difference $\mathcal{F}(z, S(s,\cdot)) - \mathcal{F}(\bar{z}, S(s, \cdot) )$, are the ones shown in \eqref{eq: S_1}. 

As we have already shown that $\tr(KR(z;K))$, $zS(s,z)$, and $S(s,z)$ are Lipschitz in $z$, we only need to bound $\|I(B(s))\|$ and $\|H(B(s))\|$ as $\|D\| \le C(|\mathcal{O}|)$. We have already shown a uniform bound on $\|H(B(s))\|$ in the proof of Proposition~\ref{prop:stability}. Notably, we showed that for $s \le \hat{\tau}_M$, we have from \eqref{eq:H_bound} that $\|H(B(s))\| \le C(L(h),\|K\|_{\sigma}) \cdot M$. As for the boundedness of $I(B(s))$, we will do an abbreviated argument, since it is analogous to the one for $H(B(s))$. Since $I$ is $\alpha$-pseudo-Lipschitz (Assumption~\ref{assumption:E_loss_pseudo_Lip}), 
\begin{equation}
\begin{aligned}\label{eq:I_bound}
\|I(B(s))\| \le L(I) \|B(s)\| (1 + \|B(s)\|^{\alpha}). 
\end{aligned}
\end{equation}
Using the representation of $B(t)$ in \eqref{eq:B_W_identity} together with the boundedness of $\Gamma$ and $\hat{\tau}_M$, we have that 
\[
\|B(s)\| \le C \oint_{\Gamma} |z| \|S(s,z)\| \, \dif |z| \le C(\|K\|_{\sigma}) \cdot M. 
\]
As such, $\|I(B(s))\| \le C \cdot M$ where the constant $C$ depends on $\alpha$, the Lipschitz constant of $I$ ($L(I)$), and $\|K\|_{\sigma}$, but independent of $d$. 
% and an analogous argument for $I(B(s))$ shows that $\|I(B(s))\| \le C \cdot M$ where $C$ is a positive constant. Moreover we have that
% \begin{align*}
% \frac{1}{d} | \tr(KR(z;K)) - \tr(KR(\bar{z}; K)) | \le \|K\|_{\sigma} \|R(z;K) - R(\bar{z};K)\|_{\sigma} \le \|K\|_{\sigma} \cdot d^{-\delta}
% \end{align*}
% where we applied \eqref{eq:resolvent_identity} and \eqref{eq:resolvent_bound_1}. 

First, by taking the supremum over $z \in \Gamma$ and then the supremum over $0 \le t \le (\hat{\tau}_M \wedge T)$ on the left-hand-side of \eqref{eq: S_1} and then using the bounds \eqref{eq:S_3} and \eqref{eq:S_2}, yields the result. 
\end{proof}

In what remains of this section, we will show that homogenized SGD and SGD are approximate solutions to \eqref{eq:ODE_resolvent_2}. To do so, it will be convenient to work directly with the stopped process $X_{t \wedge \hat{\tau}_M}$ on the iterates. Since $\hat{\tau}_M$ is a time based on $S$-values, it is often difficult to apply to iterates of SGD and homogenized SGD, so we introduce equivalent stopping times
\begin{equation}
\begin{aligned} \label{eq:stopping_time}
    \vartheta_{M} 
    &
    \defas \inf \{ t \ge 0 \, : \, \|W_{td} \|^2 > M \, \, \text{or} \, \, \ip{W_{td}^{\otimes 2},K}_{\mathcal{A}^{\otimes 2}} \not \in \mathcal{U} \} 
    \\
    \text{or} \quad  \vartheta_{M} 
    &
    \defas \inf \{t \ge 0 \, : \, 
    \|\HSGD_t\|^2  > M  \, \, \text{or} \, \, \ip{\HSGD_t^{\otimes 2}, K}_{\mathcal{A}^{\otimes 2}} \not \in \mathcal{U} \}.     
\end{aligned}
\end{equation}
% where $C(T)>0$ is some constant which depends on the last iterate of homogenized SGD, $T={n}/{d}\in [0,\infty)$. 
% \textcolor{red}{CP: Not sure what the exact stopping criteria should be.}  
We overload the notation $\vartheta_M$ to be either applied to SGD iterates, $W_{td}$ or homogenized SGD iterates, $\HSGD_t$, for which it will be made clear in the context which criterion is used. These stopping times are equivalent to $\hat{\tau}_M$ in that there exists constants $c, C > 0$ such that  $\vartheta_{c \cdot M} \le \hat{\tau}_M \le  \vartheta_{C \cdot M}$ (see Lemma~\ref{lem:S_derivative_bounds}). Moreover, we often drop the $M$ so that $\vartheta \defas \vartheta_M$. It will be convenient to work with the stopped processes, $W_{td }^{\vartheta} \defas W_{ td  \wedge \vartheta }$ and $\HSGD_t^{\vartheta} \defas \HSGD_{t \wedge \vartheta}$.

\subsection{Homogenized SGD under statistics} 
Our goal is a comparison of the dynamical behavior of SGD to another process, \textit{homogenized SGD} (HSGD) applied to the risk $\mathcal{R}_{\delta}(X)$. 
With this, we recall \textit{homogenized SGD} \eqref{eq:HSGD}
\begin{equation}
\dif \WHSGD_t = -\gamma(t) \nabla \mathcal{R}_{\delta} (\WHSGD_t) + \gamma(t) \ip{\sqrt{K / d}  \otimes \sqrt{ \EE_{a, \epsilon} [ \nabla f( \ip{\WHSGD_t \oplus X^{\star}, a}_{\mathcal{A}})^{\otimes 2} ] }, \dif B_t}_{\mathcal{A \otimes O}},
\end{equation}
where the initial conditions given by $\WHSGD_0 = X_0$ and $(B_t, t \ge 0)$ is a $\mathcal{A} \otimes \mathcal{O}$ standard Brownian motion. 

In an analogous definition for homogenized SGD, we introduce
\[
\HSGD_t \defas \WHSGD_t \oplus X^{\star} \quad \text{and} \quad 
\rho_t \defas \ip{\HSGD_t, a}_{\mathcal{A}}.
\]
% Furthermore, we recall by Assumption \ref{assumption:unbiased} 
%  and \ref{assumption:E_loss_pseudo_Lip} that 
% \[
% \mathcal{R}(X) = h \circ B(W) \quad \text{and} \quad \EE_{a}[\nabla f(\ip{W, a}_{\mathcal{A}})^{\otimes 2}] = I \circ B(W) \quad \text{with} \quad \, B(W) = \ip{W \otimes W, K}_{\mathcal{A}^{\otimes 2}},
% \] 
% and the functions $h \, : \, \mathcal{O}^+ \otimes \mathcal{O}^+ \to \mathbb{R}$ differentiable and $I \, : \, \mathcal{O}^+ \otimes \mathcal{O}^+ \to \mathbb{R}$ $\alpha$-pseudo-Lipschitz. 
% It will be useful, throughout the remaining paper, to decompose the derivative of $h$, i.e., $\Dif h$, in terms of its $\mathcal{O}$ and $\mathcal{T}$ components. The easiest and succinct way to do this is to consider a matrix structure
% \begin{equation} \label{eq:matrix_form}
%     (a \oplus b) \otimes (c \oplus d) \cong \left [ \begin{array}{c|c}
%         a \otimes c & a \otimes d\\
%         \hline
%         b \otimes c & b \otimes d
%     \end{array} \right ].
% \end{equation}
% In this regard, we express $\Dif h$ in terms of this matrix,
% \begin{align*}
% \Dif h 
% &
% \cong
% \left[ \begin{array}{c|c} 
%     \Dif h_{11} & \Dif h_{12}
%     \\
%     \hline
%     \Dif h_{21} & \Dif h_{22}
%     \end{array} \right ] \in \left[ \begin{array}{c|c} 
%     \mathcal{O} \otimes \mathcal{O} & \mathcal{O} \otimes \mathcal{T}
%     \\
%     \hline
%     \mathcal{T} \otimes \mathcal{O} & \mathcal{T} \otimes \mathcal{T}
%     \end{array} \right ].
% \end{align*}
Under this notation, as mentioned before, we will be interested in the behavior of homogenized SGD under one particular statistic, which we introduced earlier as
\[
W \in \mathcal{A} \otimes \mathcal{O}^+ \mapsto S(W,z) = \ip{W \otimes W, R(z; K)}_{\mathcal{A}^{\otimes 2}}, \quad \text{for $z \in \mathbb{C}$.}
\]
We will show that $S(\HSGD_t,z)$ is an approximate solution \eqref{def:integro_differential_equation} to the integro-differential equation \eqref{eq:ODE_resolvent_2} which we state below. 

\begin{proposition}[Homogenized SGD is an approximate solution] \label{prop:homogenized_SGD_approx_solution}
Fix a $T, M > 0$ and $0 < \delta < 1/2$ Then $S(\HSGD_t, z)$ is an $(d^{-\delta}, M, T)$-approximate solution w.o.p., that is, 
\begin{equation}
\sup_{0 \le t \le (T \wedge \tau_M)} \|S(\HSGD_t,z)- S(W_0,z) - \int_0^t \mathscr{F}(z, S(\HSGD_s, z)) \, \dif s \|_{\Gamma} \le d^{-\delta} \quad \text{w.o.p.}
\end{equation}
\end{proposition}
The proof we defer to Section~\ref{sec:HSGD_approximate_solution}.

\subsubsection{Doob decomposition for homogenized SGD.} We begin by writing homogenized SGD under any quadratic test function $\varphi \, : \, \mathcal{A} \otimes \mathcal{O} \to \mathbb{R}$ using It\^{o} calculus. By quadratic, we assume that the function $\varphi$ is smooth (all derivatives exist) and $(\nabla^{(j)} \varphi)(X) \equiv 0$ for all $X \in \mathcal{A} \otimes \mathcal{O}$ and $j \ge 3$. Note that the entries of $S(W, z)$ are quadratic. 

By using It\^{o}'s lemma \cite[Thm. 33, Chapt. 2]{protter2005stochastic}, we deduce that
\begin{equation} \label{eq:HSGD_2}
    \begin{aligned}
        \dif \varphi(\WHSGD_t) 
        &
        = 
        \ip{\nabla \varphi(\WHSGD_t), \dif \WHSGD_t} + \frac{1}{2}\ip{\nabla^2 \varphi(\WHSGD_t), (\dif \WHSGD_t)^{\otimes 2} } 
        \\
        &
        =
        -\gamma(t) \ip{ \nabla \varphi(\WHSGD_t), \nabla \mathcal{R}_{\delta}(\WHSGD_t) } \, \dif t + \frac{\gamma^2(t)}{2d} \ip{ (\nabla^2 \varphi)(\WHSGD_t), \ip{\sqrt{K} \otimes \sqrt{\EE_{a, \epsilon} [\nabla_x f(\rho_t)^{\otimes 2} ]}, \dif B_t}_{\mathcal{A} \otimes \mathcal{O}}^{{\otimes 2}} } 
        \\
        &
        \qquad 
        + 
        \frac{\gamma(t)}{\sqrt{d}} \ip{ \nabla \varphi(\WHSGD_t), \ip{\sqrt{K} \otimes \sqrt{\EE_{a, \epsilon} [\nabla_x f( \rho_t )^{\otimes 2}} ], \dif B_t}_{\mathcal{A} \otimes \mathcal{O}} }.
    \end{aligned}
\end{equation}
We seek to simplify some of the terms in \eqref{eq:HSGD_2}. For this, we flatten the last term in sum:
\begin{equation} \label{eq:martingale_HSGD_1}
    \begin{aligned}
         \ip{ \nabla \varphi (\WHSGD_t), \ip{ \sqrt{K} \otimes \sqrt{ \EE_{a, \epsilon} [ \nabla_x f(\rho_t)^{\otimes 2} ]}, \dif B_t}_{\mathcal{A} \otimes \mathcal{O}} }
        &
        =
        \ip{ \sqrt{K}, \otimes \sqrt{ \EE_{a, \epsilon} [ \nabla_x f(\rho_t)^{\otimes 2} ] }, \dif B_t \otimes \nabla \varphi (\WHSGD_t) }\\
        \text{ (by symmetry) } \qquad 
        &
        =
        \ip{ \sqrt{K}, \otimes \sqrt{ \EE_{a, \epsilon} [ \nabla_x f(\rho_t)^{\otimes 2} ] }, \nabla \varphi(\WHSGD_t) \otimes \dif B_t }.
    \end{aligned}
\end{equation}
Next, we look at the second derivative term of $\varphi$, \eqref{eq:HSGD_2}. To help show this, we use Einstein notation and $(\dif B_t)_{xw} (\dif B_t)_{yz} = \delta_{xy} \delta_{wz} \dif (t \wedge \vartheta)$
\begin{equation} \label{eq:Hessian_HSGD_1}
    \begin{aligned}
        & \ip{\nabla^2 \varphi(\WHSGD_t), \ip{\sqrt{K} \otimes \sqrt{ \EE_{a, \epsilon} [ \nabla_x f(\rho_t)^{\otimes 2} ]}, \dif B_t}_{\mathcal{A} \otimes \mathcal{O}}^{\otimes 2} }
        \\
        &
        \quad 
        =
        \nabla^2 \varphi(\WHSGD_t)_{ijkl} \sqrt{K}_{xi} \sqrt{ \EE_{a, \epsilon} [ \nabla_x f(\rho_t)^{\otimes 2}] }_{wk} \sqrt{K}_{yj} \sqrt{ \EE_{a, \epsilon} [ \nabla_x f(\rho_t)^{\otimes 2}] }_{zl} (\dif B_t)_{xw} (\dif B_t)_{yz} 
        \\ 
        &
        \quad 
        = (\Dif^2 \varphi)(\WHSGD_t)_{ijkl} \sqrt{K}_{xi} \sqrt{ \EE_{a, \epsilon} [ \nabla_x f(\rho_t)^{\otimes 2}] }_{wk} \sqrt{K}_{xj} \sqrt{ \EE_{a, \epsilon} [ \nabla_x f(\rho_t)^{\otimes 2}] }_{wl} \, \dif t
        \\
        &
        \quad
        = \nabla^2 \varphi(\WHSGD_t)_{ijkl} K_{ij} \EE_{a, \epsilon} [ \nabla_x f(\rho_t)^{\otimes 2}]_{kl} \dif t
        \\
        &
        \quad 
        = 
        \ip{ \nabla^2 \varphi(\WHSGD_t), K \otimes \EE_{a, \epsilon}[\nabla_x f(\rho_t)^{\otimes 2}]  } \dif t ,
    \end{aligned}
\end{equation}
where we used symmetry of $\sqrt{K}$ and $\EE_{a, \epsilon}[ \nabla_x f(\rho_t)^{\otimes 2}]$ in the fourth line. 

With this, we can now identify the martingale increment for homogenized SGD, 
\begin{equation} \label{eq:HSGD_statistic}
    \begin{aligned}
        \dif \varphi(\WHSGD_t) 
        &
        = 
        - \gamma(t) \ip{ \nabla \varphi(\WHSGD_t), \nabla \mathcal{R}_{\delta}(\WHSGD_t)} \dif t \\
        & \qquad 
        + \frac{\gamma^2(t)}{2 d} \ip{ \nabla^2 \varphi(\WHSGD_t), K \otimes \EE_{a,\epsilon} [\nabla_x f(\rho_t)^{\otimes 2}] } \, \dif t + \dif \mathcal{M}_{t}^{\text{HSGD}}(\varphi),
        \\
       \text{where} \qquad & \dif \mathcal{M}_{t}^{\text{HSGD}}(\varphi) 
        \defas
        \frac{\gamma(t)}{\sqrt{d}} \ip{ \sqrt{K} \otimes \sqrt{ \EE_{a, \epsilon}[ \nabla_x f(\rho_t)^{\otimes 2} ] }, \nabla \varphi(\WHSGD_t) \otimes \dif B_t } .
    \end{aligned}
\end{equation}
By integrating, we derive the Doob decomposition for $\varphi(\WHSGD_t)$
\begin{equation}
\begin{aligned}
\varphi(\WHSGD_t) 
&
= \varphi(X_0) - \int_0^t \gamma(s) \ip{(\nabla \varphi)(\WHSGD_s), \nabla \mathcal{R}_{\delta}(\WHSGD_s)} \, \dif s
\\
&
\qquad +
\frac{1}{2 d} \int_0^t \gamma^2(s) \ip{ \nabla^2 \varphi(\WHSGD_s), K \otimes \EE_{a, \epsilon} [\nabla_x f(\rho_s)^{\otimes 2}] } \, \dif s + \int_0^t \dif \mathcal{M}_s^{\text{HSGD}}(\varphi).
\end{aligned}
\end{equation}
% We note the comparison of the above formula to the one derived for SGD in \eqref{eq:integral_1} . While we will not be able to show that $\|X_t - \WHSGD_t\|$ are close, nonetheless $\varphi(X_t)$ will be close to $\varphi(\WHSGD_t)$. 

\subsubsection{$S(\HSGD_t, z)$ is an approximate solution, proof of Proposition~\ref{prop:homogenized_SGD_approx_solution}} \label{sec:HSGD_approximate_solution}

The goal in this section is to prove Proposition~\ref{prop:homogenized_SGD_approx_solution}, that is, show that 
\[
S(\HSGD_t,z) = \ip{\HSGD_t^{\otimes 2}, R(z;K)}_{\mathcal{A}^{\otimes 2}}
\]
is an approximate solution to the integro-differential equation \eqref{eq:ODE_resolvent_2}.

% the convergence of $S(\HSGD_t, z)$ to the deterministic function $\mathcal{S}(t,z)$ which solves the integro-differential equation \eqref{eq:ODE_resolvent_1}. We will also show two other representations of \eqref{eq:ODE_resolvent_1}. The first of which is a Volterra equation for $\mathcal{S}(t,z)$ and the second, a more explicit formula for $\mathcal{S}(t,z)$ when $\ip{a, X}_{\mathcal{A}} \in \mathbb{R}$. 

% construct a recurrence for the ``loss". While this is not fully possible with the risk function, there is a quantity, which we denote by $\mathcal{S}$, which closes and from which one can derive a closed form. For this, we recall the expression for homogenized SGD \eqref{eq:HSGD}

% \[
% \dif \WHSGD_t = - \gamma ( \Dif \mathcal{R}(\WHSGD_t) + (\Dif p)(\WHSGD_t) )\dif t+ \ip{ \sqrt{K / d} \otimes ( \EE_a [  \nabla f( \rho_t)^{\otimes 2} ] )^{1/2}, \dif B_t }_{\mathcal{A} \otimes \mathcal{O}} ).
% \]
Letting $\HSGD_t = \WHSGD_t \oplus X^{\star}$,
it will be useful to decompose the statistic $S(\HSGD_t, z)$ and others in terms of their $\mathcal{O}$ and $\mathcal{T}$ components. The easiest and succinct way to do this is to consider a matrix structure
\begin{equation} \label{eq:matrix_form}
    (a \oplus b) \otimes (c \oplus d) \cong \left [ \begin{array}{c|c}
        a \otimes c & a \otimes d\\
        \hline
        b \otimes c & b \otimes d
    \end{array} \right ].
\end{equation}
In their matrix forms,
{\small \begin{gather*}
S(\HSGD_t,z) 
 \cong 
 \left[ \begin{array}{c|c} 
\WHSGD_t^T R(z;K) \WHSGD_t & \WHSGD_t^T R(z;K) X^{\star}\\
\hline
(X^{\star})^T R(z;K) \WHSGD_t & (X^{\star})^T R(z;K) X^{\star}
\end{array}
 \right ]
 \cong
 \left[ \begin{array}{c|c} 
    S_{11}(\HSGD_t,z) & S_{12}(\HSGD_t,z)
    \\
    \hline
    S_{21}(\HSGD_t,z) & S_{22}(\HSGD_t,z)
    \end{array} \right ] \in \left[ \begin{array}{c|c} 
    \mathcal{O} \otimes \mathcal{O} & \mathcal{O} \otimes \mathcal{T}
    \\
    \hline
    \mathcal{T} \otimes \mathcal{O} & \mathcal{T} \otimes \mathcal{T}
    \end{array} \right ],
    \\
    \mathcal{S}(t,z) 
 \cong
 \left[ \begin{array}{c|c} 
    \mathcal{S}_{11}(t,z) & \mathcal{S}_{12}(t,z)
    \\
    \hline
    \mathcal{S}_{21}(t,z) & \mathcal{S}_{22}(t,z)
    \end{array} \right ] \in \left[ \begin{array}{c|c} 
    \mathcal{O} \otimes \mathcal{O} & \mathcal{O} \otimes \mathcal{T}
    \\
    \hline
    \mathcal{T} \otimes \mathcal{O} & \mathcal{T} \otimes \mathcal{T}
    \end{array} \right ],
    \\
 \text{and} \, \, \nabla h 
\cong
\left[ \begin{array}{c|c} 
    \nabla h_{11} & \nabla h_{12}
    \\
    \hline
    \nabla h_{21} & \nabla h_{22}
    \end{array} \right ] \in \left[ \begin{array}{c|c} 
    \mathcal{O} \otimes \mathcal{O} & \mathcal{O} \otimes \mathcal{T}
    \\
    \hline
    \mathcal{T} \otimes \mathcal{O} & \mathcal{T} \otimes \mathcal{T}
    \end{array} \right ].
\end{gather*}
}
With this notation established, the first step to proving Proposition~\ref{prop:homogenized_SGD_approx_solution} is deriving a closed equation for $S(\HSGD_t, z)$ using It\^o calculus.  

\paragraph{It\^o calculus applied to $S(\HSGD_t, z)$.}
Recall the expected risk $\mathcal{R}$ which can be expressed as a composition, $\mathcal{R}(\WHSGD_t) = h \circ B(\HSGD_t)$, for some function $h \, : \, \mathcal{O}^+ \otimes \mathcal{O}^+ \to \mathbb{R}$ and 
\[
B(\HSGD_t) = \ip{\HSGD_t \otimes \HSGD_t, K}_{\mathcal{A}^{\otimes 2}}.
\]
A simple computation yields that
\[
\nabla \mathcal{R} = \ip{\nabla h, \ip{ ( \text{Id}_{\mathcal{A} \otimes \mathcal{T}} \oplus 0_{\mathcal{A} \otimes \mathcal{T}} ) \otimes \HSGD_t, K }_{\mathcal{A}^{\otimes 2}} }_{(\mathcal{O}^{+})^{\otimes 2}} 
+ 
\ip{\nabla h, \ip{\HSGD_t \otimes ( \text{Id}_{\mathcal{A} \otimes \mathcal{T}} \oplus 0_{\mathcal{A} \otimes \mathcal{T}} ), K }_{\mathcal{A}^{\otimes 2}} }_{(\mathcal{O}^{+})^{\otimes 2}}. 
\]
We observe that $\dif \HSGD_t = \dif \WHSGD_t \oplus 0_{\mathcal{A} \otimes \mathcal{T}}$ where $0$ is the zero tensor. Using the  product rule for It\^o derivatives,
\begin{equation} \label{eq:concentration_S_1}
    \begin{aligned}
        \dif S 
        &
        = 
        \ip{ \dif \HSGD_t \otimes \HSGD_t, R(z;K)}_{\mathcal{A}^{\otimes 2}} + \ip{\HSGD_t \otimes \dif \HSGD_t, R(z; K)}_{\mathcal{A}^{\otimes 2}} 
        + \ip{ \dif \HSGD_t \otimes \dif \HSGD_t, R(z;K)}_{\mathcal{A}^{\otimes 2}}
        \\
        &
        = 
        \ip{ (\dif \WHSGD_t \oplus 0_{\mathcal{A} \otimes \mathcal{T}}) \otimes \HSGD_t, R(z;K)}_{\mathcal{A}^{\otimes 2}} + \ip{\HSGD_t \otimes (\dif \WHSGD_t \oplus 0_{\mathcal{A} \otimes \mathcal{T}}), R(z; K)}_{\mathcal{A}^{\otimes 2}} 
        \\
        & 
        \quad + 
        \ip{ (\dif \WHSGD_t \oplus 0_{\mathcal{A} \otimes \mathcal{T}}) \otimes (\dif \WHSGD_t \oplus 0_{\mathcal{A} \otimes \mathcal{T}}), R(z;K)}_{\mathcal{A}^{\otimes 2}} 
        \\
        &
        =
        -\gamma_t \cdot \ip{ ( (\nabla \mathcal{R} + \delta \WHSGD_t) \oplus 0_{\mathcal{A} \otimes \mathcal{T}}) \otimes \HSGD_t, R(z;K)}_{\mathcal{A}^{\otimes 2}} \dif t
        \\
        &
        \quad
        - 
        \gamma_t \cdot \ip{\HSGD_t \otimes ( (\nabla \mathcal{R} + \delta \WHSGD_t)  \oplus 0_{\mathcal{A} \otimes \mathcal{T}}), R(z; K)}_{\mathcal{A}^{\otimes 2}} \dif t
        \\
        & 
        \quad + \tfrac{\gamma^2_t}{d}
        \ip{ K, R(z;K)}_{\mathcal{A}^{\otimes 2}} ( \EE_{a, \epsilon} [ \nabla_x f(\rho_t)^{\otimes 2}] \oplus 0^{\otimes 2}_{\mathcal{T}} ) \, \dif t
        + \dif \mathcal{M}_t^{\text{HSGD}}(S(\cdot, z)).
    \end{aligned}
\end{equation}

\begin{remark} \label{rmk:martingal_S}
    We are interested in the behavior of $S(W,z)$ which lives in $(\mathcal{O}^+)^{\otimes 2}$, but we have only defined the martingale increments for test functions mapping into $\mathbb{R}$. To reconcile the two spaces, we consider, by moving to coordinates, $\varphi(X) = S_{o_i o_j}(W,z)$, that is the $\varphi(X)$ is the $(o_i, o_j)$-th coordinate of $S(W,z)$. Consequently, we write 
\[
\dif \mathcal{M}^{\text{HSGD}}_t(S_{o_i o_j}(z, W)) \defas \frac{\gamma_t}{\sqrt{d}} \ip{ \ip{\sqrt{K} \otimes (\EE_{a, \epsilon} [ \nabla_x f(\rho_t)^{\otimes 2} ])^{1/2}, \nabla_X (S_{o_i o_j}(W,z)) }_{\mathcal{A} \otimes \mathcal{O}}, \dif B_{t}}
\]
and then define, $\dif \mathcal{M}_t^{\text{HSGD}}(S)$ entrywise by
\[
\big ( \dif \mathcal{M}_t^{\text{HSGD}}(S(W,z)) \big )_{o_i o_j} = \dif \mathcal{M}_t^{\text{HSGD}}(S(W,z)_{o_i o_j}).
\]
Analogously, we define
\[
\mathcal{M}_t^{\text{HSGD}}(S(\cdot, z)) = \int_0^t \dif \mathcal{M}_s^{\text{HSGD}}(S(\cdot,z)).
\]
\end{remark}

We consider the first term in the summation above, and after plugging in $\nabla \mathcal{R}$, we have
\begin{align*}
& \ip{ ( (\nabla \mathcal{R} + \delta \WHSGD_t) \oplus 0_{\mathcal{A} \otimes \mathcal{T}}) \otimes \HSGD_t, R(z;K)}_{\mathcal{A}^{\otimes 2}} \, \dif t \\
& 
\qquad 
= 
\ip{ (\ip{\nabla h, \ip{ ( \text{Id}_{\mathcal{A} \otimes \mathcal{T}} \oplus 0_{\mathcal{A} \otimes \mathcal{T}} ) \otimes \HSGD_t, K }_{\mathcal{A}^{\otimes 2}} }_{(\mathcal{O}^{+})^{\otimes 2}}  \oplus 0_{\mathcal{A} \otimes \mathcal{T}}) \otimes \HSGD_t, R(z;K)}_{\mathcal{A}^{\otimes 2}} \, \dif t 
\\
&
\qquad 
\quad 
+ 
\ip{ (\ip{\nabla h, \ip{\HSGD_t \otimes ( \text{Id}_{\mathcal{A} \otimes \mathcal{T}} \oplus 0_{\mathcal{A} \otimes \mathcal{T}} ), K }_{\mathcal{A}^{\otimes 2}} }_{(\mathcal{O}^{+})^{\otimes 2}} \oplus 0_{\mathcal{A} \otimes \mathcal{T}}) \otimes \HSGD_t, R(z;K)}_{\mathcal{A}^{\otimes 2}} \, \dif t
\\
&
\qquad
\quad 
+
\delta \ip{\WHSGD_t \oplus 0_{\mathcal{A} \oplus \mathcal{T}} \otimes \HSGD_t, R(z;K) }_{\mathcal{A}^{\otimes 2}} \, \dif t.
\end{align*}
Expanding the terms with $\HSGD_t = \WHSGD_t \oplus X^{\star}$ and using our matrix conventions, we get that 
\begin{align*}
    -\gamma_t \cdot & \ip{ ( (\nabla \mathcal{R} + \delta \WHSGD_t) \oplus 0_{\mathcal{A} \otimes \mathcal{T}}) \otimes \HSGD_t, R(z;K)}_{\mathcal{A}^{\otimes 2}} \dif t
    \cong
    \left [ \begin{array}{c|c}
        A_1 + \tilde{A}_1 & E + \tilde{E}\\
        \hline
        0 & 0
    \end{array} \right ], 
    \\
    \text{where} 
    \quad A_1
    &
    \cong 
    -\gamma_t \cdot \ip{ \ip{\nabla h, \ip{ ( \text{Id}_{\mathcal{A} \otimes \mathcal{T}} \oplus 0_{\mathcal{A} \otimes \mathcal{T}} ) \otimes \HSGD_t, K }_{\mathcal{A}^{\otimes 2}} }_{(\mathcal{O}^{+})^{\otimes 2}} \otimes \WHSGD_t, R(z;K)}_{\mathcal{A}^{\otimes 2}} \dif t,
    \\
    & \qquad -\gamma_t \cdot \ip{\ip{\nabla h, \ip{\HSGD_t \otimes ( \text{Id}_{\mathcal{A} \otimes \mathcal{T}} \oplus 0_{\mathcal{A} \otimes \mathcal{T}} ), K }_{\mathcal{A}^{\otimes 2}} }_{(\mathcal{O}^{+})^{\otimes 2}}  \otimes \WHSGD_t, R(z;K)}_{\mathcal{A}^{\otimes 2}} \dif t,
    \\
    \quad \tilde{A}_1 
    &
    \cong -\gamma_t \cdot \delta \cdot \ip{\WHSGD_t \otimes \WHSGD_t, R(z;K)}_{\mathcal{A}^{\otimes 2}} \, \dif t,
    \\
    \quad
    E
    & 
    \cong
    -\gamma_t \cdot \ip{ \ip{\nabla h, \ip{ ( \text{Id}_{\mathcal{A} \otimes \mathcal{T}} \oplus 0_{\mathcal{A} \otimes \mathcal{T}} ) \otimes \HSGD_t, K }_{\mathcal{A}^{\otimes 2}} }_{(\mathcal{O}^{+})^{\otimes 2}} ) \otimes X^{\star}, R(z;K)}_{\mathcal{A}^{\otimes 2}} \dif t,
    \\
    & 
    \qquad - \gamma_t \cdot 
    \ip{ \ip{\nabla h, \ip{\HSGD_t \otimes ( \text{Id}_{\mathcal{A} \otimes \mathcal{T}} \oplus 0_{\mathcal{A} \otimes \mathcal{T}} ), K }_{\mathcal{A}^{\otimes 2}} }_{(\mathcal{O}^{+})^{\otimes 2}} \otimes X^{\star}, R(z;K)}_{\mathcal{A}^{\otimes 2}} \dif t,
    \\
    \text{and} 
    \quad
    \tilde{E} 
    &
    \cong -\gamma_t \cdot \delta \cdot \ip{ \WHSGD_t \otimes X^{\star}, R(z;K)}_{\mathcal{A}^{\otimes 2}} \, \dif t.
\end{align*}
This is to say $\ip{ ( (\nabla \mathcal{R} + \delta \WHSGD_t)  \oplus 0_{\mathcal{A} \otimes \mathcal{T}}) \otimes \HSGD_t, R(z;K)}_{\mathcal{A}^{\otimes 2}}$ only effects $\dif S_{11}$ and $\dif S_{12}$

Similarly for the other ``symmetric" term in \eqref{eq:concentration_S_1}, 
\begin{align*}
    - \gamma_t \cdot
    &
    \ip{\HSGD_t \otimes ( (\nabla \mathcal{R} + \delta \WHSGD_t) \oplus 0_{\mathcal{A} \otimes \mathcal{T}}), R(z; K)}_{\mathcal{A}^{\otimes 2}} \dif t \cong \left [
    \begin{array}{c|c}
    A_2 + \tilde{A}_2 & 0\\
    \hline
    C + \tilde{C} & 0
    \end{array}
    \right ],
    \\
    \text{where} 
    \quad A_2
    & 
    \cong
    - \gamma_t \cdot \ip{\WHSGD_t \otimes \ip{\nabla h, \ip{ ( \text{Id}_{\mathcal{A} \otimes \mathcal{T}} \oplus 0_{\mathcal{A} \otimes \mathcal{T}} ) \otimes \HSGD_t, K }_{\mathcal{A}^{\otimes 2}} }_{(\mathcal{O}^{+})^{\otimes 2}}, R(z; K)}_{\mathcal{A}^{\otimes 2}} \dif t,
    \\
    & \qquad
    - \gamma_t \cdot \ip{\WHSGD_t \otimes \ip{\nabla h, \ip{\HSGD_t \otimes ( \text{Id}_{\mathcal{A} \otimes \mathcal{T}} \oplus 0_{\mathcal{A} \otimes \mathcal{T}} ), K }_{\mathcal{A}^{\otimes 2}} }_{(\mathcal{O}^{+})^{\otimes 2}}, R(z; K)}_{\mathcal{A}^{\otimes 2}} \dif t,
    \\
    \quad
    \tilde{A}_2 
    &
    \cong
    -\gamma_t \cdot \delta \cdot \ip{ \WHSGD_t \otimes \WHSGD_t, R(z;K)}_{\mathcal{A}^{\otimes 2}} \, \dif t,
    \\
    \quad
    C 
    & 
    \cong 
    - \gamma_t \cdot \ip{X^{\star} \otimes \ip{\nabla h, \ip{ ( \text{Id}_{\mathcal{A} \otimes \mathcal{T}} \oplus 0_{\mathcal{A} \otimes \mathcal{T}} ) \otimes \HSGD_t, K }_{\mathcal{A}^{\otimes 2}} }_{(\mathcal{O}^{+})^{\otimes 2}}, R(z; K)}_{\mathcal{A}^{\otimes 2}} \dif t,
    \\
    & \qquad
    - \gamma_t \cdot \ip{X^{\star} \otimes \ip{\nabla h, \ip{\HSGD_t \otimes ( \text{Id}_{\mathcal{A} \otimes \mathcal{T}} \oplus 0_{\mathcal{A} \otimes \mathcal{T}} ), K }_{\mathcal{A}^{\otimes 2}} }_{(\mathcal{O}^{+})^{\otimes 2}}, R(z; K)}_{\mathcal{A}^{\otimes 2}} \dif t,
    \\
    \text{and}
    \quad 
    \tilde{C} 
    &
    \cong
    -\gamma_t \cdot \delta \cdot \ip{X^{\star} \otimes \WHSGD_t, R(z;K)}_{\mathcal{A}^{\otimes 2}} \, \dif t.
\end{align*}
The last term in \eqref{eq:concentration_S_1} is quite simple
\begin{align*}
    \frac{\gamma^2_t}{d}
    &
        \ip{ K, R(z;K)}_{\mathcal{A}^{\otimes 2}} ( \EE_{a, \epsilon} [ \nabla_x f(\rho_t)^{\otimes 2}] \oplus 0^{\otimes 2}_{\mathcal{T}} ) \, \dif t 
        \cong
        \left [
        \begin{array}{c|c}
             A_3 & 0  \\
             \hline
             0 &  0
        \end{array}
        \right ],
        \\
        \text{where} 
        \quad A_3 
        &
        \cong \frac{\gamma^2_t}{d} \tr(KR(z;K) ) \EE_{a, \epsilon} [ \nabla_x f(\rho_t)^{\otimes 2}] \, \dif t.
\end{align*} 
It follows then that 
\[
(\dif S)(\HSGD_t,z) \cong 
\left [ 
\begin{array}{c|c}
\dif S_{11} & \dif S_{12} \\
\hline
\dif S_{21} & \dif S_{22}
\end{array}
\right ] = 
\left [ 
\begin{array}{c|c}
A_1 + \tilde{A}_1 + A_2 + \tilde{A}_2 + A_3 + \tilde{A}_3 & E + \tilde{E} \\
\hline
C + \tilde{C} & 0
\end{array}
\right ] + \dif \mathcal{M}_t^{\text{HSGD}}(S(\HSGD_t, z)).
\]
We further seek to simplify the terms $A_1, A_2, A_3, E,$ and $C$. For this, recall $\nabla h$ viewed in its matrix form as 
\[
\nabla h \cong \left [ \begin{array}{c|c} \nabla h_{11} & \nabla h_{12} \\ \hline \nabla h_{21} & \nabla h_{22} \end{array} \right ], 
\]
and consequently, after simple computations (and $\nabla h_{12} = \nabla h_{21}$), we derive 
\begin{equation} \label{eq:S_matrices}
\begin{aligned}
    A_1 
    &
    = 
    -2\gamma_t \cdot \ip{ \ip{\nabla h_{11}, \WHSGD_t}_{\mathcal{O}} \otimes \WHSGD_t, K R(z;K)}_{\mathcal{A}^{\otimes 2}}
    -2\gamma_t \ip{ \ip{\nabla h_{12}, X^{\star}}_{\mathcal{T}} \otimes \WHSGD_t, K R(z;K)}_{\mathcal{A}^{\otimes 2}} \dif t, \\
    A_2  
    &
    =
    -2\gamma_t \cdot \ip{ \WHSGD_t \otimes \ip{\nabla h_{11}, \WHSGD_t}_{\mathcal{O}}, K R(z;K)}_{\mathcal{A}^{\otimes 2}}
    -2\gamma_t \cdot \ip{ \WHSGD_t \otimes  \ip{\nabla h_{12}, X^{\star}}_{\mathcal{T}}, K R(z;K)}_{\mathcal{A}^{\otimes 2}}  \, \dif t, \\
    A_3
    &
    =
    \frac{\gamma^2_t}{d} \tr(K R(z;K) ) \EE_{a, \epsilon} [ \nabla_x f(\rho_t)^{\otimes 2}]  \, \dif t,
    \\
    E 
    &
    =
    -2\gamma_t \cdot \ip{ \ip{\nabla h_{11}, \WHSGD_t}_{\mathcal{O}} \otimes X^{\star}, K R(z;K)}_{\mathcal{A}^{\otimes 2}}
    -2\gamma_t \cdot \ip{\ip{\nabla h_{12}, X^{\star}}_{\mathcal{T}} \otimes X^{\star}, K R(z;K)}_{\mathcal{A}^{\otimes 2}}  \, \dif t, \\
    \text{and}
    \quad 
    C
    &
    =
    -2\gamma_t \cdot \ip{ X^{\star} \otimes \ip{\nabla h_{11}, \WHSGD_t}_{\mathcal{O}}, K R(z;K)}_{\mathcal{A}^{\otimes 2}}
    -2\gamma_t \cdot \ip{X^{\star} \otimes \ip{\nabla h_{12}, X^{\star}}_{\mathcal{T}}, K R(z;K)}_{\mathcal{A}^{\otimes 2}}  \, \dif t. \\    
\end{aligned}
\end{equation}
We observe that 
\[
KR(z;K) = K (K-zI_{\mathcal{A}})^{-1} = (K -zI_{\mathcal{A}} + zI_\mathcal{A}) (K-zI_{\mathcal{A}})^{-1} = I_{\mathcal{A}} + z R(z;K). 
\]
We can now see, using the above identity, that the quantities $A_1, A_2, E$, and $C$ \eqref{eq:S_matrices} and the quantities $\tilde{A}_1, \tilde{A}_2, \tilde{E}$, and $\tilde{C}$ can be expressed back in terms of $S(\HSGD_t, z) = \ip{\HSGD_t \otimes \HSGD_t, R(z;K)}_{\mathcal{A}^{\otimes 2}}$. The result is that
\begin{equation} 
\begin{aligned}
    \dif S(\HSGD_t,z) 
    &
    =
    - 2\gamma_t \cdot \big [ V_0(\HSGD_t) (H\circ B(\HSGD_t) ) + (H^T \circ B(\HSGD_t) ) V_0(\HSGD_t) \big ] \, \dif t
    \\
    &
    + \frac{\gamma^2_t}{d}\left [ \begin{array}{c|c} 
    \tr(K R(z;K)) \EE_{a, \epsilon} [ \nabla_x f(\rho_t)^{\otimes 2}] & 0\\ \hline 0 & 0
    \end{array} \right ] \, \dif t
    \\
    &
    - \gamma_t \cdot (S(\HSGD_t,z) (2z ( H \circ B(\HSGD_t) ) + \delta D) + ( 2 z (H^T \circ B(\HSGD_t)) + \delta D) S(\HSGD_t,z)) \, \dif t
    \\
    &
    + \dif \mathcal{M}_t^{\text{HSGD}}(S),
\end{aligned}
\end{equation}

\begin{gather*}
\text{where} \quad
V_0(W) = \ip{W \otimes W, I_{\mathcal{A}}}_{\mathcal{A}^{\otimes}}, 
\quad
B(W) = \ip{W \otimes W, K}_{\mathcal{A}^{\otimes 2}},
% \left [ \begin{array}{c|c} \WHSGD_t^T \WHSGD_t & \WHSGD_t^T X^{\star} \\
%  \hline
%  (X^{\star})^T \WHSGD_t & (X^{\star})^T X^{\star} 
%  \end{array} \right ], 
\quad
 H(B) = \left [ \begin{array}{c|c} 
 \nabla h_{11}(B) & 0\\
 \hline
 \nabla h_{21}(B) & 0
 \end{array} \right ], \\
  D = \left [ \begin{array}{c|c} 
 I_{\mathcal{O}} & 0\\
 \hline
 0 & 0
 \end{array} \right ], \quad 
% \quad S(\HSGD_t,z) \cong \left[ \begin{array}{c|c} 
% \WHSGD_t^T R(z;K) \WHSGD_t & \WHSGD_t^T R(z;K) X^{\star}\\
% \hline
% (X^{\star})^T R(z;K) \WHSGD_t & (X^{\star})^T R(z;K) X^{\star}
% \end{array}
%  \right ], 
%  \\
\text{and initialized with} \quad S(0,z) = \ip{W_0 \otimes W_0, R(z; K)}_{\mathcal{A}^{\otimes 2}}. 
%  \left[ \begin{array}{c|c} 
% X_0^T R(z;K) X_0 & X_0^T R(z;K) X^{\star}\\
% \hline
% (X^{\star})^T R(z;K) X_0 & (X^{\star})^T R(z;K) X^{\star}
% \end{array}
%  \right ].
\end{gather*}
Using Cauchy integral formula identities related to the resolvent, we see 
\begin{equation}
\begin{aligned}
    V_0(W) = \frac{-1}{2 \pi i} \oint_{\Gamma} S(W, z) \, \dif z \quad \text{and} \quad B(W) = \frac{-1}{2 \pi i} \oint_{\Gamma} zS(W,z) \, \dif z,
\end{aligned}
\end{equation}
and moreover, by Assumption~\ref{assumption:E_loss_pseudo_Lip}, 
\[
\EE_{a,\epsilon} [ \nabla_x f(\rho_t)^{\otimes 2}] = I \circ B(\HSGD_t) = I \circ B \big (\tfrac{-1}{2 \pi i} \oint_{\Gamma} zS(\HSGD_t,z) \, \dif z \big ).
\]
Therefore, 
\begin{equation}\label{eq:HSGD_exact}
\dif S(\HSGD_t, \cdot) = \mathscr{F}(z, S(\HSGD_t, \cdot)) \, \dif t + \dif \mathcal{M}^{\text{HSGD}}_t(S(\HSGD_t, \cdot)), 
\end{equation}
with $S(\HSGD_0, \cdot) = \ip{W_0 \otimes W_0, R(\cdot;K)}_{\mathcal{A}^{\otimes 2}}$. We now are ready to prove Proposition~\ref{prop:homogenized_SGD_approx_solution}.

\begin{proof}[Proof of Proposition~\ref{prop:homogenized_SGD_approx_solution}]
By It\^o's Lemma, we have seen that 
\[
S(\HSGD_t, \cdot) = \ip{W_0 \otimes W_0, R(\cdot;K)}_{\mathcal{A}^{\otimes 2}} + \int_0^t \mathscr{F}(z, S(\HSGD_s, \cdot)) \, \dif s + \int_0^t \dif \mathcal{M}^{\text{HSGD}}_s(S(\HSGD_s, \cdot)).
\]
Thus to show that $S(\HSGD_t, \cdot)$ is an approximate solution of the integro-differential equation \eqref{eq:ODE_resolvent_2} it amounts to bounding the martingale term where $C$ is a positive constant independent of $d$. Let $\Gamma = \{ z \, : \, |z| = \max\{1, 2\|K\|_{\sigma}\}\}$. For all $z \in \Gamma$, we note that for some constants $C,c > 0$ such that $\vartheta_{c \cdot M} \le \hat{\tau}_{M} \le \vartheta_{C \cdot M}$ (see Lemma~\ref{lem:normequivalence}). Consequently, we can work with the stopped process $\HSGD_t^{\vartheta} = \HSGD_{t \wedge \vartheta}$ instead of using $\hat{\tau}_M$. We thus have that for all $z \in \Gamma$,
\[
\sup_{0 \le t \le T \wedge \hat{\tau}_{M} } \| S(\HSGD_t, z) - S(W_0,z) - \int_0^t \mathscr{F}(z, S(\HSGD_s, z )) \, \dif s \| \le \sup_{0 \le t \le T \wedge \vartheta_{C \cdot M} } \|\mathcal{M}_{t}^{\text{HSGD}}(S(\cdot, z)) \|.
\]
Fix a constant $\delta > 0$. Let $\Gamma_{\delta} \subset \Gamma$ such that there exists a $\bar{z} \in \Gamma_{\delta}$ such that $|z-\bar{z}| \le d^{-\delta}$ and the cardinality of $\Gamma_{\delta}$, $|\Gamma_{\delta}| = C d^{\delta}$ where $C > 0$ depending on $\|K\|_{\sigma}$. 

By the martingale error proposition,  Proposition~\ref{prop:HSGD_Martingale_Bound}, which we have deferred the proof to Section~\ref{sec:HSGD_martingale}, we have that for any $\hat{\delta} > 0$ 
\[
\sup_{0 \le t \le T} \| \mathcal{M}_{t \wedge \vartheta_{C \cdot M}}^{\text{HSGD}}(S(\cdot, z))\| \le C \cdot L(f) \cdot d^{\hat{\delta}/2 - 1/2}, \quad \text{w.o.p.}
\]
As the cardinality of $\Gamma_{\delta}$ is polynomial in $d$, we have that 
\[
\sup_{z \in \Gamma_{\delta}} \sup_{0 \le t \le T} \| \mathcal{M}_{t \wedge \vartheta_{C \cdot M}}^{\text{HSGD}}(S(\cdot, z))\| \le C \cdot L(f) \cdot d^{\hat{\delta}/2 - 1/2}, \quad \text{w.o.p.}
\]
Consequently, we deduce that 
\begin{align*}
\sup_{0 \le t \le T \wedge \hat{\tau}_{M} } \| S(\HSGD_t^{\vartheta}, z) - S(W_0,z) - \int_0^t \mathscr{F}(z, S(\HSGD_s^{\vartheta}, z )) \, \dif s \|_{\Gamma_{\delta}} 
&
\le
\sup_{0 \le t \le T} \| \mathcal{M}_{t \wedge \vartheta}^{\text{HSGD}}(S(\cdot, z))\|_{\Gamma_{\delta}}
\\
&
\le
C \cdot L(f) \cdot d^{\hat{\delta}/2 - 1/2} \quad \text{w.o.p}.
\end{align*}
An application of the net argument, Lemma~\ref{lem:net_argument}, finishes the proof after setting $\hat{\delta} = 1-2\delta$. 
\end{proof}

\subsection{SGD under the statistics} \label{sec:SGD_approx_solution}
In this section, we show that $S(W_{td},z)$ is an approximate solution \eqref{def:integro_differential_equation} to the integro-differential equation \eqref{eq:ODE_resolvent_2} which we state below. 

\begin{proposition}[SGD is an approximate solution] \label{prop:SGD_approx_solution}
Fix a $T, M > 0$ and $0 < \delta < 1/2$ Then $S(\HSGD_t, z)$ is an $(d^{-\delta}, M, T)$-approximate solution w.o.p., that is, 
\begin{equation}
\sup_{0 \le t \le (T \wedge \tau_M)} \|S(W_{td},z)- S(W_0,z) - \int_0^t \mathscr{F}(z, S(W_{sd}, z)) \, \dif s \|_{\Gamma} \le d^{-\delta} \quad \text{w.o.p.}
\end{equation}
\end{proposition}
The proof of this Procession is deferred to Section~\ref{sec:SGD_approximate_solution}.

\subsection{Doob decomposition for SGD}
We begin by writing SGD under any quadratic statistic $\varphi \, : \, \mathcal{A} \otimes \mathcal{O}$ satisfying Assumption~\ref{assumption:statistic} in terms of its Doob decomposition by 
% Because the Doob's decomposition holds for a general statistic $\varphi$, we will, in fact, write the Doob decomposition 
identifying the predictable part of $\varphi(X_k)$. We later specialize to $S(W_{td}, z)$ in Section~\ref{sec:SGD_approximate_solution} when we show that $S(W_{td}, z)$ is an approximated solution as defined in \ref{def:integro_differential_equation}. 

By Taylor's expansion, setting $\Delta_k \defas a_{k+1} \otimes \nabla_x f(r_k) + \delta X_k$, 

\begin{equation} \label{eq:taylor_theorem}
    \begin{aligned}
    \varphi(X_{k+1}) 
    = 
    \varphi(X_k - \frac{\gamma_k}{d} \Delta_k \big ) 
    &
    =
    \varphi(X_{k}) - \frac{\gamma_k}{d} \ip{\nabla \varphi(X_k), \Delta_k}
    + 
    \frac{1}{2} \cdot \frac{\gamma_k^2}{d^2} \cdot \ip{\nabla^2\varphi(X_k), \Delta_k^{\otimes 2}}
    % \\
    % & 
    % \qquad 
    % - 
    % \frac{\gamma^3}{2 d^3} \int_0^1 (1-s)^2 \ip{\Dif^3 \varphi \big (X_k - s\frac{\gamma}{d} \Delta_k \big ), \Delta_k^{\otimes 3}} \, \dif s.
    \end{aligned}
\end{equation}
% Here we have that $\Dif^2\varphi(W) \in \mathcal{L}(\mathbb{R}^{d \times \ell}, \mathcal{L}(\mathbb{R}^{d \times \ell}, \mathbb{R}) )$ and $\Dif^3 \varphi(W) \in \mathcal{L}( \mathbb{R}^{d \times \ell}, \mathcal{L}(\mathbb{R}^{d \times \ell}, \mathcal{L}(\mathbb{R}^{d \times \ell}, \mathbb{R}) ))$ representing the 2nd and 3rd derivatives respectively. We use the standard identification of $\mathcal{L}(\mathbb{R}^{d \times \ell}, \mathbb{R}) \approx \mathbb{R}^{d \times \ell}$. Moreover, as we apply the same matrix $a_{k+1} \otimes \nabla f(r_k; r^{\star})$, we simplify the notation so say $\Dif^2 \varphi (W) [(a_{k+1} \otimes \nabla f(r_k; r^{\star})^{\otimes 2} ]$ is $\Dif \varphi(W)$ is first applied to $a_{k+1} \otimes \nabla f(r_k; r^{\star})$ and then applying the outputting linear functional to $a_{k+1} \otimes \nabla f(r_k; r^{\star})$. The same applies to $\Dif^3 \varphi(W) [ (a_{k+1} \otimes \nabla f(r_k; r^{\star})^{\otimes 3}]$ so that the resulting output is a scalar. 
To write the Doob decomposition, the idea is to condition on $r_k = \ip{a_{k+1}, W_k}_{\mathcal{A}}$ and $W_{k} = X_k \oplus X^{\star}$. For this, we will introduce some notation. Define the $\sigma$-algebras, 
\[\mathcal{G}_k \defas \sigma ( \{W_i\}_{i=0}^k, \{r_i\}_{i=0}^k) \quad \text{and} \quad \mathcal{F}_k \defas \sigma ( \{W_i\}_{i=0}^k ).
\]
% The continuous counterparts $\{ \mathcal{F}_t \, : \, t \ge 0\}$ and $\{\mathcal{G}_t \, : \, t \ge 0\}$ are defined analogously via $r_{td} = \ip{a_{td + 1}, W_{td}}_{\mathcal{A}}$ and $W_{td}$.  

% For this, we introduce a lemma. 

% We will use Lemma~\ref{lem:conditioning} to evaluate various terms in \eqref{eq:taylor_theorem}. 

\paragraph{Gradient term in Taylor expansion.}
 First, we consider the gradient term in \eqref{eq:taylor_theorem},
 \begin{equation} \label{eq:grad}
     \frac{\gamma_k}{d} \ip{\nabla \varphi(X_k), a_{k+1} \otimes \nabla_x f(r_k) + \delta X_k}.
 \end{equation} 
We now define a martingale increment associated with the gradient term in \eqref{eq:taylor_theorem} as 
\begin{equation}
    \begin{aligned}
      \Delta \mathcal{M}^{\text{Grad}}_k(\varphi) 
      &
      \defas 
      \frac{\gamma_k}{d} \ip{\nabla \varphi (X_k),  a_{k+1} \otimes \nabla_x f(r_k) }  
      - 
      \frac{\gamma_k}{d} \EE \big [ \ip{ \nabla \varphi (X_k), a_{k+1} \otimes \nabla_x f(r_k) } \, | \, \mathcal{F}_k \big ].
      % \\
      % & 
      % =
      % \frac{\gamma}{d} \ip{(\Dif \varphi)(X_k), a_{k+1} \otimes \nabla f(r_k) } 
      % - 
      % \frac{\gamma}{d} \EE \big [ \EE \big [ \ip{ (\Dif \varphi)(X_k), a_{k+1} \otimes \nabla f(r_k) } \, | \, \mathcal{G}_k \big ] \, | \, \mathcal{F}_k \big ]
      % \\
      % &
      % =
      %  \frac{\gamma}{d} \ip{ (\Dif \varphi)X_k), a_{k+1} \otimes \nabla f(r_k) }
      %  \\
      % & \qquad -
      %  \frac{\gamma}{d} \EE \big [ \ip{(\Dif \varphi)(X_k), K W_k (W_k^T K W_k)^{-1} r_k  \otimes \nabla f(r_k) } \, | \, \mathcal{F}_k \big ].
    \end{aligned}
\end{equation}
where $W_k = X_k \oplus X^{\star} \in \mathcal{A} \otimes \mathcal{O}^+$. Passing the derivative under that expectation, the Jacobian of the risk function, $\nabla \mathcal{R}(X) = \EE_{a, \epsilon}[ a \otimes \nabla_x f(\ip{X,a}_{\mathcal{A}}) ]$. It immediately follows that 
\[
\EE [ \ip{ \nabla \varphi(X_k), a_{k+1} \otimes \nabla_x f(r_k) } \, | \, \mathcal{F}_k] =  \ip{ \nabla \varphi (X_k), \nabla \mathcal{R}(X_k) }. 
\]
Consequently, we can express the gradient term in \eqref{eq:grad} as simply
\begin{equation} \label{eq:grad_term}
\begin{gathered}
    \frac{\gamma_k}{d} \ip{\nabla \varphi(X_k), a_{k+1} \otimes \nabla_x f(r_k) + \delta X_k}
    =  
    \frac{\gamma_k}{d} \ip{\nabla \varphi (X_k), \nabla \mathcal{R}(X_k) + \delta X_k}
    % \EE \big [ \ip{(\Dif \varphi)(X_k), K W_k (W_k^T K W_k)^{-1} r_k \otimes \nabla f(r_k) + (\Dif p)(X_k) } \, | \, \mathcal{F}_k \big ]
    + \Delta \mathcal{M}_k^{\text{grad}}
    \\
    \text{where} \, \, \Delta \mathcal{M}_k^{\text{grad}}(\varphi) = \frac{\gamma_k}{d} \ip{\nabla \varphi (X_k),  a_{k+1} \otimes \nabla_x f(r_k) }  
      - 
      \frac{\gamma_k}{d} \EE \big [ \ip{ \nabla \varphi (X_k), a_{k+1} \otimes \nabla_x f(r_k) } \, | \, \mathcal{F}_k \big ].
    % = \frac{\gamma}{d} \ip{ (\Dif \varphi)(X_k), a_{k+1} \otimes \nabla f(r_k) }
    % - \frac{\gamma}{d} \EE \big [ \ip{(\Dif \varphi)(X_k), K W_k (W_k^T K W_k)^{-1} r_k \otimes \nabla f(r_k) } \, | \, \mathcal{F}_k \big ].
\end{gathered}
\end{equation}

\paragraph{Hessian term in the Taylor expansion.}  Next, we turn to simplifying and estimating the conditional expectation of the term that arises due to the second derivative in the Taylor expansion \eqref{eq:taylor_theorem},
\begin{equation} \label{eq:Hessian_taylor_term}
\begin{aligned}
    \frac{\gamma_k^2}{2d^2} &\ip{ \nabla^2 \varphi(X_k), \big ( a_{k+1} \otimes \nabla_x f(r_k) + \delta X_k \big )^{\otimes 2} } 
    = 
    \frac{\gamma_k^2}{2d^2} \ip{ \nabla^2 \varphi (X_k),  a_{k+1}^{\otimes 2} \otimes \nabla_x f(r_k)^{\otimes 2} }
    \\
    &
    +\frac{\gamma_k^2}{2d^2} \ip{ \nabla^2 \varphi(X_k),  ( \delta X_k )^{\otimes 2} } 
    +
    \frac{\gamma_k^2}{d^2} \ip{ \nabla^2 \varphi(X_k), a_{k+1} \otimes \nabla_x f(r_k) \otimes \delta X_k  }. 
\end{aligned}
\end{equation}
Setting $\Delta_k = a_{k+1} \otimes \nabla_x f(r_k) + \delta X_k$, let us introduce the martingale increment associated with the Hessian,
\begin{equation} \label{eq:Hessian_Martingale}
    \begin{aligned}
        \Delta \mathcal{M}_k^{\text{Hess}}(\varphi) \defas \frac{\gamma_k^2}{2d^2} \bigg ( \ip{\nabla^2 \varphi(X_k), \Delta_k^{\otimes 2} } - \EE[ \ip{\nabla^2 \varphi(X_k), \Delta_k^{\otimes 2}} \, | \, \mathcal{F}_k] \bigg ) .
    \end{aligned}
\end{equation}

Now we seek to evaluate the conditional expectation of \eqref{eq:Hessian_taylor_term} on $\mathcal{F}_k$. To do so, we begin by first conditioning on $\mathcal{G}_k$ and utilizing the following Lemma~\ref{lem:conditioning} as a way to simplify and isolate the leading order term. 
\begin{lemma}[Conditioning] \label{lem:conditioning} Let $|\mathcal{O}| < d$. Suppose $v \in \mathcal{A}$ is distributed $N(0, I_d)$ and $U \in \mathcal{A} \otimes \mathcal{O}$ has orthonormal columns. Then 
\begin{equation}
 v \, | \, \ip{U, v}_{\mathcal{A}} \sim v - U (U^T v) + U U^T v, 
 \end{equation}
where $v - U (U^T v) \sim N(0, I_d - U U^T)$ and $U U^T v \sim N(0, U U^T)$ with $v - U (U^T v)$ independent of $U U^T v$. 
\end{lemma}

A simple computation yields
\begin{equation} \label{eq:hessian_1}
    \begin{aligned}
        \EE[ \ip{ \nabla^2 \varphi(X_k), a_{k+1}^{\otimes 2} \otimes \nabla_x f(r_k)^{\otimes 2} } \, | \, \mathcal{G}_k ] 
        &
        =
        \EE [ \ip{ \nabla^2 \varphi(X_k), (a_{k+1}-\EE[a_{k+1} \, | \, \mathcal{G}_k])^{\otimes 2} \otimes \nabla_x f(r_k)^{\otimes 2} } \, | \, \mathcal{G}_k]
        \\
        &
        \quad +
        \ip{ \nabla^2 \varphi(X_k), \EE[a_{k+1} \, | \, \mathcal{G}_k]^{\otimes 2} \otimes \EE_{\epsilon_k}[\nabla_x f(r_k)^{\otimes 2}] }.
    \end{aligned}
\end{equation}

To compute the conditional mean $\EE[a_{k+1} \, | \, \mathcal{G}_k]$ and conditional covariance $(\EE[ a_{k+1} - \EE[a_{k+1} \, | \, \mathcal{G}_k ] ])^{\otimes 2}$, we use Lemma~\ref{lem:conditioning}. By Assumption~\ref{ass:data_normal}, we write $a_{k+1} = \sqrt{K} v_{k}$ where $v_{k} \sim N(0, I_d)$. Now we perform a QR-decomposition on $\ip{\sqrt{K}, W_k}_{\mathcal{A}}  \defas \ip{Q_k, R_k}_{\mathcal{O}^+}$ where $Q_k \in \mathcal{A} \otimes \mathcal{O}^+$ is orthogonal and $R_k \in (\mathcal{O}^+)^{\otimes 2}$ is upper triangular (and invertible). Set $\Pi_k \defas Q_k Q_k^T$. In distribution, 
\[ 
a_{k+1} \, | \,\ip{ a_{k+1}, W_k}_{\mathcal{A}} \overset{\text{d}}{=} \sqrt{K} v_k \, | \, R_k^T Q_k^T v_k. 
\]
As $R_k$ is invertible, by Lemma~\ref{lem:conditioning}, 
\begin{equation} \label{eq:blah_1}
	a_{k+1} \, | \, \ip{a_{k+1}, W_k}_{\mathcal{A}} \overset{\text{d}}{=} \sqrt{K} v_k \, | \,  Q_k^T v_k \overset{\text{d}}{=} \sqrt{K} \big ( v_k - \Pi_k v_k \big ) + \sqrt{K} \Pi_k v_k.
\end{equation}
We note that $(I_d - \Pi_k) v_k \sim N(0, I_{d}  - \Pi_k )$ and $\Pi_k v_k \sim N(0, \Pi_k)$ with $(I_d - \Pi_k)v_k$ independent of $\Pi_k v_k$. 
% Now we simplify the expression by computing the key quantity, the covariance $\Pi_k$. For this, we observe that 
% \begin{equation} \label{eq:Q_kQ_k}
% \begin{aligned}
% 	\Pi_k
% 	= 
% 	Q_k R_k ( R_k^T R_k)^{-1} R_k^T Q_k^T
% 	= 
% 	\sqrt{K} W_k ( W_k^T K W_k )^{-1} W_k^T \sqrt{K}.
% \end{aligned}
% \end{equation}
From this, we have that
\begin{equation} \label{eq:mean_a}
\EE[ a_{k+1} \, | \, \mathcal{G}_k ] = \sqrt{K} \Pi_k v_k, \quad \text{where $v_k \sim N(0, I_d)$.}
\end{equation}
 Moreover the conditional covariance of $a_{k+1}$ is precisely 
\begin{equation}
    \begin{gathered}
       (\EE[ a_{k+1} - \EE[a_{k+1} \, | \, \mathcal{G}_k ] ])^{\otimes 2} = \sqrt{K} (I_d - \Pi_k) \sqrt{K}, \quad \text{where $\Pi_k = Q_k Q_k^T$.} 
    \end{gathered}
\end{equation}

% = \sqrt{K} W_k ( W_k^T K W_k )^{-1} W_k^T \sqrt{K}$ as described in \eqref{eq:Q_kQ_k}. 
% Moreover, as we have already computed \eqref{eq:Q_kQ_k_1}, the conditional mean is given by
% \begin{equation} \label{eq:mean_a}
% \EE[ a_{k+1} \, | \, \mathcal{G}_k ] = K W_k ( W_k^T K W_k )^{-1} r_k  = \sqrt{K} \Pi_k v_k,
% \end{equation}
% where $v_k \sim N(0, I_d)$. 

Next, we now expand \eqref{eq:hessian_1} to get the leading order behavior
\begin{equation} \label{eq:hessian_2}
\begin{aligned}
    \EE[ \ip{ \nabla^2 \varphi(X_k), a_{k+1}^{\otimes 2} \otimes \nabla_x f(r_k)^{\otimes 2} } \, | \, \mathcal{G}_k ] 
    &
    = 
    \ip{ \nabla^2 \varphi(X_k), K \otimes  \EE_{\epsilon_k}[\nabla_x f(r_k)^{\otimes 2} ] } \\
    &
    -
    \ip{\nabla^2 \varphi(X_k), \sqrt{K} \Pi_k \sqrt{K} \otimes \EE_{\epsilon_k} [\nabla_x f(r_k)^{\otimes 2}]} 
    \\
    & 
    + 
    \ip{ \nabla^2 \varphi (X_k), \big ( \sqrt{K}\Pi_k v_k \big )^{\otimes 2} \otimes \EE_{\epsilon_k}[\nabla_x f(r_k)^{\otimes 2}]}.
\end{aligned}
\end{equation}
We will later see, in Section~\ref{sec:lower_term_Hessian}, that the term, 
\begin{align*}
    \mathcal{E}_{k,1}^{\text{Hess}}(\varphi) 
    &
    \defas
    \ip{\nabla^2 \varphi(X_k), \sqrt{K} \Pi_k \sqrt{K}  \otimes \EE_{\epsilon_k} [\nabla_x f(r_k)^{\otimes 2}]} \\
    &
    \quad 
    +
   \ip{ \nabla^2 \varphi (X_k), \big ( \sqrt{K} \Pi_k v_k \big )^{\otimes 2} \otimes \EE_{\epsilon_k}[\nabla_x f(r_k)^{\otimes 2}]},
\end{align*}
is of lower order and will disappear as $d \to \infty$. So, we may write
\begin{equation}
    \begin{aligned}
        \frac{\gamma_k^2}{2d^2} \EE [ \ip{\nabla^2 \varphi (X_k), a_{k+1}^{\otimes 2} \otimes \nabla_x f(r_k)^{\otimes 2} }   \, | \, \mathcal{F}_k ] 
        &
        = 
        \frac{\gamma_k^2}{2d^2} \ip{\nabla^2 \varphi(X_k), K \otimes \EE[\nabla_x f(r_k)^{\otimes 2}  \, | \, \mathcal{F}_k ]} \\
        & \qquad + \EE[ \mathcal{E}_{k,1}^{\text{Hess}} \, | \, \mathcal{F}_k].
    %     \\
    %     \text{where} \qquad \mathcal{E}_{k,1}^{\text{Hess}}(\varphi) \defas  -\frac{\gamma^2}{2d^2} \ip{(\Dif^2 \varphi)(X_k), \sqrt{K} \Pi_k \sqrt{K} \otimes \nabla f(r_k)^{\otimes 2}} 
    % \\
    % + \frac{\gamma^2}{2d^2}
    % \ip{ (\Dif^2 \varphi)(X_k), \big ( \sqrt{K} \Pi_k v_k \big )^{\otimes 2} \otimes \nabla f(r_k)^{\otimes 2}}.
    \end{aligned}
\end{equation}

For the other terms in \eqref{eq:Hessian_taylor_term}, indeed, it is clear
\[
\frac{\gamma_k^2}{2d^2} \EE[ \ip{\nabla^2 \varphi(X_k), (\delta X_k )^{\otimes 2} } \, | \, \mathcal{F}_k ] = \frac{\gamma_k^2}{2d^2}  \ip{\nabla^2 \varphi(X_k), ( \delta X_k )^{\otimes 2} }. 
\]
Due to the factor of $\tfrac{1}{d^2}$, this term will be of lower order and disappear as $d \to \infty$ (see Section~\ref{sec:lower_term_Hessian}). As such, we define it as
\begin{equation} \label{eq:Hess_lower_term_2}
\mathcal{E}_{k,2}^{\text{Hess}}(\varphi) \defas \frac{\gamma_k^2}{2d^2} \ip{ (\Dif^2 \varphi)(X_k), (\delta X_k)^{\otimes 2}}. 
\end{equation}
Lastly, for the cross term in \eqref{eq:Hessian_taylor_term}, 
\[
\frac{\gamma_k^2}{d^2} \ip{\nabla^2 \varphi (X_k), a_{k+1} \otimes \nabla_x f(r_k) \otimes \delta X_k}. 
\]
As we saw in \eqref{eq:mean_a}, the conditional expectation is
\begin{equation}
\begin{aligned}
\EE[ &\ip{\nabla^2 \varphi(X_k), a_{k+1} \otimes \nabla_x f(r_k) \otimes \delta X_k} \, | \, \mathcal{F}_k ] \\
&
= 
\EE [ \ip{\nabla^2 \varphi(X_k), \sqrt{K} \Pi_k v_k \otimes \nabla_x f(r_k) \otimes \delta X_k} \, | \, \mathcal{F}_k],
\end{aligned}
\end{equation}
with $v_k \sim N(0,I_d)$. Also due to the $\tfrac{1}{d^2}$, this term will be of lower order and disappear as $d \to \infty$ (see Section~\ref{sec:lower_term_Hessian}), and thus, we define
\begin{equation}
    \mathcal{E}_{k,3}^{\text{Hess}}(\varphi) \defas \frac{\gamma_k^2}{d^2} \ip{\nabla^2 \varphi(X_k), \sqrt{K} \Pi_k v_k \otimes \nabla_x f(r_k) \otimes \delta X_k}. 
\end{equation}

Putting this all back together, we get the following for the Hessian term in \eqref{eq:taylor_theorem}
\begin{equation} \label{eq:Hessian_term}
\begin{aligned}
    \frac{\gamma_k^2}{2d^2} \ip{ \nabla^2 \varphi(X_k),  \Delta_k^{\otimes 2} } 
    &
    =
    \frac{\gamma_k^2}{2d^2} \ip{\nabla^2 \varphi(X_k), K \otimes \EE[ \nabla_x f(r_k)^{\otimes 2} \, | \, \mathcal{F}_k ]}
    + 
    \Delta \mathcal{M}_{k}^{\text{Hess}}(\varphi)
    + \EE[ \mathcal{E}_k^{\text{Hess}}(\varphi) \, | \, \mathcal{F}_k] 
    \\
    \text{where} \quad \Delta \mathcal{M}_k^{\text{Hess}}(\varphi)
    &
    = 
    \frac{\gamma_k^2}{2d^2} \bigg ( \ip{\nabla^2 \varphi(X_k), \Delta_k^{\otimes 2}} - \EE[ \ip{\nabla^2 \varphi(X_k), \Delta_k^{\otimes 2}} \, | \, \mathcal{F}_k] \bigg ),
    \\
    \text{and} \quad \mathcal{E}_k^{\text{Hess}}(\varphi)
    &
    = 
    \mathcal{E}_{k,1}^{\text{Hess}}(\varphi) + \mathcal{E}_{k,2}^{\text{Hess}}(\varphi) + \mathcal{E}_{k,3}^{\text{Hess}}(\varphi)
    \\
    &
    = -\frac{\gamma_k^2}{2d^2} \ip{\nabla^2 \varphi (X_k), \sqrt{K} \Pi_k \sqrt{K} \otimes \nabla_x f(r_k)^{\otimes 2}} 
    \\
    &
    + \frac{\gamma_k^2}{2d^2}
    \ip{ \nabla^2 \varphi(X_k), \big ( \sqrt{K} \Pi_k v_k \big )^{\otimes 2} \otimes \nabla_x f(r_k)^{\otimes 2}}
    \\
    &
    + \frac{\gamma_k^2}{2d^2} \ip{\nabla^2 \varphi(X_k), \delta X_k)^{\otimes 2}}
    \\
    &
    + \frac{\gamma_k^2}{d^2} \ip{ \nabla^2 \varphi(X_k), \sqrt{K} \Pi_k v_k \otimes \nabla_x f(r_k) \otimes \delta X_k }.
\end{aligned}
\end{equation}

% \subsubsection{Doob decomposition of SGD} 
We have successfully identified the martingale increments of a \textit{single} update of SGD, that is, by \eqref{eq:Hessian_term} and \eqref{eq:grad_term} in the Taylor expansion \eqref{eq:taylor_theorem},  
\begin{equation} \label{eq: taylor_expansion_result}
    \begin{aligned}
        \varphi(X_{k+1}) 
        & 
        =
        \varphi(X_k) - \frac{\gamma_k}{d} \ip{\nabla \varphi(X_k), \nabla \mathcal{R}_\delta(X_k)}
        + 
        \frac{\gamma_k^2}{2d^2} \ip{\nabla^2 \varphi(X_k), K \otimes \EE[ \nabla_x f(r_k)^{\otimes 2} \, | \, \mathcal{F}_k ]}
        \\ 
        & + \Delta \mathcal{M}_k^{\text{Grad}}(\varphi) + \Delta \mathcal{M}_k^{\text{Hess}}(\varphi) + \EE[ \mathcal{E}_k^{\text{Hess}}(\varphi) \, |\, \mathcal{F}_k] 
        % + \mathcal{E}_k^{\text{Higher}},
    \end{aligned}
\end{equation}
where the error terms look like
\begin{equation} \label{eq:error_terms}
    \begin{aligned}
    \Delta \mathcal{M}_k^{\text{grad}}(\varphi) 
    &
    = \frac{\gamma_k}{d} \ip{ \nabla \varphi (X_k), a_{k+1} \otimes \nabla_x f(r_k) }
    - \frac{\gamma_k}{d} \EE \big [ \ip{ \nabla \varphi(X_k), a_{k+1} \otimes \nabla_x f(r_k) } \, | \, \mathcal{F}_k \big ]
    \\
    \Delta \mathcal{M}_k^{\text{Hess}} (\varphi)
    &
    = 
    \frac{\gamma_k^2}{2d^2} \bigg ( \ip{\nabla^2 \varphi(X_k), \Delta_k^{\otimes 2} } - \EE[ \ip{\nabla^2 \varphi(X_k), \Delta_k^{\otimes 2}} \, | \, \mathcal{F}_k] \bigg ) 
    \\
        \mathcal{E}_k^{\text{Hess}} (\varphi)
        & = -\frac{\gamma_k^2}{2d^2} \ip{\nabla^2 \varphi(X_k), \sqrt{K} \Pi_k \sqrt{K} \otimes \nabla_x f(r_k)^{\otimes 2}} 
    \\
    &
    + \frac{\gamma_k^2}{2d^2}
    \ip{ \nabla^2 \varphi (X_k), \big ( \sqrt{K} \Pi_k v_k \big )^{\otimes 2} \otimes \nabla_x f(r_k)^{\otimes 2}}
    \\
    &
    + \frac{\gamma_k^2}{2d^2} \ip{\nabla^2 \varphi(X_k), (\delta X_k)^{\otimes 2}} 
    \\
    &
    + \frac{\gamma_k^2}{d^2} \ip{ \nabla^2 \varphi (X_k), \sqrt{K} \Pi_k v_k \otimes \nabla_x f(r_k) \otimes \delta X_k }
    % \\
    % \mathcal{E}_k^{\text{Higher}} (\varphi)
    % & \defas
    % - 
    % \frac{\gamma_k^3}{2 d^3} \int_0^1 (1-s)^2 \ip{(\nabla^{(3)} \varphi \big (X_k - s\frac{\gamma}{d} \Delta_k^{\otimes 3} \big ), \Delta_k^{\otimes 3}} \, \dif s.
    \end{aligned}
\end{equation}
Here $W_k = X_k \oplus X^{\star} \in \mathcal{A} \otimes \mathcal{O}^+$, $v_k \sim N(0, I_d)$, $\Pi_k = Q_k Q_k^T$, $r_k = \ip{a_{k+1}, W_k}_{\mathcal{A}}$, $\Delta_k = a_{k+1} \otimes \nabla f(r_k) + \delta X_k$, and $K = \EE[a \otimes a]$. 

Indeed, we now utilize our continuous time to sum up (integrate). For this, we introduce the forward difference
\[
(\Delta \varphi)(X_{j}) \defas \varphi(X_{j+1}) - \varphi(X_{j}),
\]
and thus, 
\begin{equation}
\varphi(X_{td}) = \varphi(X_0) + \sum_{j=0}^{\lfloor td \rfloor - 1} (\Delta \varphi)(X_j). 
\end{equation}
Therefore, we have
\[
\varphi(X_{td}) = \varphi(X_0) + \sum_{j=0}^{\lfloor td \rfloor-1} (\Delta \varphi)(X_j)  \defas \varphi(X_0) + \int_0^{t} d \cdot (\Delta \varphi)(X_{sd}) \, \dif s + \xi_{td},
\]
where $ \displaystyle |\xi_{td}|= \bigg | \int_{(\lfloor td\rfloor-1)/d}^t d \cdot \Delta \varphi(X_{sd})  \,\dif s \bigg | \le \max_{0 \le j \le \lceil td \rceil} \{ | \Delta \varphi(X_j) | \}.$ Note 
an analogous definition for the martingale (and its increment) hold
\[
 \mathcal{M}_{td} = \sum_{j=0}^{\lfloor td \rfloor-1} \Delta \mathcal{M}_j.
 % =  d \int_0^{t} \Delta \mathcal{M}_{sd} \, \dif s. 
\]
With this, we have our Doob decomposition for SGD
\begin{align}
\varphi(X_t) 
&
= \varphi(X_0) - \int_0^{t} \gamma(s) \ip{ \nabla \varphi(X_{sd}), \nabla \mathcal{R}_{\delta}(X_{sd})} \, \dif s \label{eq:integral_1} \\
&
        + 
        \frac{1}{2d} \int_0^{t} \gamma(s)^2 \ip{\nabla^2 \varphi(X_{sd}), K \otimes \EE[ \nabla_x f(r_{sd})^{\otimes 2} \, | \, \mathcal{F}_{sd} ]} \, \dif s \label{eq:integral_2}
        \\
        & +  \sum_{j=0}^{\lfloor td \rfloor-1} \Delta \mathcal{M}_{j}^{\text{Grad}}(\varphi) + \Delta \mathcal{M}_{j}^{\text{Hess}}(\varphi) + \EE [\mathcal{E}_{j}^{\text{Hess}}(\varphi) \, | \, \mathcal{F}_{j} ]
        % + \mathcal{E}_{sd}^{\text{Higher}}(\varphi) 
        + \xi_{td}(\varphi). \label{eq:error_terms_integrated}
\end{align}

In Section~\ref{sec:error_bounds}, we prove that the term \eqref{eq:error_terms_integrated} is negligible as $d \to \infty$. The other two terms \eqref{eq:integral_1} and \eqref{eq:integral_2} survive the limit. Next, we show that SGD on $S$ is an $(\varepsilon, M,T)$ approximated solution.
% and will be compared to homogenized SGD applied to $\varphi$ via It\^o Lemma in the section below. 
\subsubsection{$S(W_{td}, z)$ is an approximate solution, proof of Proposition \ref{prop:SGD_approx_solution}} 
\label{sec:SGD_approximate_solution}
The goal in this section is to prove Proposition~\ref{prop:homogenized_SGD_approx_solution}, that is, show that 
\[
S(W_{td},z) = \ip{(W_{td}^{\otimes 2}, R(z;K)}_{\mathcal{A}^{\otimes 2}}
\]
is an approximate solution to the integro-differential equation \eqref{eq:ODE_resolvent_2}. 
% We move now to the proof of Proposition~\ref{prop:SGD_approx_solution}.

\begin{proof}[Proof of Proposition~\ref{prop:SGD_approx_solution}]
Appling Eq. \eqref{eq:integral_1}, Eq. \eqref{eq:integral_2}, and Eq. \eqref{eq:error_terms_integrated} for each matrix element, following the same computation as in section \ref{sec:HSGD_approximate_solution} replacing $\HSGD_t$ with $W_{td}$, and $\rho_t$ with $r_{td}$, 
\begin{align} 
S(W_{td},z)
&
=S(W_{0},z) + \int_0^{t} \mathscr{F}(z, S(W_{sd}, z)) \, \dif s 
\\
& +  \sum_{j=0}^{\lfloor td \rfloor-1} \Delta \mathcal{M}_{j}^{\text{Grad}}(S) + \Delta \mathcal{M}_{j}^{\text{Hess}}(S) + \EE [\mathcal{E}_{j}^{\text{Hess}}(S) \, | \, \mathcal{F}_{j} ] 
% + \mathcal{E}_{sd}^{\text{Higher}}(\varphi) 
+ \xi_{td}(S). 
\end{align}
Thus to show that $S(W_{td}, \cdot)$ is an approximate solution of the integro-differential equation \eqref{eq:ODE_resolvent_2} it amounts to bounding the martingales and error terms where $C$ is a positive constant independent of $d$. Let $\Gamma = \{ z \, : \, |z| = \max\{1, 2\|K\|_{\sigma}\}\}$.  
We thus have that for all $z \in \Gamma$,
\begin{equation}
\begin{aligned}
&\sup_{0 \le t \le T \wedge \hat{\tau}_{M} } \| S(W_{td}, z) - S(W_0,z) - \int_0^t \mathscr{F}(z, S(W_{sd}, z )) \, \dif s \| \\ &\le 
\sup_{0 \le t \le T \wedge \hat{\tau}_{M} } \|\mathcal{M}_{td}^{\text{Grad}}(S(\cdot, z)) \|+\sup_{0 \le t \le T \wedge \hat{\tau}_{M} } \|\mathcal{M}_{td}^{\text{Hess}}(S(\cdot, z)) \|
\\&+\sup_{0 \le t \le T \wedge \hat{\tau}_{M} } \|\sum_{j=0}^{\lfloor td \rfloor-1}\EE [\mathcal{E}_{j}^{\text{Hess}}(S) \, | \, \mathcal{F}_{j} ]  \|+\sup_{0 \le t \le T \wedge \hat{\tau}_{M} } \|\xi_{td}(S) \|.
\end{aligned}
\end{equation}
Next, fix a constant $\delta > 0$. Let $\Gamma_{\delta} \subset \Gamma$ such that there exists a $\bar{z} \in \Gamma_{\delta}$ such that $|z-\bar{z}| \le d^{-\delta}$ and the cardinality of $\Gamma_{\delta}$, $|\Gamma_{\delta}| = C d^{\delta}$ where $C > 0$ depending on $\|K\|_{\sigma}$. 
For all $z \in \Gamma$, we note that for some constants $C,c > 0$ such that $\vartheta_{c \cdot M} \le \hat{\tau}_{M} \le \vartheta_{C \cdot M}$ (see Lemma~\ref{lem:normequivalence}). Consequently, we evaluate the error with the stopped process $W_{td}^{\vartheta} = W_{td \wedge \vartheta}$ instead of using $\hat{\tau}_M$.
By the martingale errors proposition,  Proposition~\ref{prop:gradmartingale}, and  Proposition~\ref{prop:hessmartingale} which we have deferred the proof to Section~\ref{subsec:SGD_martingales}, we have that for any $\hat{\delta} > 0$ 
\begin{equation}
\sup_{z \in \Gamma_\delta} \sup_{0 \leq t\leq T} \|\cM_{d(t \wedge \vartheta_{CM})}^{\grad}(S(\cdot, z))\|<d^{-\frac12+\hat{\delta}}\quad \text{w.o.p.},
\end{equation}
and,
\begin{equation}
\sup_{z \in \Gamma_\delta} \sup_{0 \leq t\leq T}\| \cM_{(t \wedge \vartheta_{CM})d}^{\hess} (S(\cdot, z))\|<d^{-1+\hat{\delta}} \quad \text{w.o.p.}
\end{equation}
In addition, for the Hessian error by proposition~\ref{prop:Hessian-error} which we have deferred the proof to Section~\ref{sec:lower_term_Hessian} together with Jensen's inequality,
\begin{equation}
   \sup_{z \in \Gamma_\delta} \sup_{0 \le t \le T} \sum_{j=0}^{\lfloor( t \wedge \vartheta_{CM}) d\rfloor -1} \| \EE[ \mathcal{E}_j^{\text{\rm Hess}}(S(\cdot, z)) \, | \, \mathcal{F}_{j}] \| \le C (L(f))^2  d^{-1+\hat{\delta}},\quad \text{w.o.p.}
\end{equation}
% As the cardinality of $\Gamma_{\delta}$ is polynomial in $d$, we have that 
% \[
% \sup_{z \in \Gamma_{\delta}} \sup_{0 \le t \le T} \| \mathcal{M}_{t \wedge \vartheta_{C \cdot M}}^{\text{HSGD}}(S(\cdot, z))\| \le C \cdot L(f) \cdot d^{\hat{\delta}/2 - 1/2}, \quad \text{w.o.p.}
% \]
Last, 
\begin{equation}
\sup_{0 \le t \le T \wedge \hat{\tau}_{M} }   \|\xi_{td}(S) \| \le \sup_{0 \le t \le T\wedge \hat{\tau}_{M}  } \, \|\Delta S_{td} \| = \frac{1}{d}\sup_{0 \le t \le T\wedge \hat{\tau}_{M}  } \, \|\mathscr{F}(z, S(W_{td}, z)\| 
\end{equation}
where 
\begin{equation}
\begin{aligned}
\|\mathscr{F}(z, S(W_{td}, \cdot))\|& \le     
     \bar{\gamma}C(\|K\|_\sigma) \| H( {B}_{td})\| +  \frac{\bar{\gamma}^2}{d}
    \tr(K R(z;K)) \|I({B}_{td})\| 
    \\
    & +\bar{\gamma}\delta|\mathcal{O}|\|{S}(W_{td}, z)\|
    + \bar{\gamma} \|{S}(W_{td}, z)z\|\|H({B}_{td})\|
\end{aligned}   
\end{equation}
such that $B_{td} \defas  \frac{-1}{2\pi i} \oint_{\Gamma_{\delta }} z{S}(W_{td},z) \, \dif z $. Next, using Assumptions \ref{assumption:unbiased} and Assumption \ref{assumption:E_loss_pseudo_Lip} and plugging Eq. (\ref{eq:H_bound}), Eq. \eqref{eq:I_bound}, and $
\|S(W_{td},z)\| 
% \le \|W_{td}\|^2 \|R(z; K)\|_{\sigma} 
\le C(\|K\|_{\sigma}) \cdot M  
$, there is a positive constant positive $C = C(\|K\|_\sigma, \bar{\gamma}, |\mathcal{O}|,M, L(h),L(I))$, such that 
$\|\mathscr{F}(z, {S}(W_{td}, \cdot))\|\le C.$ Therefore, 
\begin{equation}
\sup_{0 \le t \le T \wedge \hat{\tau}_{M} }   \|\xi_{td}(S) \| \le Cd^{-1}.
\end{equation}
Consequently, combining all the errors, we deduce that for some $C>0$, which does not depend on $d$, or $n$
\begin{align*}
\sup_{0 \le t \le T \wedge \hat{\tau}_{M} } \| S(W_{td}, z) - S(W_0,z) - \int_0^t \mathscr{F}(z, S(W_{sd}, z )) \, \dif s \|_{\Gamma_{\delta}} 
% &
% \le
% \sup_{0 \le t \le T} \| \mathcal{M}_{t \wedge \vartheta}^{\text{HSGD}}(S(\cdot, z))\|_{\Gamma_{\delta}}
\le
C d^{\hat{\delta}/2 - 1/2} \quad \text{w.o.p}.
\end{align*}
An application of the net argument, Lemma~\ref{lem:net_argument}, finishes the proof after setting $\hat{\delta} = 1-2\delta$ for $\delta\in (0, 1/2)$. 
\end{proof}

\subsection{Error bounds} \label{sec:error_bounds}
Recall, letting $W = X \oplus X^{\star}$, we are interested in the statistic 
\[ S(W,z) = \ip{W \otimes W, R(z;K)}_{\mathcal{A}^{\otimes 2}},\]
where $R(z;K) = (K-zI_{\mathcal{A}})^{-1}$ and throughout this section, the contour
\[
\Gamma = \{ z \, : \, |z| = \max\{1, 2 \|K\|_{\sigma} \} \}.
\]
This section is devoted to controlling the error terms that arise when comparing SGD and homogenized SGD under $S$ with $\mathscr{F}$ from the integro-differential equation \eqref{eq:ODE_resolvent_2}.

Before proceeding, we present some bounds on the derivatives of $S$. 

\begin{lemma} \label{lem:S_derivative_bounds}
There exists constants $c, C = C(|\mathcal{O}^+|) > 0$ such that
\[
c \|W\|^2 \le \|S(W,z)\|_{\Gamma} \le C \|W\|^2, \quad \|\nabla_X S(W,z)\|_{\Gamma} \le C \|W\|, \quad 
\text{and} \quad \|\nabla^2_X S(W,z)\|_{\Gamma} \le C. 
\]
Moreover, 
\[
\ip{W^{\otimes 2}, K}_{\mathcal{A}^{\otimes 2}} = \frac{-1}{2 \pi i} \oint_{\Gamma} z S(W,z) \, \dif z.
\]
\end{lemma}

\begin{proof} First, by Neumann series, $(K-zI_{\mathcal{A}})^{-1} = -1/z (I_{\mathcal{A}} - 1/z K)^{-1} = -\tfrac{1}{z} \sum_{j=0}^\infty ( \tfrac{1}{z} K)^j$. Using $|z| = \max \{ 1, 2 \|K\|_{\sigma}\}$, we immediately get $ \displaystyle \sup_{z \in \Gamma} \|R(\cdot;K)\|_{\sigma} \le 2$. The upper bound for the first term immediately follows from $\displaystyle \|S(W,z)\|_{\Gamma} \le \|W\|^2 \sup_{z \in \Gamma} \|R(\cdot;K)\|_{\sigma}$. 

On the other hand, we have that for $\Gamma = \{z \, : \, |z| = \max\{1, 2\|K\|_{\sigma} \} \}$, we can express
\[
\|\ip{W^{\otimes 2}, I_{\mathcal{A}}}\| = \| \frac{-1}{2 \pi i} \oint_{\Gamma} S(W,z) \, \dif z  \|^2 \le c \|S(W,z)\|_{\Gamma}, \quad \text{for some constant $c > 0$.}
\]
This proves the first result. 

For the derivative, a simple computation shows that 
\[
\nabla_X S(W,z) \cong ( \text{Id}_{\mathcal{O}} \oplus 0_{\mathcal{T}} ) \otimes \ip{W, R(z;K)}_{\mathcal{A}^{\otimes 2}} + \ip{W, R(z;K)}_{\mathcal{A}^{\otimes 2}} \otimes ( \text{Id}_{\mathcal{O}} \oplus 0_{\mathcal{T}} ).
\]
Taking norms and using that $\displaystyle \sup_{z \in \Gamma} \|R(\cdot;K)\|_{\sigma} \le 2$, the second result follows. 

Finally, for the Hessian, we have
\[
\nabla^2_X S(W,z) \cong ( \text{Id}_{\mathcal{O}} \oplus 0_{\mathcal{T}} ) \otimes \ip{( \text{Id}_{\mathcal{O}} \oplus 0_{\mathcal{T}} ), R(z;K)}_{\mathcal{A}^{\otimes 2}} + \ip{( \text{Id}_{\mathcal{O}} \oplus 0_{\mathcal{T}} ), R(z;K)}_{\mathcal{A}^{\otimes 2}} \otimes ( \text{Id}_{\mathcal{O}} \oplus 0_{\mathcal{T}} ).
\]
It immediately follows the bound on the Hessian. 
% We introduce the projection operation $\text{Proj}_i \, : \, (\mathcal{O}^+)^{\otimes 2} \to \mathbb{R}$ which projects an element of $(\mathcal{O}^+)^{\otimes 2}$ onto the $i$-th coordinate. In coordinates, $\text{Proj}_i$ can be represented as multiplication by a tensor. Thus the $\text{Proj}_i$ is linear and $C^2$-smooth with 
% \[ 
% \| \text{Proj}_i(S(W,z)) \| = 1, \quad \|\nabla \text{Proj}_i(S(W,z)) \| = 1, \quad \text{and} \quad \nabla^2 \text{Proj}_i(S(W,z)) \equiv 0.
% \]

The last statement follows from Cauchy's integral formula which relates the resolvent, $R(z;K)$, with analytic functions of $f(K)$. In particular, we use the identity that
\[
K = \frac{-1}{2 \pi i } \oint_{\Gamma} z R(z;K) \, \dif z. 
\]
\end{proof}

To control the errors, we will need to make an \textit{a priori} estimate that effectively shows that the iterates of homogenized SGD and SGD remain bounded. Thus, recall, our definition, for fixed $M > 0$, the stopping times
\begin{equation}
\begin{aligned} \label{eq:stopping_time_1}
    \vartheta_{M} 
    &
    \defas \inf \{ t \ge 0 \, : \, \|W_{td} \|^2 > M \, \, \text{or} \, \, \ip{W_{td}^{\otimes 2}, K}_{\mathcal{A}^{\otimes 2}} \not \in \mathcal{U}\}\\
    \text{or} \quad  \vartheta_{M} 
    &
    \defas \inf \{t \ge 0 \, : \, 
    \|\HSGD_t\|^2  > M \, \, \text{or} \, \, \ip{\HSGD_t^{\otimes 2}, K}_{\mathcal{A}^{\otimes 2}} \not \in \mathcal{U} \},     
\end{aligned}
\end{equation}
% where $C(T)>0$ is some constant which depends on the last iterate of homogenized SGD, $T={n}/{d}\in [0,\infty)$. 
% \textcolor{red}{CP: Not sure what the exact stopping criteria should be.}  
depending on whether we are working with SGD iterates or homogenized SGD iterates. We often drop the $M$ so that $\vartheta \defas \vartheta_M$. It will be convenient to work with the stopped processes, $W_{td }^{\vartheta} \defas W_{ t \wedge \vartheta d }$ and $\HSGD_t^{\vartheta} \defas \HSGD_{t \wedge \vartheta}$. 
% It will be convenient along the proof to introduce a stopping time, 
% \begin{equation}
% \begin{gathered}
% \vartheta \defas \inf \{ t \ge 0 \, : \, \|W_{\lfloor td \rfloor}\| > C(T)  \, \, \text{or} \, \, 
%     \|\HSGD_t\|  > C(T) \}. 
% \end{gathered}
% \end{equation}
% We first start by comparing the stopped processes, $\WHSGD_t^\vartheta=\WHSGD_{t\wedge \vartheta}$, and $X_k^\vartheta=X_{k\wedge \vartheta}$ under the statistics. 
% Later we show that if the norm of HSGD is bounded then the norm of SGD is also bounded with overwhelming probability, and we can remove the stopping time. 

\begin{remark} The stopping time $\hat{\tau}_M = \inf \{ t \ge 0 \, : \, \|S(\HSGD_t, z)\|_{\Gamma} >  M \, \, \text{or} \, \, \frac{-1}{2 \pi i} \oint_{\Gamma} z S(\HSGD_t, z) \, \dif z \not \in \mathcal{U} \} \quad \text{and} \quad \quad \hat{\tau}_M = \inf \{ t \ge 0 \, : \, \|S(W_{td}, z)\|_{\Gamma} > M, \, \, \text{or} \, \, \frac{-1}{2 \pi i} \oint_{\Gamma} z S(W_{td}, z) \, \dif z \not \in \mathcal{U} \}$ are related to $\vartheta_M$ by positive constants $c, C > 0$, $\vartheta_{c \cdot M} \le \hat{\tau}_{M} \le  \vartheta_{C \cdot M}$ (see Lemma~\ref{lem:S_derivative_bounds}).  
\end{remark}

In the remainder of this section, we prove a series of propositions, bounding the martingale terms that arise from homogenized SGD and SGD respectively.  Throughout these proofs, we use $C$ to denote a constant that may depend on various bounded quantities, namely $\gamma$, $T$, $\delta$, $|\mathcal{O}^+|$, $\alpha$, $\|K\|_\sigma$, and $M$, but does not depend on $d$.  The value of $C$ may change throughout these proofs and is not necessarily the same as $C$ in Lemma~\ref{lem:S_derivative_bounds}.

\subsubsection{Homogenized SGD Martingale Error} \label{sec:HSGD_martingale}
In this section, we control the martingale that arises in homogenized SGD, that is, for a test function $\varphi \, : \, \mathcal{A} \otimes \mathcal{O} \to \mathbb{R}$,
\begin{equation}
\begin{aligned}
    \mathcal{M}^{\text{HSGD}}_{t}(\varphi)
    &
    \defas \int_0^t \dif \mathcal{M}_s^{\text{HSGD}}(\varphi)  
    = 
    \frac{1}{\sqrt{d}} \int_0^t \gamma(s) \cdot \ip{ \ip{\sqrt{K} \otimes (\EE_{a, \epsilon} [ \nabla_x f(\rho_s)^{\otimes 2} ])^{1/2}, \nabla \varphi(\WHSGD_s)}_{\mathcal{A} \otimes \mathcal{O}}, \dif B_{s}}. 
\end{aligned}
\end{equation}
As introduced in Remark~\ref{rmk:martingal_S}, we are interested in controlling $\mathcal{M}_t^{\text{HSGD}}(S(\cdot, z))$.

To control the fluctuations of this martingale, we need to control its quadratic variation, defined as follows. Consider a partition of time for $[0,t]$, that is, $0 = t_0 < t_1 < \hdots < t_n = t$ such that the size of the partition $ \displaystyle \Delta t = \max_{i} \{ t_i - t_{i-1}\} \to 0$. We define for the continuous process $Y$,
\[ [Y_t(n)] = \sum_{k=1}^n (Y_{t_k} - Y_{t_{k-1}})^2.
\]
If, for every partition of time $[0,t]$ such that $\Delta t \to 0$, the process $[Y_t(n)]$ converges in probability to a process $[Y_t]$ as $n \to \infty$, we call $[Y_t]$ the \textit{quadratic variation} of $Y$ (see \cite[Chapter 1]{protter2005stochastic} for details). % The \textit{quadratic variation} $[Y_t(n)]$ is the limit of the sum of squares of all jumps of the process as the size of the partition $ \displaystyle \Delta t = \max_{i} \{ t_i - t_{i-1}\} \to 0$, that is, 
% \[
% [Y_t] \defas \lim_{\Delta t \to 0} \sum_{k=1}^n (\Delta Y_{t_k})^2. 
% \]
Using the quadratic variation of $\mathcal{M}_t^{\text{HSGD}}$, we will show that the martingale arising from homogenized SGD is small. 

\begin{proposition}[Homogenized SGD martingale small.] \label{prop:HSGD_Martingale_Bound} Suppose $f \, : \, \mathcal{O} \oplus \mathcal{T} \oplus \mathcal{T} \to \mathbb{R}$ is $\alpha$-pseudo-Lipschitz function with constant $L(f)$ (see Assumption~\ref{ass:pseudo_lipschitz}). Let the statistic $S \, : \, \mathcal{A} \otimes \mathcal{O} \to (\mathcal{O}^+)^{\otimes 2}$ be defined as in \eqref{eq:S_statistic}. For any $T > 0$,  $\zeta > 0$ and fix $z \in \Gamma$, there is some constant $C$ such that, with overwhelming probability, 
\begin{equation}
    \sup_{0 \le t \le T} \| \mathcal{M}^{\text{HSGD}}_{t \wedge \theta} (S(\cdot,z)) \| \le C L(f) \, \, d^{\zeta / 2 - 1/2}.
%    \sup_{0 \le t \le T} | \mathcal{M}_{t \wedge \vartheta}^{\text{HSGD}}(\varphi)| \le C L(f) L(\varphi) \, \, d^{\varepsilon \max \{1, 2 \alpha + 2 \alpha'\} - 1/2}
\end{equation}
%where $C$ is a constant depending on $\gamma, T, \|K\|_{\sigma}, \alpha, M, \text{ and }  |\mathcal{O}^+|$, which in particular, is independent of $d$. 
\end{proposition}

\begin{proof} Let $S_{ij} \defas S_{ij}(\cdot,z)$ be the $ij$-coordinate of $S$ for a fixed $z \in \Gamma$. First, we rewrite the martingale increment, $\dif \mathcal{M}_{t}^{\text{HSGD}}$, 
    \begin{equation}
        \dif \mathcal{M}_{t}^{\text{HSGD}}(S_{ij}) = \frac{\gamma_t}{\sqrt{d}} \ip{ \ip{\sqrt{K} \otimes (\EE_{a, \epsilon} [\nabla_x f(\rho_t)^{\otimes 2} ])^{1/2} , \nabla_X S_{ij} (\HSGD_t,z ) }_{\mathcal{A} \otimes \mathcal{O}}, \dif B_t }.
    \end{equation}
    The quadratic variation of $\mathcal{M}_t^{\text{HSGD}}$ is 
    \begin{equation} \label{eq:quad_variation_HSGD}
       [ \mathcal{M}_t^{\text{HSGD}}(S_{ij}) ] = \frac{1}{d} \int_0^t \gamma^2_s \| \ip{\sqrt{K} \otimes (\EE_{a, \epsilon} [\nabla_x f(\rho_s)^{\otimes 2}])^{1/2}, \nabla_X S_{ij}(\HSGD_s, z)}_{\mathcal{A} \otimes \mathcal{O}} \|^2 \, \dif s. 
    \end{equation}
    We need to compute $\displaystyle \sup_{0 \le t \le T} [\mathcal{M}_{t \wedge \vartheta}^{\text{HSGD}}(S_{ij})]$ and show that this quantity is small. In particular, we only need to show that the norm $\| \cdot \|^2$ inside the integral is small.  For this, we see that 
    \begin{equation}
        \begin{aligned}
            & 
            \| \ip{\sqrt{K} \otimes (\EE_{a, \epsilon} [\nabla_x f(\rho_s^{\vartheta})^{\otimes 2}])^{1/2}, \nabla_X S_{ij}(\HSGD_s^{\vartheta}, z)}_{\mathcal{A} \otimes \mathcal{O}} \|^2
            \\
            & \qquad \qquad \qquad \qquad 
            = 
            \ip{K \otimes \EE_{a, \epsilon} [\nabla_x f(\rho_s)^{\otimes 2}], \big (\nabla_X S_{ij} (\HSGD_s^{\vartheta},z) \big )^{\otimes 2} }
            \\
            & \qquad \qquad \qquad \qquad 
            = 
            \ip{K, \ip{\EE_{a, \epsilon} [\nabla_x f(\rho_s^{\vartheta})^{\otimes 2}], \big ( \nabla_X S_{ij} (\HSGD_s^{\vartheta},z) \big )^{\otimes 2} }_{\mathcal{O}^{\otimes 2}} } 
            \\
            & \qquad \qquad \qquad \qquad 
            \le 
            \|K\|_{\sigma} \| \|\ip{ \EE_{a, \epsilon}[\nabla_x f(\rho_s^{\vartheta})^{\otimes 2}], \big ( \nabla_X S_{ij}(\HSGD_s^{\vartheta}, z ) \big )^{\otimes 2}}\|
            \\
            & \qquad \qquad \qquad \qquad 
            \le 
            \|K\|_{\sigma} \EE_{a, \epsilon}[ \|\nabla_x f(\rho_s^{\vartheta}) \|^2 ] \|\nabla_X S_{ij} (\HSGD_s^{\vartheta}, z)\|^2.
        \end{aligned}
    \end{equation} 
    By Lemma~\ref{lem:S_derivative_bounds}, we have a bound on $\|\nabla_X S_{ij}(W,z)\| \le \|\nabla_X S(W,\cdot)\|_{\Gamma} \le C\|W\|$. From Lemma~\ref{lem:growth_grad_f}, the growth condition on $\EE_{a, \epsilon}[\|\nabla_x f(\rho)\|^2]$ yields
    \begin{equation} \label{eq:quad_variation_HSGD_1}
        \begin{aligned}
            \| &\ip{\sqrt{K} \otimes (\EE_{a,\epsilon}[\nabla_x f(\rho_s^{\vartheta})^{\otimes 2}])^{1/2}, \nabla_X S_{ij} (\HSGD_s^{\vartheta}, z)}_{\mathcal{A} \otimes \mathcal{O}} \|^2
            \le
             \|K\|_{\sigma} \EE_{a,\epsilon}[ \|\nabla f(\rho_s^{\vartheta}) \|^2 ] \|\nabla_X S_{ij} (\HSGD_s^{\vartheta},z)\|^2
            \\
            & \qquad \qquad
            \le
            C \cdot (L(f))^2 \|\HSGD_t^{\vartheta}\|^2 (1 + \|K\|_{\sigma}^{1/2} \|\HSGD_t^{\vartheta}\| )^{\max\{1, 2\alpha\}}
            \\
            &\qquad \qquad
            \le
            C \cdot (L(f))^2 M (1 + \sqrt{M} )^{\max\{1, 2 \alpha\}}.
        \end{aligned}
    \end{equation}
   % where the positive constant $C = C(\alpha, \|K\|_{\sigma}, |\mathcal{O}^+|)$. 
    Thus, \eqref{eq:quad_variation_HSGD} and \eqref{eq:quad_variation_HSGD_1}, together
    \begin{equation}
        \sup_{0 \le t \le T} [\mathcal{M}_{t \wedge \vartheta}^{\text{HSGD}}(S_{ij})] \le C (L(f))^2 \cdot \bar{\gamma}^2 \cdot d^{- 1}.
    \end{equation}
   % where the constant $C \defas C(\alpha, \|K\|_{\sigma}, T, |\mathcal{O}^+|, M)$. 
    Using the fact, if $\displaystyle \sup_{0 \le t \le T}  [ \mathcal{M}_{t \wedge \vartheta}^{\text{HSGD}} (S_{ij})] \le b$ a.s, then $\Pr ( \displaystyle \sup_{0 \le t \le T} | \mathcal{M}_{t \wedge \vartheta}^{\text{HSGD}}(S_{ij}) | > p ) \le \exp(-p^2 / 2b)$. By letting $p \defas \sqrt{C} L(f) d^{\zeta/2 - 1/2}$ for any $\zeta > 0$,
    \[
    \Pr( \sup_{0 \le t \le T} | \mathcal{M}_{t \wedge \vartheta}^{\text{HSGD}}(S_{ij}) | > p ) \le C \exp( - d^{\zeta} ).
    \]
    The result immediately follows after noting that the number of $ij$ coordinates is $|\mathcal{O}^+|^2$ which is independent of $d$. 
\end{proof}

\subsubsection{Bounds on the martingales $\cM_k^{\grad}$ and $\cM_k^{\hess}$ \label{subsec:SGD_martingales}}

In this section, we work with martingale increments coming from SGD applied to test functions $\varphi$. Recall, the expressions for the martingale increments for any quadratic statistics $\varphi$
\begin{align*}
\Delta \mathcal{M}_k^{\grad}(\varphi) 
   & = \frac{\gamma}{d} \ip{ \nabla \varphi(X_k), a_{k+1} \otimes \nabla_x f(r_k,\epsilon_{k+1}) }
    - \frac{\gamma}{d} \EE \big [ \ip{ \nabla \varphi(X_k), a_{k+1} \otimes \nabla_x f(r_k,\epsilon_{k+1}) }
     \, | \, \mathcal{F}_k \big ]
     \\
     \Delta \mathcal{M}_k^{\text{Hess}} (\varphi)
    &
    = 
    \frac{\gamma^2}{2d^2} \bigg ( \ip{\nabla^2 \varphi(X_k), \Delta_k^{\otimes 2} } - \EE[ \ip{\nabla^2 \varphi(X_k), \Delta_k^{\otimes 2}} \, | \, \mathcal{F}_k] \bigg )
\end{align*}
with
\[
\mathcal{M}_k(\varphi) = \sum_{j=1}^{k-1} \Delta \mathcal{M}_j(\varphi).
\]

\begin{proposition}[Gradient martingale]\label{prop:gradmartingale}
Suppose $f \, : \, \mathcal{O} \oplus \mathcal{T} \oplus \mathcal{T} \to \mathbb{R}$ is $\alpha$-pseudo-Lipschitz function with constant $L(f)$ (see Assumption~\ref{ass:pseudo_lipschitz}). Let the statistic $S \, : \, \mathcal{A} \otimes \mathcal{O} \to (\mathcal{O}^+)^{\otimes 2}$ be defined as in \eqref{eq:S_statistic}. Then, for any $\zeta>0$ and $T > 0$, and with overwhelming probability,
\beq
\sup_{0 \leq t\leq T} \|\cM_{d(t \wedge \vartheta)}^{\grad}(S(\cdot, z))\|<d^{-\frac12+\zeta}.
\eeq
\end{proposition}
\begin{proof}
Let $\varphi(X) \defas S_{ij}(W,z)$ be the $ij$-coordinate of $S$. Throughout the proof of this proposition, we will be working on the stopped version of the martingale, $\cM^{\grad}_{(t\wedge \vartheta)d}$.  However, to lighten the notation, we will suppress the $\vartheta$ dependence in the subscript as well as the $\varphi$ and simply write $\cM_{td}^{\grad} \defas \cM_{(t \wedge \vartheta) d}^{\grad}(\varphi)$.  We have the martingale increments
\beq\begin{split}
\Delta \mathcal{M}_{k}^{\grad}
   & = \frac{\gamma_k}{d} \ip{ \nabla \varphi(X_k), a_{k+1} \otimes \nabla_x f(r_k,\epsilon_{k+1}) }
    - \frac{\gamma_k}{d} \EE \big [ \ip{ \nabla \varphi(X_k), a_{k+1} \otimes \nabla_x f(r_k,\epsilon_{k+1}) }
     \, | \, \mathcal{F}_k \big ]\\
     &=\frac{\gamma_k}{d} \ip{\ip{ \nabla \varphi(X_k), a_{k+1}}_\cA, \nabla_x f(r_k,\epsilon_{k+1}) }
    - \frac{\gamma_k}{d} \EE \big [  \ip{\ip{ \nabla \varphi(X_k), a_{k+1}}_\cA, \nabla_x f(r_k,\epsilon_{k+1}) }
     \, | \, \mathcal{F}_k \big ]
\end{split}\eeq
We define $\cM_k^{\grad,\beta}$ to be a new martingale with increments
\beq\begin{split}
\Delta\cM_k^{\grad,\beta}
=&\frac{\gamma_k}{d} \ip{\textstyle\proj_\beta\ip{ \nabla \varphi(X_k), a_{k+1}}_\cA, \nabla_x f\circ\textstyle\proj_\beta(r_k,\epsilon_{k+1}) }\\
   & - \frac{\gamma_k}{d} \EE \big [  \ip{\textstyle\proj_\beta\ip{ \nabla \varphi(X_k), a_{k+1}}_\cA, \nabla_x f\circ\textstyle\proj_\beta(r_k,\epsilon_{k+1}) }
     \, | \, \mathcal{F}_k \big ],
\end{split}\eeq
 where we note that there are two projections and the projection of $(r_k,\epsilon_{k+1})$ is in all coordinates of $\mathcal{O} \oplus \mathcal{T} \oplus \mathcal{T}$, even though the gradient $\nabla_x f$ is only with respect to the $x$ coordinates (i.e. the coordinates in $\mathcal{O})$.  We take the projection radius to be $\beta = d^\zeta$ for some $\zeta>0$ to be determined later.
We will bound $\cM_k^{\grad,\beta}$ first, and then bound the difference between $\cM_k^{\grad}$ and $\cM_k^{\grad,\beta}$.  

We begin by computing subgaussian bounds on the quantities that are going to be projected, namely $(r_k,\epsilon_{k+1})$ and $\ip{ \nabla \varphi(X_k), a_{k+1}}_\cA$.  
For the purposes of this section, when we refer to a vector as ``subgaussian,'' we mean that its entries individually satisfy the stated subgaussian concentration bound.
We can rewrite the quantities $r_k$ and $\ip{ \nabla \varphi(X_k), a_{k+1}}_\cA$ as
\beq\begin{split}
&r_k=\ip{W_k,a_{k+1}}_\cA
=\ip{W_k,\sqrt{K}v_{k+1}}_\cA
=\ip{\ip{W_k,\sqrt{K}}_\cA,v_{k+1}}_\cA\\
&\ip{ \nabla \varphi(X_k),a_{k+1}}_\cA
=\ip{\nabla \varphi(X_k),\sqrt{K}v_{k+1}}_\cA
=\ip{\ip{\nabla \varphi(X_k),\sqrt{K}}_\cA,v_{k+1}}_\cA.
\end{split}\eeq
so $r_k$ is $\| W_k\| _\sigma\| \sqrt{K}\| _\sigma$-subgaussian and $\ip{\nabla \varphi (X_k),a_{k+1}}_\cA$ is $\| \nabla \varphi(X_k)\| _\sigma\| \sqrt{K}\| _\sigma$-subgaussian where
$
\| \nabla \varphi(X_k)\|_\sigma = \sup_{z \in \Gamma} \|S_{ij}(W_k,z)\|_{\sigma} \le \|S(W_k,z)\|_{\Gamma} \leq C \| W_k\|
$
 by Lemma~\ref{lem:S_derivative_bounds}.  Furthermore, $\epsilon_{k+1}$ is 1-subgaussian by assumption.  Thus, since we are working on the stopped processes,
 \beq\label{eq:Mgrad_subgaussianbounds}
 \|r_{k},\epsilon_{k+1}\|_{\psi_2}=C,\qquad
 \|\ip{\nabla \varphi(X_{k}),a_{k+1}}_\cA\|_{\psi_2}=C \eeq
   These subgaussian bounds will be used to bound the difference between $\cM_{k}^{\grad}$ and $\cM_{k}^{\grad,\beta}$.
 
  Furthermore,  from the projections and the growth bound on $\nabla_x f$ in Lemma \ref{lem:growth_grad_f}, we get the norm bounds 
  \begin{align}
  \| \nabla _xf\circ\textstyle\proj_\beta(r_{k},\epsilon_{k+1})\| &\leq L(f)C\beta ^{\max\{1,\alpha \}}, \\
  \| \textstyle\proj_\beta \ip{\nabla\varphi(X_{k}),a_{k+1}}_\cA\| &\leq \beta .
  \end{align}
This gives us the bound
\beq
|\ip{\textstyle\proj_\beta \ip{ (\nabla \varphi)(X_{k}), a_{k+1}}_\cA, \nabla_x f\circ\textstyle\proj_\beta (r_{k},\epsilon_{k+1}) }|
\leq 
L(f)C\beta ^{2+\alpha }
\eeq
and, since this is an almost sure bound, it holds for the expectation as well, and we get 
\beq
|\Delta\cM_{k}^{\grad,\beta }|\leq\frac{2\gamma}{d}L(f)C\beta ^{2+\alpha }.
\eeq
Applying Azuma's inequality with the assumption $n=O(d)$, we obtain
\beq
\sup_{1\leq k\leq n}\Pr(|\cM_{k}^{\grad,\beta }|>t)<2\exp\left( \frac{-t^2}{2n\cdot (Cd^{-1}\beta ^{2+\alpha })^2} \right)
\leq 2\exp\left( \frac{-t^2}{C'd^{-1}\beta ^{2(2+\alpha )}} \right).
\eeq
Thus, with overwhelming probability,
\beq
\sup_{1\leq k\leq n}|\cM_{k}^{\grad,\beta }|<d^{-\frac12}\beta ^{3+\alpha }
\eeq
Finally, we bound the difference between $\{\cM_{k}^{\grad}\}_{k=1}^n$ and $\{\cM_{k}^{\grad,\beta }\}_{k=1}^n$.  For ease of notation, we write
\beq\begin{split}
G_k:=&\frac{\gamma}{d} \ip{\ip{ \nabla \varphi(X_{k}), a_{k+1}}_\cA, \nabla_x f(r_{k},\epsilon_{k+1}) },\\
G_{k,\beta }:=&\frac{\gamma}{d} \ip{\textstyle\proj_\beta \ip{ \nabla \varphi(X_{k}), a_{k+1}}_\cA, \nabla_x f\circ\textstyle\proj_\beta (r_{k},\epsilon_{k+1}) }.
\end{split}\eeq
The quantity we are trying to bound is
\beq
 |(G_k-\EE G_k)-(G_{k,\beta }-\EE G_{k,\beta })|\leq |G_k-G_{k,\beta }|+|\EE (G_k- G_{k,\beta })|
 \eeq
First, we will show that $G_k-G_{k,\beta }=0$ with overwhelming probability. Using the subgaussian bounds on $(r_{k},\epsilon_{k+1})$ and $\ip{\nabla\varphi(X_{k}),a_{k+1}}_\cA$, we have
\beq\begin{split}
\Pr(G_k\ne G_{k,\beta })&\leq\Pr(\| r_{k},\epsilon_{k+1}\| >\beta )\;+\;
\Pr(\|\ip{\nabla\varphi(X_{k}),a_{k+1}}_\cA\|>\beta )\\
&<4\exp\left(-\frac{\beta ^2}{2C}\right).
\end{split}\eeq
Since $\beta=d^\zeta$ for some $\zeta>0$, the probability bounds above imply that $G_k-G_{k,\beta }=0$ with overwhelming probability, and it remains to bound the difference in their expectations.  %We define an event 
For this, we have
\begin{equation}
\begin{aligned}
|\EE [G_k- G_{k,\beta }]|
&
=
\left|\EE [ (G_k- G_{k,\beta })\cdot 1\{G_k\ne G_{k,\beta }\} ]\right|
\\
&
\leq
|\EE [ G_k\cdot 1\{G_k\ne G_{k,\beta }\}]|+|\EE [G_{k,\beta }\cdot 1\{G_k\ne G_{k,\beta } \}]|
\end{aligned}
\end{equation}
For $\EE [ G_{k,\beta }\cdot 1 \{G_k\ne G_{k,\beta }\}]$, we have
\begin{equation}
\begin{aligned}
|\EE [G_{k,\beta} \cdot 1\{G_k\ne G_{k,\beta } \}] | 
&
\leq
\max |G_{k,\beta }|\;\Pr(G_k\ne G_{k,\beta })
\\
&\leq
d^{-1}L(f)C\beta ^{2+\alpha }\cdot4\exp(-\beta ^2/(2C)).
\end{aligned}
\end{equation}
For $\EE [ G_{k}\cdot 1\{G_k\ne G_{k,\beta}\} ]$, we have
\beq\begin{split}
|\EE [G_k\cdot 1\{G_k\ne G_{k,\beta }\}]|\leq&
\EE [ |G_k \cdot 1\{E_1\}| ]+\EE[ |G_k \cdot 1\{E_2\}|]+\EE[|G_k \cdot 1\{E_3\}|],\\
\text{where }\;E_1\defas & \{\| r_{k}\| \leq\beta \}\cap\{\| \ip{\nabla\varphi(X_{k}),a_{k+1}}_\cA\|>\beta \},\\
E_2\defas &\{\| r_{k}\| >\beta \}\cap\{\| \ip{\nabla\varphi(X_{k}),a_{k+1}}_\cA\|\leq\beta \},\\
E_3\defas &\{\| r_{k}\| >\beta \}\cap\{\| \ip{\nabla\varphi(X_{k}),a_{k+1}}_\cA\|>\beta \}.\\
\end{split}\eeq
The term $\EE|G_k 1\{E_1\}|$ can be bounded as
\beq
\EE|G_k \cdot 1\{E_1\}|\leq L(f)C\beta ^{\max\{1,\alpha \}}
\cdot\EE\left(\| \ip{\nabla\varphi(X_{k}),a_{k+1}}_\cA\| \cdot 1\{\| \ip{\nabla\varphi(X_{k}),a_{k+1}}_\cA\|>\beta \}\right),
\eeq
where the expectation on the right-hand side is exponentially small due to being a tail of a sub-Gaussian first moment (where $\beta ^2$ is larger than the sub-Gaussian variance and grows with $d$).  By similar reasoning, $\EE|G_k 1\{E_2\}|$ is also exponentially small (using the growth bound on $\nabla_x f$).  For $\EE|G_k 1\{E_3\}|$, we have
\beq\begin{split}
\EE&[|G_k \cdot 1\{E_3\}|]\\
&\leq\EE[ \|\nabla_x f(r_{k},\epsilon_{k+1}) \cdot 1\{\|r_{k},\epsilon_{k+1}\|>\beta \}\| \cdot \| \ip{\nabla\varphi(X_{k}),a_{k+1}}_\cA \cdot 1\{\| \ip{\nabla\varphi(X_{k}),a_{k+1}}_\cA\|>\beta \}\| ]\\
&\leq \EE [ \|\nabla_x f(r_{k},\epsilon_{k+1}) \cdot 1\{\|r_{k},\epsilon_{k+1}\|>\beta \}\|^2\cdot \EE\| \ip{\nabla\varphi(X_{k}),a_{k+1}}_\cA \cdot 1\{\| \ip{\nabla\varphi(X_{k}),a_{k+1}}_\cA\|>\beta \}\|^2.
\end{split}\eeq
This is a product of tails of Gaussian moments, which is again exponentially small.  Thus, we conclude that, with overwhelming probability, $\sup_{1\leq k\leq n}|\Delta\cM_{k}^{\grad,\beta }-\Delta\cM_{k}^{\grad}|$ is exponentially small and thus, taking $\beta =d^{\zeta}$, we conclude that,
with overwhelming probability,
\beq
\sup_{1\leq k\leq n}|\cM_{k}^{\grad}|<d^{-\frac12+\zeta(3+\alpha )}.
\eeq
Adjusting the value of $\zeta$, and recalling that all of this has been proved on the stopped process, we obtain the Proposition.
\end{proof}

\begin{proposition}[Hessian martingale]\label{prop:hessmartingale}
Suppose $f \, : \, \mathcal{O} \oplus \mathcal{T} \oplus \mathcal{T} \to \mathbb{R}$ is $\alpha$-pseudo-Lipschitz function with constant $L(f)$ (see Assumption~\ref{ass:pseudo_lipschitz}). Let the statistic $S \, : \, \mathcal{A} \otimes \mathcal{O} \to (\mathcal{O}^+)^{\otimes 2}$ be defined as in \eqref{eq:S_statistic}. Then, for any $\zeta>0$, and with overwhelming probability,
\beq
 \sup_{0 \leq t\leq T}\| \cM_{(t \wedge \vartheta)d}^{\hess} (S(\cdot, z))\|<d^{-1+\zeta}.
\eeq
\end{proposition}
\begin{proof}
As in the proof of the previous proposition, we will work on the stopped version of the martingale but will suppress the $\vartheta$ dependence in the subscript in order to lighten the notation. We also, as before, set $\varphi(X) = S_{ij}(W,z)$ to be the $ij$-th entry of the matrix $S(W,z)$. We have the martingale increments
\beq
\Delta\cM_k^{\hess}=\Delta\cM_k^{H1}+\Delta\cM_k^{H2}
\eeq
where
\beq\begin{split}
\Delta\cM_k^{H1}=&\frac{\gamma^2}{2d^2}\ip{\nabla^2\varphi (X_k),a_{k+1}^{\otimes2}\otimes\nabla_x f(r_k,\epsilon_{k+1})^{\otimes2}}\\
&-\frac{\gamma^2}{2d^2}\EE\left[\ip{\nabla^2\varphi (X_k),a_{k+1}^{\otimes2}\otimes\nabla_x f(r_k,\epsilon_{k+1})^{\otimes2}} |\cF_k\right],\\
\Delta\cM_k^{H2}=&\frac{\gamma^2}{d^2}\ip{\nabla^2\varphi (X_k), \delta X_k\otimes a_{k+1}\otimes\nabla_x f(r_k,\epsilon_{k+1})}\\
&-\frac{\gamma^2}{d^2}\EE\left[ \ip{\nabla^2\varphi (X_k),\delta X_k\otimes a_{k+1}\otimes\nabla_x f(r_k,\epsilon_{k+1})}|\cF_k\right].\\
\end{split}\eeq
We begin by bounding $\cM_k^{H2}$.  Since this increment is linear in $a_{k+1}$, the procedure is almost identical to what we did for $\cM_k^{\grad}$.  We rewrite the increment as
\beq\begin{split}
\Delta\cM_k^{H2}=&\frac{\gamma^2}{d^2}\ip{\ip{\ip{\nabla^2\varphi(X_k), \delta X_k}_{\cA\otimes\cO}, a_{k+1}}_{\cA},\nabla_x f(r_k,\epsilon_{k+1})}_{\cO}\\
&-\frac{\gamma^2}{d^2}\EE\left[\ip{\ip{\ip{\nabla^2\varphi (X_k),\delta X_k}_{\cA\otimes\cO}, a_{k+1}}_{\cA},\nabla_x f(r_k,\epsilon_{k+1})}_{\cO}|\cF_k \right]
\end{split}\eeq
and we introduce another martingale $\cM_k^{H2,\beta}$ with increments
\beq\begin{split}
\Delta\cM_k^{H2,\beta}=&\frac{\gamma^2}{d^2}\ip{\textstyle\proj_\beta\ip{\ip{\nabla^2\varphi (X_k),\delta X_k}_{\cA\otimes\cO}, a_{k+1}}_{\cA},\nabla_x f\circ\textstyle\proj_\beta(r_k,\epsilon_{k+1})}_{\cO}\\
&-\frac{\gamma^2}{d^2}\EE\left[\ip{\textstyle\proj_\beta\ip{\ip{\nabla^2\varphi (X_k),\delta X_k}_{\cA\otimes\cO}, a_{k+1}}_{\cA},\nabla_x f\circ\textstyle\proj_\beta(r_k,\epsilon_{k+1})}_{\cO}|\cF_k\right].
\end{split}\eeq
Using Lemma~\ref{lem:S_derivative_bounds} and similar reasoning as in \eqref{eq:Mgrad_subgaussianbounds},
\beq
 \|r_k,\epsilon_{k+1}\|_{\psi_2}=C,\qquad
 \|\ip{\ip{\nabla^2\varphi (X_{k}),\delta X_k}_{\cA\otimes\cO},a_{k+1}}_\cA\|_{\psi_2}=C.
 \eeq
 Following the steps from the proof of Proposition \ref{prop:gradmartingale}, we get
 \beq
 |\Delta\cM_k^{H2,\beta}|\leq\frac{2\gamma^2}{d^2}L(f)C\beta^{2+\alpha},
 \quad\text{and thus }
% \eeq
% and, applying Azuma's inequality, we conclude
 %\beq
 \sup_{1\leq k\leq n}|\cM_k^{H2,\beta}|<d^{-3/2}\beta^{3+\alpha}.
 \eeq
 This is smaller than what was obtained for $\cM_k^{\grad,\beta}$ due to the extra factor of $d^{-1}$ in the martingale.   Finally, we can show that $|\cM_k^{H2,\beta}-\cM_k^{H2}|$ is exponentially small with overwhelming probability, using the same procedure as in the proof of Proposition \ref{prop:gradmartingale}.
 
It remains to bound $\cM_k^{H1}$, the portion of the martingale that is quadratic in $a_{k+1}$.  The increments are
\beq\begin{split}
\Delta\cM_k^{H1}%=\frac{\gamma^2}{2d^2}\ip{\nabla^2\varphi(X_k),a_{k+1}^{\otimes2}\otimes\nabla_x f(r_k,\epsilon_{k+1})^{\otimes2}}
%-\frac{\gamma^2}{2d^2}\EE\left[\ip{\nabla^2\varphi(X_k),a_{k+1}^{\otimes2}\otimes\nabla_x f(r_k,\epsilon_{k+1})^{\otimes2}} |\cF_k\right]\\
=&\frac{\gamma^2}{2d^2}\ip{\ip{ \nabla^2\varphi (X_k),a_{k+1}^{\otimes2}}_{\cA^{\otimes2}},\nabla_x f(r_k,\epsilon_{k+1})^{\otimes2} }_{\cO^{\otimes2}}\\
&-\frac{\gamma^2}{2d^2}\EE\left[\ip{\ip{ \nabla^2\varphi (X_k),a_{k+1}^{\otimes2}}_{\cA^{\otimes2}},\nabla_x f(r_k,\epsilon_{k+1})^{\otimes2} }_{\cO^{\otimes2}}|\cF_k\right],
\end{split}\eeq
and we define $\cM_k^{H1,\beta}$ to be a new martingale with increments
\beq\begin{split}
\Delta\cM_k^{H1,\beta}=&\frac{\gamma^2}{2d^2}\ip{\textstyle\proj_{d^{\frac12}\beta}\ip{ \nabla^2\varphi (X_k),a_{k+1}^{\otimes2}}_{\cA^{\otimes2}},\nabla_x f\circ\textstyle\proj_\beta (r_k,\epsilon_{k+1})^{\otimes2} }_{\cO^{\otimes2}}\\
&-\frac{\gamma^2}{2d^2}\EE\left[\ip{\textstyle\proj_{d^{\frac12}\beta}\ip{ \nabla^2\varphi (X_k),a_{k+1}^{\otimes2}}_{\cA^{\otimes2}},\nabla_x f\circ\textstyle\proj_{\beta}(r_k,\epsilon_{k+1})^{\otimes2} }_{\cO^{\otimes2}}|\cF_k\right].
\end{split}\eeq
The approach here is similar to the procedure for bounding $\cM_k^{\grad}$ and $\cM_k^{H2}$, although we note that the projection radii for $\ip{ \nabla^2\varphi (X_k),a_{k+1}^{\otimes2}}_{\cA^{\otimes2}}$ and $(r_k,\epsilon_{k+1})$ are different because, while both quantities exhibit concentration of measure, their fluctuations are on different scales.  As we saw in the proof of the previous Proposition, $(r_k,\epsilon_{k+1})$ is $\| W_k\| _\sigma\| \sqrt{K}\| _\sigma$-subgaussian in each entry.  To obtain a concentration bound for $\ip{ \nabla^2\varphi (X_k),a_{k+1}^{\otimes2}}_{\cA^{\otimes2}}$, we rewrite it as

\beq
\ip{ \nabla^2\varphi (X_k),a_{k+1}^{\otimes2}}_{\cA^{\otimes2}}
=\ip{ \nabla^2\varphi (X_k),(\sqrt{K}v_{k+1})^{\otimes2}}_{\cA^{\otimes2}}
=\ipa{\ip{\nabla^2\varphi (X_k),\sqrt{K}^{\otimes2}}_{\cA^{\otimes2}},v_{k+1}^{\otimes2}}_{\cA^{\otimes2}}.
\eeq
Since $\nabla^2\varphi (X_k)\in(\cA\otimes\cO^+)^{\otimes2}$ and $\sqrt{K}^{\otimes2}\in(\cA^{\otimes2})^{\otimes2}$, we get $\ip{\nabla^2\varphi(X_k),\sqrt{K}^{\otimes2}}_{\cA^{\otimes2}}\in(\cO^+)^{\otimes2}\otimes\cA^{\otimes2}$.  Using this ordering of coordinates, in Einstein notation, we write
\beq
\ipa{\ip{\nabla^2\varphi (X_k),\sqrt{K}^{\otimes2}}_{\cA^{\otimes2}},v_{k+1}^{\otimes2}}_{\cA^{\otimes2}}
=\left(\ip{\nabla^2\varphi (X_k),\sqrt{K}^{\otimes2}}_{\cA^{\otimes2}}\right)_{ijk\ell}(v_{k+1})_k(v_{k+1})_\ell.
\eeq
Thus, for each pair $i,j$, the contraction with $v_{k+1}^{\otimes2}$ produces a quadratic form that we can bound using the Hanson-Wright inequality.  More specifically, for each pair $i,j$,
\beq
\Pr\left(\left( \ip{\ip{\nabla^2\varphi (X_k),\sqrt{K}^{\otimes2}}_{\cA^{\otimes2}},v_{k+1}^{\otimes2}}_{\cA^{\otimes2}}\right)_{ij}>t\right)
<2\exp\left(-C\min\left\{\frac{t^2}{\|M(i,j)\|^2},\frac{t}{\|M(i,j)\|_{\text{op}}}\right\}\right)
\eeq
where $M(i,j)$ denotes the $d\times d$ matrix obtained by fixing the $\cO^{\otimes2}$ coordinates of the tensor $\ip{\nabla^2\varphi (X_k),\sqrt{K}^{\otimes2}}_{\cA^{\otimes2}}$ as $i,j$.
For the operator norm, we have
\beq
\|M(i,j)\|_{\text{op}}\leq \|\ip{\nabla^2\varphi (X_k),\sqrt{K}^{\otimes2}}_{\cA^{\otimes2}}\|_\sigma
\leq C
\eeq
where the constant bound comes from the norm bound on $\nabla^2\varphi (X)$ in Lemma~\ref{lem:S_derivative_bounds}.  Using this and the fact that $\|M(i,j)\|^2\leq d\|M(i,j)\|_{\text{op}}^2$, we conclude that
\beq
\Pr\left(\left( \ip{\ip{\nabla^2\varphi (X_k),\sqrt{K}^{\otimes2}}_{\cA^{\otimes2}},v_{k+1}^{\otimes2}}_{\cA^{\otimes2}}\right)_{ij}>t\right)<2\exp\left(-\frac{\min\left\{t^2d^{-1},t\right\}}{C}\right)
\eeq
and this holds uniformly in $i,j$, so
\beq
\Pr\left(\left\| \ip{\ip{\nabla^2\varphi (X_k),\sqrt{K}^{\otimes2}}_{\cA^{\otimes2}},v_{k+1}^{\otimes2}}_{\cA^{\otimes2}}\right\|>t\right)<2|\mathcal{O}|^2\exp\left(-\frac{\min\left\{t^2d^{-1},t\right\}}{C}\right).
\eeq
In particular, this tells us that, for any $\zeta>0$,
\beq
\| \ip{\ip{\nabla^2\varphi (X_k),\sqrt{K}^{\otimes2}}_{\cA^{\otimes2}},v_{k+1}^{\otimes2}}_{\cA^{\otimes2}}\|<d^{\frac12+\zeta}
\eeq
with overwhelming probability.

Having obtained concentration bounds for $(r_k,\epsilon_{k+1})$ and $\ip{\nabla^2\varphi (X_k),a_{k+1}^{\otimes2}}_{\cA^{\otimes2}}$, we proceed to bound $\cM_k^{H1,\beta}$ and show that it is close to $\cM_k^{H1}$.  From the projections and the growth bound on $\nabla_x f$ in Lemma \ref{lem:growth_grad_f}, we get the norm bounds 
  \beq
  \| (\nabla_x f\circ\textstyle\proj_\beta(r_k,\epsilon_{k+1}))^{\otimes2}\| 
  \leq (L(f)C\beta ^{\max\{1,\alpha \}})^2, \qquad
  \| \textstyle\proj_{d^{\frac12}\beta} \ip{\nabla^2\varphi (X_{k}),a_{k+1}^{\otimes2}}_{\cA^{\otimes2}}\| 
  \leq d^{\frac12}\beta,
  \eeq
and thus
\beq
\left|\ipa{\textstyle\proj_{d^{\frac12}\beta} \ip{\nabla^2\varphi (X_{k}),a_{k+1}^{\otimes2}}_{\cA^{\otimes2}},
 (\nabla_x f\circ\textstyle\proj_\beta(r_k,\epsilon_{k+1}))^{\otimes2}
}\right|\leq(L(f)C)^2d^{\frac12}\beta^{3+2\alpha}.
\eeq
 Since this is an almost sure bound, it holds for the expectation as well and we get
\beq
|\Delta\cM_k^{H1,\beta}|\leq\gamma^2(L(f)C)^2d^{-\frac32}\beta^{3+2\alpha}.
\eeq
Applying Azuma's inequality with $n=O(d)$, we obtain
\beq
\sup_{1\leq k\leq n}\Pr(|\cM_k^{H1,\beta}|>t)<2\exp
\left(\frac{-t^2}{2n(Cd^{-\frac32}\beta^{3+2\alpha})^2}\right)
\leq2\exp\left(\frac{-t^2}{2n(C'd^{-2}\beta^{2(3+2\alpha)}}\right)
\eeq
so, with overwhelming probability,
\beq\label{eq:Mhessbetabound}
\sup_{1\leq k\leq n}|\cM_k^{H1,\beta}|<d^{-1}\beta^{4+2\alpha}.
\eeq
It remains only to bound the difference between 
 $\{\cM_{k}^{H1}\}_{k=1}^n$ and $\{\cM_{k}^{H1,\beta }\}_{k=1}^n$.  This follows a very similar argument to what was in the proof of Proposition \ref{prop:gradmartingale}, we write
\beq\begin{split}
G^{H1}_k:=&\frac{\gamma^2}{2d^2} \ip{\ip{ (\nabla^2 \varphi)(X_{k}), a_{k+1}^{\otimes2}}_\cA, \nabla_x f(r_k,\epsilon_{k+1})^{\otimes2} },\\
G^{H1}_{k,\beta }:=&\frac{\gamma^2}{2d^2} \ip{\textstyle\proj_\beta \ip{ (\nabla^2 \varphi)(X_{k}), a_{k+1}^{\otimes2}}_\cA, (\nabla_x f\circ\textstyle\proj_\beta (r_k,\epsilon_{k+1}))^{\otimes2} }.
\end{split}\eeq
The quantity we are trying to bound is
\beq
 |(G^{H1}_k-\EE [G^{H1}_k] )-(G^{H1}_{k,\beta }-\EE [ G^{H1}_{k,\beta } ])|\leq |G^{H1}_k-G^{H1}_{k,\beta }|+|\EE [(G^{H1}_k- G^{H1}_{k,\beta })]|.
 \eeq
As in the proof of Proposition \ref{prop:gradmartingale}, the first of the terms on the right-hand side is 0 with overwhelming probability, while the second is exponentially small.  Computing the bound for $|\EE [(G^{H1}_k- G^{H1}_{k,\beta })]|$ is similar to what was done in the previous proof and is not repeated here.  To see that $|G^{H1}_k-G^{H1}_{k,\beta }|=0$ with overwhelming probability, we write
\beq\begin{split}
\Pr(G^{H1}_k\ne G^{H1}_{k,\beta })&\leq\Pr(\| r_k,\epsilon_{k+1}\| >\beta )\;+\;
\Pr(\|\ip{\nabla^2\varphi(X_{k}),a_{k+1}^{\otimes2}}_{\cA^{\otimes2}}\|>d^{\frac12}\beta )\\
&<2\exp\left(-\frac{\beta ^2}{2C}\right)
+2|\mathcal{O}^+|^2\exp\left(-\frac{\min\{\beta^2,d^{\frac12}\beta\}}{2C}\right).
\end{split}\eeq
Thus, $|\cM_k^{H1,\beta}-\cM_k^{H1}|$ is exponentially small with overwhelming probability.  Using \eqref{eq:Mhessbetabound} along with the bound on $\cM_k^{H2}$ and setting $\beta$ to be an arbitrarily small power of $d$, we obtain the proposition.

\end{proof}

\subsubsection{Bounds on the lower order terms in the Hessian, $\mathcal{E}_t^{\text{Hess}}$}\label{sec:lower_term_Hessian}
We now bound the error term, $\displaystyle \sup_{0 \le t \le T} \sum_{k=0}^{(t \wedge \vartheta)d-1} \| \EE[\mathcal{E}_{k} ^{\text{Hess}} \, | \, \mathcal{F}_k] \|$, in \eqref{eq:error_terms}. For this, we utilize the $\sigma$-norm bound and its dual norm, the nuclear norm. 

\begin{proposition}[Hessian error term]\label{prop:Hessian-error} Suppose $f \, : \, \mathcal{O} \oplus \mathcal{T} \oplus \mathcal{T} \to \mathbb{R}$ is $\alpha$-pseudo-Lipschitz function with constant $L(f)$ (see Assumption~\ref{ass:pseudo_lipschitz}). Let the statistic $S \, : \, \mathcal{A} \otimes \mathcal{O} \to (\mathcal{O}^+)^{\otimes 2}$ be defined as in \eqref{eq:S_statistic}. Then, for any $T > 0$,
\begin{equation}
   \sup_{z \in \Gamma} \sup_{0 \le t \le T} \sum_{k=0}^{(t \wedge \vartheta) d-1} \| \EE[ \mathcal{E}_k^{\text{\rm Hess}}(S(\cdot, z)) \, | \, \mathcal{F}_{k}] \| \le C (L(f))^2  d^{-1}.
\end{equation}
%where the constant $C$ depends on $\|K\|_{\sigma}, |\mathcal{O}^+|, \bar{\gamma}, T, \text{ and } \alpha$ (independent of $d$). 
\end{proposition}

\begin{proof}
We do this entry-wise on the statistic $S(\cdot, z)$, that is, we let $\varphi(X) = S_{ij}(W, z)$ where $S_{ij}$ is the $ij$-th entry of the matrix $S(W,z)$. Define $\Pi_k \defas Q_{k} Q_{k}^T$ and note that $\|\Pi_k\|^2 = \text{rank}(\Pi_k) = |\mathcal{O}^+|$. First, we consider the following term
\begin{equation}
\begin{aligned}
        |\ip{\nabla^2 \varphi(X_k), \sqrt{K} \Pi_k \sqrt{K}  \otimes \nabla_x f(r_k)^{\otimes 2}}| 
        &
        = 
        | \ip{\ip{\nabla^2 \varphi (X_k), \nabla_x f(r_k)^{\otimes 2}}_{\mathcal{O}^{\otimes 2}}, \sqrt{K} \Pi_k \sqrt{K} } |\\
        &
        \le 
        \|\sqrt{K} \Pi_k \sqrt{K}\|_* \|\ip{\nabla^2 \varphi(X_k), \nabla_x f(r_k)^{\otimes 2}}_{\mathcal{O}^{\otimes 2}}\|_{\sigma}
        \\
        &
        \le 
         \|\sqrt{K} \Pi_k \sqrt{K}\|_* \|\nabla^2 \varphi(X_k)\|_{\sigma} \|\nabla_x f(r_k)\|^2
        \\
        &
        \le 
        \|K\|_{\sigma} \|\Pi_k\|_* \|\nabla^2 \varphi (X_k)\|_{\sigma} \|\nabla_xf(r_k)\|^2.
        % \\
        % &
        % \le \|\Dif^2 \varphi(X_k)\|_{\sigma} \|\nabla f(r_k)^{\otimes 2}\| \| \sqrt{K} \Pi_k \sqrt{K} \|
        % \\
        % &
        % \le 
        % \|\Dif^2 \varphi(X_t)\|_{\sigma} \|\nabla f(r_k)\|^2 \|\Pi_k\|_* \|K\|_{\sigma}
\end{aligned}
\end{equation}
From Lemma~\ref{lem:growth_grad_f}, we have $\EE[ \|\nabla_x f(r_k)\|^2 \, | \, \mathcal{F}_k] \le L(f)^2 ( 1 + \|K\|_{\sigma}^{1/2} \|W_k\| )^{\max \{1, 2\alpha\} }$. Moreover, we also, by Lemma~\ref{lem:S_derivative_bounds}, have $\|\nabla^2 \varphi(X_k)\|_{\sigma} \le \|\nabla^2_X S(W, z)\|_{\Gamma} \le C$. Noting that $k \le (t \wedge \theta) d$, 
\begin{equation} \label{eq:Hessian_error_1}
\begin{aligned}
    \EE[| 
    & 
    \ip{\nabla^2 \varphi(X_k), \sqrt{K} \Pi_k \sqrt{K}  \otimes \nabla_x f(r_k)^{\otimes 2}}| \, | \, \mathcal{F}_k] \\
    &
    \le 
    C L^2(f) ( 1 + \|K\|_{\sigma}^{1/2}\|W_k\|)^{\max\{1, 2 \alpha \}}.
\end{aligned}
\end{equation}
% We note, provided $\mathcal{W}_k$ stays within a compact set, $\|\nabla f(r_k; r_k^{\star})\|^2$ will be bounded independent of $n$ and $d$ as $f$ is $C^1$-
% smooth. Moreover by assumption $\|\Dif^2 \varphi(W_k)\|_{\sigma}$ is bounded (again provided $\mathcal{W}_k$ lives in a compact set). Lastly we note that the nuclear norm of $\|Q_k Q_k^T\|_*$ is at most rank $\ell + \ell^{\star}$ as $\ell + \ell^{\star} < d,n$. 

Similarly we get that 
\begin{equation}
    \begin{aligned}
        |\ip{\nabla^2 \varphi(X_k), (\sqrt{K} \Pi_k v_k)^{\otimes 2} \otimes \nabla_x f(r_k)^{\otimes 2}} |
        % &
        % = 
        % \ip{ \ip{ \Dif^2 \varphi(W_k), \nabla f(r_k)^{\otimes 2}}_{\mathcal{O}^{\otimes 2}}, (\sqrt{K} \Pi_k v_k )^{\otimes 2}}_{\mathcal{A}^{\otimes 2}}
        % \\
        &
        \le 
        \|\nabla^2 \varphi(X_k)\|_{\sigma} \|\nabla_xf(r_k) \|^2 \|\sqrt{K} \Pi_k v_k\|^2
        \\
        & 
        \le 
        \|\nabla^2 \varphi(X_k) \|_{\sigma} \|\nabla_x f(r_k)\|^2 \|K\|_{\sigma} \|\Pi_k v_k\|^2.
    \end{aligned}
\end{equation}
Upon taking expectations, with $v_k \sim N(0, I_d)$ independent of $r_k$, we have that $\EE[\|\nabla_x f(r_k)\|^2 \, | \, \mathcal{F}_k] \le C L(f)^2 (1+\|K\|_{\sigma}^{1/2}\|W_k\|)^{\max\{1, 2 \alpha\}}$ (Lemma~\ref{lem:growth_grad_f}) and $\EE[ \|\Pi_k v_k\|^2 \, | \, \mathcal{F}_k] = \|\Pi_k\|^2 = \text{rank}(\Pi_k) = |\mathcal{O}^+|$ as $\Pi_k$ is a projection. Using Lemma~\ref{lem:S_derivative_bounds} on the growth of $\varphi$, 
\begin{equation} \label{eq:Hessian_error_2}
    \EE[ \ip{\nabla^2 \varphi(X_k), (\sqrt{K} \Pi_k v_k)^{\otimes 2} \otimes \nabla_x f(r_k)^{\otimes 2}} \, | \, \mathcal{F}_k] \le C L(f)^2  ( 1 + \|K\|_{\sigma}^{1/2}\|W_k\|)^{\max\{1, 2 \alpha \}}. 
\end{equation}
%where $C$ is a constant depending on $\|K\|_{\sigma}$, $\alpha$, and $|\mathcal{O}^+|$. In particular, this constant is independent of $d$. 

Let us now consider the next term, 
\begin{equation} \label{eq:Hessian_error_3}
\begin{aligned}
    | \ip{\nabla^2 \varphi(X_k), ( \delta X_k )^{\otimes 2} } | 
    &
    \le \delta^2
    \|\nabla^2 \varphi(X_k)\|_{\sigma} \|X_k\|^2
    \le C \|W_k\|^2.
\end{aligned}
\end{equation}
 Note the result also holds in expectation conditioned on $\mathcal{F}_k$. 

Lastly, we consider the term 
\begin{equation} 
\begin{aligned}
    | 
    &
    \ip{\nabla^2 \varphi(X_k), \sqrt{K} \Pi_k v_k \otimes \nabla_x f(r_k) \otimes \delta X_k } |  \\
    &
    \le 
    \delta^2 \|\nabla^2 \varphi(X_k)\| \| \sqrt{K}\|_{\sigma} \|\Pi_k v_k\| \|\nabla_x f(r_k)\| \|X_k\|
\end{aligned}
\end{equation}
As in \eqref{eq:Hessian_error_2}, upon taking expectations, we have that $\EE[ \| \nabla_x f(r_k) \| \, |\, \mathcal{F}_k]  \le C L(f) ( 1 + \|K\|_{\sigma}^{1/2} \|W_k\|)^{\max\{1, \alpha\} }$ (Lemma~\ref{lem:growth_grad_f}) and $\EE[ \|\Pi_k v_k\| \, | \, \mathcal{F}_k] = \|\Pi_k\| = |\mathcal{O}^+|$. Using Lemma~\ref{lem:S_derivative_bounds}, we have 
\begin{equation}
    \label{eq:Hessian_error_4}
    \begin{aligned}
        \EE[ & |\ip{\nabla^2 \varphi(X_k), \sqrt{K} \Pi_k v_k \otimes \nabla_x f(r_k) \otimes \delta X_k } | \, | \, \mathcal{F}_k ]\\
        &
        \le 
        C  L(f)  ( 1 + \|K\|_{\sigma}^{1/2} \|W_k\|)^{\max\{ 1, 2 \alpha\}}.
    \end{aligned}
\end{equation}
%where the constant $C$ is independent of $d$ and depends on $\|K\|_{\sigma}, \alpha, \alpha'$ and $|\mathcal{O}^+|$. 

As $k \le (t \wedge \vartheta )d$, then $\|W_k\| \le M$. The result then immediately follows by combining \eqref{eq:Hessian_error_1}, \eqref{eq:Hessian_error_2}, \eqref{eq:Hessian_error_3}, and \eqref{eq:Hessian_error_4} and summing up with the extra factor $\gamma^2 / d^2$.
\end{proof}

\section{Optimization}\label{sec:optimization}
In this section, we provide criteria for showing distance to optimality descent and convergence for several examples (i.e., bounds on the learning rates) under various assumptions on the outer function $f$. In particular, in this section, we provide proofs of Proposition~\ref{prop:nonexplosiveness}, Proposition~\ref{prop:descent_SGD_main}, Corollary~\ref{cor:convexbounded}, Proposition~\ref{prop:RSI}, and Proposition~\ref{prop:Courtneyrate_main}.
%The main result is that, while typical stepsize stability threshold bounds involve $\|K\|_{\sigma}$, one can relax this condition and only need the average eigenvalue, $\tfrac{1}{d} \tr(K)$.  This is a large improvement as typical datasets have $\|K\|_{\sigma} \gg \tfrac{1}{d} \tr(K)$. We end the section by providing some convergence rates. 

We will do this analysis using the coupled ODEs $(\mathrsfs{B}_i(t) : 1 \leq i \leq d)$, which will also give probability-1 statements.  
% Similar statements can also be made directly on homogenized SGD; our goal here is not to be exhaustive, but simply to illustrate that this system admits a nontrivial and useful analysis and which gives nontrivial conclusions for the optimization theory of these problems.
All these conclusions will be drawn by considering 
the evolution of various quadratic functionals.
For example, in the case $\mathcal{O}=\mathcal{T},$
we will consider
the deterministic counterpart for $\|X-X^{\star}\|^2$.
When evolving according to solution to the \eqref{eq:coupledODElimit} or the integro-differential equation \eqref{eq:ODE_resolvent_2} $\mathcal{S}(t, z)$, 
\begin{equation} \label{eq:deterministic_distance}
\begin{aligned}
    \mathrsfs{D}^2(t) 
    &=
    \frac{1}{d}\sum_{i=1}^d 
    \tr\biggl(
    \mathrsfs{B}_{11,i}(t) 
    - 2\mathrsfs{B}_{12,i}(t) 
    +\mathrsfs{B}_{22,i}(t) 
    \biggr)
    \\
    &\defas \tr \bigg ( \frac{-1}{2 \pi i} \oint_{\Gamma} \mathcal{S}_{11}(t,z) - \mathcal{S}_{12}(t,z) - \mathcal{S}_{21}(t,z) + \mathcal{S}_{22}(t,z) \, \dif z \bigg ),
\end{aligned}
\end{equation}
where we have identified $\mathcal{S}(t,z)$ as a block $2 \times 2$ matrix such that
\[
\mathcal{S}(t,z) = \begin{pmatrix}
\mathcal{S}_{11}(t,z) & \mathcal{S}_{12}(t,z)\\
\mathcal{S}_{21}(t,z) & \mathcal{S}_{22}(t,z)
\end{pmatrix} \in \begin{bmatrix}
\mathcal{O}^{\otimes 2} & \mathcal{O} \otimes \mathcal{T}\\
\mathcal{T} \otimes \mathcal{O} & \mathcal{T}^{\otimes 2}
\end{bmatrix}.
\]
It will turn out that this statistic has a simple evolution which is amenable to analysis.
To motivate this, we consider applying It\^o's lemma to the statistic $\varphi(X) \defas \|X-X^\star\|^2$ applied to homogenized SGD, which produces
\begin{equation}\label{eq:phiIto}
    \dif \varphi(\WHSGD_t) = - \gamma_t \ip{\WHSGD_t - {X}^\star, \nabla \mathcal{R}(\WHSGD_t) } \dif t + \frac{\gamma^2_t}{2d} \tr(K)  \EE_{a,\epsilon} [ \|\nabla_x f(\rho_t)\|^2] \, \dif t + \dif \mathcal{M}_{t}^{\text{HSGD}} (\varphi),
\end{equation}
where we recall $\rho_t = \ip{\WHSGD_t,a}_{\mathcal{A}}$ and 
where $\mathcal{M}_{t}^{\text{HSGD}} (\phi)$ is a martingale.
The function  $\EE_{a,\epsilon} [ \|\nabla_x f(\rho_t)\|^2]$ has a representation as $I({B}(\WHSGD_t))$.  
We also observe that
\begin{equation}\label{eq:A}
    \ip{\WHSGD_t - {X}^\star, \nabla \mathcal{R}(\WHSGD_t) } = \Exp_{a,\epsilon} [ \ip { \ip{\WHSGD_t - {X}^\star,a}, \nabla_x f(\rho_t)} ] \defas A( B(\WHSGD_t)) ,
\end{equation}
as it is again a Gaussian expectation.  Hence, we have
\[
    \dif \varphi(\WHSGD_t) = - \gamma_t A( B(\WHSGD_t)) \dif t + \frac{\gamma^2_t}{2d} \tr(K) I( B(\WHSGD_t))\, \dif t + \dif \mathcal{M}_{t}^{\text{HSGD}} (\varphi).
\]
Moreover, it turns out that this evolution precisely carries over to $\mathrsfs{D}^2$, without a martingale error.
\begin{lemma}[It\^o correction for $\mathrsfs{D}^2$]\label{lem:D2}
$\mathrsfs{D}^2$ solves the differential equation
\[
\frac{\dif}{\dif t} \mathrsfs{D}^2(t) 
=
- \gamma_t A( \mathrsfs{B}(t)) + \frac{\gamma^2_t}{2d} \tr(K) I(\mathrsfs{B}(t)) ,
\]
where
\[
\left.
\begin{aligned}
&A( \mathrsfs{B}) = \Exp_{a,\epsilon} [\ip { x-x^\star, \nabla_x f(x\oplus x^\star)}], \\
&I( \mathrsfs{B}) = \Exp_{a,\epsilon} [\|\nabla_x f(x\oplus x^\star)\|^2],
\end{aligned}
\right\}
\quad\text{where}\quad (x \oplus x^\star) \sim N(0, \mathrsfs{B}).
\]
\end{lemma}
\begin{proof}
The semi-martingale decomposition of an It\^o process is unique.
On the one-hand, It\^o's lemma gives \eqref{eq:phiIto}.  On the other hand, we can give a second decomposition using the representation
\[
\varphi(\WHSGD_t)
=
\tr \bigg ( \frac{-1}{2 \pi i} \oint_{\Gamma} {S}_{11}(\HSGD_t,z) - {S}_{12}(\HSGD_t,z) - {S}_{21}(\HSGD_t,z) + {S}_{22}(\HSGD_t,z) \, \dif z \bigg ).
\]
Applying \eqref{eq:HSGD_exact}, for some local martingale $\mathcal{M}$,
\[
\dif \varphi(\WHSGD_t) = 
\tr \bigg ( \frac{-1}{2 \pi i} \oint_{\Gamma} {\mathscr{F}}_{11}(z, S(\HSGD_t, \cdot)) - {\mathscr{F}}_{12}(z, S(\HSGD_t, \cdot)) - {\mathscr{F}}_{21}(z, S(\HSGD_t, \cdot)) + {\mathscr{F}}_{22}(z, S(\HSGD_t, \cdot)) \, \dif z \bigg )
+\dif \mathcal{M}_t.
\]
Hence we have equality between the finite variation terms.  But from the definition of the integro-differential equation, this finite variation terms is precisely the derivative of $\mathrsfs{D}^2(t),$ i.e.
\[
\frac{\dif}{\dif t} \mathrsfs{D}^2(t) = 
\tr \bigg ( \frac{-1}{2 \pi i} \oint_{\Gamma} {\mathscr{F}}_{11}(z, \mathcal{S}(t, \cdot)) - {\mathscr{F}}_{12}(z, \mathcal{S}(t, \cdot)) - {\mathscr{F}}_{21}(z, \mathcal{S}(t, \cdot)) + {\mathscr{F}}_{22}(z, \mathcal{S}(t, \cdot)) \, \dif z \bigg ),
\]
and hence the claim follows.
\end{proof}
\begin{remark}
    We note that the key to this lemma was that, first, the statistic we consider is linear in $\mathcal{S}$ and second, the finite variation portions of the evolution of $S(\mathscr{X}_t,\cdot)$ are exactly the same as those for $\mathcal{S}$.  Hence, in particular, the same conclusion holds for any other linear functional of $\mathcal{S}$.
\end{remark}
We mention a second important example which also holds regardless of whether or not $\mathcal{O} = \mathcal{T}$:
\begin{corollary}\label{cor:NODE}
    The analogue $\mathrsfs{N}(t)$ of $\|\mathscr{X}_t\|^2 + \|X^*\|^2$,
    given by
    $\mathrsfs{N}(t) = \frac{-1}{2 \pi i} \oint_{\Gamma} \tr(\mathcal{S}(t,z)) \, \dif z$
    evolves by
\[
\frac{\dif}{\dif t} \mathrsfs{N}(t) 
=
- \gamma_t A_0( \mathrsfs{B}(t)) + \frac{\gamma^2_t}{2d} \tr(K) I(\mathrsfs{B}(t)) ,
\]
where
\[
\left.
\begin{aligned}
&A_0( \mathrsfs{B}) = \Exp_{a,\epsilon} [\ip { x, \nabla_x f(x\oplus x^\star)}], \\
&I( \mathrsfs{B}) = \Exp_{a,\epsilon} [\|\nabla_x f(x\oplus x^\star)\|^2],
\end{aligned}
\right\}
\quad\text{where}\quad (x \oplus x^\star) \sim N(0, \mathrsfs{B}).
\]
\end{corollary}

Before continuing, we record for convenience
that the curves $\mathrsfs{N}(t)$ and $\mathrsfs{D}^2(t)$ are naturally related, as one would expect from the norms to which they correspond. Namely,
\begin{equation}\label{eq:DNnorm}
\mathrsfs{N}(t) \le 2 \mathrsfs{D}^2(t) + 3 \|X^*\|^2.
%\tr( \mathrsfs{P}_{22}(t) ). 
\end{equation}
For this, we need to use that $\mathrsfs{B}_i(t)$ for $i = 1, 2, \hdots, d$ (see \eqref{eq:Bi}) are positive semi-definite. Define $\mathrsfs{P}(t) \defas \tfrac{1}{d} \sum_{i=1}^d \mathrsfs{B}_i(t) = \frac{-1}{2 \pi i} \oint_{\Gamma} \mathcal{S}(t,z) \, \dif z$ which is positive semi-definite, and $\mathrsfs{P}_{ij}(t) = \frac{-1}{2 \pi i} \oint_{\Gamma} \mathcal{S}_{ij}(t,z) \, \dif z$. 

Writing in terms of $\mathrsfs{P}$, \eqref{eq:DNnorm} is equivalent to, 
\begin{equation}
    \tr( \mathrsfs{P}_{11}(t) + \mathrsfs{P}_{22}(t) ) \le 2 \tr( \mathrsfs{P}_{11}(t) + \mathrsfs{P}_{22}(t) - \mathrsfs{P}_{12}(t) - \mathrsfs{P}_{21}(t) ) + 3 \tr ( \mathrsfs{P}_{22}(t)).
\end{equation}
This is equivalent to
\[
0 \le \tr ( \mathrsfs{P}_{11}(t) + \mathrsfs{P}_{22}(t) - 2\mathrsfs{P}_{12}(t) - 2\mathrsfs{P}_{21}(t) ) + 3 \tr (\mathrsfs{P}_{22}(t)) = \tr \left ( \mathrsfs{P}(t) \begin{bmatrix} I & -2 I\\ -2I & 4I \end{bmatrix} \right ).
\]
This inequality is immediate after noting that $\mathrsfs{P}(t) \succeq 0$ and $\begin{bmatrix} I & -2 I\\ -2I & 4I \end{bmatrix} \succeq 0$ so the trace of a product of symmetric positive semi-definite matrix is non-negative.

\subsection{Non-explosiveness}
\label{sec:nonexplosiveness}

We have formulated our main theorems as a comparison between processes up to the first time that one of the processes explodes or exits the domain of definition $\mathcal{U}$.
In this section, we give a simple criterion under which one can show that \emph{a priori}, the deterministic ODEs exist for all time. We restate and prove the Proposition~\ref{prop:nonexplosiveness} below. 

\begin{proposition}[Non-explosiveness]
Suppose that 
Assumptions \ref{ass:pseudo_lipschitz},
\ref{assumption:scaling}, \ref{ass:data_normal} and \ref{assum:learning_rate} hold.
Suppose further that the objective function $f$ is $\alpha$-pseudo-Lipschitz with $\alpha=1$.  
Then there is a constant $C$ depending on $\|K\|_{\sigma}$, $\bar{\gamma}$, $\|X_0\|$, $\|X^\star\|$, $L(f)$ so that
\[
\mathrsfs{N}(t) \leq (1+ \mathrsfs{N}(0))e^{C t}
\]
for all time $t$ such that $\mathrsfs{B}(t)$ is in $\mathcal{U}.$
\end{proposition}
\begin{proof}
    From Corollary \ref{cor:NODE},
    \[
    \frac{\dif}{\dif t} \mathrsfs{N}(t) 
    =
    - \gamma_t A_0( \mathrsfs{B}(t)) + \frac{\gamma^2_t}{2d} \tr(K) I(\mathrsfs{B}(t)).
    \]
    From the assumption that $f$ is $1$-pseudo-Lipschitz, we conclude that
    \[
    \|\nabla x f\| \leq L(f)(1 + \|r\| + \|\epsilon\|).
    \]
    It follows by Cauchy-Schwarz that for some constant $C>0$ depending on $L(f)$
    \[
    |A_0( \mathrsfs{B}(t))|, I(\mathrsfs{B}(t))
    \leq C(1+\mathrsfs{N}(t)).
    \]
    Hence for some other constant depending on $\|K\|_{\sigma}$, $L(f)$ and $\bar{\gamma}$,
    \[
    \frac{\dif}{\dif t} \mathrsfs{N}(t) 
    \leq C(1+\mathrsfs{N}(t)).
    \]
    Hence by Gronwall's inequality, $(1 + \mathrsfs{N}(t)) \leq (1 + \mathrsfs{N}(0))e^{Ct}$, which completes the proof.
\end{proof}

\subsection{Distance to optimality descent} \label{sec:distance_optimality}
 We will show that for standard outer function assumptions and some upper bound on the learning rate $\gamma_t < \bar{\gamma}$ that \textit{the function $\mathrsfs{D}^2(t)$ is decreasing in $t$}. Since $\|X-X^{\star}\|^2$ is a statistic that satisfies Assumption~\ref{assumption:statistic}, fixing a $T > 0$, we have by Corollary~\ref{cor:statistic_bounded} for some $\varepsilon > 0$,
\[
\sup_{0 \le t \le T} | \|X_{\lfloor td \rfloor}-X^{\star}\|^2 - \mathrsfs{D}^2(t) | \le d^{-\varepsilon} \quad \text{w.o.p.} 
\]
In this way, $\mathrsfs{D}^2(t) \approx \|X_{\lfloor td \rfloor}-X^{\star}\|^2$ and since $\mathrsfs{D}^2(t)$ is decreasing, so is the distance to optimality of SGD. Consequently, we say SGD is \textit{descending} if $\mathrsfs{D}^2(t)$ is decreasing. Surprisingly, for this to happen, we will see that the upper bound on the learning rate $\bar{\gamma}$ depends on the average eigenvalue of $K$, $\tfrac{1}{d} \tr(K)$, instead of on the largest eigenvalue, $\lambda_{\max}(K)$. As $\tfrac{1}{d} \tr(K) \ll \lambda_{\max}(K)$ for typical datasets, our result shows a larger learning rate can be used in practice and one will still observe decrease. In this section, we will not provide a rate of convergence; we only show learning rates which guarantee decrease of the function $\mathrsfs{D}^2(t)$. 
% While we will not give a rate of convergence in this section, we will provide several conditions on the learning rate for which decrease at each iteration is guaranteed. 

We will work in a simplified setting. 
% We will give several examples on the loss function $f$ for which the condition \eqref{eq:stability_1} holds.
First, throughout the rest of this section, we will assume that there is no regularization 
\[\delta = 0.\]

We now recall Proposition~\ref{prop:descent_SGD_main} below and prove the result. 

% \begin{assumption}[Risk and loss minimizer] \label{assumption:risk_loss_minimizer_1} Suppose that
%   \[
%     X^{\star} \in \argmin_{X} \big \{ \mathcal{R}(X) = \EE_{a, \epsilon}[f(\ip{X,a}_{\mathcal{A}} \oplus \ip{ X^{\star},a}_{\mathcal{A}})] \big \}
%   \]
%   exists and has norm bounded independent of $d.$
%   Then one has, 
%   \[ 
%     \ip{X^{\star},a}_{\mathcal{A}} \in \argmin_x 
%     \{
%     f(x \oplus \ip{X^{\star},a}_{\mathcal{A}}) \}, \qquad \text{for a.s. $a \sim N(0, K)$.}
%   \]
% \end{assumption}
% \noindent While at first, this assumption seems quite strong, in fact, in a typical student-teacher setup when label noise is $0$ (i.e., $\epsilon = 0$), where the targets have the same model as the outputs, in which case the assumption is satisfied.

\begin{proposition}[Descent of SGD] \label{prop:descent_SGD} Fix a constant $T > 0$ and $\eta > 0$. Consider an outer function $f \, : \mathcal{O} \otimes \mathcal{T} \otimes \mathcal{T} \to \mathbb{R}$. Suppose the Assumptions of Theorem~\ref{thm:main_concentration_S} hold  and suppose that $\sup_{0 \le t \le T} \sup_{V \in \mathcal{U}^c} \| \mathrsfs{B}(t)- V\| > \eta$. Moreover, suppose the following inequality holds for some constant $q > 0$,
\begin{equation} \label{eq:stability_2}
    q \cdot \EE_{a, \epsilon} \big [ \| \nabla_x f(\ip{W,a}_{\mathcal{A}})\|^2 \big ] \le \ip{X - X^{\star}, (\nabla \mathcal{R})(X)}, \quad \text{for all $X \in \mathcal{A} \otimes \mathcal{O}$.}
\end{equation}
If the learning rate $\displaystyle \gamma_t < \bar{\gamma}$ for all $t \ge 0$, where 
\begin{equation} 
% \label{eq:stepsize_stability_1}
\bar{\gamma} = \frac{2 q}{\tfrac{1}{d} \tr (K)},
\end{equation}
then, the function $\mathrsfs{D}^2(t)$ defined in \eqref{eq:deterministic_distance} is decreasing for all $t \ge 0$. Moreover, for some $\varepsilon > 0$, the iterates of SGD $\{X_k\}$ satisfy
\begin{equation}
\label{eq:descent_101}
\sup_{0 \le t \le T} | \|X_{\lfloor td \rfloor }-X^{\star}\|^2 - \mathrsfs{D}^2(t) | \le d^{-\varepsilon}, \quad \text{w.o.p.}
\end{equation}
\end{proposition}

\begin{proof} First, we show that $\mathrsfs{D}(t)$ is a decreasing function. For this, we see by \eqref{eq:stability_2} and Lemma~\ref{lem:D2} that
\begin{equation*}
\begin{aligned}
\dif \mathrsfs{D}^2(t) 
&
= - \gamma_t A(\mathrsfs{B}(t)) \, \dif t + \frac{\gamma_t^2}{2d} \tr(K) I(\mathrsfs{B}(t))\\
&
= -\gamma_t \EE_{a, \epsilon} [ \ip{x-x^{\star}, \nabla_x f(x \oplus x^{\star}}] + \frac{\gamma_t^2}{2d} \tr(K) \EE_{a, \epsilon} [ \|\nabla_x f(x \oplus x^{\star})\|^2 ], \quad \text{where $(x \oplus x^{\star}) \sim N(0, \mathrsfs{B})$}
\\
&
= 
-\gamma_t \EE_{a, \epsilon} [ \ip{X-X^{\star}, a \otimes \nabla_x f(x \oplus x)}] + \frac{\gamma_t^2}{2d} \tr(K) \EE_{a, \epsilon} [ \|\nabla_x f(x \oplus x^{\star})\|^2 ]
\\
&
\le 
\gamma_t \big [ \frac{\gamma_t}{2} \cdot \frac{1}{d} \tr(K) - q \big ] \big [ \EE_{a, \epsilon} [ \|\nabla_x f(x \oplus x^{\star})\|^2 ] \big ] < 0.
\end{aligned}
\end{equation*}
Thus, the function $\mathrsfs{D}(t)$ is decreasing. 

Now as $\mathrsfs{D}^2(t)$ is non-increasing,
then using \eqref{eq:DNnorm}, we have that 
\[
\sup_{0 \le t \le T} \mathrsfs{N}(t) \le 2 \mathrsfs{D}^2(0) + 3 \|X^{\star}\|^2 \le C. 
\]
Hence the assumptions of Corollary~\ref{cor:statistic_bounded} are satisfied and the conclusions of Corollary~\ref{cor:statistic_bounded} give the result \eqref{eq:descent_101}.  
\end{proof}

Next, we will need to assume a result about our outer function $f$, that is, it attains a \textit{global minimizer} at the same point as the global minimizer of the risk $\mathcal{R}$, that is, Assumption~\ref{assumption:risk_loss_minimizer_1_main} holds. Moreover, we give a value for $q$ in \eqref{eq:stability_2} when the outer function $f$ (and \textit{not} the objective function $\mathcal{R}$) is \textit{$\hat{L}$-smooth}. We again restate Corollary~\ref{cor:convexbounded} and provide a proof. 

% we will assume the loss function $ f$ (and \textit{not} the objective function $\mathcal{R}$) is \textit{$\hat{L}$-smooth}. This type of assumption is typical of many optimization convergence algorithms.

% \begin{definition}[$\hat{L}$-smoothness of loss function $f$] A $C^1$-smooth function $ f \, : \, \mathcal{O} \to \mathbb{R}$ is \textit{$\hat{L}(f)$-smooth} if the following quadratic upper bound holds for any $x, \hat{x} \in \mathcal{O}$
% \begin{equation} \label{eq:quadratic_upper_bound_1}
%      f(\hat{x}) \le   f(x) + \ip{\nabla_x f(x), \hat{x} - x} + \tfrac{\hat{L}(f)}{2} \|\hat{x}-x\|^2. 
% \end{equation}
% \end{definition}
% \noindent Note that if $\nabla_x  f$ is $\hat{L}(f)$-Lipschitz, i.e., $\|\nabla f(x) - \nabla f(\hat{x})\| \le \hat{L}(f) \| x - \hat{x}\|$, then the inequality \eqref{eq:quadratic_upper_bound} holds with constant $\hat{L}$. Suppose $\displaystyle x^{\star} \in \argmin_x \{f(x)\}$ exists. An immediate consequence of \eqref{eq:quadratic_upper_bound} is that 
% \begin{equation} \label{eq:L_smoothness_bound}
% \frac{1}{2\hat{L}(f)}\| \nabla f(x) \|^2 \leq  f(x) -f(x^{\star})\leq  \frac{\hat{L}(f)}{2}\|x-x^\star\|^2.
% \end{equation}

\begin{corollary}[Descent of convex, $\hat{L}(f)$-smooth outer function] Fix a constant $T > 0$. Suppose the Assumptions of Theorem~\ref{thm:main_concentration_S} hold and suppose that $\sup_{0 \le t \le T} \sup_{V \in \mathcal{U}^c} \| \mathrsfs{B}(t)- V\| > \eta$. In addition, let the outer function $f \, : \, \mathcal{O} \otimes \mathcal{T} \otimes \mathcal{T} \to \mathbb{R}$ be a convex and $\hat{L}(f)$-smooth function with respect to $x\in \mathcal{O}$. Suppose $X^{\star} \in argmin_{X} \{\mathcal{R}(X)\}$ exists bounded, independent of $d$ and Assumption~\ref{assumption:risk_loss_minimizer_1_main} holds. Then the inequality \eqref{eq:stability_2} holds with $q = \tfrac{1}{2\hat{L}(f)}$. Moreover, if $\displaystyle  \gamma_t \le \bar{\gamma}$ for all $t$ where
\[
\bar{\gamma} = \frac{1}{\hat{L}(f)\tfrac{1}{d} \tr(K)}, 
\]
then, the function $\mathrsfs{D}^2(t)$ defined in \eqref{eq:deterministic_distance} is decreasing for all $t \ge 0$. Moreover, for some $\varepsilon > 0$, the iterates of SGD $\{X_k\}$ satisfy
\[
\sup_{0 \le t \le T} | \|X_{\lfloor td \rfloor }-X^{\star}\|^2 - \mathrsfs{D}^2(t) | \le d^{-\varepsilon}, \quad \text{w.o.p.}
\]

% then, with overwhelming probability,
% \[
% \sup_{0 \le t \le T} \|\WHSGD_t \| \le C
% \]
% for some positive constant $C = C(\|K\|_{\sigma}, \|X_0\|, \|X^{\star}\|, \alpha)$. 
\label{cor:convexbounded_1}
\end{corollary}

\begin{proof} By convexity of $f$, we have that $f(\ip{X,a}_{\mathcal{A}})$ is convex in $X$ and thus, $\mathcal{R}(X) = \EE_{a, \epsilon} [ f(\ip{X,a}_{\mathcal{A}})]$ is convex. Therfore, we deduce that 
\begin{align} \label{eq:stability_convex_1}
\ip{X-X^{\star}, (\nabla \mathcal{R})(X)} \ge \mathcal{R}(X) - \mathcal{R}(X^{\star}), \quad \text{for all $X \in \mathcal{A} \otimes \mathcal{O}$.}
\end{align}
In addition, Assumption~\ref{assumption:risk_loss_minimizer_1_main} together with $\hat{L}(f)$-smoothness of $f$ \eqref{eq:L_smoothness_bound_main} implies
\begin{align*}
    \frac{1}{2 \hat{L}(f)} \|\nabla_x f(\ip{X,a}_{\mathcal{A}})\|^2 \le f(\ip{X,a}_{\mathcal{A}}) - \inf_x f(x) = f(\ip{X,a}_{\mathcal{A}}) - f(\ip{ X^{\star}, a}_{\mathcal{A}}),
\end{align*}
for almost surely any $a \sim N(0,K)$. Taking expectation, we have that 
\begin{align} \label{eq:stability_convex_2}
    \frac{1}{2\hat{L}(f)} \EE_{a,\epsilon} \big [ \|\nabla_x f(\ip{X,a}_{\mathcal{A}})\|^2 \big ] \le \EE_{a,\epsilon} [ f(\ip{X,a}_{\mathcal{A}})] - \EE_{a,\epsilon}[f(\ip{X^{\star},a}_{\mathcal{A}}) ] = \mathcal{R}(X) - \mathcal{R}(X^{\star}).
\end{align}
The inequality \eqref{eq:stability_2} immediately follows from \eqref{eq:stability_convex_1} and \eqref{eq:stability_convex_2} with $q = \frac{1}{2 \hat{L}(f)}$. The result, then follows, by applying Proposition~\ref{prop:descent_SGD}. 
\end{proof}

Under the RSI assumption on the outer function $f$, we can show that the inequality \eqref{eq:stability_2} holds in Proposition~\ref{prop:descent_SGD}.  
\begin{corollary}[Descent of $\hat{L}(f)$-smooth, RSI with $\hat{\mu}(f)$ outer function] \label{cor: stability_RSI_smooth} Fix $T > 0$. Suppose the Assumptions of Theorem~\ref{thm:main_concentration_S} hold and suppose that $\sup_{0 \le t \le T} \sup_{V \in \mathcal{U}^c} \| \mathrsfs{B}(t)- V\| > \eta$ w.o.p. In addition, let the outer function $f \, : \, \mathcal{O} \otimes \mathcal{T} \otimes \mathcal{T} \to \mathbb{R}$ be a $\hat{L}(f)$-smooth and
$\hat{\mu}(f)$--RSI with respect to $x \in \mathcal{O}$. 
Suppose $X^{\star} \in \argmin_X \{ \mathcal{R}(X)\}$ is bounded, independent of, $d$ and Assumption~\ref{assumption:risk_loss_minimizer_1_main} holds. 
%Then the inequality \eqref{eq:stability_1} holds with $q = \frac{\hat{\mu}(f)}{(\hat{L}(f))^2}$. 
Then provided $\gamma_t \le \bar{\gamma}$ for all $t \ge 0$ where 
\[
\bar{\gamma} = \frac{2 \hat{\mu}(f) }{(\hat{L}(f))^2 \tfrac{1}{d} \tr(K)},
\]
then, the function $\mathrsfs{D}^2(t)$ defined in \eqref{eq:deterministic_distance} is decreasing for all $t \ge 0$. Moreover, for some $\varepsilon > 0$, the iterates of SGD $\{X_k\}$ satisfy
\[
\sup_{0 \le t \le T} | \|X_{\lfloor td \rfloor }-X^{\star}\|^2 - \mathrsfs{D}^2(t) | \le d^{-\varepsilon}, \quad \text{w.o.p.}
\]
\label{cor:descent_RSI}
\end{corollary}

\begin{proof} By the RSI (with constant $\hat{\mu}(f)$) condition on $f$, we have that
\begin{equation}
\begin{aligned} \label{eq:RSI_1}
    \ip{X - X^{\star}, \nabla_X \mathcal{R}(X)} 
    &
    = 
    \ip{X-X^{\star}, \EE_{a,\epsilon} [a \otimes \nabla_x f(\ip{X,a}_{\mathcal{A}})] }
    \\
    &
    =
    \EE_{a, \epsilon} \big [ \ip{x-x^{\star}, \nabla_x f(x) } \big ]
    \\
    &
    \ge 
    \hat{\mu}(f) \EE_{a, \epsilon} [ \|x-x^{\star}\|^2],
\end{aligned}
\end{equation}
where $x= \ip{ X,a}_{\mathcal{A}}$ and $x^{\star} = \ip{X^{\star},a}_{\mathcal{A}}$.

By $\hat{L}(f)$-smoothness, 
\[
\frac{1}{2\hat{L}(f)} \| \nabla_x f(x)\|^2 \le \frac{\hat{L}(f)}{2} \|x-x^{\star}\|^2.
\]
This implies that 
\begin{equation} \label{eq:RSI_2}
\frac{1}{(\hat{L}(f))^2} \EE_{a,\epsilon} \big [ \|\nabla_x f(\ip{X,a}_{\mathcal{A}})\|^2 \big ] \le \EE_{a, \epsilon} \big [ \|x-x^{\star}\|^2\big ] .
\end{equation}
Thus by \eqref{eq:RSI_1} and $\eqref{eq:RSI_2}$, we have that the inequality \eqref{eq:stability_2} holds with $q = \frac{\hat{\mu}(f)}{(\hat{L}(f))^2}$. The result then follows by applying Proposition~\ref{prop:descent_SGD}. 
\end{proof}

\subsection{Convergence analysis} \label{sec:convergence}

We provide a simple complexity analysis under various scenarios. The first result is, for strongly convex risks, a linear rate that only depends on the \textit{average condition number}, $\frac{\tr(K) /d }{\lambda_{\min}(K)}$, where $\lambda_{\min}(K)$ is the smallest eigenvalue of $K$. Typical convergence rates usually depend on $\frac{\|K\|_{\sigma}}{\lambda_{\min}(K)}$ which for many datasets, especially those in machine learning, the average eigenvalue is much smaller than the maximum eigenvalue of $K$. We restate below Proposition~\ref{prop:RSI} and \ref{prop:Courtneyrate_main} and provide proofs. 

\begin{proposition}[Global convergence rate for fixed stepsize, $\hat{\mu}(f)$-RSI, $\hat{L}(f)$-smooth function, with covariance $K \succ 0$] Fix a constant $T > 0$. Suppose the Assumptions of Theorem~\ref{thm:main_concentration_S} hold and suppose that $\sup_{0 \le t \le T} \sup_{V \in \mathcal{U}^c} \| \mathrsfs{B}(t)- V\| > \eta$. Let the outer function $f \, : \, \mathcal{O} \otimes \mathcal{T} \otimes \mathcal{T} \to \mathbb{R}$ be a $\hat{L}(f)$-smooth function satisfying the RSI condition with $\hat{\mu}(f)$ with respect to $x \in \mathcal{O}$. Suppose $X^{\star} \in \argmin_X \{ \mathcal{R}(X)\}$ is bounded, independent of, $d$ and Assumption~\ref{assumption:risk_loss_minimizer_1_main} holds. Let the covariance matrix $K$ have a smallest eigenvalue bounded away from $0$, that is $\lambda_{\min}(K) > 0$. If the learning rate satisfies
\[
\gamma_t = \gamma = \frac{2 \hat{\mu}(f) }{(\hat{L}(f))^2 \tfrac{1}{d} \tr(K)} \zeta,
\]
for some $0 < \zeta < 1$, then for all $t \ge 0$
\[
\mathrsfs{D}^2(t) \le e^{-a t} \mathrsfs{D}^2(0), 
\]
where $a = \gamma (1-\zeta) \hat{\mu}(f) \lambda_{\min}(K)$. Moreover, for some $\varepsilon> 0$, the iterates of SGD $\{X_k\}$ satisfy
\begin{equation} \label{eq:convergence_10}
\sup_{0 \le t \le T} | \|X_{\lfloor td \rfloor }-X^{\star}\|^2 - \mathrsfs{D}^2(t) | \le d^{-\varepsilon}, \quad \text{w.o.p.}
\end{equation}
\label{prop:Courtneyrate}
\end{proposition}

\begin{proof} The assumptions and choice of $\gamma_t$ ensure that the Assumptions of Corollary~\ref{cor:descent_RSI} hold. Thus, it immediately follows that \eqref{eq:convergence_10} holds. It remains to show the linear rate of decrease of $\mathrsfs{D}^2(t)$. 

% Let $\phi(X) = \tfrac{1}{2}\|X - \hat{X}\|^2$. Similar to the proof of Proposition~\ref{prop:stability_threshold}, using It\^o's Lemma with the function $\phi(\WHSGD_t)$ and homogenized SGD \eqref{eq:HSGD}, we have that
% \begin{equation}
% \begin{aligned} \label{eq:RSI_convergence_1}
% \dif \phi(\WHSGD_t) = - \gamma \ip{\WHSGD_t - \hat{X}, (\Dif \mathcal{R})(\WHSGD_t) } \dif t + \frac{\gamma^2}{2d} \tr(K)  \EE_{(a,\varepsilon)} [ \|\nabla f(\rho_t)\|^2] \, \dif t + \dif \mathcal{M}_{t}^{\text{HSGD}} (\phi) 
%     \\
%     \text{where} \quad \dif \mathcal{M}_{t}^{\text{HSGD}}(\phi) = \frac{\gamma}{\sqrt{d}} \ip{ \sqrt{K} \otimes \sqrt{\EE_{(a,\varepsilon)}[\nabla f(\rho_t)^{\otimes 2}]}, (\WHSGD_t-\hat{X}) \otimes \dif B_t}.
% \end{aligned}
% \end{equation}
By \eqref{eq:RSI_1} and \eqref{eq:RSI_2},
\[
\frac{\hat{\mu(f)}}{(\hat{L}(f))^2} \EE_{a,\epsilon}[ \|\nabla_x f( \ip{X,a}_{\mathcal{A}})\|^2 ] \le \ip{X - X^{\star}, (\nabla \mathcal{R})(X)}, \quad \text{for any $X \in \mathcal{A} \otimes \mathcal{O}$}
\]
Setting $q = \frac{\hat{\mu(f)}}{(\hat{L}(f))^2}$, we have that
\begin{equation} \label{eq:RSI_convergence_2}
\begin{aligned}
-\gamma \EE_{a, \epsilon} \big [ \ip{x-x^{\star}, \nabla_x f(x \oplus x^{\star}) } \big ] &+ \frac{\gamma^2}{2 d} \tr( K ) \EE_{a,\epsilon} [ \|\nabla_x f(x \oplus x^{\star})\|^2] 
\\
&
= 
-\gamma \ip{X- X^{\star}, (\nabla \mathcal{R})(X)} + \frac{\gamma^2}{2 d} \tr( K ) \EE_{a,\epsilon} [ \|\nabla_x f(\ip{X,a}_{\mathcal{A}})\|^2] 
\\
& 
\le 
\gamma \big [ \frac{\gamma}{2q} \cdot \frac{1}{d} \tr(K) - 1 \big ] \big [ \ip{X-X^{\star}, (\nabla \mathcal{R})(X)} \big ]
\\
&
=
- \gamma (1-\zeta)  \big [ \ip{X-X^{\star}, (\nabla \mathcal{R})(X)} \big ]
\\
&
=- \gamma (1-\zeta)\ip{X-X^{\star}, \EE_{a,\epsilon} [a \otimes \nabla_x f(\ip{X,a}_{\mathcal{A}})] }
    \\
    &
    =
    - \gamma (1-\zeta) \EE_{a, \epsilon} \big [ \ip{\ip{ X,a}_{\mathcal{A}}-\ip{X^{\star},a}_{\mathcal{A}}, \nabla_x f(\ip{X,a}_{\mathcal{A}}) } \big ]
    \\
    &
    =
    - \gamma (1-\zeta) \EE_{a, \epsilon} \big [ \ip{x-x^{\star}, \nabla_x f(x \oplus x^{\star}) } \big ].
\end{aligned}
\end{equation}
Here $(x \oplus x^{\star}) \sim N(0, \mathrsfs{B})$. By the RSI (with constant $\hat{\mu}(f)$) assumption, \begin{equation}
\begin{aligned} \label{eq:RSI_convergence_3}
    % \ip{X - X^{\star}, (\nabla \mathcal{R})(X)} 
    % &
    % = 
    % \ip{X-X^{\star}, \EE_{a,\epsilon} [a \otimes \nabla_x f(\ip{X,a}_{\mathcal{A}})] }
    % \\
    % &
    % =
    % \EE_{a, \epsilon} \big [ \ip{\ip{ X,a}_{\mathcal{A}}-\ip{X^{\star},a}_{\mathcal{A}}, \nabla_x f(\ip{X,a}_{\mathcal{A}}) } \big ]
    \EE_{a, \epsilon} \big [ \ip{x-x^{\star}, \nabla_x f(x \oplus x^{\star}) } \big ]
    &
    % \ge 
    % \hat{\mu}(f) \EE_{a, \epsilon} [ \|x-x^{\star}\|^2]
    % \\
    % &
    \ge \hat{\mu}(f) \EE_{a, \epsilon} [ \|x - x^{\star}\|^2 ]
    \\
    &
    =
    \hat{\mu}(f) \tr( \mathrsfs{B}_{11}(t) - \mathrsfs{B}_{12}(t) - \mathrsfs{B}_{21}(t) + \mathrsfs{B}_{22}(t) )
    \\
    &
    \ge \hat{\mu}(f) \lambda_{\min}(K) \tr \bigg ( \frac{1}{d} \sum_{i=1}^d \big ( \mathrsfs{B}_{11, i}(t) - \mathrsfs{B}_{12,i}(t) - \mathrsfs{B}_{21,i} + \mathrsfs{B}_{22,i}(t) \big ) \bigg )
    \\
    &
    = \hat{\mu}(f) \lambda_{\min}(K) \mathrsfs{D}^2(t),
    % \\
    % &
    % =
    % \hat{\mu}(f) \ip{(X - \hat{X}) \otimes (X-\hat{X}), K}_{\mathcal{A}^{\otimes 2}}
    % \\
    % &
    % \ge 
    % \hat{\mu}(f) \lambda_{\min}(K) \phi(X),
\end{aligned}
\end{equation}
where $\lambda_{\min}(K)$ is the smallest eigenvalue of $K$ and $\frac{-1}{2 \pi i} \oint_{\Gamma} \mathcal{S}_{k\ell}(t,z) \, \dif z = \frac{1}{d} \sum_{i=1}^d \mathrsfs{B}_{k\ell, i} $. 

Now by Lemma~\ref{lem:D2}, with $(x \oplus x^{\star}) \sim N(0, \mathrsfs{B})$,
\begin{align*}
\frac{\dif}{\dif t} \mathrsfs{D}^2(t) 
&
=
- \gamma A( \mathrsfs{B}(t)) + \frac{\gamma^2}{2d} \tr(K) I(\mathrsfs{B}(t))
\\
&
=
- \gamma \Exp_{a,\epsilon} [\ip { x-x^\star, \nabla_x f(x\oplus x^\star)}] + \frac{\gamma^2}{2d} \tr(K) \EE_{a,\epsilon} [\|\nabla_x f(x \oplus x^{\star})\|^2]
\\
&
\le 
-\gamma (1-\zeta) \EE_{a, \epsilon} \big [ \ip{x-x^{\star}, \nabla_x f(x \oplus x^{\star}) } \big ]
\\
&
\le 
-\gamma (1-\zeta) \hat{\mu}(f) \lambda_{\min}(K) \mathrsfs{D}^2(t)
\end{align*}
By Gronwall's inequality,
\[
\mathrsfs{D}^2(t) \le e^{-a t} \mathrsfs{D}^2(0).
\]
where $a = \gamma (1-\zeta) \hat{\mu}(f) \lambda_{\min}(K)$.

\end{proof}

We now provide a local convergence rate statement. This will mainly be applied to the multi-class logistic regression problem which is (strictly) convex, but locally strongly convex. 

\begin{proposition}[Local convergence rate for fixed stepsize, $(\hat{\mu}(f),\hat{\theta}(f))$-RSI, $\hat{L}(f)$-smooth function, with covariance $K \succ 0$] 
Fix a constant $T > 0$. Suppose the Assumptions of Theorem~\ref{thm:main_concentration_S} hold and suppose that $\sup_{0 \le t \le T} \sup_{V \in \mathcal{U}^c} \| \mathrsfs{B}(t)- V\| > \eta$. Let the outer function $f \, : \, \mathcal{O} \otimes \mathcal{T} \otimes \mathcal{T} \to \mathbb{R}$ be a $\hat{L}(f)$-smooth function satisfying $(\hat{\mu}(f),\hat{\theta}(f))$--RSI with respect to $x \in \mathcal{O}$. Suppose $X^{\star} \in \argmin_X \{ \mathcal{R}(X)\}$ is bounded, independent of, $d$ and Assumption~\ref{assumption:risk_loss_minimizer_1_main} holds. Let the covariance matrix $K$ have a smallest eigenvalue bounded away from $0$, that is $\lambda_{\min}(K) > 0$.
% Fix $T > 0$ and suppose the initialization $X_0$, $X^{\star}$ are bounded independent of $d$. Let the loss function $f \, : \, \mathcal{O} \to \mathbb{R}$ be a $\hat{L}(f)$-smooth, $\alpha$-pseudo-Lipschitz function satisfying the $(\hat{\mu}(f),\hat{\theta}(f))$--RSI. Suppose $\hat{X} \in \argmin_X \{ \mathcal{R}(X)\}$ is bounded, independent of, $d$ and Assumption~\ref{assumption:risk_loss_minimizer} holds. Let the covariance matrix $K$ have the smallest eigenvalue bounded away from $0$, that is $\lambda_{\min}(K) > 0$. 

Suppose the initialization $X_0$ satisfies that for some $\zeta_0 \in (0,1)$
\[
10 \exp \biggl( -\frac{\hat{\theta}(f)}{8 \|K\|_\sigma^2
  \max\{\|X_0-X^{\star}\|^2, \|X^{\star}\|^2\}
} \biggr) < \zeta_0,
\]
Suppose that $0 < \zeta < 1-\zeta_0$ and that
\[
\gamma_t = \gamma = \frac{2 \hat{\mu}(f) }{(\hat{L}(f))^2 \tfrac{1}{d} \tr(K)} \zeta,
\]
Then with $a = \gamma (1-\zeta_0 - \zeta) \hat{\mu}(f) \lambda_{\min}(K)$,
we have for all $t \ge 0$
\[
\mathrsfs{D}^2(t) \le 2 e^{-at} \|X_0-X^{\star}\|^2
\]
% \[
% \|\WHSGD_t - \hat{X}\|^2 \le 2 e^{-at} \|X_0-\hat{X}\|^2, \quad \text{w.o.p}.
% \]
Moreover, for some $\varepsilon> 0$, the iterates of SGD $\{X_k\}$ satisfy
\begin{equation} \label{eq:convergence_101}
\sup_{0 \le t \le T} | \|X_{\lfloor td \rfloor }-X^{\star}\|^2 - \mathrsfs{D}^2(t) | \le d^{-\varepsilon}, \quad \text{w.o.p.}
\end{equation}
\end{proposition}

\begin{proof}
By hypothesis on $f$,
\begin{equation}
\begin{aligned} \label{eeq:RSI_1}
    \ip{X - X^{\star}, (\nabla \mathcal{R})(X)} 
    &
    = 
    \ip{X-X^{\star}, \EE_{a,\epsilon} [a \otimes \nabla_x f(\ip{X,a}_{\mathcal{A}})] }
    \\
    &
    =
    \EE_{a, \epsilon} \big [ \ip{x-x^{\star}, \nabla_x f(x) } \big ]
    \\
    &
    \ge 
    \hat{\mu}(f) \EE_{a, \epsilon} [ \|x-x^{\star}\|^2 \mathbf{}{1}\{ \|x-x^{\star}\|^2 \leq \hat\theta(f) \, \text{and} \, \|x^{\star}\|^2 \le \hat{\theta}(f)\} ],
\end{aligned}
\end{equation}
where $x = \ip{X,a}_{\mathcal{A}}$ and $x^{\star} = \ip{X^{\star},a}_{\mathcal{A}}$.
%We note that by Hanson-Wright, there is an absolute constant $C> 0$ so that
%\[
%\Pr(\|r-\hat{r}\|^2 \geq \Exp \|r - \hat{r}\|^2 + t)
%\leq \exp( - \min\{
%\tfrac{t^2}{C V^2}, \tfrac{t}{C V}
%\} ),
%\quad
%V = \Exp \|r - \hat{r}\|^2 = \tr\langle K, (X_t-\hat{X})^{\otimes 2} \rangle,
%\]
%where we have used that $\|r-\hat{r}\|^2 \sim \tr\langle z^{\otimes 2}, (\sqrt{K}(X-\hat{X}))^{\otimes 2}\rangle$ for standard Gaussian $z$.  This is a quadratic form in matrix $M = \langle \sqrt{K}(X-\hat{X}),\sqrt{K}(X-\hat{X}) \rangle_{\mathcal{O}}$.  As $M$ is positive definite, its trace, given by $V$ is an upper bound for its operator norm.  Further, its Hilbert-Schmidt norm is bounded by $V^2.$
%It follows there is an absolute constant $C> 1$ so that for $t\geq 1$
%\[
%\Pr(\|r-\hat{r}\|^2 \geq CVt) \leq e^{-t},
%\]
%and hence for $u = \hat\theta(f)/(CV) \geq 1$ and integration by parts
%\[
%\EE_{(a, \varepsilon)} [ \|r-\hat{r}\|^2 \mathbf{}{1}\{ \|r-\hat{r}\|^2  \geq \hat\theta(f)\} ]
%\leq CVe^{-u}+CV \int_u^\infty \Pr(\|r-\hat{r}\|^2 \geq CVx)\dif x
%\leq 2CVe^{-u}.
%\]
Using Lemma \ref{lem:GVN}, with $V = \Exp \|x-x^{\star}\|^2$,
\begin{equation}\label{eeq:normlb}
\EE_{a, \epsilon} 
[ \|x-x^{\star}\|^2
\mathbf{}{1}\{ \|x-x^{\star}\|^2  \geq \hat\theta(f)\}
] \leq 5V\exp(-\hat{\theta}(f)/4V).
\end{equation}
%where we note the statement is vacuous for $\hat\theta(f)/(V)$ 
%less than some absolute constant.
We need to do the same estimate for the contribution from large $x^{\star}$,
but correlations complicate the analysis.
So by Cauchy Schwarz
\[
\EE_{a, \epsilon} 
(
\|x-x^{\star}\|^2
1\{x^{\star} \geq \hat{\theta}(f)\}
)
\leq 
\sqrt{ \EE \|x-x^{\star}\|^4 \times \Pr (\|x^{\star}\|^2 \geq \hat{\theta}(f))}.
\]
From Wick's formula, 
the $4$-th moment can be bounded by $3(\EE \|x-x^{\star}\|^2)^2$.
Using Lemma \ref{lem:GVN} we can also bound the tail of $\|x^{\star}\|^2.$

So suppose for some $\zeta_0 \in (0,1)$ that we work up to the stopping time $\vartheta$, defined as the first time,
\[
5\exp(-\hat{\theta}(f)/8P))
+ 5\exp(-\hat{\theta}(f)/4b_t)) < \zeta_0,
\quad
\left\{
\begin{aligned}
&P = \tr\langle (X^{\star})^{\otimes 2},K \rangle \\
&b_t = \tr( \mathrsfs{B}_{11} - \mathrsfs{B}_{12} - \mathrsfs{B}_{21} + \mathrsfs{B}_{22} )
\end{aligned}
\right.
\]
Then the stopped process (system of ODEs satisfies the conditions of $\vartheta$) $\mathrsfs{B}(t \wedge \vartheta)$ satisfies the conclusions of Proposition \ref{prop:Courtneyrate} with effective RSI constant $\hat \mu (1-\zeta_0)$.  

It remains to show that we can remove the stopping time.  For this purpose, we need to ensure the process $b_t$ remains in control.  
In particular, provided $\gamma \leq 
\frac{2\hat \mu}
{
(\hat{L}(f))^2 \tfrac{1}{d} \tr(K)
}
\zeta$
for $\zeta < 1-\zeta_0$,
then with overwhelming probability
\[
b_t \leq \|K\|_\sigma^2 \mathrsfs{D}^2(t) \le 2\|K\|_\sigma^2
\max\{\|X_0-X^{\star}\|^2, \|X^{\star}\|^2\}
\defas I.
\]
So provided that 
\[
10\exp(-\hat{\theta}(f)/(4I))) < \zeta_0,
\]
the stopping time $\vartheta = \infty$, i.e., never occurs.
\end{proof}

We need the following Gaussian lemma.
\begin{lemma}\label{lem:GVN}
  If $Z \sim N(0, \Id)$, $A$ is a $d\times d$ matrix, and $X = \|AZ\|^2$.
  Then with $V = \Exp X$ and for any $u \geq 0$
  \[
    \Exp( X \cdot 1\{ X \geq u\})
    \leq 5Ve^{-u/(4V)}
    \quad
    \text{and}
    \quad
    \Pr( X \geq u)
    \leq 2e^{-u/(4V)}.
  \]
\end{lemma}
\begin{proof}
  By rotation invariance of the Gaussian, we may assume $A = \operatorname{diag}( a_j : 1 \leq j \leq d).$
  Then provided $\lambda < 1/a_j^2$ for all $j$,
  \[
    \Exp e^{\lambda X} =
    \prod_{j=1}^d \tfrac{1}{\sqrt{1-2\lambda a_j^2}}.
  \]
  Taking $\lambda = 1/(4\sum a_j^2)$ and using that for $x \leq \tfrac12$, we have 
  \(
    \tfrac{1}{\sqrt{1-x}} \leq e^{x},
  \)
  and 
  we conclude
  \[
    \Exp e^{\lambda X} \leq e^{1/2}.
  \]
  Thus we have $\Pr(X \geq t) \leq e^{-\lambda t+ 1/2}$ for all $t \geq 0.$
  %and hence we havefor all $t \geq 0$
  %\[
  %  \Pr( X-\Exp X \geq 2(\Exp X)t) \leq e^{-t}.
  %\]
  Hence from integration by parts
  \[
  \Exp( X \cdot 1\{ X \geq u\})
  \leq
  u e^{-\lambda u +1/2} + \int_u^\infty e^{-\lambda x + 1/2} \dif x
  \leq (u+\tfrac{1}{\lambda})e^{-\lambda u + 1/2}.
  \]

\end{proof}

\appendix
\section{Integro-Differential Equation Analysis} \label{sec:Volterra_equation}
In this section, we provide some alternative characterization for the solution to the integro-differential equation \eqref{eq:ODE_resolvent_2}. We recall below the formula. 

\begin{mdframed}[style=exampledefault]
\textbf{Integro-Differential Equation for $\mathcal{S}(t, z)$.} For any contour $\Gamma \subset \mathbb{C}$ enclosing the eigenvalues of $K$, we have an expression for the derivative of $\mathcal{S}$:
\begin{equation}
    \dif \mathcal{S}(t,\cdot) 
    = \mathscr{F}(z, \mathcal{S}(t, \cdot)) \, \dif t
\end{equation}
\begin{align}
    \text{where} \, \, 
    \mathscr{F}(z, \mathcal{S}(t, \cdot))
    &
    \defas
    - 2\gamma_t \bigg ( \bigg ( \frac{-1}{2\pi i} \oint_{\Gamma} \mathcal{S}(t,z) \, \dif z \bigg ) H( \mathrsfs{B}(t)) + H^T(\mathrsfs{B}(t)) \bigg ( \frac{-1}{2\pi i} 
 \oint_{\Gamma} \mathcal{S}(t,z) \, \dif z \bigg ) \bigg ) \,  \nonumber
 \\
    & \qquad
    + \frac{\gamma_t^2}{d}\left [ \begin{array}{c|c} 
    \tr(K R(z;K)) I(\mathrsfs{B}(t)) & 0\\ \hline 0 & 0
    \end{array} \right ] \\
    & \qquad 
    - \gamma_t (\mathcal{S}(t,z) (2z H(\mathrsfs{B}(t)) + \delta D) + ( 2 z H^T( \mathrsfs{B}(t) ) + \delta D) \mathcal{S}(t,z)). \nonumber
\end{align}
\begin{gather*}
\text{Here} \, \, \mathrsfs{B}(t) = \frac{-1}{2\pi i} \oint_{\Gamma} z \mathcal{S}(t,z) \, \dif z, 
 \quad 
 H(\mathrsfs{B}) = \left [ \begin{array}{c|c} 
\nabla h_{11}(\mathrsfs{B}) & 0\\
 \hline
 \nabla h_{21}(\mathrsfs{B}) & 0
 \end{array} \right ],  \quad \text{and} \quad D = \left [ \begin{array}{c|c} 
 I_{\mathcal{O}} & 0\\
 \hline
 0 & 0
 \end{array} \right ], \\
 \text{and initialization} \quad \mathcal{S}(0,z) = \ip{W_0 \otimes W_0, R(z; K)}_{\mathcal{A}^{\otimes 2}}. 
\end{gather*}
\end{mdframed}

We can derive a Volterra equation for $\mathcal{S}$. Now we let $\Phi$ be the fundamental matrix for the ODE:
\[
\dot{\Phi} = \gamma_t ( 2 z H(\mathrsfs{B}(t)) + \delta D) \Phi, \quad \Phi(0) = I_{\mathcal{O}^+}.
\]
Then it follows that $\dot{\Phi}^{-T} = -\gamma_t (2 z H^T(\mathrsfs{B}(t)) + \delta D) \Phi^{-T}$. Defining, 
\begin{align*}
U_0(t) 
&
\defas
- 2\gamma_t \bigg ( \bigg ( \frac{-1}{2\pi i} \oint_{\Gamma} \mathcal{S}(t,z) \, \dif z \bigg ) H( \mathrsfs{B}(t)) + H^T(\mathrsfs{B}(t)) \bigg ( \frac{-1}{2\pi i} 
 \oint_{\Gamma} \mathcal{S}(t,z) \, \dif z \bigg ) \bigg ) \,  \nonumber
 \\
    & \qquad
    + \frac{\gamma_t^2}{d}\left [ \begin{array}{c|c} 
    \tr(K R(z;K)) I(\mathrsfs{B}(t)) & 0\\ \hline 0 & 0
    \end{array} \right ] 
\end{align*}
then we observe that the ODE in \eqref{eq:ODE_resolvent_2} becomes
\begin{align*} 
    \dot{(\Phi^T \mathcal{S} \Phi)} 
    &
    = 
    \dot{\Phi^T} \mathcal{S} \Phi + \Phi^T \dot{\mathcal{S}} \Phi + \Phi^T \mathcal{S} \dot{\Phi} 
    \\
    &
    =
    \gamma_t \Phi^T (2z H^T + \delta D) \mathcal{S} \Phi + \Phi^T [ U_0 - \gamma_t (\mathcal{S} (2zH + \delta D) + (2z H^T + \delta D) \mathcal{S}) ] \Phi
    \\
    & \quad + \gamma_t \Phi^T \mathcal{S} (2z H + \delta D) \Phi 
    \\
    &
    = 
    \Phi^T U_0 \Phi. 
\end{align*}
This ODE is, of course, solvable, and thus we get that $\mathcal{S}$ satisfies the equation below.

\begin{mdframed}[style=exampledefault]
\textbf{Resolvent formula.}
\begin{equation}
\label{eq:resolvent_form}
\mathcal{S}(t,z) = \Phi^{-T}(t,z) \mathcal{S}(0,z) \Phi^{-1}(t,z) + \int_0^t \Phi^{-T}(t,z) \Phi^T(s,z) U_0(s, z) \Phi(s, z) \Phi^{-1}(t,z) \, \dif s
\end{equation}

\begin{gather*}
\text{where} \quad
% \quad \mathcal{S}(t,z) \cong \left[ \begin{array}{c|c} 
% \WHSGD_t^T R(z;K) \WHSGD_t & \WHSGD_t^T R(z;K) X^{\star}\\
% \hline
% (X^{\star})^T R(z;K) \WHSGD_t & (X^{\star})^T R(z;K) X^{\star}
% \end{array}
%  \right ], 
%  \\
 \mathcal{S}(0,z) = \ip{W_0^{\otimes 2}, R(z;K)}_{\mathcal{A}^{\otimes 2}},  \quad 
 H(B) = \left [ \begin{array}{c|c} 
 \nabla h_{11}(B) & 0\\
 \hline
 \nabla h_{21}(B) & 0
 \end{array} \right ], \quad D = \left [ \begin{array}{c|c} 
 I & 0\\
 \hline
 0 & 0
 \end{array} \right ]
 \\
 \text{$\Phi(t,z)$ is the solution to $\dot{\Phi} = \gamma_t ( 2 z H(\mathrsfs{B}(t)) + \delta D) \Phi$ \, \ with \, \, $\Phi(0, z) = I_{\mathcal{O}^+}$,} \\
 \text{and} \quad 
 U_0(t,z) 
 = - 2\gamma_t \bigg ( \bigg ( \frac{-1}{2\pi i} \oint_{\Gamma} \mathcal{S}(t,z) \, \dif z \bigg ) H( \mathrsfs{B}(t)) + H^T(\mathrsfs{B}(t)) \bigg ( \frac{-1}{2\pi i} 
 \oint_{\Gamma} \mathcal{S}(t,z) \, \dif z \bigg ) \bigg ) \,  \nonumber
 \\
    \qquad
    + \frac{\gamma_t^2}{d}\left [ \begin{array}{c|c} 
    \tr(K R(z;K)) I(\mathrsfs{B}(t)) & 0\\ \hline 0 & 0
    \end{array} \right ].
 \end{gather*}
\end{mdframed}

As one can see, this requires that one be able to solve the ODE, $\dot \Phi = \gamma_t (2 z H(\mathrsfs{B}(t)) + \delta D)\Phi$. This, in general, has no closed form solution when $H$ is not a $2 \times 2$ matrix (i.e., scalar setting where $\ip{X,a}_{\mathcal{A}}, \ip{X^{\star},a}_{\mathcal{A}} \in \mathbb{R}$). In some cases, there is a general solution to $\Phi$ especially when $H$ is a constant matrix, as in least squares. In the next section, we focus on the scalar setting. 

\paragraph{Scalar setting.} We restrict to the setting where $\ip{X,a}_{\mathcal{A}} \in \mathbb{R}$ and $\ip{X^{\star},a}_{\mathcal{A}} \in \mathbb{R}$, that is, where $X$ and $X^{\star}$ are vectors. To derive the deterministic dynamics of the risk function $\mathcal{R}(X)$, we introduce
\[
\mathrsfs{R}(t) = h \circ \mathrsfs{B}(t), \quad \text{where $\mathrsfs{B}(t) = \frac{-1}{2 \pi i} \oint_{\Gamma} z \mathcal{S}(t,z) \, \dif z$.}
\]

In the scalar setting, we will simplify the equations for the resolvent formula and show that $\mathrsfs{B}(t)$ solves Volterra equation. By solving this Volterra equation, one can derive the deterministic dynamics of the risk function $\mathrsfs{R}$ by applying the function $h$. We can do this because in the scalar setting, the ODE for $\Phi$ decouples into 2 first-order linear ODEs. First-order linear ODEs have an explicit formula via the integrating factor.   

\begin{mdframed}[style=exampledefault]
\textbf{Evolution of $\mathrsfs{B}(t) = \frac{-1}{2 \pi i} \oint_{\Gamma} z \mathcal{S}(t,z) \, \dif z$.}
In the scalar setting, we will be able to give a more explicit formula for the function $\mathrsfs{B}(t)$, that is, we will show 
\begin{equation} \label{eq:C_scalar}
    \begin{aligned}
        \mathrsfs{B}(t) 
        &
        = \begin{bmatrix}
            \mathrsfs{B}_{11}(t) & \mathrsfs{B}_{12}(t)
            \\
            \mathrsfs{B}_{21}(t) & \mathrsfs{B}_{22}(t)
        \end{bmatrix},
        \\
        \text{where} \quad \mathrsfs{B}_{11}(t) 
        &
        = 
        X_0^T \tfrac{K}{\Phi_{11}^{2}(t,K)}X_0 - 2 X_0^T \tfrac{K\Phi_{21}(t, K)}{\Phi_{11}^{2}(t, K)}  X^{\star} + (X^{\star})^T \tfrac{K \Phi_{21}^2(t, K)}{\Phi_{11}^{2}(t,K)}  X^{\star}  
        \\
        &
        + \frac{1}{d} \int_0^t \gamma_s^2  I(\mathrsfs{B}(s)) \tr \big ( K^2 \tfrac{\Phi_{11}^2 (s, K)}{\Phi_{11}^2(t, K)} \big ) \, \dif s,
        \\
        \mathrsfs{B}_{12}(t)
        &
        =
        X_0^T \tfrac{K}{\Phi_{11}(t,K)} X^{\star} - (X^{\star})^T \tfrac{K}{\Phi_{11}(t,K)} X^{\star},
        \\
        \mathrsfs{B}_{21}(t) 
        &
        =
        \mathrsfs{B}_{12}^T(t), \quad \text{and} \quad
        \mathrsfs{B}_{22}(t) = (X^{\star})^T K X^{\star}. 
    \end{aligned}
\end{equation}
The function $\Phi_{11}(t, z)$ and $\Phi_{21}(t,z)$, by solving a differential equation, are given by 
\begin{equation}
\begin{aligned}
     \Phi_{11}(t,z) 
     &
     = \exp \left ( \int_0^t \gamma_s ( 2z \nabla h_{11}(\mathrsfs{B}(s)) + \delta) \, \dif s \right ) \, \, 
     \\
     \text{and} \, \,
    \Phi_{21}(t, z) 
    &
    = \int_0^t 2 \gamma_s z \nabla h_{21}(\mathrsfs{B}(s)) \Phi_{11}(s, z) \, \dif s. 
    \end{aligned}
\end{equation}
\end{mdframed}

To this end, we need to solve the expression for $\mathcal{S}(t, z)$, in \eqref{eq:ODE_resolvent_2}. The most challenging part, of course, is solving the linear ODE that arises in the computation of $\Phi$, that is, 
\begin{equation} \label{eq:phi_linear_scalar_ode}
\dot{\Phi} = \gamma(t) \begin{bmatrix}
    2z\nabla h_{11}(\mathrsfs{B}(t)) + \delta & 0\\
    2z\nabla h_{21}(\mathrsfs{B}(t)) & 0
\end{bmatrix} \Phi, \qquad \Phi(0) = I. 
\end{equation}
In the scalar case, we can do so since each term of $\Phi$ reduces down to a system of first-order linear ODE:  
\begin{equation} \label{eq:ODE_system_1}
    \begin{aligned}
        \dot{\Phi}_{11} &= \gamma_t (2z \nabla h_{11}(\mathrsfs{B}(t)) + \delta) \Phi_{11}, \quad \Phi_{11}(0) = 1\\
        \dot{\Phi}_{21} &= 2 \gamma_t z \nabla h_{21}(\mathrsfs{B}(t)) \Phi_{11}, \quad \Phi_{21}(0) = 0 \\
    \end{aligned}
\end{equation}
Note that the differential equation for $\Phi_{12}$ ($\Phi_{22}$) is the same as $\Phi_{11}$ ($\Phi_{21}$) but with different initial condition, $\Phi_{12}(0) = 0 $ ($\Phi_{22}(0) = 1$), respectively. 

This system decouples so that $\Phi_{11}$ is a scalar 1st-order linear ODE; therefore we can use an integrating factor to get give an explicit solution. The system of ODEs \eqref{eq:ODE_system_1} becomes
\begin{equation}
\begin{gathered}
    \Phi_{11}(t,z) = \exp \left ( \int_0^t \gamma_s (2z \nabla h_{11}(\mathrsfs{B}(s)) + \delta) \, \dif s \right ), \qquad 
    \Phi_{21}(t, z) = \int_0^t 2 \gamma_s z \nabla h_{21}(\mathrsfs{B}(s)) \Phi_{11}(s, z) \, \dif s,
    \\
    \Phi_{22}(t, z)  = 1, \quad \text{and} \quad \Phi_{12}(t , z) = 0.
\end{gathered}
\end{equation}
As $\Phi$ is a $2 \times 2$ matrix, we can give an explicit representation for its inverse
\begin{equation}
    \Phi^{-1}(t, z) = \frac{1}{\Phi_{11}(t, z)} \begin{bmatrix}
        1 & 0 \\
        -\Phi_{21}(t, z) & \Phi_{11}(t, z)
    \end{bmatrix}. 
\end{equation}
Now it is a matter of computing the quantities in \eqref{eq:resolvent_form} using the solution of $\Phi$, e.g., 
\begin{equation}
    \begin{aligned}
        &\Phi^{-T}(t,z) S(0, z) \Phi^{-1}(t,z) \\
        &
        =
        \frac{1}{\Phi_{11}^2(t,z)} \begin{bmatrix}
            \substack{X_0^T R(z; K) X_0\\ - 2 \Phi_{21}(t,z) X_0^T R(z;K) X^{\star} \\+ \Phi_{21}^2(t,z) (X^{\star})^T R(z; K) X^{\star}}
            & \substack{\Phi_{11}(t,z) X_0^T R(z;K) X^{\star}\\ - \Phi_{11}(t,z) \Phi_{21}(t,z) (X^{\star})^T R(z; K) X^{\star}   }
            \\
            \star & \Phi_{11}^2(t,z) (X^{\star})^T R(z; K) X^{\star}
        \end{bmatrix}.
    \end{aligned}
\end{equation}
Furthermore, we also have (via a simple computation), 
\begin{equation}
    \begin{aligned}
        \Phi(s) \Phi^{-1}(t) 
        &
        = 
        \begin{bmatrix}
            \frac{\Phi_{11}(s)}{\Phi_{11}(t)} & 0\\0 & 1
        \end{bmatrix},
    \end{aligned}
\end{equation}
and thus, we get that
\begin{equation*}
    \begin{aligned}
        \Phi^{-T}(t) \Phi^T(s) 
        &
        \left [ \begin{array}{c|c} 
    \frac{\gamma(t)^2}{d} \tr(K R(z;K)) I(\mathrsfs{B}(t)) & 0\\ \hline 0 & 0
    \end{array} \right ] \Phi(s) \Phi^{-1}(t) 
        \\
        &
        =
        \frac{\gamma(t)^2}{d} I(\mathrsfs{B}(t)) \begin{bmatrix}
        \frac{\Phi_{11}^2(s)}{\Phi_{11}^2(t)} \tr(K R(z;K)) & 0 \\
            0 & 0
        \end{bmatrix}.
    \end{aligned}
\end{equation*}
By setting $V_0(t) = \frac{-1}{2\pi i} \oint_{\Gamma} \mathcal{S}(t,z) \, \dif z$, it follows that 
\begin{equation} \label{eq:deep_appendix}
    \begin{aligned}
        \mathcal{S}(t,z) 
        &
        = 
        \frac{1}{\Phi_{11}^2(t,z)} \begin{bmatrix}
            \substack{X_0^T R(z; K) X_0\\ - 2 \Phi_{21}(t,z) X_0^T R(z;K) X^{\star} \\+ \Phi_{21}^2(t,z) (X^{\star})^T R(z; K) X^{\star}}
            & \substack{\Phi_{11}(t,z) X_0^T R(z;K) X^{\star}\\ - \Phi_{11}(t,z) \Phi_{21}(t,z) (X^{\star})^T R(z; K) X^{\star}   }
            \\
            \star & \Phi_{11}^2(t,z) (X^{\star})^T R(z; K) X^{\star}
        \end{bmatrix}
        \\
        & 
        \qquad - 2\gamma_t (V_0 H(\mathrsfs{B}(t)) + H^T(\mathrsfs{B}(t)) V_0)
        +
        \frac{1}{d} \int_0^t \gamma_s^2 I(\mathrsfs{B}(s)) \begin{bmatrix}
            \frac{\Phi_{11}^2(s)}{\Phi_{11}^2(t)} \tr(K R(z;K)) & 0 \\
            0 & 0
        \end{bmatrix} \, \dif s. 
    \end{aligned}
\end{equation}
We now apply Cauchy's integral formula to $z \mathcal{S}(t,z)$, in that,  $\mathrsfs{B}(t) = - \tfrac{1}{2\pi i} \oint z\mathcal{S}(t,z) \, \dif z$. We see that the term $- 2\gamma_t  z(V_0 H(\mathrsfs{B}(t)) + H^T(\mathrsfs{B}(t)) V_0)$ is analytic in $z$ ($V_0$ and $H$ do not depend on $z$). Therefore, this term, when Cauchy's integral formula is applied to it, is $0$. The result \eqref{eq:C_scalar} immediately follows. 

Piggybacking on the solution of $\mathrsfs{B}$ via the Volterra equation expression, we can derive the dynamics of any statistic satisfying Assumption~\ref{assumption:statistic}, we simply need to derive an expression for the following quantity 
\[
\mathcal{Q}(t) \defas \frac{-1}{2 \pi i} \oint_{\Gamma} q(z) \mathcal{S}(t,z) \, \dif z, 
\]
as one can recover the deterministic statistics dynamics of SGD/HSGD by
\[
\phi(t) = g \circ \mathcal{Q}(t). 
\]

Having derived an equation for $\mathcal{S}$ in \eqref{eq:deep_appendix}, we can get $\mathcal{Q}(t)$ by Cauchy's integral formula. The result is below. 

\begin{mdframed}[style=exampledefault]
\textbf{Evolution of $\mathcal{Q}(t) = \frac{-1}{2 \pi i} \oint_{\Gamma} q(z) \mathcal{S}(t,z) \, \dif z$.}
In the scalar setting, piggybacking off of the Volterra equation for $\mathrsfs{B}$, we will be able to give a more explicit formula for the function $\mathcal{Q}(t)$, that is, we show 
\begin{equation} 
    \begin{aligned}
        \mathcal{Q}(t) 
        &
        = \begin{bmatrix}
            \mathcal{Q}_{11}(t) & \mathcal{Q}_{12}(t)
            \\
            \mathcal{Q}_{21}(t) & \mathcal{Q}_{22}(t)
        \end{bmatrix},
        \\
        \text{where} \quad \mathcal{Q}_{11}(t) 
        &
        = 
        X_0^T \tfrac{q(K)}{\Phi_{11}^{2}(t,K)}X_0 - 2 X_0^T \tfrac{q(K) \Phi_{21}(t, K)}{\Phi_{11}^{2}(t, K)}  X^{\star} + (X^{\star})^T \tfrac{q(K) \Phi_{21}^2(t, K)}{\Phi_{11}^{2}(t,K)}  X^{\star}  
        \\
        &
        + \frac{1}{d} \int_0^t  \gamma_s^2 I(\mathrsfs{B}(s)) \tr \big ( K^2 \tfrac{\Phi_{11}^2 (s, K)}{\Phi_{11}^2(t, K)} \big ) \, \dif s,
        \\
        \mathcal{Q}_{12}(t)
        &
        =
        X_0^T \tfrac{q(K)}{\Phi_{11}(t,K)} X^{\star} - (X^{\star})^T \tfrac{q(K)}{\Phi_{11}(t,K)} X^{\star},
        \\
        \mathcal{Q}_{21}(t) 
        &
        =
        \mathcal{Q}_{12}^T(t), \quad \text{and} \quad
        \mathcal{Q}_{22}(t) = (X^{\star})^T K X^{\star}. 
    \end{aligned}
\end{equation}
The function $\Phi_{11}(t, z)$ and $\Phi_{21}(t,z)$, by solving a differential equation, are given by 
\begin{equation}
\begin{aligned}
     \Phi_{11}(t,z) 
     &
     = \exp \left ( \int_0^t \gamma_s ( 2z \nabla h_{11}(\mathrsfs{B}(s)) + \delta) \, \dif s \right ) \, \, 
     \\
     \text{and} \, \,
    \Phi_{21}(t, z) 
    &
    = \int_0^t 2 \gamma_s z \nabla h_{21}(\mathrsfs{B}(s)) \Phi_{11}(s, z) \, \dif s. 
    \end{aligned}
\end{equation}
\end{mdframed}

\section{Analysis of Examples} \label{sec:analysis_examples}

In this section, we derive the function $h$ (and its derivative), $f$ (and its derivative, as well as $\EE_{a, \epsilon}[\nabla_x f(\ip{W,a}_{\mathcal{A}})^{\otimes 2}]$. We do so in the case when the learning rate is constant $\gamma$. These quantities are exactly what you need to solve the Volterra equation for $\mathrsfs{B}$. From $\mathrsfs{B}$, one can derive other statistics, particularly important are the statistics corresponding to the norm $\ip{X^{\otimes 2}, K}_{\mathcal{A}^{\otimes 2}}$ and cross term $\ip{X, \ip{K, X^{\star}}_{\mathcal{A}} }_{\mathcal{A}}$.

Throughout this section, we use the notation
\[
\mathrsfs{B}(t) = \begin{bmatrix}
\mathrsfs{B}_{11}(t) & \mathrsfs{B}_{12}(t)\\
\mathrsfs{B}_{21}(t) &
\mathrsfs{B}_{22}(t) 
\end{bmatrix} = \frac{-1}{2 \pi i} \oint_{\Gamma} z \mathcal{S}(t,z) \, \dif z. 
\]
The correspondence of $\mathrsfs{B}$ with iterates is given by
\[
\mathrsfs{B}(t) \approx \ip{W^{\otimes 2}, K}_{\mathcal{A}^{\otimes 2}}. 
\]

\subsection{Example 1: Least squares (matrix outputs)} \label{sec:lsq_analysis}
We consider the dynamics of the least squares (with matrix outputs) in which we are interested in minimizing $X \in \mathcal{A} \otimes \mathcal{O}$ over the risk function,
\begin{equation}
\begin{aligned}
\mathcal{R}(X) 
&
\defas 
\tfrac{1}{2} \EE_{a, \epsilon} [\|\ip{X,a}_{\mathcal{A}} - (\ip{X^{\star}, a}_{\mathcal{A}} + \varepsilon) \|^2]
\\
&= 
\tfrac{1}{2} \EE[ \ip{ \ip{ X-X^{\star}, a}_{\mathcal{A}}, \ip{ X-X^{\star}, a}_{\mathcal{A}} } ] + \tfrac{1}{2} \EE[ \|\epsilon\|^2 ] \\ 
&
= 
\tfrac{1}{2} \tr \big ( \ip{K, (X-X^{\star}) \otimes (X-X^{\star}) }_{\mathcal{A} \otimes \mathcal{A} } \big ) + \tfrac{1}{2} \EE[ \|\epsilon\|^2 ]\\
&
=
\tfrac{1}{2} \tr \big ( \ip{X \otimes X, K}_{\mathcal{A}} \big ) - \tfrac{1}{2} \tr \big ( \ip{X \otimes X^{\star}, K}_{\mathcal{A}} \big ) - \tfrac{1}{2} \tr \big ( \ip{ X^{\star} \otimes X, K}_{\mathcal{A}} \big )
\\
& \quad + \tfrac{1}{2} \tr \big ( \ip{K, X^{\star} \otimes X^{\star}}_{\mathcal{A}}  \big ) + \tfrac{1}{2} \EE[\|\epsilon\|^2]
\end{aligned}
\end{equation}
Here we assume that the targets $y = \ip{ X^{\star}, a}_{\mathcal{A}} + \epsilon$ where $\varepsilon$ is independent of $a$ and the expectation is taken over both the label noise $\epsilon$ and the data $a$. 

The function $h \, : \, \mathcal{O}^+ \otimes \mathcal{O}^+ \to \mathbb{R}$ must satisfy $h ( \ip{K, W \otimes W}_{\mathcal{A}} ) = \mathcal{R}(X)$. For this we make the identification, 
\begin{equation*}
\begin{gathered}
z_{11} = X^T K X, \quad z_{12} = X^T K X^{\star},\\
z_{21} = (X^{\star})^T K X, \quad \text{and} \quad z_{22} = (X^{\star})^T K X^{\star}. 
\end{gathered}
\end{equation*}
Under this identification, 
\[
h \left ( \begin{bmatrix} 
z_{11} & z_{12}\\
z_{21} & z_{22}
\end{bmatrix} \right )
= \tfrac{1}{2} \tr(z_{11}) - \tfrac{1}{2} \tr(z_{12}) - \tfrac{1}{2}\tr(z_{21}) + \tfrac{1}{2} \tr(z_{22}) + \EE[\|\varepsilon\|^2]. 
\]
As $(\nabla \tr)(x) = I$ (here $\mathcal{T} = \mathcal{O}$), we get that 
\begin{equation} \label{eq:Dh_linear}
\nabla h(B) = \left [ \begin{array}{c|c} 
\nabla h_{11}(B) & \nabla h_{12}(B)\\
\hline
\nabla h_{21}(B) & \nabla h_{22}(B)
\end{array} \right ] 
=
\left [
\begin{array}{cc}
\tfrac{1}{2} I_{\mathcal{O}} & -\tfrac{1}{2} I_{\mathcal{O}}\\
-\tfrac{1}{2} I_{\mathcal{T}} & \tfrac{1}{2} I_{\mathcal{T}}
\end{array}
\right ]. 
\end{equation}
Hence, we conclude that 
\begin{align*}
H(\mathrsfs{B}(t)) = 
\begin{bmatrix}
\tfrac{1}{2} I & 0\\
-\tfrac{1}{2} I & 0
\end{bmatrix}.
\end{align*}

Moreover, we also need to identify the function $f$, which in this case is simply $r \mapsto \tfrac{1}{2} \| r -  (\ip{X^{\star},a}_{\mathcal{A}} + \epsilon)\|^2$. The derivative, 
\[
\nabla_x f(x) = x - (\ip{X^{\star}, a}_{\mathcal{A}} + \epsilon), 
\]
satisfies evaluated at $r= \ip{X,a}_{\mathcal{A}}$
\[
\EE_{a, \epsilon} [ \nabla_x f(\ip{X,a}_{\mathcal{A}})^{\otimes 2} ] = \ip{K, X-X^{\star} \otimes X-X^{\star}}_{\mathcal{A}} +  \EE[ \epsilon^{\otimes 2} ].
\]
Thus, it follows that 
\begin{align*}
I(\mathrsfs{B}(t)) = \mathrsfs{B}_{11}(t) - \mathrsfs{B}_{12}(t) - \mathrsfs{B}_{21}(t) + \mathrsfs{B}_{22}(t) + \EE[\epsilon^{\otimes 2}]. 
\end{align*}

We now have all the components to find $\mathcal{S}(t,z)$ in \eqref{eq:resolvent_form}. One of the most challenging components to get an explicit formula is being able to solve the ODE 
\[
\dot{\Phi}(t,z) = 2\gamma z H(\mathrsfs{B}) \Phi, \quad \text{where} \quad \Phi(0,z) = I \quad \text{and} \quad H(\mathrsfs{B}) = \left [ \begin{array}{c|c} 
\nabla h_{11}(\mathrsfs{B}(t)) & 0\\
\nabla h_{21}(\mathrsfs{B}(t)) & 0
\end{array} \right ].
\]
In this case, because $\nabla h$ is quite simple, that is composed of identities (see \eqref{eq:Dh_linear}), we can solve the constant coefficient system of ODEs:
\begin{equation*}
\begin{gathered}
    \dot{\Phi} = 2 \gamma z \begin{bmatrix}
        \tfrac{1}{2} I & 0\\
        -\tfrac{1}{2} I & 0
    \end{bmatrix} \Phi,  \qquad \Phi(0) = I,
\end{gathered}
\end{equation*}
where the matrix $H$ diagonalized by
\[
   \begin{bmatrix}
        I & 0\\
        -I & 0
    \end{bmatrix} = \begin{bmatrix} 
    0 & -I\\
    I & I
    \end{bmatrix}
    \begin{bmatrix}
        0 & 0\\
        0 & I
    \end{bmatrix}    
    \begin{bmatrix}
        I & I\\
        -I & 0
    \end{bmatrix}. 
\]
The solution $\Phi(t,z)$ is simply given by taking the exponential and thus, 
\[
\Phi(t,z) = \begin{bmatrix}
    e^{\gamma z t} I & 0\\
    (1-e^{\gamma z t}) I & I
\end{bmatrix} \quad \text{and} \quad \Phi^{-1}(t,z) = \begin{bmatrix}
    e^{-\gamma zt} I & 0\\
    (1-e^{-\gamma zt})I & I
\end{bmatrix}. 
\]
A simple computation yields that 
\begin{align*}
    [\Phi^{-T}(t,z) \mathcal{S}(0,z) \Phi^{-1}(t,z)]_{11} 
    &
    = 
    e^{-2 \gamma z t}(X_0^T R(z;K) X_0) + e^{-\gamma z t}  (1-e^{-\gamma z t}) X_0^T R(z;K) X^{\star}
    \\
    &
    + (1-e^{-\gamma z t})e^{-\gamma z t} (X^{\star})^T R(z; K) X_0 + (1-e^{-\gamma z t})^2 (X^{\star})^T R(z; K) X^{\star}\\
    [\Phi^{-T}(t,z) \mathcal{S}(0,z) \Phi^{-1}(t,z)]_{12}
    &
    =
    e^{-\gamma z t} X_0^T R(z;K) X^{\star} + (1-e^{-\gamma z t}) (X^{\star})^T R(z;K) X^{\star}
    \\
    [\Phi^{-T}(t,z) \mathcal{S}(0,z) \Phi^{-1}(t,z)]_{21} 
    &
    =
    [\Phi^{-T}(t,z) \mathcal{S}(0,z) \Phi^{-1}(t,z)]_{12}^T\\
    [\Phi^{-T}(t,z) \mathcal{S}(0,z) \Phi^{-1}(t,z)]_{22}
    &
    = (X^{\star})^T R(z;K) X^{\star}. 
\end{align*}
and, we have that
\[
\Phi(s) \Phi^{-1}(t) = \begin{bmatrix}
    e^{-\gamma z (t-s)} & 0\\
    (1-e^{\gamma z s})e^{-\gamma z t} + (1-e^{-\gamma z t}) & 1
\end{bmatrix}.
\]
Using this term, we get that 
\begin{equation*}
\begin{aligned}
\frac{\gamma^2}{d} 
& 
\Phi^{-T}(t,z) \Phi^T(s,z) \begin{bmatrix} 
\tr(K R(z; K)) I(\mathrsfs{B}(t)) & 0\\
0 & 0
\end{bmatrix} \Phi(s,z) \Phi^{-1}(t,z) 
\\
&
=
\frac{\gamma^2}{d} \tr(K R(z;K)) e^{-2\gamma z(t-s)} \begin{bmatrix}
    \mathrsfs{B}_{11}(t) - \mathrsfs{B}_{12}(t) - \mathrsfs{B}_{21}(t) + \mathrsfs{B}_{22}(t) + \EE[\epsilon^{\otimes 2}] & 0\\
    0 & 0
\end{bmatrix}.
\end{aligned}
\end{equation*}
We can recover the $\mathrsfs{B}(t)$ and hence the risk $\mathcal{R}(X)$ by Cauchy's integral formula, that is, $-\frac{1}{2\pi i} \oint z \mathcal{S}(t,z) \, dz = \mathrsfs{B}(t)$. Note that the term $-2\gamma \big ( (\tfrac{-1}{2 \pi i} \oint_{\Gamma} \mathcal{S}(t,z) \, \dif z)  H(\mathrsfs{B}) + H^T(\mathrsfs{B}) (\tfrac{-1}{2 \pi i} \oint_{\Gamma} \mathcal{S}(t,z) \, \dif z) \big ) $ is analytic in $z$ and thus will integrate $0$ when performing the contour integral. Doing this contour integral, we get that 
\begin{align*}
\mathrsfs{B}_{11}(t)
&
=
 X_0^T e^{-2 \gamma K t} K X_0 +  X_0^T e^{-\gamma K t} (1-e^{-\gamma K t}) K X^{\star} \\
&
+  (X^{\star})^T K(1-e^{-\gamma K t})e^{-\gamma K t} X_0 +  (X^{\star})^T (1-e^{-\gamma K t})^2 K X^{\star}\\
&
+ \frac{\gamma^2}{d} \int_0^t \tr( K^2 e^{-2\gamma K(t-s)} ) \big ( \mathrsfs{B}_{11}(t) - \mathrsfs{B}_{12}(t) - \mathrsfs{B}_{21}(t) + \mathrsfs{B}_{22}(t) + \EE[\epsilon^{\otimes 2}] \big ) \, \dif s
\\
\mathrsfs{B}_{12}(t)
&
=
X_0^T K e^{-\gamma K t} X^{\star} + (X^{\star})^T K (1-e^{-\gamma K t}) X^{\star}
\\
\mathrsfs{B}_{21}(t)
& 
= 
\mathrsfs{B}_{12}(t).
\end{align*}
We note that 
\[
2\mathrsfs{R}(t) = 
\mathrsfs{B}_{11}(t) - \mathrsfs{B}_{12}(t) - \mathrsfs{B}_{21}(t) + \mathrsfs{B}_{22}(t) + \EE[\epsilon^{\otimes 2}]. 
\]
Then we can get a formula for the deterministic dynamics of the risk $\mathrsfs{R}$:
\begin{equation}
\begin{aligned}
\mathcal{R}(W_{td}) \to \mathrsfs{R}(t)
% = 
% \tfrac{1}{2} \tr( \mathrsfs{B}(t)_{11} - \mathrsfs{B}(t)_{12} - \mathrsfs{B}(t)_{21} + \mathrsfs{B}(t)_{22} ) + \tfrac{1}{2} \EE[ \| \varepsilon \|^2 ]
% \\
& 
= 
\tfrac{1}{2} \tr( \ip{(X_0-X^{\star}) \otimes (X_0-X^{\star}), K e^{-2K \gamma t} }_{\mathcal{A}^{\otimes 2}} )  + \tfrac{1}{2} \EE[ \| \epsilon \|^2 ] \\
&
\qquad + \frac{\gamma^2}{d} \int_0^t \tr(K^2 e^{-2 \gamma K(t-s)} \mathrsfs{R}(s) \, \dif s.
\end{aligned}
\end{equation}

\subsection{Example 2: (Real) Phase Retrieval}

In the (real) phase retrieval problem, we are trying to find an unknown signal $X^{\star}$ from linear observations of the modulus of the signal, that is, the target is $y = \|\ip{ X^{\star}, a}_{\mathcal{A}}\|^2$. For this setting, we will consider $\ip{X^{\star}, a}_{\mathcal{A}} \in \mathbb{R}$, the scalar setting. The (noiseless) phase retrieval problem can be formulated as 

\[
\min_{X} \EE_{a} [ \big ( (\ip{X,a}_{\mathcal{A}})^2 - (\ip{X^{\star}, a}_{\mathcal{A}})^2 \big )^2 ].
\]

To apply our result, we need to identify the functions $h$ and $f$. Let's first compute the function $h$. For this, we need to use Wick's formula:
\begin{equation}
\begin{aligned}
    \EE_a[ (\ip{ X,a}_{\mathcal{A}}^2  - \ip{ X^{\star},a}_{\mathcal{A}}^2 )^2 ]
    &
    =
    3 \ip{X \otimes X, K}_{\mathcal{A}^{\otimes 2}}^2 - 2 \ip{X \otimes X, K}_{\mathcal{A}^{\otimes 2}} \ip{X^{\star} \otimes X^{\star}, K}_{\mathcal{A}^{\otimes 2}}
    \\
    &
    \quad 
    -4 \ip{X \otimes X^{\star}, K}_{\mathcal{A}^{\otimes 2}} \ip{X^{\star} \otimes X, K}_{\mathcal{A}^{\otimes 2}} + 3 \ip{X^{\star} \otimes X^{\star}, K}_{\mathcal{A}^{\otimes 2}}^2.
\end{aligned}
\end{equation}
We can express the risk in terms of $\mathrsfs{B}$
\[
\mathrsfs{R}(t) = 3 \mathrsfs{B}_{11}^2 - 2 \mathrsfs{B}_{11} \mathrsfs{B}_{22} - 4 \mathrsfs{B}_{12} \mathrsfs{B}_{21} + 3 \mathrsfs{B}_{22}^2. 
\]
Therefore the function $h$ is 
\begin{equation*}
\begin{gathered}
h \left ( \begin{bmatrix} B_{11} & B_{12} \\ B_{21} & B_{22} \end{bmatrix} \right ) = 3 B_{11}^2 - 2 B_{11} B_{22} - 4 B_{12} B_{21} + 3 B_{22}^2
\\
\text{and} \quad (\nabla h) (\mathrsfs{B}(t)) = \begin{bmatrix} 
6 \mathrsfs{B}_{11}(t) -2 \mathrsfs{B}_{22}(t) & - 4 \mathrsfs{B}_{21}(t)\\
-4 \mathrsfs{B}_{12}(t) & 6 \mathrsfs{B}_{22}(t) - 2 \mathrsfs{B}_{11}(t)
\end{bmatrix}.
\end{gathered}
\end{equation*}
The $\Phi(t)$ from the ODE is thus
\begin{equation*}
    \begin{aligned}
        \Phi_{11}(t,z) 
        &
        = 
        \exp \left (\int_0^t 2\gamma z [ 6 \mathrsfs{B}_{11}(s) - 2 \mathrsfs{B}_{22}(s) ]\, \dif s \right )\\
        \Phi_{21}(t,z) 
        &
        =
        -8\gamma z \int_0^t \exp \left (\int_0^s 2\gamma z [ 6 \mathrsfs{B}_{11}(s') - 2 \mathrsfs{B}_{22}(s') ]\, \dif s' \right ) \mathrsfs{B}_{12}(s) \, \dif s
    \end{aligned}
\end{equation*}

We also need to find the function $f$. For this, we see that 
\begin{equation*}
    \begin{aligned}
        f(x) 
        = 
        (x^2 - \ip{ X^{\star}, a}_{\mathcal{A}}^2 )^2 
        \quad \text{and} \quad 
        \nabla_x f(x) 
        =
        4 x (x^2 - \ip{ X^{\star}, a}_{\mathcal{A}}^2 ).
    \end{aligned}
\end{equation*}
It follows by another application of Wick's formula:
\begin{align*}
    \EE_a[ \nabla_x f( \ip{W,a}_{\mathcal{A}} )^{\otimes 2}] = I(\mathrsfs{B}(t)) &  = 16 \big ( 15 (\mathrsfs{B}_{11}(t))^3 - 6 (\mathrsfs{B}_{11}(t))^2 \mathrsfs{B}_{22}(t) - 24 \mathrsfs{B}_{11}(t) (\mathrsfs{B}_{12}(t))^2\\
    & \qquad + 3 \mathrsfs{B}_{11}(t) (\mathrsfs{B}_{22}(t))^2 + 12 \mathrsfs{B}_{22}(t) (\mathrsfs{B}_{12}(t))^2 \big ).
\end{align*}
Plugging this into the \eqref{eq:C_scalar} gives you an implicit formula for the dynamics of $\mathrsfs{B}(t)$. 

\subsection{Example 3: (Real) Phase Retrieval, Lipschitz version}
As in the previous example, we are trying to recover an unknown signal $X^{\star}$ from linear observations of the modulus of the signal. The target function, which we assume is noiseless, follows $y = | \ip{X^{\star}, a}_{\mathcal{A}}|$ where $y \in \mathbb{R}$. Another popular formulation for the (noiseless) phase retrieval problem is the non-smooth, Lipschitz version 

\begin{equation} \label{eq:phase_retrieval_nonsmooth_loss_1}
\mathcal{R}(X) \defas  \tfrac{1}{2} \EE_a[ \big ( |\ip{ X, a}_{\mathcal{A}}| - |\ip{X^{\star},a}_{\mathcal{A}}| \big )^2 ]. 
\end{equation}

As before, to apply our result, we need to identify the function $h$ and $f$. Let's first compute the function $h$ in terms of the tensor
\[
B = \ip{W \otimes W, K}_{\mathcal{A}^{\otimes 2}} = \begin{pmatrix} X^T K X & X^T K X^{\star} \\
(X^{\star})^T K X & (X^{\star})^T K X^{\star}
\end{pmatrix} = \begin{pmatrix} B_{11} & B_{12}\\
B_{21} & B_{22} \end{pmatrix}.
\]
For this, we expand the population risk \eqref{eq:phase_retrieval_nonsmooth_loss}
\begin{align*}
\mathcal{R}(X) 
&
= 
\tfrac{1}{2} \EE_a[ \ip{X,a}^2_{\mathcal{A}} ] + \tfrac{1}{2} \EE_a [ \ip{ X^{\star},a}_{\mathcal{A}}^2 ] - \EE_a [ |\ip{X,a}_{\mathcal{A}}| |\ip{X^{\star},a}_{\mathcal{A}}| ]
\\
&
=
\tfrac{1}{2} B_{11} + \tfrac{1}{2} B_{22} - \tfrac{1}{2} \EE_a [ |\ip{X,a}_{\mathcal{A}}| |\ip{X^{\star},a}_{\mathcal{A}} | ] - \tfrac{1}{2} \EE_a [ |\ip{X,a}_{\mathcal{A}}| |\ip{X^{\star},a}_{\mathcal{A}} | ] .
\end{align*}
To compute the last term, we use a result from \cite[Table 1]{louart2018random}, 
\begin{align*}
\EE_a [ |\ip{X,a}_{\mathcal{A}}| | \ip{ X^{\star}, a}_{\mathcal{A}}| ] = 
\tfrac{2}{\pi} \sqrt{B_{11}} \sqrt{B_{22}} \left ( \tfrac{B_{12}}{\sqrt{B_{11}} \sqrt{B_{22}} } \arcsin \left (  \tfrac{B_{12}}{\sqrt{B_{11}} \sqrt{B_{22}} } \right ) + \sqrt{ 1 - \left ( \tfrac{B_{12}}{\sqrt{B_{11}} \sqrt{B_{22}} } \right )^2} \right ).
\end{align*}
Therefore, we have
\begin{align*}
    \mathcal{R}(X) = h \left ( \begin{pmatrix} B_{11} & B_{12}\\
    B_{21} & B_{22}
    \end{pmatrix} \right ) 
    &
    = 
    \tfrac{1}{2} B_{11} + \tfrac{1}{2} B_{22} 
    \\
    &
    -
    \tfrac{1}{\pi} \sqrt{B_{11}} \sqrt{B_{22}} \left ( \tfrac{B_{12}}{\sqrt{B_{11}} \sqrt{B_{22}} } \arcsin \left (  \tfrac{B_{12}}{\sqrt{B_{11}} \sqrt{B_{22}} } \right ) + \sqrt{ 1 - \left ( \tfrac{B_{12}}{\sqrt{B_{11}} \sqrt{B_{22}} } \right )^2} \right )
    \\
    &
    -
    \tfrac{1}{\pi} \sqrt{B_{11}} \sqrt{B_{22}} \left ( \tfrac{B_{21}}{\sqrt{B_{11}} \sqrt{B_{22}} } \arcsin \left (  \tfrac{B_{21}}{\sqrt{B_{11}} \sqrt{B_{22}} } \right ) + \sqrt{ 1 - \left ( \tfrac{B_{21}}{\sqrt{B_{11}} \sqrt{B_{22}} } \right )^2} \right ).
\end{align*}
Taking the derivative, we get that
\begin{align*}
(\nabla h)(B) = \begin{bmatrix} 
\frac{1}{2} - \frac{1}{\pi} \sqrt{\frac{B_{22}}{B_{11}} - \frac{B_{12}^2}{ B_{11}^2 } } 
& 
-\frac{1}{\pi} \arcsin \left ( \frac{B_{21}}{ \sqrt{B_{11}} \sqrt{B_{22}} } \right )
\\
-\frac{1}{\pi} \arcsin \left ( \frac{B_{12}}{ \sqrt{B_{11}} \sqrt{B_{22}} } \right ) & *
\end{bmatrix}.
\end{align*}

Next, we consider the function $f$ and its gradient $\nabla f$. It is clear from \eqref{eq:phase_retrieval_nonsmooth_loss_1} that 
\begin{align*}
f(x) 
= 
\tfrac{1}{2} (|x| - |\ip{X^{\star}, a}_{\mathcal{A}}| )^2
\quad \text{and} \quad
\nabla_x f(x) 
= 
x - \text{sign}(r) |\ip{X^{\star}, a}_{\mathcal{A}}|,
\end{align*}
where $\text{sign} \, : \, \mathbb{R} \to \mathbb{R}$ is the sign function. In particular, we need to compute $\EE_a [ \nabla_xf( \ip{X,a}_{\mathcal{A}} )^{\otimes 2} ]$. A simple computation shows that 
\[
\EE_a [\nabla_x f(\ip{X,a}_{\mathcal{A}})^{\otimes 2} ] = 2\mathcal{R}(X). 
\]

\subsubsection{Vector field computations}
In this section, we work with identity covariance, and we are interested in understanding the dynamics of the norm and cross term, that is,
\[
B_{11} = X^TX \quad \text{and} \quad B_{12} = X^TX^{\star}.
\]
% \begin{equation}
% \min_{X \in \mathbb{R}^d}\bigg \{\mathcal{R}(X) =  \tfrac{1}{2} \EE_a \big [ \big ( |\ip{a,X}_{\mathcal{A}}| - |\ip{a,X^{\star}}_{\mathcal{A}}| \big )^2\big ]\bigg\}. 
% \end{equation}
First, let us define the following variables consistent with the notation for the Volterra equation
\begin{align*}
B_{11} \defas X^T X, \quad 
B_{12} \defas X^T X^{\star}, \quad B_{21} \defas (X^{\star})^T X, \quad \text{and} \quad B_{22} \defas (X^{\star})^T (X^{\star}).
\end{align*}
Note in the scalar case $B_{12} = B_{21}$, but for purposes of making a unifying theory with the matrix case, we think of these two as independent variables. We can express $\mathcal{R}(X) = h( B_{11}, B_{12}, B_{21}, B_{22})$ where $h$ is some function of the variables $B_{11}, B_{12}, B_{21}, B_{22}$ and, in particular,
\begin{align*}
\mathcal{R}(X) 
&
= 
\frac{1}{2} B_{11} + \frac{1}{2} B_{22} - \frac{2}{\pi} \left [ B_{12} \sin^{-1} \left ( \frac{B_{12}}{\sqrt{B_{11} B_{22}}} \right ) + \sqrt{B_{11} B_{22}} \sqrt{ 1 - \frac{B_{12}^2}{B_{11} B_{22}} } \right ].
\end{align*}
Using chain rule, we have that 
\begin{align*}
\nabla \mathcal{R}(X) 
&
= 2X (\partial_{B_{11}} h) + 2X^{\star} (\partial_{B_{12}} h), \\
\text{where} \quad \nabla h 
& = \begin{bmatrix}
\frac{\partial h}{\partial B_{11}} & \frac{\partial h}{\partial B_{12}}\\
\frac{\partial h}{\partial B_{21}} & \frac{\partial h}{\partial B_{22}}
\end{bmatrix}
= 
\begin{bmatrix}
\frac{1}{2}-\frac{1}{\pi} \sqrt{\frac{B_{22}}{B_{11}}} \sqrt{1 - \frac{B_{12}^2}{B_{11} B_{22}} } & -\frac{1}{\pi} \sin^{-1} \left ( \frac{B_{12}}{\sqrt{B_{11} B_{22}}} \right )\\
-\frac{1}{\pi} \sin^{-1} \left ( \frac{B_{21}}{\sqrt{B_{11} B_{22}}} \right ) & \star
\end{bmatrix}.
\end{align*}
Therefore, the gradient of $\mathcal{R}$ is 
\begin{align*}
\nabla \mathcal{R}(X) = 2 X \left ( \frac{1}{2}-\frac{1}{\pi} \sqrt{\frac{B_{22}}{B_{11}}} \sqrt{1 - \frac{B_{12}^2}{B_{11} B_{22}} }\right ) - 2 X^{\star} \left ( \frac{1}{\pi} \sin^{-1} \left ( \frac{B_{12}}{\sqrt{B_{11} B_{22}}} \right ) \right ).
\end{align*}
Now we compute via Ito's the derivative of the norm
\begin{align*}
\dif B_{11} 
& 
=
2 \ip{X_t, \dif X_t} + \ip{\dif X_t, \dif X_t} 
=
-2 \gamma \ip{X_t, \nabla \mathcal{R}(X_t)} \dif t + 2 \gamma^2 \mathcal{R}(X_t) \, \dif t
\\
&
=
-4 \gamma B_{11} \left ( \frac{1}{2}-\frac{1}{\pi} \sqrt{\frac{B_{22}}{B_{11}}} \sqrt{1 - \frac{B_{12}^2}{B_{11} B_{22}} }\right ) + 4 \gamma B_{12}  \left ( \frac{1}{\pi} \sin^{-1} \left ( \frac{B_{12}}{\sqrt{B_{11} B_{22}}} \right ) \right )
\\
&
\quad 
+ \gamma^2 \left ( B_{11} + B_{22} - \frac{4}{\pi} \left [ B_{12} \sin^{-1} \left ( \frac{B_{12}}{\sqrt{B_{11} B_{22}}} \right ) + \sqrt{B_{11} B_{22}} \sqrt{ 1 - \frac{B_{12}^2}{B_{11} B_{22}} } \right ] \right ).
\end{align*}
A similar Ito computation gives the overlap term
\begin{align*}
\dif B_{12} 
&
= 
\ip{X^{\star}, \dif X_t} = -\gamma \ip{X^{\star}, \nabla \mathcal{R}(X_t)} \, \dif t\\
&
= 
-2 \gamma B_{12} \left ( \frac{1}{2}-\frac{1}{\pi} \sqrt{\frac{B_{22}}{B_{11}}} \sqrt{1 - \frac{B_{12}^2}{B_{11} B_{22}} }\right ) + 2 \gamma B_{22}  \left ( \frac{1}{\pi} \sin^{-1} \left ( \frac{B_{12}}{\sqrt{B_{11} B_{22}}} \right ) \right ).
\end{align*}

\subsection{Example 4: Binary logistic regression.}
In this setting, we consider a binary logistic regression problem where we are trying to classify two classes. We will follow a Student-Teacher model: let $X^{\star} = X^{\star} \oplus 0$ and generated targets $y$ by 
\begin{equation} \label{eq:logistic_student_teacher}
    y = \frac{\exp( \ip{X^{\star} \oplus 0, a}_{\mathcal{A}} )}{\tr( \exp (\ip{X^{\star} \oplus 0, a \otimes 1}_{\mathcal{A}} ) ) } = \frac{\exp( \ip{X^{\star}, a}_{\mathcal{A}}) \oplus 1}{\exp(\ip{X^{\star}, a}_{\mathcal{A}}) + 1}.
\end{equation}
The classification problem for $X= X \oplus 0$ is
\begin{equation} \label{eq:logistic_bineary_expected_loss}
    \min_{X} \EE_a \bigg [ - \ip{X,a}_{\mathcal{A}} \cdot \frac{\exp( \ip{X^{\star}, a}_{\mathcal{A}})}{\exp(\ip{X^{\star}, a}_{\mathcal{A}}) + 1} + \log \left ( \exp( \ip{X,a}_{\mathcal{A}} ) + 1 \right ) \bigg ]. 
\end{equation}
We begin by computing the function $h$, which is defined by via the risk as $\mathcal{R}(X) 
= h(\ip{W \otimes W, K}_{\mathcal{A}^{\otimes 2}})$. In this case, the function $\mathcal{R}(X)$, \eqref{eq:logistic_bineary_expected_loss}, consists of two terms. Following the notation in Section~\ref{sec:Volterra_equation}, we will think of $h$ as a function of $B$ where 
\begin{equation} \label{eq:C_log}
B = \begin{bmatrix} B_{11} & B_{12} \\
B_{21} & B_{22} 
\end{bmatrix}  \cong \begin{bmatrix} 
\ip{X \otimes X, K}_{\mathcal{A}^{\otimes 2}} & \ip{X \otimes X^{\star}, K}_{\mathcal{A}^{\otimes 2}}\\
\ip{X^{\star} \otimes X, K}_{\mathcal{A}^{\otimes 2}} & \ip{X^{\star} \otimes X^{\star}, K}_{\mathcal{A}^{\otimes 2}}
\end{bmatrix}.
\end{equation}
We will start, with the slightly easier term manage:
$h_2(B) \defas \EE_{a}[ \log ( \exp( \ip{X,a}_{\mathcal{A}} ) + 1)$. To isolate $h_2$,  by letting $z = \ip{X,a}_{\mathcal{A}} \sim N(0, \ip{X \otimes X, K}_{\mathcal{A}^{\otimes 2} })$, we see that
\begin{equation}
\begin{aligned}
    h_2( B) 
    &
    = 
    \EE_a \big [ \log ( \exp(\ip{X,a}_{\mathcal{A}} ) + 1)  \big ]
    = 
    \EE_z \big [ \log (\exp(z)  + 1 ) \big ]
    \\
    &
    =
    \EE_w \big [ \log (\exp (\sqrt{X^T K X} w) + 1) \big ]
\end{aligned}
\end{equation}
where $w$ is standard normal $N(0, 1)$. From this, the function 
\begin{equation} \label{eq:h_2}
h_2(B) = \EE_w \big [\log (\exp(w\sqrt{B_{11}})  + 1 ) \big ], \quad w \sim N(0,1).
\end{equation}

Now let us consider the other term in \eqref{eq:logistic_bineary_expected_loss}, that is, the function, $h_1(B) \defas \EE_{a} [ - \ip{X,a}_{\mathcal{A}} \cdot \frac{\exp(\ip{X^{\star},a}_{\mathcal{A}}}{\exp(\ip{X^{\star}, a}_{\mathcal{A}}) + 1} ]$ and let us identify the inputs of $B$. First, we observe that $r = \ip{X,a}_{\mathcal{A}}$ and $r^{\star} = \ip{X^{\star},a}_{\mathcal{A}}$ are jointly Gaussian with $r^{\star} \sim N(0, \ip{X^{\star} \otimes X^{\star}, K}_{\mathcal{A}^{\otimes 2}})$ and $r \sim N(0, \ip{X \otimes X, K}_{\mathcal{A}^{\otimes 2}})$. Under this identification, we can express $h_1(B)$ as 
\[
h_1(B) = \EE_a \bigg [ - \ip{X,a}_{\mathcal{A}} \cdot \frac{\exp(\ip{X^{\star},a}_{\mathcal{A}}) }{\exp(\ip{X^{\star}, a}_{\mathcal{A}}) + 1} \bigg ] = \EE_{(r,r^{\star})} \bigg [ -r \cdot  \frac{\exp(r^{\star})}{\exp(r^{\star}) + 1}\bigg ].
\]
We can express $r^{\star} = \lambda r + U$ where $U$ is normally distributed (mean $0$) and independent of $r$ and the constant $\lambda$ is chosen so that $\EE[ r^{\star} \cdot r] = \lambda \EE[r^2]$. In particular, by noting that $\EE[r^{\star} \cdot r] = \EE[ X^{\star} a a^T X] = \ip{X \otimes X^{\star}, K}_{\mathcal{A}^{\otimes 2}} = B_{21}$ and $\EE[r^2] = \EE[ X^T a a^T X] = B_{11}$, it follows that the constant $\lambda = \frac{B_{21}}{B_{11}}$. Using this identity, we have that 
\begin{align*}
    \EE_{(r,r^{\star})} \bigg [ -r \cdot \frac{\exp(r^{\star})}{\exp(r^{\star}) + 1} \bigg] 
    &
    =
    \EE_{(r,U)} \bigg [ - r \cdot \frac{\exp(\lambda r + U)}{\exp(\lambda r + U) + 1} \bigg ]
    \\
    &
    =
    -\ip{X \otimes X, K}_{\mathcal{A}^{\otimes 2}} \EE_{(r,U)} \bigg [ \partial_r \left( \frac{\exp(\lambda r + U)}{\exp(\lambda r + U)+1 } \right ) \bigg ]
    \\
    &
    = - \lambda \cdot B_{11} \cdot
    \EE_{(r,U)} \bigg [ \frac{\exp(\lambda r + U)}{(1 + \exp(\lambda r + U))^2} \bigg ]
    \\
    &
    = - \lambda \cdot B_{11} \cdot
    \EE_{r^{\star}} \bigg [ \frac{\exp(r^{\star})}{(1 + \exp(r^{\star}))^2} \bigg ].
\end{align*}
Here the 2nd equality is a direct result of Stein's Lemma. Using that $\lambda = \frac{B_{21}}{B_{11}}$, and by letting $r^{\star} = \sqrt{\ip{X^{\star} \otimes X^{\star}, K}_{\mathcal{A}^{\otimes 2}}} \cdot z = \sqrt{B_{22}} \cdot z$ where $z \sim N(0,1)$, we have
\begin{align*}
    h_1(B) 
    &
    = 
    \EE_a \bigg [ - \ip{X,a}_{\mathcal{A}} \cdot \frac{\exp(\ip{X^{\star},a}_{\mathcal{A}}) }{\exp(\ip{X^{\star}, a}_{\mathcal{A}}) + 1} \bigg ] 
    = -\lambda \cdot B_{11} \cdot \EE_{r^{\star}} \bigg [ \frac{\exp(r^{\star})}{(1 + \exp(r^{\star}))^2} \bigg ]
    \\
    &
    = - B_{21} \cdot \EE_z \bigg [\frac{\exp(\sqrt{B_{22}} \cdot z)}{(1 + \exp(\sqrt{B_{22}} \cdot z) )^2} \bigg ], \quad \text{where $z \sim N(0,1)$.}
\end{align*}
Putting this together, we have that 
\begin{equation}
    \begin{aligned}
    h(B) 
    &
    =
    \EE_a \bigg [ - \ip{X,a}_{\mathcal{A}} \cdot \frac{\exp( \ip{X^{\star}, a}_{\mathcal{A}})}{\exp(\ip{X^{\star}, a}_{\mathcal{A}}) + 1} + \log \left ( \exp( \ip{X,a}_{\mathcal{A}} ) + 1 \right ) \bigg ]
    \\
    &
    = 
    h_1(B) + h_2(B)
    \\
    &
    = 
    -B_{21}  \EE_z \bigg [\frac{\exp(\sqrt{B_{22}} \cdot z)}{(1 + \exp(\sqrt{B_{22}} \cdot z) )^2} \bigg ] + \EE_w \big [\log (\exp(w\sqrt{B_{11}})  + 1 ) \big ],
    \end{aligned}
\end{equation}
where $z,w \sim N(0,1)$.

Furthermore, to use our expression in \eqref{eq:resolvent_form}, we need to compute the derivative of $h$, $\nabla h$, with respect to $B$. This is a little tricky because we are needed to use the ``symmetric" version of this derivative, that is, it must respect $\tfrac{\partial h}{\partial B_{12}} = \tfrac{\partial h}{\partial B_{21}}$. We will need a different representation for the function $h_1$ in order to do this. First, we begin with the easier of the two derivatives, that is, $\nabla h_2(B)$:
\begin{equation}
    \nabla h_2(B) = \begin{bmatrix} 
    \frac{1}{2 \sqrt{B_{11}}} \EE_{w} \bigg [ \frac{w \exp(\sqrt{B_{11}} w)}{ 1 + \exp( \sqrt{B_{11}} w)} \bigg ] & 0 \\
    0 & 0
    \end{bmatrix}, \quad \text{where $w \sim N(0,1)$.}
\end{equation}

For $h_1(B)$, we use a different representation, that is, using a multi-variate normal distribution, we have that 
\begin{equation}
\begin{aligned}
    h_1(B)
    &
    = 
    \EE_a \bigg [ - \ip{X,a}_{\mathcal{A}} \cdot \frac{\exp(\ip{X^{\star},a}_{\mathcal{A}}) }{\exp(\ip{X^{\star}, a}_{\mathcal{A}}) + 1} \bigg ] 
    \\
    &
    = 
   \frac{ 1 }{2 \pi \sqrt{ \text{det}(B)}}  \int_{-\infty}^\infty \int_{\infty}^\infty -x  \cdot \frac{\exp(y)}{1 + \exp(y)} \exp \bigg ( -\frac{1}{2} \begin{pmatrix} x \\ y \end{pmatrix}^T B^{-1} \begin{pmatrix} x \\ y \end{pmatrix} \bigg ) \, \dif x \dif y,
\end{aligned}
\end{equation}
where the matrix $B$ is defined as in \eqref{eq:C_log}. With this expression in hand, we can take the derivative with respect to $B_{11}$ and $B_{21}$. A simple computation shows
\begin{equation}
    \begin{aligned}
        & \frac{\partial}{\partial B_{11}} \bigg ( \frac{1}{\sqrt{\text{det}(B)}} \exp \bigg ( -\frac{1}{2} \begin{pmatrix} x \\ y \end{pmatrix}^T B^{-1} \begin{pmatrix} x \\ y \end{pmatrix} \bigg ) \bigg ) 
        \\
        &
        =
       - \frac{1}{2}  \cdot \frac{1}{\sqrt{\text{det}(B)}} \exp \bigg ( -\frac{1}{2} \begin{pmatrix} x \\ y \end{pmatrix}^T B^{-1} \begin{pmatrix} x \\ y \end{pmatrix} \bigg ) \bigg ( 
       \frac{y^2}{\text{det}(B)} - \frac{B_{22}}{\text{det}(B)} \begin{pmatrix} x \\ y \end{pmatrix}^T B^{-1} \begin{pmatrix} x \\ y \end{pmatrix} + \frac{B_{22}}{\text{det}(B)}
       \bigg ),
    \end{aligned}
\end{equation}
and, for the other derivative, 
\begin{equation}
    \begin{aligned}
        & \frac{\partial}{\partial B_{21}} \bigg ( \frac{1}{\sqrt{\text{det}(B)}} \exp \bigg ( -\frac{1}{2} \begin{pmatrix} x \\ y \end{pmatrix}^T B^{-1} \begin{pmatrix} x \\ y \end{pmatrix} \bigg ) \bigg ) 
        \\
        &
        =
       - \frac{1}{2}  \cdot \frac{1}{\sqrt{\text{det}(B)}} \exp \bigg ( -\frac{1}{2} \begin{pmatrix} x \\ y \end{pmatrix}^T B^{-1} \begin{pmatrix} x \\ y \end{pmatrix} \bigg ) \bigg ( 
       -\frac{xy}{\text{det}(B)} + \frac{B_{12}}{\text{det}(B)} \begin{pmatrix} x \\ y \end{pmatrix}^T B^{-1} \begin{pmatrix} x \\ y \end{pmatrix} + \frac{B_{12}}{\text{det}(B)}
       \bigg ).
    \end{aligned}
\end{equation}
Using the Cholesky decomposition on $B$, we now express the $\Dif h(B)$ 
\begin{equation}
    \begin{aligned}
        &\frac{\partial (h_1 + h_2)}{\partial B_{11}}  
        =  \frac{1}{2 \sqrt{B_{11}}} \EE_{w} \bigg [ \frac{w \exp(\sqrt{B_{11}} w)}{ 1 + \exp( \sqrt{B_{11}} w)} \bigg ]
        \\
        &
        +\frac{1}{2\pi} \int_{\mathbb{R}^2} x \cdot \frac{\exp(y)}{1+ \exp(y)} \cdot \exp \bigg ( - \begin{pmatrix} u \\ v \end{pmatrix}^T \begin{pmatrix} u \\ v \end{pmatrix} \bigg ) \bigg ( 
       \frac{y^2}{\text{det}(B)} - \frac{2B_{22}}{\text{det}(B)} \begin{pmatrix} u \\ v \end{pmatrix}^T \begin{pmatrix} u \\ v \end{pmatrix} + \frac{B_{22}}{\text{det}(B)}
       \bigg ) \dif u \dif v,
    \end{aligned}
\end{equation}
and, for the other term,
\begin{equation}
    \begin{aligned}
        &\frac{\partial (h_1 + h_2)}{\partial B_{21}} \\
        &
        = 
        \frac{1}{2\pi} \int_{\mathbb{R}^2} x \cdot \frac{\exp(y)}{1+ \exp(y)} \cdot \exp \bigg ( - \begin{pmatrix} u \\ v \end{pmatrix}^T \begin{pmatrix} u \\ v \end{pmatrix} \bigg ) \bigg ( 
       \frac{-xy}{\text{det}(B)} - \frac{2B_{12}}{\text{det}(B)} \begin{pmatrix} u \\ v \end{pmatrix}^T \begin{pmatrix} u \\ v \end{pmatrix} + \frac{B_{12}}{\text{det}(B)}
       \bigg ) \dif u \dif v,
    \end{aligned}
\end{equation}
where we have 
\begin{equation}
    \begin{pmatrix} x \\ y \end{pmatrix} = \sqrt{2} L \begin{pmatrix} u \\ v \end{pmatrix} \quad \text{and} \quad B = L L^T. 
\end{equation}

Lastly, the function $f \, : \, \mathcal{O} \to \mathbb{R}$ is 
\[
f(x) = -x \cdot \frac{\exp( \ip{X^{\star}, a}_{\mathcal{A}})}{\exp(\ip{X^{\star}, a}_{\mathcal{A}}) + 1} + \log( \exp(x) + 1 ). 
\]
The derivative of $f$ is
\begin{equation}
    \begin{aligned}
        \nabla_x f(\ip{X,a}_{\mathcal{A}}) = - \frac{\exp( \ip{X^{\star},a}_{\mathcal{A}})}{\exp(\ip{X^{\star}, a}_{\mathcal{A}}) + 1} + \frac{\exp(\ip{X,a}_{\mathcal{A}})}{\exp(\ip{X,a}_{\mathcal{A}}) + 1}.
    \end{aligned}
\end{equation}
Therefore, we deduce with $g(x) \defas \frac{\exp(x)}{1 + \exp(y)}$
\begin{equation}
    \begin{aligned}
        \EE_a[ \nabla f(\ip{X,a}_{\mathcal{A}})^{\otimes 2}]
        &
        = 
        \frac{1}{2 \pi \sqrt{\text{det}(B)} }
        \int_{\mathbb{R}^2} (g(x)-g(y))^2 \exp \bigg ( -\frac{1}{2} \begin{pmatrix} x \\ y \end{pmatrix}^T B^{-1} \begin{pmatrix} x \\ y \end{pmatrix} \bigg )  \, \dif x \dif y.
    \end{aligned}
\end{equation}
This can also be reduced by doing a Cholesky decomposition on $B = LL^T$ and then using a transformation $\begin{pmatrix} x \\ y \end{pmatrix} = \sqrt{2} L \begin{pmatrix} u \\ v \end{pmatrix}$.

\subsubsection{SGD dynamics on the landscape of logistic regression} 
We focus on binary logistic regression, particularly the behavior near the optimum. In this section, we examine the dynamics of SGD as it evolves. We focus on the trajectories of the cross term, $X^T K X^{\star}$, and the norm $X^T K X$, as it changes from updates of SGD. First, under the student-teacher setup described in \eqref{eq:logistic_student_teacher}, we have a unique solution to the loss \eqref{eq:logistic_bineary_expected_loss}. 

\begin{proposition}[Unique minimizer of logistic loss] Suppose we consider the student-teacher set-up for binary logistic regression described in \eqref{eq:logistic_student_teacher} for the loss \eqref{eq:logistic_bineary_expected_loss}. Let $K = \mathbb{E}_a[a a^T]$ be positive-definite, i.e. non-degenerate covariance. Then there exists a unique minimizer of \eqref{eq:logistic_bineary_expected_loss}, $\tilde{X} \in \mathcal{A} \otimes \mathcal{O}$ such that $\tilde{X} = X^{\star}$.
\end{proposition}

\begin{proof} Using the definition of the logistic regression risk, we have that
\begin{equation}
\nabla \mathcal{R}(X) = \mathbb{E}_a \left [ - \frac{\exp(\ip{X^{\star},a }_{\mathcal{A}})}{1+ \exp(\ip{X^{\star}, a}_{\mathcal{A}} )} \cdot a + \frac{\exp(\ip{X,a }_{\mathcal{A}})}{1+ \exp(\ip{X, a}_{\mathcal{A}} )} \cdot a
\right ]. 
\end{equation}
Let $g(r) = \frac{\exp(r)}{1 + \exp(r)}$. Since $a \sim N(0,K)$, by setting $a = \sqrt{K} z$ for $z \sim N(0, I_d)$, we get that 
\begin{equation}
    \nabla \mathcal{R}(X) = \sqrt{K} \EE_z \left [ \bigg (g(\ip{\sqrt{K} X, z}_{\mathcal{A}} ) - g(\ip{\sqrt{K} X^{\star}, z}_{\mathcal{A}}) \bigg ) z \right ].
\end{equation}
By applying Stein's lemma, we then deduce that
\begin{align*}
    \nabla \mathcal{R}(X) 
    &
    = 
    \sqrt{K} \EE_z \bigg [ \bigg ( g'(\ip{\sqrt{K}X, z}_{\mathcal{A}}) \cdot \sqrt{K} X -g'(\ip{\sqrt{K}X^{\star}, z}_{\mathcal{A}}) \cdot \sqrt{K} X^{\star} \bigg ) \bigg ]
    \\
    &
    =
    K \EE_z \bigg [ g'(\ip{\sqrt{K}X, z}_{\mathcal{A}}) \cdot X - g'(\ip{\sqrt{K}X^{\star}, z}_{\mathcal{A}}) \cdot X^{\star} \bigg ].
\end{align*}
It is clear that when $X = X^{\star}$, $\nabla \mathcal{R}(X) = 0$ and thus $X^{\star}$ is a global minimizer of $\mathcal{R}$ (logistic regression is convex). Now we consider cases. \\

\noindent \textit{Case 1}: Suppose $X$ is not parallel to $X^{\star}$, i.e., $X \neq c X^{\star}$ for any $c \in \mathbb{R}$. Then we see that $(\Dif \mathcal{R})(X) = 0$ if and only if
\begin{equation} \label{eq:logistic_minimizer_1}
0 = \EE_z [ g'(\ip{\sqrt{K}X, z}_{\mathcal{A}}) ] = \EE_z [ g'(\ip{\sqrt{K} X^{\star},z }_{\mathcal{A}})]. 
\end{equation}
Note we used explicitly that the covariance $K$ is non-degenerate. A simple computation shows that $g'(r) > 0$ and thus \eqref{eq:logistic_minimizer_1} can never occur. 

Next, we consider when $X^{\star} = 0$. By Case 1, we know that $X^{\star} = X$. Therefore we can exclude this case so for the following cases $X^{\star} \neq 0$.\\

\noindent \textit{Case 2}: Suppose $X = -c X^{\star}$ where $c \ge 0$ and $X^{\star} \neq 0$. Then we have that 
 \[
 \nabla \mathcal{R}(X) = - K X^{\star} \cdot \EE_z \big [ c g'(-\ip{\sqrt{K}X^{\star},z}_{\mathcal{A}}) + g'( \ip{\sqrt{K} X^{\star}, z}_{\mathcal{A}} )\big ].
 \]
 Since $g'(r) > 0$, then $\EE_z \big [ c g'(-\ip{\sqrt{K}X^{\star},z}_{\mathcal{A}}) + g'( \ip{\sqrt{K} X^{\star}, z}_{\mathcal{A}} )\big ] > 0$ and hence $(\Dif \mathcal{R})(X) \neq 0$. \\

\noindent \textit{Case 3}: Suppose $X = c X^{\star}$ where $c > 0$, $c \neq 1$, and $X^{\star} \neq 0$. We have $\nabla \mathcal{R}(X) = 0$ implied that 
$\EE_z [ cg'(c\ip{\sqrt{K} X^{\star}, z}_{\mathcal{A}})] =  \EE_z [ g'(\ip{\sqrt{K} X^{\star}, z}_{\mathcal{A}})]$. Let $y = \ip{z, \sqrt{K}X^{\star}}_{\mathcal{A}}$. Then $y \sim N(0, \sigma^2)$ for some $\sigma > 0$ and, thus, we can write $y = \sigma w$ for $w \sim N(0, 1)$. Consequently, $\nabla \mathcal{R}(X) = 0$ implies that $\EE_w [cg'(\sigma c w)] = \EE_w[g'(\sigma w)]$. 

By Stein's Lemma, 
\begin{align*}
    \EE_w [ c g'(c \sigma w) ] 
    &
    = 
    \tfrac{1}{\sigma} \EE_w [ g( \sigma c w) w]
    \\
    &
    =
    \frac{1}{\sigma \sqrt{2 \pi}} \int_0^{\infty} \frac{\exp(\sigma c w)}{1 + \exp(\sigma c w)} w e^{-w^2/2} \dif w- \frac{1}{\sigma \sqrt{2\pi}} \int_0^{\infty} \frac{\exp(-\sigma cw)}{1 + \exp(-\sigma cw) } w e^{-w^2/2} \, \dif w.
\end{align*}
Note that $c \mapsto \exp(\sigma c w) / (1 + \exp(\sigma c w))$ is strictly increasing and $c \mapsto \exp(-\sigma c w) / (1 + \exp(-\sigma c w))$ is strictly decreasing in $c$ when $\sigma w > 0$. Consequently, $\EE_w[c g'(c\sigma w)] = \tfrac{1}{\sigma} \EE_w \big [ \frac{\exp(c \sigma w)}{1 + \exp( c \sigma w)} \big ]$ is a strictly increasing function of $c$. 

Since at $c = 1$, $\EE_w [c g'(c \sigma w)] = \EE_w[g'(\sigma w)]$, and $c \mapsto \EE_w [c g'(c \sigma w)]$ is strictly increasing, we have that $\EE_w[c g'(c \sigma w)] \neq \EE_w[g'(\sigma w)]$ for any $c \neq 1$. The result then immediately follows. 
\end{proof} 

\subsection{Example 5: Simple, 2-layer Neural Networks with Activation Functions}

In this setting, we consider a simple 2-layer neural network whose output layer is a single node and the loss is the mean-squared error
\begin{equation}
\mathcal{R}(X) \defas \tfrac{1}{2} \EE_{(a,y)} [ \big ( \sigma( \ip{a, X}_{\mathcal{A}} ) - y \big )^2 ] = \tfrac{1}{2} \EE_{a} [ \big ( \sigma(\ip{a, X}_{\mathcal{A}}) - \sigma(\ip{a, X^{\star}}_{\mathcal{A}} ) \big )^2 ],
\end{equation}
where the Lipschitz continuous function $\sigma \, : \, \mathbb{R} \to \mathbb{R}$ is an activation function which is applied entry-wise on the vector $\ip{a,X}_{\mathcal{A}}$ and then the entries are added before squaring. 

For this case, the function $f$ and its gradient are
\[
f\, : \, x \mapsto \tfrac{1}{2} \big ( \sigma(x) - \sigma(\ip{ X^{\star},a }_{\mathcal{A}} )  \big )^2 \quad \text{and} \quad \nabla_x f \, : \, x \mapsto \sigma'(x) \big ( \sigma(x) - \sigma( \ip{ X^{\star}, a}_{\mathcal{A}} ) \big ).
\]
In this way, we see that 
\[
\EE_{a} \big [ \nabla f( \ip{X,a}_{\mathcal{A}} )^{\otimes 2}] = \EE_a [2 ( \sigma'(\ip{X,a}_{\mathcal{A}}) )^2 f(\ip{X,a}_{\mathcal{A}}) \big ].
\]
The function $h$, in general, can be quite complicated owing to the activation function $\sigma$. In Table~\ref{table:activation_functions} (see \cite[Table 1]{liao2018dynamics}), we provide some examples of various activation functions written in terms of the matrix $B = \ip{W \otimes W, K}_{\mathcal{A}^{\otimes 2}}$. 

\ctable[notespar,
caption = {{\bfseries $h$ function and its derivatives for different activation functions.} Summary of different activation functions and the corresponding $h$ in terms of $\ip{W \otimes W, K}_{\mathcal{A}^{\otimes 2}} = \begin{pmatrix} B_{11} & B_{12} \\ B_{21} & B_{22} \end{pmatrix}$. Results were taken from Table 1 in \cite{liao2020Random}.},
label = {table:activation_functions},
captionskip=2ex,
pos =!t
]{l l l}{}{
\toprule
$\sigma(r)$ & $h (B)$ \\
\midrule
$r$ & $\tfrac{1}{2} B_{11} + \tfrac{1}{2} B_{22} - \tfrac{1}{2} B_{12} - \tfrac{1}{2} B_{21}$ &\\
\midrule
ReLU, $\max\{r,0\}$ & \begin{minipage}{0.8 \textwidth} $\tfrac{B_{11}}{4} + \tfrac{B_{22}}{4} - \tfrac{1}{4 \pi} \sqrt{B_{11} B_{22}} \left ( \frac{B_{12}}{\sqrt{B_{11} B_{22}}} \cos^{-1} \left ( -\frac{B_{12}}{\sqrt{B_{11} B_{22} }} \right ) + \sqrt{ 1 - \left ( \frac{B_{12}}{\sqrt{B_{11} B_{22}} } \right )^2 }\right )$
\\
$- \tfrac{1}{4 \pi} \sqrt{B_{11} B_{22}} \left ( \frac{B_{21}}{\sqrt{B_{11} B_{22}}} \cos^{-1} \left ( -\frac{B_{21}}{\sqrt{B_{11} B_{22} }} \right ) + \sqrt{ 1 - \left ( \frac{B_{21}}{\sqrt{B_{11} B_{22}} } \right )^2 }\right )$
\end{minipage}
\\
\midrule
$\text{erf}(r)$ & \begin{minipage}{0.8\textwidth} $\frac{1}{\pi} \sin^{-1} \left ( \frac{2B_{11}}{ (1 + 2 B_{11})} \right ) + \frac{1}{\pi} \sin^{-1} \left ( \frac{2 B_{22}}{(1+2 B_{22})} \right ) - \frac{1}{\pi} \sin^{-1} \left ( \frac{2 B_{12}}{\sqrt{(1+2 B_{11})(1+2 B_{22}) }}  \right ) \\
- \frac{1}{\pi} \sin^{-1} \left ( \frac{2 B_{21}}{\sqrt{(1+2 B_{11})(1+2 B_{22}) }}  \right )
$
\end{minipage}
\\
% \midrule
% $1_{\{r > 0\}}$ & $ \frac{1}{4\pi} \cos^{-1} \left ( \frac{C_{12}}{\sqrt{C_{11} C_{22}}} \right ) + \frac{1}{4\pi} \cos^{-1} \left ( \frac{C_{21}}{\sqrt{C_{11} C_{22}}} \right ) $ \\
\midrule
$\text{sign}(r)$ & $1 - \frac{1}{\pi} \sin^{-1} \left ( \frac{B_{12}}{\sqrt{B_{11} B_{22} }} \right )- \frac{1}{\pi} \sin^{-1} \left ( \frac{B_{21}}{\sqrt{B_{11} B_{22} }} \right )$\\
\midrule
$\cos(r)$ & \begin{minipage}{0.8\textwidth} $\frac{1}{2} \big [ \exp(-B_{11}) \cosh(B_{11}) + \exp(-B_{22}) \cosh(B_{22}) - \exp( -\frac{1}{2} (B_{11} + B_{22}) )\cosh(B_{12})$\\
$-\exp( -\frac{1}{2} (B_{11} + B_{22}) )\cosh(B_{21}) \big ]$ \end{minipage} \\
\midrule
$\sin(r)$ & \begin{minipage}{0.8\textwidth} 
$\frac{1}{2} \big [ \exp(-B_{11}) \sinh(B_{11}) + \exp( -B_{22}) \sinh(B_{22}) - \exp( -\frac{1}{2} (B_{11} + B_{22}) ) \sinh(B_{12})   $\\
$ - \exp( -\frac{1}{2} (B_{11} + B_{22}) ) \sinh(B_{21})  \big ]$
\end{minipage}\\
 \bottomrule
}

\subsection{Phase chase problem} 
In this problem, we consider a $X = (X_1, X_2) \in \mathcal{A} \otimes \mathbb{R}^2$ where $X_1, X_2 \in \mathcal{A \otimes \mathbb{R}}$, that is $\mathcal{O} = \mathbb{R}^2$ and we consider the no target setting (i.e., $X^{\star} = 0$). Like the phase retrieval, the phases of $\ip{a,X_1}_{\mathcal{A}}$ and $\ip{a,X_2}_{\mathcal{A}}$ are lost, and we are trying to recover a $X_1$ close to $X_2$. We can formulate this as the optimization problem
\begin{equation} \label{eq:chasing_problem_1}
    \min_{X_1, X_2 \in \mathcal{A} \otimes \mathbb{R}} 
    \bigg \{ \mathcal{R}(X) = \EE_{a}\big [ \big ( (\ip{a,X_1}_{\mathcal{A}})^2 - (\ip{a, X_2}_{\mathcal{A}})^2 \big )^2 \big ] \bigg \}.
\end{equation}

There are many solutions to this problem, all of which satisfy $X_1 = X_2$ or $X_1 = -X_2$, provided $K$ is non-degenerate (in the case of degenerate $K$, you get equality outside the kernel of $K$). Therefore, the dynamics of this problem are such that $X_1$ is \textit{chasing} $X_2$. 

\subsubsection{Dynamics of the $\mathcal{S}$ matrix for phase chase, non-symmetric} To understand these dynamics better and, in particular, the role of SGD noise, we invoke our homogenized SGD theorem. For this, we need the expressions for $h, \nabla h, \nabla_x f,$ and $\EE_a[ \nabla f(r)^{\otimes 2}]$. First, we note the target $X^{\star} = 0$ and thus, $B_{12} = \ip{X \otimes X^{\star}, K}_{\mathcal{A} \otimes 2}$ and $B_{22} = \ip{X^{\star} \otimes X^{\star}, K}_{\mathcal{A}^{\otimes 2}}$ are both identically $0$. This leaves the $B_{11} = \ip{X \otimes X, K}_{\mathcal{A}^{\otimes 2}}$ which is itself a $2 \times 2$ matrix and can be viewed as a norm and cross term with $x_1$ and $x_2$.  

With this in mind, we introduce notation to represent the norm and cross term between $X_1$ and $X_2$, as represented by a symmetric matrix, 
\begin{equation}
\begin{aligned}
B_{11} \defas Q = \begin{pmatrix} Q_{11} & Q_{12}\\
Q_{12} & Q_{22}
\end{pmatrix} = \ip{(X_1 \oplus X_2) \otimes (X_1 \oplus X_2), K}_{\mathcal{A}^{\otimes 2}} = \begin{pmatrix} \|X_1\|_K^2 & X_1^T K X_2\\
X_1^T K X_2 & \|X_2\|_K^2
\end{pmatrix},
\end{aligned}
\end{equation}
where we use the $K$-norm, $\|\cdot\|_K = \ip{\cdot \otimes \cdot, K}_{\mathcal{A}^{\otimes 2}}$. 

Under this notation, we represent the function $h$ and $\Dif h$:
\begin{align*}
    h(Q, B_{12}, B_{22}) 
    &
    = 
    3( Q_{11}^2 + Q_{22}^2 ) - 2 ( Q_{11} Q_{22} ) - 4 Q_{12}^2
    \\
    \nabla h(Q, B_{12}, B_{22}) 
    &
    =
    \begin{pmatrix}
        6 Q_{11} - 2 Q_{22} & -4 Q_{12}\\
        - 4 Q_{21} & 6 Q_{22} - 2Q_{11}
    \end{pmatrix}.
\end{align*}

The expression for the function $f$ is simply
\[
f(x_1, x_2) = ( x_1^2 - x_2^2)^2 \quad \text{and} \quad \nabla_x f(x) = 4 (x_1^2 - x_2^2) \begin{bmatrix} x_1 \\
-x_2 
\end{bmatrix},
\]
where $x_1 = \ip{X_1,a}_{\mathcal{A}}$ and $x_2 = \ip{ X_2,a}_{\mathcal{A}}$. An application of Wick's formula yields that 
\begin{equation}
\begin{gathered}
\EE_a [ \nabla f( \ip{ X,a}_{\mathcal{A}})^{\otimes 2} ] 
= 
16 \begin{bmatrix}
G_{11} & G_{12}\\
G_{12} & G_{22}
\end{bmatrix} \\
\text{where} 
\qquad 
G_{11} = 15 Q_{11}^3 - 6 Q_{11}^2 Q_{22} - 24 Q_{11} Q_{12}^2 + 3 Q_{11} Q_{22}^2 + 12 Q_{12}^2 Q_{22}
\\
G_{12} = - ( 15 Q_{12} Q_{22}^2 + 15 Q_{12} Q_{11}^2 - 18 Q_{11} Q_{12} Q_{22} - 12 Q_{12}^3 )
\\
G_{22} = 15 Q_{22}^3 - 6 Q_{22}^2 Q_{11} - 24 Q_{22} Q_{12}^2 + 3 Q_{22} Q_{11}^2 + 12 Q_{12}^2 Q_{11}.
\end{gathered}
\end{equation}
It is through these quantities that we can derive an expression for $\mathcal{S}$ when applied to homogenized SGD. 
% The expected matrix $\mathcal{S}$ becomes
% \[
% \EE [ \dot{\mathcal{S}}] = -2 z \gamma \left ( \EE \big [ \mathcal{S} \nabla h + \Dif h \mathcal{S} \big ]\right ) + \gamma^2 \left ( \tfrac{1}{d} \tr(K R(z;K)) \right ) G + \text{constants}. 
% \]

Note an important \textit{symmetry} between $Q_{11} = \|X_1\|_K^2$ and $Q_{22} = \|X_2\|_{K}^2$. Provided that at initialization $X_1$ and $X_2$ have the same norm value, the evolution of $Q_{11}$ will be the same as $Q_{22}$. In essence, we can simplify and look at the dynamics of only two quantities $Q_{11}$ and $Q_{12}$ \textit{and} replace $Q_{22}$ with $Q_{11}$ in the expressions.

We will see from homogenized SGD that the evolution of $Q$ has interesting properties. In particular, for SGD, the cross term $Q_{12}$ evolves depending on the stepsize, and thus, the learning rate affects the solution that SGD converges to. This does not occur for gradient flow, and hence gradient descent-- all learning rates go to the same optimum. 

\subsubsection{Dynamics when $K = \Id$} 
When the covariance is identity, the expressions for the dynamics of $Q$ simplify to a system of ODEs
\begin{equation}
\begin{aligned}
\dot{Q_{11}} 
&
=
-16 \gamma ( Q_{11}^2 - Q_{12}^2 ) + 192 \gamma^2 (Q_{11}^2 - Q_{12}^2) Q_{11}
\\
\dot{Q_{12}} 
&
=
- 192 \gamma^2 (Q_{11}^2 -Q_{12}^2) Q_{12}.
\end{aligned}
\end{equation}
In comparison to gradient flow, we have that
\begin{equation}
\begin{aligned}
\dot{Q_{11}} 
&
= 
- 16 \gamma (Q_{11}^2 - Q_{12}^2)
\\
\dot{Q_{12}}
&
=
0. 
\end{aligned}
\end{equation}
In particular, we see that the rate at which $Q_{11}(t)-Q_{12}(t) \to 0$ is slowed down
\[
\dot{(Q_{11}-Q_{12})} = -16 \gamma (Q_{11}^2 - Q_{12}^2) + 192 \gamma^2 (Q_{11}^2 - Q_{12}^2) (Q_{11} + Q_{12}).
\]

We expect for both SGD and gradient flow that $Q_{11} = Q_{12}$ at the optimum, but they go about it differently. As we see, for gradient flow (and hence gradient descent scaled by stepsize), the cross term $Q_{12}$ remains constant. The norm, $Q_{11}$, and the risk $\mathcal{R}$, do change, reflecting that \textit{for all stepsizes} gradient descent finds the optimum for which $Q_{11}(t) = Q_{12}(0)$. 

On the other hand, SGD noise, as illustrated through the $\gamma^2$ terms, does three things:
\begin{enumerate}
    \item SGD noise slows down the rate at which $Q_{11}(t)-Q_{12}(t) \to 0$
    \item The movement in the cross-term, $Q_{12}$, is solely due to the noise in SGD
    \item Since both the cross term and norm move, SGD finds an optimum where the first time $Q_{11}(t) = Q_{12}(t)$. Moreover, because of this, larger learning rates lead to slower movement in $Q_{11} \to Q_{12}$ and faster movement in $Q_{12}$. The result is an optimum, $x^*$, with lower $K$-norm values, that is, $\|X_1^*\|_K$ and $\|X_2^*\|_K$ have smaller values as learning rate $\gamma$ increases. In this sense, SGD is doing some form of implicit $\ell^2$-regularization. 
\end{enumerate}

\bibliographystyle{plainnat}
\bibliography{references}
\end{document}

%% file: macros.tex
\def\N{\mathbb{N}}
\def\R{\mathbb{R}}

\def\dif{\mathop{}\!\mathrm{d}}

\def\EE{{\mathbb E}\,}

\DeclareMathOperator*{\argmin}{{arg\,min}}

\DeclareMathOperator{\tr}{\text{Tr}}
\def\defas{\stackrel{\text{def}}{=}}

\DeclareDocumentCommand{\Prto} {o} {
  \IfNoValueTF {#1}
  {\overset{\Pr}{\longrightarrow}}
  { \xrightarrow[ #1 \to \infty]{\Pr }}
}

\DeclareDocumentCommand{\Asto} {o} {
  \IfNoValueTF {#1}
  {\overset{\text{\rm a.s.}}{\longrightarrow}}
  { \xrightarrow[ #1 \to \infty]{\text{\rm a.s.} }}
}

\DeclareDocumentCommand{\law} {o} {
  \IfNoValueTF {#1}
  {\overset{\text{law}}{=}}
  { \xrightarrow[ #1 \to \infty]{\Pr }}
}

\newcommand{\Id}{I}
\DeclareMathOperator{\Exp}{\mathbb{E}}

\newcommand{\e}{\varepsilon}
\newcommand{\cA}{\mathcal{A}}
\newcommand{\cF}{\mathcal{F}}

\newcommand{\cO}{\mathcal{O}}
\newcommand{\cM}{\mathcal{M}}

\DeclareMathOperator{\grad}{Grad}
\DeclareMathOperator{\hess}{Hess}
\newcommand{\beq}{ \begin{equation} }
\newcommand{\eeq}{ \end{equation} }
\DeclareMathOperator*{\proj}{Proj}

\newcommand{\ip}[1]{\langle {#1} \rangle }
\newcommand{\ipa}[1]{\left\langle {#1} \right\rangle }

\newcommand{\vertiii}[1]{{\left\vert\kern-0.25ex\left\vert\kern-0.25ex\left\vert #1 
    \right\vert\kern-0.25ex\right\vert\kern-0.25ex\right\vert}}

\newcommand{\WHSGD}{ \mathscr{X} }

\newcommand{\HSGD}{ \mathscr{W} }

\def\Dif{\mathop{}\!\mathrm{D}}